\documentclass[11pt]{amsart}
\usepackage{mathrsfs}
\usepackage{amssymb}
\usepackage{graphicx}
\usepackage[T1]{fontenc}
\usepackage[all]{xy}
\usepackage{arydshln}
\usepackage{amscd}
\usepackage{amsmath, amssymb}
\usepackage{amsthm}
\usepackage{amsfonts}
\usepackage[colorlinks,linkcolor=blue,citecolor=blue, pdfstartview=FitH]
{hyperref}
\usepackage{amscd}
\usepackage{easybmat}
\usepackage{mathrsfs}
\usepackage{amsfonts}
\usepackage{color}
\usepackage{pifont}
\usepackage{upgreek}
\usepackage{bm}
\usepackage{hyperref}
\usepackage{shorttoc}
\usepackage{amsmath,amstext,amsthm,a4,amssymb,amscd}
\usepackage[mathscr]{eucal}
\usepackage{mathrsfs}
\usepackage{epsf}
\usepackage{tikz}

\numberwithin{equation}{section}

\def\p{\partial}
\def\o{\overline}
\def\b{\bar}
\def\mb{\mathbb}
\def\mc{\mathcal}
\def\mr{\mathrm}
\def\ms{\mathscr}
\def\mbf{\mathbf}
\def\n{\nabla}
\def\v{\varphi}

\def\wt{\widetilde}
\def\mf{\mathfrak}

\def\op{\operatorname}
\def\ra{\rightarrow}
\def\bc{\mathbb C}
\def\br{\mathbb R}
\def\bz{\mathbb Z}
\def\bs{\boldsymbol}
\def\rro{\bs{\rho}}

\def\cal{\mathcal}

\theoremstyle{plain}
\newtheorem{thm}{Theorem}[section]
\newtheorem{lemma}[thm]{Lemma}
\newtheorem{prop}[thm]{Proposition}
\newtheorem{cor}[thm]{Corollary}
\theoremstyle{definition}
\newtheorem{rem}[thm]{Remark}

\newtheorem{ex}[thm]{Example}
\theoremstyle{definition}
\newtheorem{defn}[thm]{Definition}
\newtheorem{sthm}{Theorem}
\newcommand{\comment}[1]{}

\usepackage{amscd}
\usepackage{fancyhdr}
\makeatletter
\@namedef{subjclassname@2020}{2020 Mathematics Subject Classification}

\begin{document}

\title{Signature for flat unitary
bundles over surfaces with boundary}
\author{InKang Kim}
\author{Pierre Pansu}
\author{Xueyuan Wan}

\address{Inkang Kim: School of Mathematics, KIAS, Heogiro 85, Dongdaemun-gu Seoul, 02455, Republic of Korea}
\email{inkang@kias.re.kr}

\address{Pierre Pansu:
Universit\'e Paris-Saclay, CNRS, Laboratoire de Math\'ematiques d'Orsay\\ 91405 Orsay C\'edex, France}
\email{pierre.pansu@universite-paris-saclay.fr}

\address{Xueyuan Wan: Mathematical Science Research Center, Chongqing University of Technology, Chongqing 400054, China}
\email{xwan@cqut.edu.cn}
\begin{abstract}

This paper deals with the representations of the fundamental groups of compact surfaces with boundary into classical simple Lie groups of Hermitian type. We relate work on the signature of the associated local systems of  Atiyah-Patodi-Singer, to Burger-Iozzi-Wienhard's Toledo invariant. To measure the difference, we extend Atiyah-Patodi-Singer's rho invariant, initially defined on $\mathrm{U}(p)$, to discontinuous class functions, first on $\mathrm{U}(p,q)$, and then on other classical groups via embeddings into $\mathrm{U}(p,q)$. In this way, we present three different invariants -- signature, Toledo  and rho invariant -- in a unifying way, which is a version of the classical signature formula of Atiyah-Patodi-Singer for manifolds with boundary.
\end{abstract}

  \thanks{Research by Inkang Kim is partially supported by Grant NRF-2019R1A2C1083865 and KIAS Individual Grant (MG031408), and Xueyuan Wan is supported by  NSFC (No. 12101093), Scientific Research Foundation of Chongqing University of Technology.}

\maketitle
\tableofcontents

\section*{Introduction}

\subsection{Motivation}

We are concerned with the signature of flat (indefinite) unitary bundles over surfaces with boundary. M. Atiyah, V. Patodi and I. Singer's index theorem provides an expression of the signature involving two terms, an integral and a boundary term. The goal of the present paper is to give a modern interpretation of the integral term, and to compute the boundary term accordingly.

The main input comes from Atiyah-Patodi-Singer's theorem. It takes the form
$$
\op{sign}(M)=\int_M \text{ Hirzebruch L-class} - \eta (\partial M),
$$
where $\eta$ is a spectral invariant of an elliptic self-adjoint differential operator $A$ acting on a Hermitian vector bundle over $\partial M$. For the signature of unitary local systems over surfaces with boundary, $A$ is a first order ordinary differential operator depending in a subtle manner on the local system, not merely on its holonomy. When the holonomy belongs to $\text{SL}(2, \mathbb Z)\subset U(1,1)$, M. Atiyah \cite{Atiyah} has designed a class function on  $\text{SL}(2, \mathbb Z)$ from the eta invariant, recovering a result of H. Rademacher \cite{Ra} in relation to the study of Dedekind's $\bs{\eta}$ function
$$
\bs{\eta}(\tau)=e^{i\pi \tau/12} \Pi_{n=1}^\infty(1- e^{2\pi i n\tau}),
$$
on the upper half plane. Our expression for the boundary term will be a slightly different class function on $U(p,q)$, due to our different interpretation of the integral term.

\medskip

Let us provide more details. Our inspiration comes from the theses of G. Lusztig \cite{Lus} and W. Meyer \cite{Meyer, Meyer1}. On the $1$-cohomology of a flat Hermitian bundle $(\mc E,\Omega)$ over a closed surface $\Sigma$, there is a natural Hermitian form, the intersection form. Its signature $\op{sign}(\mc E,\Omega)$ can be expressed as the index of a first order differential operator, whence, thanks to the index theorem, an expression for $\op{sign}(\mc E,\Omega)$ as a characteristic number of $(\mc E,\Omega)$, twice the first Chern number of a specific complex line bundle $L$,
\begin{align*}
\op{sign}(\mc E,\Omega)=\int_{\Sigma}2c_1(L).
\end{align*}

M. Atiyah \cite{Atiyah} extends the discussion to compact surfaces with nonempty boundary. $c_1(L^2)$ must be thought of as a relative Chern class for a line bundle together with a trivialization $\sigma$ along the boundary depending on the boundary holonomy. He expresses $\sigma$ as a discontinuous section of a central extension of the unitary group $\mathrm{U}(p,q)$, arising from a $2$-cocycle on $\mathrm{U}(p,q)$, the \emph{signature cocycle}. M. Atiyah computes the map $\sigma$ on semi-simple elements of $\mathrm{U}(p,q)$, he highlights but leaves open the calculation on general holonomies. Completing Atiyah's program is one of the main results of the present paper.

\medskip

A second approach, which goes back to W. Goldman's thesis \cite{GoldPhD}, interprets integrals like $\int_{\Sigma}c_1(L)$ as follows. There is a universal homogeneous complex line bundle $L$ on the symmetric space of the indefinite unitary group $\mathrm{U}(p,q)$, the bounded complex domain $\mathrm{D}^{\mathrm{I}}_{p,q}$. A flat Hermitian bundle over $\Sigma$ gives rise to a homomorphism $\phi:\pi_{1}(\Sigma)\to \mathrm{U}(p,q)$. There exist equivariant maps from the universal cover $\tilde\Sigma$ to $\mathrm{D}^{\mathrm{I}}_{p,q}$ and since any two are homotopic, the first Chern class of the pulled-back bundle is thus uniquely defined. The resulting Chern number is known as a \emph{Toledo invariant} \cite{Toledo}. When $\Sigma$ is closed, it takes only finitely many values, controlled by the \emph{Milnor-Wood inequality} \cite{Milnor, Wood}. Thus the Toledo invariant is the pull-back of a certain cohomology class $\kappa$ of $\mathrm{U}(p,q)$ viewed as a discrete group.

The extension to surfaces with nonempty boundary is due to M. Burger, A. Iozzi and A. Wienhard \cite{BIW}. They observe that the cohomology class $\kappa$ is \emph{bounded}, and that the \emph{bounded cohomology} of a surface with boundary ignores the boundary, since the fundamental groups of its components are amenable. Thus they define the Toledo invariant $\mr{T}(\Sigma,\phi)$ of a homomorphism $\phi:\pi_{1}(\Sigma)\to \mathrm{U}(p,q)$ by means of bounded cohomology classes. They express it in terms of \emph{rotation numbers}, i.e. real valued continuous functions on the universal cover of $\mathrm{U}(p,q)$. Unlike his sibling for closed surfaces, this relative Toledo invariant takes all values in an interval, defined again by an avatar of the Milnor-Wood inequality, see inequality (\ref{BIW-MW}) below.

\medskip

On a compact oriented surface with nonempty boundary, one expects that the sum $\op{sign}(\mc E,\Omega)+2\mr{T}(\Sigma,\phi)$ depends on boundary holonomy. From Atiyah-Patodi-Singer's theory \cite{APS}, one expects some eta invariant to show up. However, there should be a second correction term, due to the fact that the bounded cohomology Toledo class, by construction, lives in the relative cohomology $\mr{H}^2(\Sigma,\partial\Sigma)$.

\subsection{Main result}

In the present paper, inspired by V. Koziarz and J. Maubon \cite{KM,KM2},\footnote{Circa 2008, both groups of authors Koziarz-Maubon and Burger-Iozzi-Wienhard convinced themselves that their respective avatars of Toledo invariant coincide in rank one ($p=1$), but neither group cared to publish details.}{}  
we introduce the rho invariant of the boundary, a real number $\rro_\phi$ attached to a representation $\phi:\pi_1(\partial \Sigma)\to \mathrm{U}(p,q)$, which completes the expression of signature. The notation rho is borrowed from \cite{APSII}, where the same invariant is introduced in the positive definite case. Then we show that the rho invariant of the boundary is a sum of contributions of its connected components.

\begin{sthm}
\label{0}
Let $\Sigma$ be a compact oriented surface with nonempty boundary. Let $E$ be a complex vector space equipped with a (possibly indefinite) Hermitian form $\Omega$. Let $\phi:\pi_1(\Sigma)\to \mathrm{U}(E,\Omega)$, the unitary group of the Hermitian space $(E,\Omega)$, be a homomorphism, and $\mc E$ be the corresponding flat vectorbundle over $\Sigma$. Then
\begin{align*}
\op{sign}(\mc E,\Omega)=-2\mr{T}(\Sigma,\phi)+\rro_\phi(\partial\Sigma).
\end{align*}
Furthermore, 
$$
\rro_\phi(\partial\Sigma)=\sum_{\text{boundary components }c}\rro(\phi(c)),
$$
where $\rro:\mathrm{U}(E,\Omega)\to\br$ is a discontinuous real-valued class function. 
\end{sthm}

The rho invariant, as a function on $\mathrm{U}(p,q)$, is the sum of two terms which depend on an equivariant map $\mbf{J}:\br=\widetilde{(\br/\bz)}\to\mathrm{D}^{\mathrm{I}}_{p,q}$,
\begin{align*}
\rro(L)=\iota(L,\mbf{J})+\eta(L,\mbf{J}),
\end{align*}
where $\iota$ is the integral along $\mbf{J}$ of an $L$-invariant primitive of the K\"ahler form of $\mathrm{D}^{\mathrm{I}}_{p,q}$, and $\eta$ is the eta invariant of an elliptic operator whose index gives the signature. $\iota$ depends continuously on $L\in \mathrm{U}(p,q)$ but $\eta$ does not, it jumps when $1$ arises as an eigenvalue of $L$.

\medskip

The rho invariant $\rro:\mathrm{U}(E,\Omega)\to\br$ can be computed as follows. 

\begin{sthm}
\label{1}
Let $E$ be a complex vector space equipped with a (possibly indefinite) Hermitian form $\Omega$. Let $L\in \mathrm{U}(E,\Omega)$. 
\begin{enumerate}

  \item $(E,\Omega,L)$ canonically splits into three summands,
  \begin{align*}
(E,\Omega,L)=(E_{hu},\Omega_{hu},L_{hu})\oplus (E_{eu},\Omega_{eu},L_{eu})\oplus (E_{u},\Omega_{u},L_{u}),
\end{align*}
where the hyperbolic-unipotent summand $L_{hu}$ has nonunit eigenvalues, the elliptic-unipotent summand $L_{eu}$ has unit eigenvalues different from $1$, and the unipotent summand has only $1$ as an eigenvalue. Furthermore, 
$$
\rro(L)=\rro(L_{hu})+\rro(L_{eu})+\rro(L_{u}).
$$

  \item If $L$ is hyperbolic-unipotent, $\rro(L)=0$.
  
  \item If $L$ is elliptic-unipotent, $\rro(L)$ depends only on the semi-simple part $S$ of $L$. $S$ can be uniquely written $S=\exp(2\pi i B)$ where $\op{spectrum}(B)\subset(0,1)$. Then
$$
\rro(L)=\op{sign}(\Omega)-2\op{Trace}_{\Omega}(B).
$$
(Here $\op{Trace}_{\Omega}(B)=\sum_{j}\Omega(Be_j,e_j)$ for an $\Omega$-orthonormal basis $\{e_j\}$ of eigenvectors of $B$).

  \item If $L$ is unipotent, $E$ admits an orthogonal decomposition into $L$-invariant subspaces $E_j$ which are single Jordan blocks,
$$
(E,\Omega,L)=\bigoplus_j (E_j,\Omega_j,L_j),\quad \rro(L)=\sum_{j} \rro(L_j).
$$
For a Jordan block $L$ of dimension $n$, 
$\rro(L)=0$ if $n$ is odd. If $n$ is even, write $L=\exp(2\pi B)$ where $B$ is nilpotent. Then $\rro(L)$ is minus the signature of the Hermitian form on the $1$-dimensional space $E/BE$ induced by
\begin{align*}
(u,v)\mapsto \Omega((iB)^{n-1}u,v).
\end{align*}

\end{enumerate}
\end{sthm}

As a corollary, we are 	able to complete Atiyah's determination of the signature cocycle and the section $\sigma$ alluded to above. It is built from the rho invariant and Burger-Iozzi-Wienhard's rotation numbers. The rotation number associated with twice the K\"ahler bounded cohomology class of an element $L\in\mathrm{U}(p,q)$ is an element of $\br/\bz$ that depends only on the semi-simple elliptic component of $L$. When $L=(L_+,L_-)$ belongs to the subgroup $\mathrm{U}(p)\times\mathrm{U}(q)$,
$$
e^{2\pi i\mathrm{Rot}(L)}=\left(\frac{\mathrm{det}(L_+)}{\mathrm{det}(L_-)}\right)^2.
$$
The map $\mathrm{Rot}:\mathrm{U}(p,q)\to\br/\bz$ is continuous. Let $p_2:\mathrm{U}(p,q)_2\to \mathrm{U}(p,q)$ be the central extension of $\mathrm{U}(p,q)$ defined by the signature cocycle. Then $\mathrm{Rot}$ lifts to a continuous map $\mathrm{Rot}_2:\mathrm{U}(p,q)_2\to\br$ that restricts to the identity on the kernel of $p_2$.

\begin{sthm}
\label{thmsigma}
The image of Atiyah's section $\sigma:\mathrm{U}(p,q)\to \mathrm{U}(p,q)_2$ is the zero level set of the function $\mathrm{Rot}_2+\rro\circ p_2$. The coboundary of this function (viewed as a $0$-cochain) is the pull-back by $p_2$ of Meyer and Atiyah's signature cocycle.
\end{sthm}
A concrete expression for $\sigma$ is given in Theorem \ref{thmsigmabis}.

\medskip

Theorem \ref{thmsigma} thus provides a link between three indirectly defined objects on unitary groups:
\begin{itemize}
  \item Meyer and Atiyah's integer valued signature cocycle,
  \item Burger-Iozzi-Wienhard's $\br/\bz$-valued rotation numbers,
  \item our real valued rho invariant.
\end{itemize}

\subsection{Other classical groups of Hermitian type}

The method applies simultaneously to the groups $G=\mathrm{SO}^*(2n)$, $\mathrm{Sp}(2n,\mathbb{R})$ and $\mathrm{SO}_0(n,2)$, since they embed in unitary groups. We refer to sections \ref{SO}, \ref{SP} and \ref{SO0} for the definitions of the corresponding intersection forms and bounded cohomology Toledo invariants. We define the rho invariant by composition $G\to \mathrm{U}(p,q)\to\br$.

\begin{sthm}\label{thm0.5}
For any representation $\phi:\pi_1(\Sigma)\to G$, one can define the signature $\op{sign}(\mc{E},\Omega)$ of the associated flat vector bundles. Then, 
\begin{itemize}
  \item [(i)] For $G=\mathrm{U}(p,q)$, 
  $$
  \op{sign}(\mc{E},\Omega)=-2\op{T}(\Sigma,\phi)+\rro_\phi(\p\Sigma),\quad |\op{sign}(\mc{E},\Omega)|\leq (p+q) |\chi(\Sigma)|.
  $$
  
  \item [(ii)] For $G=\mathrm{SO}^*(2n)$,
  $$
  \op{sign}(\mc{E},\Omega)=-4\op{T}(\Sigma,\phi)+\rro_\phi(\p\Sigma),\quad |\op{sign}(\mc{E},\Omega)|\leq 2n |\chi(\Sigma)|.
  $$
  
  \item [(iii)] For $G=\mathrm{Sp}(2n,\mathbb{R})$,
  $$
  \op{sign}(\mc{E},\Omega)=2\op{T}(\Sigma,\phi)+\rro_\phi(\p\Sigma),\quad |\op{sign}(\mc{E},\Omega)|\leq 2n |\chi(\Sigma)|.
  $$
  
  \item[(iv)] For $G=\mathrm{SO}_0(n,2)$,
  $$
  \op{sign}(\mc{E},\Omega)=0.
  $$
\end{itemize}
\end{sthm}

Since $\text{Hom}(\pi_1(\Sigma), G)$ is connected for surface $\Sigma$ with boundary, people also study the relative representation variety, by fixing the holonomies of boundaries $b_i, i=1,\cdots,n$. For a given $\cal C=(\cal C_1,\cdots,\cal C_n)$, a set of conjugacy classes in $G$, one defines the relative representation variety
$$\text{Hom}^{\cal C}(\pi_1(\Sigma), G)=\{\rho\in\text{Hom}(\pi_1(\Sigma),G): \rho(b_i)\in \cal C_i\},$$ which is a real semialgebraic set.
\begin{cor}
The signature is constant on connected components of the relative representation variety $\text{Hom}^{\cal C}(\pi_1(\Sigma), G)$.
\end{cor}
\begin{proof}The Toledo invariant is constant on connected components of the relative representation variety $\text{Hom}^{\cal C}(\pi_1(\Sigma), G)$ by \cite{BIW}[Cor. 8.11], and the value $\rro$ is fixed.
\end{proof}

Along the way, we obtain Milnor-Wood-type inequalities. According to \cite{BIW}, for a simple Lie group $G$ with Hermitian symmetric space $\ms X$, the Milnor-Wood inequality reads
\begin{align}
\label{BIW-MW}
|\op{T}(\Sigma,\phi)|\leq \op{rank}(\ms X)|\chi(\Sigma)|.
\end{align}
Here, we replace the Toledo invariant with signature. The inequality follows when we estimate the absolute value of the signature $|\op{sign}(\mc E,\Omega)|$ by the dimension of the vector space on which the intersection form is defined, i.e. the image of $\mr{H}^{1}(\Sigma,\partial\Sigma;\mc E)$ in $\mr{H}^{1}(\Sigma;\mc E)$,
$$
\op{Im}(\mr{H}^{1}(\Sigma,\partial\Sigma;\mc E)\to \mr{H}^{1}(\Sigma;\mc E)),
$$
which is generically equal to $\op{dim}(E)|\chi(\Sigma)|$ (a deformation argument allows to reduce to this generic case).

Note that (i) is not sharp if $p\not=q$. In other cases, when $\mr{T}$ and $\rro$ have opposite signs, our inequalities may sharpen Burger-Iozzi-Wienhard's inequality (\ref{BIW-MW}). For instance, one can always modify a homomorphism $\phi:\pi_1(\Sigma)\ra \mathrm{Sp}(2,\br)$, replacing its elliptic boundary rho invariants with their fractional parts, except for possibly one of them, without changing the Toledo invariant. This yields

\begin{prop}\label{TI3}
For every homomorphism $\phi:\pi_1(\Sigma)\ra \mathrm{Sp}(2,\br)$,  
\begin{equation*}
\op{T}(\Sigma,\phi) \leq |\chi(\Sigma)|+1 - \sum_{c\,;\, \phi(c)\text{ elliptic}}\left\{\frac{\rro(\phi(c))}{2}\right\}.
\end{equation*}
where $\{\bullet\}:=\bullet-\lfloor\bullet\rfloor$ denotes the fractional part of $\bullet$.
\end{prop}

Refer to Remark \ref{improve} for the whole statement.
Examples showing that Theorem \ref{thm0.5} and Proposition \ref{TI3} are sharp will be given in Section \ref{Class}.

\subsection{Final remarks}

For compact target groups such as $\mathrm{U}(n)$, our Milnor-Wood-type inequality is nontrivial, but follows from the solution of the multiplicative Horn problem, see the Appendix, Subsection \ref{Horn}.

In the case of $SO_0(n,2)$, the statement (iv) of Theorem \ref{thm0.5} is disappointing. We expect the Toledo invariant to be related to a topological invariant, that would play the role played by the signature for the other families of simple groups of Hermitian type.

The method of parabolic Higgs bundles provides an alternative approach to flat bundles over surfaces with boundary, which possibly encompasses our results, see \cite{BGPMR, DeroinTholozan, TholozanToulisse}.

\subsection{Organization of the paper}

In Section \ref{Sym}, we introduce the first order differential operator on a flat Hermitian vector bundle over the circle which is related to the signature operator on a surface. In Section \ref{Sig}, we define the signature for flat Hermitian vector bundles, and we relate it to the index of a differential operator and express it using Atiyah-Patodi-Singer's index theorem. In Section \ref{Tol}, the first Chern class relevant to signature is related to the bounded cohomology Toledo invariant, culminating with the proof of the first part of Theorem \ref{0}. The rho invariant is defined and computed in Section \ref{Rho}, where the proofs of Theorems \ref{0} and \ref{1} are completed. It is exploited in Section \ref{Atiyahsigma}, where the connection with rotation numbers is made and the proof of Theorem \ref{thmsigma} is given. In Section \ref{AMW}, a Milnor-Wood type inequality for signature is proven for unitary groups. Variants for the other families of simple Lie groups of Hermitian type are discussed in Sections \ref{SO}--\ref{SO0}, leading to the proof of Theorem \ref{thm0.5}. In Section \ref{Class}, we prove Proposition \ref{TI3}.
  
In the Appendix, we provide a direct calculation of the eta invariant and the rho invariant for the group $\op{U}(1,1)$. This allows to double-check the calculations of Section \ref{Rho}. We check that our Milnor-Wood type inequality for the simple Lie group $\mathrm{U}(n)$ follows from results on the multiplicative Horn problem. For the reader's convenience, we include the details of a classical theorem: the classification of unipotents in unitary groups. Finally, we give a geometric proof of the true Milnor-Wood inequality (for the Toledo invariant) for general Hermitian symmetric spaces: it follows from Domic-Toledo's evaluation of the Gromov norm of the K\"ahler form.

\section{Flat Hermitian vector bundles and eta invariants over a circle}\label{Sym}

In this section, we shall consider the flat Hermitian vector bundle $E_\phi$ associated with a representation $\phi$ of the fundamental group of a circle $S^1$ into the group $\mathrm{U}(p,q)$. We shall define a first order elliptic self-adjoint differential operator $A_{\mbf{J}}$, and  recall the definition of eta invariant $\eta(A_{\mbf{J}})$ for the operator. 

Let $E=\mb{C}^{p+q}$ be a complex vector space of dimension $p+q$. Let 
\begin{align}\label{Hermitian form}
\begin{split}
  \Omega=|dz^1|^2+\cdots+|dz^p|^2-|dz^{p+1}|^2-\cdots-|dz^{p+q}|^2
 \end{split}
\end{align}
 be a non-degenerate Hermitian form with the signature $(p,q)$. Denote by $\mathrm{U}(E,\Omega)$ the space of all linear transformations on $E$ preserving the Hermitian form $\Omega$, it is  called the $\mathrm{U}(p,q)$-group.
 For any representation 
 $$\phi:\pi_1(S^1)\to \op{U}(E,\Omega)$$
from the fundamental group of circle $S^1$ into the $\op{U}(p,q)$-group $\op{U}(E,\Omega)$. Denote $L:=\phi(\gamma_0)$, where $\gamma_0$ denotes the generator of $\pi_1(S^1)$, which is given by $\gamma_0(x)=e^{ix},0\leq x\leq 2\pi$. The representation $\phi$ gives rise to a flat vector bundle 
$$E_\phi:=\mb{R}\times_{\phi}E=(\mb{R}\times E)/\sim$$
over $S^1$, where $(x_1,e_1)\sim (x_2,e_2)$ if $x_2=x_1+2\pi k,k\in \mb{Z}$ and $e_2=L^{-k}(e_1)$. Each global section of $E_\phi$ is equivalent to a smooth map $s:\mb{R}\to E$ satisfying the $\phi$-equivariant condition $s(x+2\pi)=L^{-1}s(x)$.

Let $\mc{J}(E,\Omega)(\subset\op{U}(E,\Omega))$ be the space of all linear transformations in $\op{U}(E,\Omega)$ such that $J^2=-\op{Id}$ and $i\Omega(\cdot,J\cdot)$ is positive definite.
Denote by $\mc{J}(E_\phi,\Omega)=C^{\infty}(S^1,\mb{R}\times_\phi \mc{J}(E,\Omega))$ the space of all $\phi$-equivariant sections with values in  $ \mc{J}(E,\Omega)$.
{For any  $L\in \op{U}(E,\Omega)$ which can be written as  $L=\pm e^{i\theta}\exp(2\pi B)$, where $0\leq \theta<2\pi$ and
 $B\in\mf{u}(E,\Omega)$, i.e. $B^*\Omega+\Omega B=0$, we can find a canonical element $\mbf{J}\in \mc{J}(E_\phi,\Omega)$ for any given $J\in \mc{J}(E,\Omega)$.
 In fact, for any $J\in\mc{J}(E,\Omega)$, we  define }
 $$\mbf{J}(x)=\exp(-xB)J\exp(xB)\in \mc{J}(E,\Omega),$$
which satisfies $\mbf{J}(x+2\pi)=L^{-1}\mbf{J}(x)L$ and thus $\mbf{J}\in \mc{J}(E_\phi,\Omega)$.   { But in general, $L$ cannot be written as $L=\pm e^{i\theta}\exp(2\pi B)$ except for the group $\op{U}(1,1)$, hence one needs to choose another $\mbf{J}$ on $E_\phi$.}

There exists a canonical flat connection $d$ on $E_\phi$, which is induced from the trivial vector bundle $\mb{R}\times E\to \mb{R}$. The holonomy representation of the flat connection $d$ is just the representation $\phi$.
 Denote by $A^{0}(S^1,E_\phi)$ the space of all smooth sections of $E_\phi$, which can be identified with the space $A^0(\mb{R},E)^L$ of all $\phi$-equivariant smooth maps $s:\mb{R}\to E$. There is a standard $\mathrm{L}^2$-metric on the space $A^{0}(S^1,E_\phi)\cong A^0(\mb{R},E)^L$ with respect to the inner product $i\Omega(\cdot,\mbf{J}\cdot)$ and the metric $dx\otimes dx$ on $S^1$, i.e. 
$
 	\int_{S^1}i\Omega(\cdot,\mbf{J}\cdot)dx.
$

Consider the following $\mb{C}$-linear first order 
differential operator
\begin{align}\label{ABJ1}
A_{\mbf{J}}:=\mbf{J}\frac{d}{dx},
\end{align}
which acts on the space  $A^0(\mb{R},E)^L\cong A^0(S^1,E_\phi)$. 
\begin{prop}\label{prop2.0}
$A_{\mbf{J}}$ is a $\mb{C}$-linear formally self-adjoint elliptic first order differential operator  in the space $A^0(S^1,E_\phi)$.
\end{prop}
\begin{proof}
It is obvious that $A_{\mbf{J}}$ is $\mb{C}$-linear, first order and elliptic, so we just need to prove $A_{\mbf{J}}$ is formally self-adjoint. For any $s_1,s_2\in  A^0(S^1,E_\phi)$, one has 
\begin{align*}
	&\quad \langle A_{\mbf{J}}s_1,s_2\rangle-\langle s_1,A_{\mbf{J}}s_2\rangle\\
	&=\int_{S^1}\left(\Omega\left({\frac{d}{dx}}s_1,{s_2}\right)+\Omega\left(s_1,{\frac{d}{dx}}{s_2}\right)\right)dx\\
	&=\int_{S^1}d\left(\Omega(s_1,{s_2})\right)=0,
\end{align*}
which completes the proof.
\end{proof}
\begin{rem}
The operator $A_{\mbf{J}}$ has a natural extension in the Hilbert space $\mathrm{L}^2(S^1,E_\phi)$, we also denote it by $A_{\mbf{J}}$, see e.g.  \cite[Definition 7.1 in Appendix]{Kodaira}. From Proposition \ref{prop2.0}, $A_{\mbf{J}}$ is formally self-adjoint and elliptic, so $A_{\mbf{J}}$ is self-adjoint in the Hilbert space $\mathrm{L}^2(S^1,E_\phi)$, see e.g.  \cite[Theorem 7.2 in Appendix]{Kodaira}.   
\end{rem}

For every elliptic self-adjoint differential operator $A$, which acts on a Hermitian vector bundle over a closed manifold, the operator $A$ has a discrete spectrum with real eigenvalues. 
 Let $\lambda_j$ run over the eigenvalues of $A$, then the eta function of $A$ is defined as 
$$\eta_A(s)=\sum_{\lambda_j\neq 0}\frac{\op{sgn}\lambda_j}{|\lambda_j|^s},$$
where $s\in \mb{C}$. The eta function admits a meromorphic continuation to the whole complex plane and is holomorphic at $s=0$. The special value $\eta_A(0)$ is then called the {\it eta invariant} of the operator $A$, and we denote the eta invariant by
\begin{align}\label{eta invariant}
\eta(A)=\eta_A(0).
\end{align}
One can refer to \cite{APS, Muller} for the definition of eta invariant. 

The  complex vector bundle $E_\phi=\wt{\Sigma}\times_\phi E$ is a Hermitian vector bundle over $S^1$ with the Hermitian metric $i\Omega(\cdot,\mbf{J}\cdot)$, and the operator $A_{\mbf{J}}$ is an elliptic operator which is formally self-adjoint with respect to the inner product $\int_{S^1}i\Omega(\cdot,\mbf{J}\cdot)dx$, then $A_{\mbf{J}}$ has discrete spectrum consisting of real eigenvalues $\lambda$ of finite multiplicity, and the eta invariant $\eta(A_{\mbf{J}})$ of $A_{\mbf{J}}$  is defined by \eqref{eta invariant}.

\begin{ex}
Given a representation $\phi:\pi_1(S^1)\to \op{U}(1,1)$, consider the operator $$A_{\mbf{J}}=\mbf{J}\frac{d}{dx},$$
	where $\mbf{J}:=\exp(-xB)J\exp(xB)$, and  $L=\pm e^{i\theta}\exp(2\pi B)\in \op{U}(1,1)$ denotes the representation of the generator of $\pi_1(S^1)$, then the eta invariant $\eta(A_{\mbf{J}})$ is calculated in Appendix \ref{Appeta}, see \eqref{etadim2}.
\end{ex}

\section{The signature of a flat Hermitian vector bundle}\label{Sig}

In this section, we will define the signature of a flat $\op{U}(p,q)$-Hermitian vector bundle, and express it as the difference of the $L^2$-indices of two operators $d^+$ and $d^-$ on a completion $\hat\Sigma$ of $\Sigma$ with cylindrical ends, see Subsection \ref{indices}. Atiyah-Patodi-Singer's formula for the $L^2$-index involves two boundary terms, the eta invariant and half the dimension of the kernel of the operator induced on the boundary, see paragraph \ref{APSindex}. We shall show that the dimensions of the relevant kernels for $d^+$ and $d^-$ are equal (paragraph \ref {limiting}), therefore such terms do not appear in the formula for signature in Theorem \ref{thmsign2}.

\medskip

Let $\Sigma$ be a connected oriented surface with smooth boundary $\p\Sigma$, each component of $\p\Sigma$ is homeomorphic to $S^1$, $\iota:\p\Sigma\to \Sigma$ denotes the natural inclusion. Let $(E,\Omega)$ be a Hermitian vector space, where $E=\mb{C}^{p+q}$ is a complex vector space of dimension $p+q$, and $\Omega$ is a non-degenerate Hermitian form (possibly indefinite) with signature $(p,q)$, $p\geq 0,q\geq 0$.  Let $\phi:\pi_1(\Sigma)\to \mathrm{U}(E,\Omega)$ be a representation from the fundamental group $\pi_1(\Sigma)$ of $\Sigma$ into the $\mathrm{U}(p,q)$-group $\mathrm{U}(E,\Omega)$. The representation $\phi$ gives rise to a flat vector bundle $$\mc{E}=\wt{\Sigma}\times_\phi E$$ over $\Sigma$. Any element of $A^*(\Sigma,\mc{E})$ can be viewed as a $\phi$-equivariant element in $A^*(\wt{\Sigma},\mb{R})\otimes E$, where $\phi$-equivariant means $(\gamma^{-1})^*\omega\otimes \phi(\gamma)v=\omega\otimes v$ for $\omega\in A^*(\wt{\Sigma},\mb{R})$ and $v\in E$. There exists a {\it canonical flat connection} $d$ on the flat bundle $\mc{E}$, which is defined by $d(\omega\otimes v):=d\omega\otimes v$. One can also refer to \cite[Section 1.1]{BM} for the representations, flat bundles and the canonical flat connection.

\subsection{Definition of signature}\label{Definition of signature}

Let $\mathrm{H}^*(\Sigma,\mc{E})$ (resp. $\mr{H}^*(\Sigma,\p\Sigma,\mc{E})$) denote the (resp. relative) twisted singular cohomology, one can refer to \cite[Chapter 5]{DK} for its definitions. Set
$$\widehat{\mathrm{H}}^1(\Sigma;\mc{E}):=\mathrm{Im}(\mathrm{H}^1(\Sigma,\p\Sigma;\mc{E})\to \mr{H}^1(\Sigma;\mc{E})).$$
There exists a natural quadratic form 
\begin{align*}
\begin{split}
  &Q:\widehat{\mr{H}}^1(\Sigma;\mc{E})\times \widehat{\mr{H}}^1(\Sigma;\mc{E})\to \mb{C}\\
  &Q([a],[b])=\int_\Sigma \Omega([a]\cup[b]).
 \end{split}
\end{align*}
By the same argument as in \cite[Page 65]{APS}, the form $Q$ is non-degenerate due to Poincar\'e duality. Moreover, $Q$ is a skew-Hermitian form, i.e. $Q([a],[b])=-\o{Q([b],[a])}$, then $iQ$ is a Hermitian form. If $\widehat{\mr{H}}^1(\Sigma,\mc{E})=\mathscr{H}^+\oplus \mathscr{H}^-$ such that $iQ$ is  positive definite on $\mathscr{H}^+$ and negative definite on $\mathscr{H}^-$, the signature of the flat Hermitian vector bundle $(\mc{E},\Omega)$ is defined as the signature of the Hermitian form $iQ$. Then 
\begin{align*}
\begin{split}
  \op{sign}(\mc{E},\Omega):=\op{sign}(iQ)=\dim\mathscr{H}^+-\dim\mathscr{H}^-.
 \end{split}
\end{align*}

\subsection{Relation to indices of operators}
\label{indices}

Suppose that on the collar neighborhood $I\times \p\Sigma\subset \Sigma$ of $\p\Sigma$, $I=[0,1]$, the Riemannian metric of $\Sigma$ is equal to the product metric $g_\Sigma=du^2+g_{\p\Sigma}$. Let 
$$\widehat{\Sigma}=\Sigma\cup((-\infty,0]\times \p\Sigma)$$
be the complete manifold obtained from $\Sigma$ by gluing the negative half-cylinder $(-\infty,0]\times \p\Sigma$ to the boundary of $\Sigma$.
 \begin{center}
\begin{tikzpicture}[x=1cm,y=1cm]

\begin{scope}[shift={(2,0)}, thick]
\clip(-1.8,-2)rectangle(3,2);
\draw (0,0) circle [x radius=2, y radius=1];
\end{scope}

\draw[shift={(2,0)}, yscale=cos(70), thick] (-1,0) arc (-180:0:1);
\draw[shift={(2,0)}, yscale=cos(70), thick] (-0.7,-0.6) arc (180:0:0.7);

\draw[thick] (-4,-0.3) -- (0,-0.3);
\draw[thick] (-4,0.3) -- (0,0.3);

\draw [xscale=cos(70), dashed, thick] (0,-0.3) arc (-90:90:0.3);
\draw [xscale=cos(70), thick] (0,0.3) arc (90:270:0.3);

\draw [shift={(-0.7,0)}, xscale=cos(70), dashed, thick] (0,-0.3) arc (-90:90:0.3);
\draw [shift={(-0.7,0)}, xscale=cos(70), thick] (0,0.3) arc (90:270:0.3);

\draw [thick] (0,0.3) .. controls (0.1,0.3) and (0.1,0.34) .. (0.21,0.446);
\draw [thick] (0,-0.3) .. controls (0.1,-0.3) and (0.1,-0.34) .. (0.21,-0.446);

\draw (-4,-0.5) node{{\tiny $-\infty$}};
\draw (-0.7,-0.5) node{{\tiny $0$}};
\draw (-0.35,-0.6) node{{\tiny $I$}};
\draw (-1.1,0) node{{\tiny $\partial \Sigma$}};
\draw (0,-0.5) node{{\tiny $1$}};
\draw (2.7, 0.5) node{{\tiny $\Sigma$}};
\draw (0,-1.3) node{{$\widehat{\Sigma}$}};

\end{tikzpicture}
\end{center}

For any $a,b\in \wedge^*T^*\Sigma$, the Hodge $*$ operator is defined as 
\begin{align}\label{star}
\begin{split}
  a\wedge *b=g_\Sigma(a,b)\op{Vol}_g
 \end{split}
\end{align}
where the volume element $\mr{Vol}_g:=\sqrt{\det(g_{ij})}dx^1\wedge dx^2$. One can check that $*^2a=(-1)^{|a|}a$. Denote by $\mc{J}(\mc{E},\Omega)$ the space of all smooth sections  $\mbf{J}$ of $\mr{End}(\mc{E})$ preserving $\Omega$, such that $i\Omega(\cdot,\mbf{J}\cdot)$ is a positive definite Hermitian form and $\mbf{J}^2=-\op{Id}$. For any $\mbf{J}\in \mc{J}(\mc{E},\Omega)$, there exists a natural inner products in $A^*(\Sigma,\mc{E})$ by 
$$(\alpha\otimes e_1,\beta\otimes e_2)_x:=g_\Sigma(\alpha,\beta)\cdot i\Omega(e_1,\mbf{J}e_2),\quad \alpha,\beta\in \wedge^*T^*\Sigma|_x,e_1,e_2\in \mc{E}_x$$ 
and 
$$\left\langle \cdot,\cdot \right\rangle =\int_\Sigma \left(\cdot,\cdot\right)\op{Vol}_g.$$
 Denote by $d^*$ the formally adjoint operator of $d$ with respect to $\left\langle\cdot,\cdot\right\rangle$. Then 
\begin{equation}
  d^*=\mbf{J}*d*\mbf{J}.
\end{equation}
Moreover, one has $*\mbf{J}=\mbf{J}*$ and $(*\mbf{J})^2=\op{Id}$ on the space $\wedge^1 T^*\Sigma\otimes\mc{E}$, that is, $*\mbf{J}$ is an involution on the space $\wedge^1T^*\Sigma\otimes\mc{E}$. Let
\begin{equation*}
  \pi^\pm:=\frac{1\pm*\mbf{J}}{2}:\wedge^1T^*\Sigma\otimes \mc{E}\to \wedge^\pm
\end{equation*}
denote the natural projections onto the $\pm 1$-eigenspaces $\wedge^\pm$ of $*\mbf{J}$, 
and set
\begin{equation*}
  d^\pm=\pi^\pm\circ d.
\end{equation*}
From Corollary \ref{cor4.2}, the operators $d^+,d^-$ have the form 
\begin{equation*}
  d^+=\sigma^+(\frac{\p}{\p u}+A^+_{\mbf{J}}),\quad d^-=\sigma^-(\frac{\p}{\p u}+A^-_{\mbf{J}}),
\end{equation*}
where both $\sigma^+:\mc{E}\to\wedge^+$ and $\sigma^-:\mc{E}\to\wedge^-$ are bundle isomorphisms, and $A^+_{\mbf{J}}$, $A^-_{\mbf{J}}$ are the first order elliptic formally self-adjoint operators on the boundary $\p\Sigma$.

For any $a\in \wedge^+$ and $b\in \wedge^-$, one has
\begin{align*}
\begin{split}
  \Omega(a\wedge *\mbf{J}{b})=-\Omega(a\wedge b)=-\Omega(\mbf{J}*a\wedge b)=\Omega(*a\wedge \mbf{J}b)=-\Omega(a\wedge *\mbf{J}b),
 \end{split}
\end{align*}
from which it follows that $ \Omega(a\wedge *\mbf{J}{b})=0$. Thus the Hermitian inner product of $a$ and $b$ satisfies
\begin{align*}
\begin{split}
  (a,b)\mr{Vol}_g=ig_\Sigma(\Omega(a,\mbf{J}b))\mr{Vol}_g=i\Omega(a\wedge *\mbf{J}b)=0,
 \end{split}
\end{align*}
that is, the two subspaces $\wedge^+$ and $\wedge^-$ are orthogonal to each other. Hence $(d^{\pm})^*=d^*$ on $\wedge^\pm$.
Note that \cite[Proposition 4.9]{APS} works as well for the cohomology with local coefficients, that is
\begin{align}
	\widehat{\mathrm{H}}^*(\Sigma,\mc{E})\cong \mathscr{H}^*(\widehat{\Sigma},\mc{E}):=\{\phi\in \mathrm{L}^2(\widehat{\Sigma},\mc{E}): (d+d^*)\phi=0\}.
\end{align}
By identification with the above two groups, $iQ$ is also a Hermitian quadratic form on $\mathscr{H}^1(\widehat{\Sigma},\mc{E})$. We can extend the bundle $\mc{E}$ and the connection $d$ to $\widehat{\Sigma}$, we still denote it by $\mc{E}$ and $d$.
\begin{prop}\label{prop3}
The splitting
$$\mathscr{H}^1(\widehat{\Sigma},\mc{E})=\op{Ker}(d^+)^*\cap \mathrm{L}^2(\widehat{\Sigma},\wedge^+)\oplus \op{Ker}(d^-)^*\cap \mathrm{L}^2(\widehat{\Sigma},\wedge^-)$$ is such that $iQ$ is positive definite on $\op{Ker}(d^+)^*\cap \mathrm{L}^2(\widehat{\Sigma},\wedge^+)$ and negative definite on $\op{Ker}(d^-)^*\cap \mathrm{L}^2(\widehat{\Sigma},\wedge^-)$, the decomposition is orthogonal with respect to $iQ$.
\end{prop}
\begin{proof}
Firstly, we show that $\op{Ker}(d^\pm)^*\cap \mr{L}^2(\widehat{\Sigma},\wedge^{\pm})\subset\mathscr{H}^1(\widehat{\Sigma},\mc{E})$. If $a\in \op{Ker}(d^\pm)^*\cap \mr{L}^2(\widehat{\Sigma},\wedge^{\pm})$, then $d^*a=0$ and $a\in \mr{L}^2(\widehat{\Sigma},\wedge^{\pm})$. From \cite[Proposition 3.11]{APS}, the $L^2$ sections in $\op{Ker}(d^\pm)^*\cap \op{L}^2(\widehat{\Sigma},\wedge^{\pm})$ are exponentially decaying as $t\to -\infty$, so 
$$d^*a=*\mbf{J}d*\mbf{J}a=\pm *\mbf{J}da.$$
Thus, $d^*a=0$ implies that $da=0$, which means that $a$ is harmonic. Hence $\op{Ker}(d^\pm)^*\cap \mr{L}^2(\widehat{\Sigma},\wedge^{\pm})\subset\mathscr{H}^1(\widehat{\Sigma},\mc{E})$. On the other hand, for any nonzero $a\in \op{Ker}(d^+)^*\cap \op{L}^2(\widehat{\Sigma},\wedge^{+})$ and $b\in \op{Ker}(d^-)^*\cap \op{L}^2(\widehat{\Sigma},\wedge^{-})$, one has
\begin{align*}
\begin{split}
  Q(a,b)=\int_{\widehat{\Sigma}}\Omega(a\wedge b)=- \int_{\widehat{\Sigma}}\Omega(a\wedge *\mbf{J}b)=0
 \end{split}
\end{align*}
since $\wedge^+$ and $\wedge^-$ are orthogonal to each other, and 
\begin{align*}
\begin{split}
  iQ(a,a)=i\int_{\widehat{\Sigma}}\Omega(a\wedge a)=i \int_{\widehat{\Sigma}}\Omega(a\wedge *\mbf{J}a)>0,
 \end{split}
\end{align*}
\begin{align*}
\begin{split}
  iQ(b,b)=i\int_{\widehat{\Sigma}}\Omega(b\wedge b)=-i \int_{\widehat{\Sigma}}\Omega(b\wedge *\mbf{J}b)<0.
 \end{split}
\end{align*}
 The proof is complete.
\end{proof}

The $\op{L}^2$-index of $d^\pm$ is well-defined and is given by
$$\op{L}^2\mathrm{Index}(d^\pm):=\dim \op{Ker}(d^\pm)\cap \op{L}^2(\widehat{\Sigma},\mc{E})-\dim \op{Ker}(d^{\pm})^*\cap \op{L}^2(\widehat{\Sigma},\wedge^\pm).$$
Note that 
\begin{align}\label{H0}
\begin{split}
  \op{Ker}(d^\pm)\cap \op{L}^2(\widehat{\Sigma},\mc{E})=\mathscr{H}^0(\widehat{\Sigma},\mc{E}).
 \end{split}
\end{align}
In fact, if $d^-a=0$, then $da=d^+a$, and so $da=*\mbf{J}da$, which follows that $(d^+)^*d^+a=d^*da=0$. From \cite[Proposition 3.15]{APS}, the $\op{L}^2$-solutions of $d^+$ and $(d^+)^*d^+$ are coincide, so $da=d^+a=0$. Therefore,
$$\op{L}^2\mathrm{Index}(d^\pm)=\dim \mathscr{H}^0(\widehat{\Sigma},\mc{E})-\dim \op{Ker}(d^{\pm})^*\cap \op{L}^2(\widehat{\Sigma},\wedge^\pm).$$
By \eqref{dP}, one can define the operators $d^{\pm}_P$ by the restriction of $d^{\pm}$. From \cite[Proposition 3.11]{APS}, $\op{Ker} d^\pm_P$ is isomorphic to the space of $\op{L}^2$-solutions of $d^\pm\varphi=0$ on $\widehat{\Sigma}$ and $\op{Ker}(d^{\pm}_P)^*$ is isomorphic to the space of extended $\op{L}^2$-solutions of $(d^\pm)^*\varphi=0$ on $\widehat{\Sigma}$, where extended solution means that on $(-\infty,0]\times\p\Sigma$, $\varphi$ can be written as $\varphi=\phi+\psi$ with $\phi\in \op{Ker}\sigma^{\pm}A^{\pm}_{\mbf{J}}(\sigma^{\pm})^{-1}$ and $\psi\in \op{L}^2$. The section $\phi$ is called the limiting value of the extended solution $\varphi$.

We denote by $h_{\infty}(\wedge^\pm)$ the dimension of the subspace of $\op{Ker}(A^\pm_{\mbf{J}})$ consisting of limiting values of extended $\op{L}^2$-sections $a$ of $\wedge^\pm$ satisfying $(d^{\pm})^*a=0$. Therefore, 

\begin{prop}\label{prop3.2}
	The signature of $(\mc{E},\Omega)$ is given by
	\begin{align*}
		\op{sign}(\mc{E},\Omega)&=\mathrm{L}^2\op{Index}(d^{-})-\mathrm{L}^2\op{Index}(d^{+})\\
		&=\op{Index}(d^-_P)-\op{Index}(d^+_P)+h_{\infty}(\wedge^-)-h_{\infty}(\wedge^+).
	\end{align*}
\end{prop}

\begin{proof}
By the definition of signature, Proposition \ref{prop3} and \cite[Corollary 3.14]{APS}, one has
\begin{align*}
\begin{split}
  \op{sign}(\mc{E},\Omega) &=\dim \op{Ker}(d^{+})^*\cap \op{L}^2(\widehat{\Sigma},\wedge^+)-\dim \op{Ker}(d^{-})^*\cap \op{L}^2(\widehat{\Sigma},\wedge^-)\\
  &=\op{L}^2\op{Index}(d^-)-\op{L}^2\op{Index}(d^+)\\
  &=\op{Index}(d^-_P)-\op{Index}(d^+_P)+h_{\infty}(\wedge^-)-h_{\infty}(\wedge^+).
 \end{split}
\end{align*}
	The proof is complete.
\end{proof}

\subsection{A formula for signature}

In this subsection, by using Atiyah-Patodi-Singer's index theorem, we will give a formula for the signature of flat Hermitian bundles. 

\subsubsection{The Atiyah-Patodi-Singer index theorem}
\label{APSindex}

Let $\Sigma$ be a connected oriented surface with smooth boundary $\p\Sigma$. Let $g_\Sigma$ be a Riemannian metric on $\Sigma$ such that $g_\Sigma=du^2+g_{\p\Sigma}$ on the collar neighborhood $\p\Sigma\times I\subset\Sigma$ of $\p\Sigma$, $I=[0,1]$. The bundles $\mc{E}$ and $\wedge^{\pm}$ are Hermitian vector bundles, and 
$$d^{\pm}:A^0(\Sigma,\mc{E})\to A^0(\Sigma,\wedge^\pm)$$
are two first order elliptic differential operators. On $\p\Sigma\times I$, the volume element is given by
$$\op{Vol}_g=\frac{dx\wedge du}{|dx|}.$$
From the definition of $*$ \eqref{star}, one has 
$$*(dx\otimes e)=|dx|du\otimes e,\quad*(du\otimes e)=-\frac{1}{|dx|}dx\otimes e$$
for any $e\in \mc{E}$. Thus 
\begin{align*}
*\left(\frac{dx}{|dx|}+idu\right)=-i\left(\frac{dx}{|dx|}+idu\right),\quad 	*\left(\frac{dx}{|dx|}-idu\right)=i\left(\frac{dx}{|dx|}-idu\right).
\end{align*}
Let $\mc{E}=\mc{E}^+\oplus \mc{E}^-$ be the decomposition of $
\mc{E}$ into $\pm i$-eigenspaces of $\mbf{J}$. Then 
\begin{align*}
\wedge^+=\mb{C}\left\{\frac{dx}{|dx|}-idu\right\}\otimes\mc{E}^-\oplus\mb{C}\left\{\frac{dx}{|dx|}+idu\right\}\otimes\mc{E}^+	
\end{align*}
and 
\begin{align*}
\wedge^-=\mb{C}\left\{\frac{dx}{|dx|}-idu\right\}\otimes\mc{E}^+\oplus\mb{C}\left\{\frac{dx}{|dx|}+idu\right\}\otimes\mc{E}^-.	
\end{align*}
Note that 
\begin{equation*}
  \dim \mc{E}^+=p,\quad \dim\mc{E}^-=q.
\end{equation*}
Thus $\wedge^{\pm}\cong\mc{E}$ and the isomorphisms are given by 
$$\sigma^+:\mc{E}\to\wedge^+, \sigma^+(e)=-\frac{i}{2}\left(\frac{dx}{|dx|}+idu\right)\otimes e^++\frac{i}{2}\left(\frac{dx}{|dx|}-idu\right)\otimes e^-$$
and 
$$\sigma^-:\mc{E}\to\wedge^-, \sigma^-(e)=\frac{i}{2}\left(\frac{dx}{|dx|}-idu\right)\otimes e^+-\frac{i}{2}\left(\frac{dx}{|dx|}+idu\right)\otimes e^-.$$
The maps $\pi^{\pm}:\wedge^1T^*\Sigma\otimes\mc{E}\to\wedge^\pm$ can be expressed as
$$\pi^+\left(\frac{dx}{|dx|}\otimes e\right)=\frac{1}{2}\left(\frac{dx}{|dx|}+idu\right)\otimes e^++\frac{1}{2}\left(\frac{dx}{|dx|}-idu\right)\otimes e^-=\sigma^+(i(e^+-e^-)),$$
$$\pi^-\left(\frac{dx}{|dx|}\otimes e\right)=\frac{1}{2}\left(\frac{dx}{|dx|}-idu\right)\otimes e^++\frac{1}{2}\left(\frac{dx}{|dx|}+idu\right)\otimes e^-=\sigma^-(-i(e^+-e^-)),$$
$$\pi^+(du\otimes e)=-\frac{i}{2}\left(\frac{dx}{|dx|}+idu\right)\otimes e^++\frac{i}{2}\left(\frac{dx}{|dx|}-idu\right)\otimes e^-=\sigma^+(e),$$
$$\pi^-(du\otimes e)=\frac{i}{2}\left(\frac{dx}{|dx|}-idu\right)\otimes e^+-\frac{i}{2}\left(\frac{dx}{|dx|}+idu\right)\otimes e^-=\sigma^-(e).$$
\begin{prop}\label{prop4.1}
For any $C\in A^0(\p\Sigma\times I,\op{End}(\mc{E}))$, one has
$$\pi^\pm(d+Cdx)=\sigma^{\pm}\left(\frac{\p}{\p u}\pm |dx|\mbf{J}\left(\frac{\p}{\p x}+C\right)\right).$$	
\end{prop}
\begin{proof}
For any local smooth section $e$ of $\mc{E}$, one has
	\begin{align*}
&\quad\sigma^{\pm}\left(\frac{\p}{\p u}\pm |dx|\mbf{J}\left(\frac{\p}{\p x}+C\right)\right)e\\
	&=\sigma^{\pm}\left(\frac{\p e}{\p u}\pm i|dx|\left(\left(\frac{\p e}{\p x}\right)^+-\left(\frac{\p e}{\p x}\right)^-\right)\pm i|dx|\left((Ce)^+-(Ce)^-\right)\right)\\
	&=\pi^{\pm}\left(\frac{\p e}{\p u}du+dx\otimes\frac{\p e}{\p x}+dx\otimes Ce\right)\\
	&=\pi^{\pm}(d+Cdx)e,
	\end{align*}
which completes the proof.
\end{proof}
As a corollary, we get
\begin{cor}\label{cor4.2}
$d^{\pm}=\sigma^{\pm}(\frac{\p}{\p u}+A^{\pm}_{\mbf{J}})$, where $A^{\pm}_{\mbf{J}}=\pm |dx|\mbf{J}\frac{\p}{\p x}$.	
\end{cor}
When restricted on $\p\Sigma$, the metric is $g_{\p\Sigma}=g(x)dx\otimes dx$, and $|dx|=\frac{1}{\sqrt{g(x)}}$. By taking an another parameter 
$x'=\int_0^x\sqrt{g(\ell)}d\ell,$
then $dx'=\sqrt{g(x)}dx$, and so the metric $g_{\p\Sigma}=dx'\otimes dx'$, the operator $A_{\mbf{J}}$ is 
$$A_{\mbf{J}}=\mbf{J}|dx|\frac{d}{d x}= \mbf{J}\frac{d}{d x'}.$$
which is exactly the operator defined by \eqref{ABJ1}.

 Let $P_{\pm}$ denote the orthogonal projections of $\mathrm{L}^2(\p\Sigma,\mc{E})$ onto the subspace spanned by all eigenfunctions of $A^{\pm}_{\mbf{J}}$ with eigenvalues $\lambda>0$. Then $P_\pm$ are  pseudo-differential operators. Let $A^0(\Sigma,\mc{E};P_\pm)$ be the subspaces of $A^0(\Sigma,\mc{E})$ consisting of all sections $\v$ which satisfying the boundary conditions
$$P_\pm(\v|_{\p\Sigma})=0.$$
Denote by 
\begin{align}\label{dP}
	d^{\pm}_{P}:A^0(\Sigma,\mc{E}; P_\pm)\to A^0(\Sigma,\wedge^{\pm})
\end{align}
the restriction of $d^{\pm}$. By Atiyah-Patodi-Singer's index theorem \cite[Theorem 3.10]{APS}, $d^{+}_P, d^-_P$ are Fredholm operators and 
\begin{align}\label{APSformula}
\op{Index}(d^{\pm}_P)=\int_{\Sigma}\alpha_{{\pm}}(z)d\mu_g-\frac{\eta(A_{\mbf{J}}^{\pm})+\dim\op{Ker}A^{\pm}_{\mbf{J}}}{2},	
\end{align}
where $d\mu_g$ denotes the volume form of the Riemannian metric $g$ on $\Sigma$,
and $\alpha_{\pm}(z)$ is the constant term in the asymptotic expansion  (as $ t\to 0$) of
$$
\sum e^{-t \mu^{\prime}_\pm}\left|\phi_{\mu_\pm}^{\prime}(x)\right|^{2}-\sum e^{-t \mu_\pm^{\prime \prime}}\left|\phi_{\mu_\pm}^{\prime \prime}(x)\right|^{2},
$$
where $\mu^{\prime}_\pm, \phi_{\mu_\pm}^{\prime}$ denote the eigenvalues and eigenfunctions of $(d^\pm)^{*} d^\pm$ on the double $\Sigma\cup_{\p\Sigma}\Sigma$ of $\Sigma,$ and $\mu_\pm^{\prime \prime},$ $\phi_{\mu_\pm}^{\prime \prime}$ are the corresponding objects for $d^\pm (d^\pm)^{*}$.

 Since $\eta(A_{\mbf{J}}^{\pm})=\eta(\pm A_{\mbf{J}})=\pm\eta(A_{\mbf{J}})$ and
 \begin{align}\label{KerAJ}
 	\op{Ker}A^{\pm}_{\mbf{J}}=\op{Ker}|dx|\mbf{J}\frac{\p}{\p x}=\op{Ker}|dx| \frac{\p}{\p x}=\op{Ker}d|_{A^0(\p\Sigma,\mc{E})}=\mathrm{H}^0(\p\Sigma,\mc{E}),
 \end{align}
so we obtain
\begin{align*}
\op{Index}(d^-_{P})-\op{Index}(d^+_P)=\int_{\Sigma}\alpha_-(z)d\mu_g-\int_{\Sigma}\alpha_+(z)d\mu_g+\eta(A_{\mbf{J}}).
\end{align*}
By Proposition \ref{prop3.2}, the signature of the flat Hermitian vector bundle $(\mc{E},\Omega)$ is 
\begin{align}\label{sign}
\begin{split}
\op{sign}(\mc{E},\Omega)&=\int_{\Sigma}\alpha_-(z)d\mu_g-\int_{\Sigma}\alpha_+(z)d\mu_g\\
&\quad+h_{\infty}(\wedge^-)-h_{\infty}(\wedge^+)+\eta(A_{\mbf{J}}).
\end{split}
\end{align}

\subsubsection{The Atiyah-Singer integrands}

In this subsection, we will deal with the Atiyah-Singer integrands $\int_{\Sigma}\alpha_-(z)d\mu_g$ and $\int_{\Sigma}\alpha_+(z)d\mu_g$.

Let
  $$g_\Sigma=g(x,y)(dx^2+dy^2)=\frac{g(z)}{2}(dz\otimes d\b{z}+d\b{z}\otimes dz)$$ be a Riemannian metric on the surface $\Sigma$, and is a  product metric on the collar neighborhood $\p\Sigma\times I$ of $\p\Sigma$, 
  where $z=x+iy$ denotes the holomorphic coordinate of $\Sigma$, and one has
$*dz=-idz$ and $*d\b{z}=id\b{z}$. 

 Following \cite{APS}, we will consider the double $\Sigma\cup_{\p\Sigma}\Sigma$ of $\Sigma$, which is a closed surface.
\begin{center}
\begin{tikzpicture}[x=1cm,y=1cm]

\begin{scope}[shift={(2,0)}, thick]
\clip(-1.8,-2)rectangle(3,2);
\draw (0,0) circle [x radius=2, y radius=1];
\end{scope}

\draw[shift={(2,0)}, yscale=cos(70), thick] (-1,0) arc (-180:0:1);
\draw[shift={(2,0)}, yscale=cos(70), thick] (-0.7,-0.6) arc (180:0:0.7);

\draw [xscale=cos(70), dashed, thick] (0,-0.3) arc (-90:90:0.3);
\draw [xscale=cos(70), thick] (0,0.3) arc (90:270:0.3);

\draw [thick] (0,0.3) .. controls (0.1,0.3) and (0.1,0.34) .. (0.21,0.446);
\draw [thick] (0,-0.3) .. controls (0.1,-0.3) and (0.1,-0.34) .. (0.21,-0.446);

\begin{scope}[shift={(-4,0)}, thick]
\clip(-3,-2)rectangle(1.8,2);
\draw (0,0) circle [x radius=2, y radius=1];
\end{scope}

\draw[shift={(-4,0)}, yscale=cos(70), thick] (-1,0) arc (-180:0:1);
\draw[shift={(-4,0)}, yscale=cos(70), thick] (-0.7,-0.6) arc (180:0:0.7);

\draw [shift={(-2,0)}, xscale=cos(70), dashed, thick] (0,-0.3) arc (-90:90:0.3);
\draw [shift={(-2,0)}, xscale=cos(70), thick] (0,0.3) arc (90:270:0.3);

\draw [thick,shift={(-2,0)}] (0,0.3) .. controls (-0.1,0.3) and (-0.1,0.34) .. (-0.21,0.446);
\draw [thick,shift={(-2,0)}] (0,-0.3) .. controls (-0.1,-0.3) and (-0.1,-0.34) .. (-0.21,-0.446);

\draw[thick] (-2,-0.3) -- (0,-0.3);
\draw[thick] (-2,0.3) -- (0,0.3);

\draw [shift={(-1,0)}, xscale=cos(70), dashed, thick] (0,-0.3) arc (-90:90:0.3);
\draw [shift={(-1,0)}, xscale=cos(70), thick] (0,0.3) arc (90:270:0.3);

\draw (-1,-0.5) node{{\tiny $0$}};
\draw (-1.4,0) node{{\tiny $\partial \Sigma$}};
\draw (0,-0.5) node{{\tiny $1$}};
\draw (-2,-0.5) node{{\tiny $1$}};
\draw (2.7, 0.5) node{{\tiny $\Sigma$}};
\draw (-4.7, 0.5) node{{\tiny $\Sigma$}};
\draw (-1,-1.3) node{{$\Sigma\cup_{\partial\Sigma}\Sigma$}};

\end{tikzpicture}
\end{center}
The vector bundle $\mc{E}$ and the  operators $d^{+}, d^-$ can be extended canonically on the double $\Sigma\cup_{\p\Sigma}\Sigma$. 
Let $\mc{F}=\mc{F}^+\oplus\mc{F^-}$ be a $\mb{Z}_2$-graded vector bundle over the double $\Sigma\cup_{\p\Sigma}\Sigma$, where 
$$\mc{F}^+:=\mc{E},\quad \mc{F}^-:=\wedge^-.$$
Let $D:\Gamma(\Sigma\cup_{\p\Sigma}\Sigma,\mc{F})\to \Gamma(\Sigma\cup_{\p\Sigma}\Sigma,\mc{F})$ be an  operator defined as follows:
\begin{align}\label{Dirac operator}
D=\left(\begin{array}{cc}
0 & D^-=(d^-)^*	\\
D^+=d^- &0
\end{array}
\right):\Gamma(\Sigma\cup_{\p\Sigma}\Sigma,\mc{F}^{\pm})\to \Gamma(\Sigma\cup_{\p\Sigma}\Sigma,\mc{F}^{\mp}).
\end{align}
 \begin{prop}
 $D$ is a self-adjoint Dirac operator.
  \end{prop}
\begin{proof}
Since $D^-=(D^+)^*$, so $D$ is self-adjoint. On the other hand, one has
\begin{align*}
D^2=\left(\begin{array}{cc}
	(d^-)^*d^- &0\\
	0&d^-(d^-)^*
\end{array}
\right)	
\end{align*}
so $D^2$ is a generalized Laplacian. In fact, for any local section $s=fe$ of $\mc{E}$, where $e$ is a local parallel section, i.e.  $d e=0$. Then 
\begin{align*}
(d^-)^*d^-s &=(d^-)^*\frac{1-*\mbf{J}}{2}(\p f+\b{\p}f)e\\
&=\frac{1}{2}(d^-)^*\left((\p f+\b{\p}f)e-(-i\p f+i\bar{\p} f)\mbf{J}e\right),
\end{align*}
 since $(d^-)^*=-2g(z)^{-1}(\p_{z}i_{\frac{\p}{\p\b{z}}}+\p_{\b{z}}i_{\frac{\p}{\p z}})+\text{zeroth terms}$, so the second order terms of $(d^-)^*d^-s$ is 
 $-2g(z)^{-1}\frac{\p^2f}{\p z\p{\b{z}}}e$. Similar for the local sections of $\wedge^-$. Thus $D$ is a Dirac operator.
\end{proof}
By the definition of $\alpha_\pm(z)$, one has
\begin{align*}
\int_{\Sigma}\alpha_-(z)d\mu_g &=\lim_{t\to 0}\int_{\Sigma}\left(\operatorname{tr}\left(e^{-t (d^{-})^{*} d^{-}}(z, z)\right)-\operatorname{tr}\left(e^{-t d^{-}  (d^{-})^*}(z, z)\right)	\right)\\
&=\lim_{t\to 0}\int_{\Sigma}\op{Str}\langle z|e^{-tD^2}|z\rangle d\mu_g,
\end{align*}
where $d\mu_g=\frac{i}{2}g(z)dz\wedge d\b{z}$ and $\op{Str}$ denotes the supertrace, see e.g. \cite[Section 1.5]{BGV} for its definition.

For any $\mbf{J}\in \mc{J}(\mc{E},\Omega)$, denote by $\mc{E}=\mc{E}^+\oplus \mc{E}^-$ the decomposition corresponding to the $\pm i$-eigenspaces of $\mbf{J}$.
\begin{defn}\label{connection}
	A connection $\n$ on $\mc{E}$ is called a {\it peripheral connection} if it satisfies the following conditions on a collar neighborhood of $\p\Sigma$:
	\begin{itemize}
\item[(i)] 	$\n=d+C(x)dx$ for some $C=C(x)\in A^0(\p\Sigma,\op{End}(\mc{E}))$;
\item[(ii)] $[\n,\mbf{J}]=0$;
\item[(iii)] $\n$ preserves the Hermitian form $\Omega$.
\end{itemize}
\end{defn}

\begin{rem}
Like in the proof of Proposition \ref{prop2.0}, the condition (iii) is equivalent to  $C(x)\in \mf{u}(p,q)$, i.e., $\o{C(x)}^\top\Omega+\Omega C(x)=0$, which implies that   
 the operator $|dx|\mbf{J}(\frac{\p}{\p x}+C(x))$ is a $\mb{C}$-linear formally self-adjoint elliptic first order differential operator  in the space $A^0(\p\Sigma,\mc{E})$.	
\end{rem}

\begin{rem}\label{remexample}
Note that $\mc{E}^+$ is a Hermitian vector bundle with a Hermitian metric $2i\Omega(\cdot,\mbf{J}\cdot)=2\Omega(\cdot,\cdot)$. Let $\n^+$ be a Hermitian connection on $\mc{E}^+$ which depends only on $x$ on a collar neighborhood of $\p\Sigma$. This can be done since we can take $\mbf{J}=\mbf{J}(x)$ near $\p\Sigma$. Similarly, there exists such a Hermitian connection $\n^-$ on the Hermitian vector bundle $(\mc{E}^-,-\Omega)$. Hence $\n=\n^+\oplus \n^-$ defines a connection on $\mc{E}=\mc{E}^+\oplus \mc{E}^-$. One can check easily that $\n$ is a peripheral connection on $\mc{E}$.
\end{rem}

Let $\n$ be any peripheral connection on $\mc{E}$, when restricted to a small collar neighborhood of $\p\Sigma$,  we assume it has the form $\n=d+C(x)dx$. Without loss of generality, we assume that it satisfies the above conditions $(i)-(iii)$ on $\p\Sigma\times [0,1)$. Denote 
$\n^{\mc{E}^+}:=\n|_{\mc{E}^+},\quad \n^{\mc{E}^-}:=\n|_{\mc{E}^-}$, and set
$$\n^{\mc{E}}:=\n^{\mc{E}^+}\oplus \n^{\mc{E}^-}.$$
Then $\n^{\mc{E}}$ is a connection on $\mc{E}$, and $\n^{\mc{E}}=\n$ on $\p\Sigma\times [0,1)$.

Now we consider the operator
$$D^{\mc{F}}=\pi^-\n^{\mc{E}}+(\pi^-\n^{\mc{E}})^*$$
on the superbundle $\mc{F}=\mc{E}\oplus \wedge^-$. One can check that  $D^{\mc{F}}$ is also a self-adjoint Dirac operator. Moreover, by Proposition \ref{prop4.1}, one has
$$\pi^-\n^{\mc{E}}=\pi^-\n=\sigma^-\left(\frac{\p}{\p u}-\mbf{J}|dx|\left(\frac{\p}{\p x}+C(x)\right)\right)$$
 on $\p\Sigma\times[0,1)$, and the operator $-\mbf{J}|dx|(\frac{\p}{\p x}+C(x))$ is a formally self-adjoint elliptic first order differential operator.
\begin{lemma}\label{lemma4.7}
It holds
	$$\lim_{t\to 0}\int_{\Sigma}\op{Str}\langle z|e^{-tD^2}|z\rangle d\mu_g=\lim_{t\to 0}\int_{\Sigma}\op{Str}\langle z|e^{-t(D^{\mc{F}})^2}|z\rangle d\mu_g.$$
\end{lemma}
\begin{proof}
	Denote  $\n_s=(1-s)d+s\n^{\mc{E}}$, $s\in [0,1]$. Then 
	$$\pi^-\n_s=(1-s)\pi^-d+s\pi^-\n^{\mc{E}}=\sigma^-\left(\frac{\p}{\p u}-\mbf{J}|dx|\frac{\p}{\p x}-s\mbf{J}|dx|C(x)\right).$$ 
The first order operator $-\mbf{J}|dx|\frac{\p}{\p x}-s\mbf{J}|dx|C(x)$ is formally self-adjoint and elliptic.
	The adjoint operator of $\pi^-\n_s$ is given by 
	$(\pi^-\n_s)^*=(1-s)(\pi^-d)^*+s(\pi^-\n^{\mc{E}})^*$, so
	$$D_s=(1-s)D+sD^{\mc{F}}.$$
From   \cite[Corollary 2.50]{BGV}, one has 
	\begin{align*}
		\lim_{t\to 0}\frac{\p}{\p s}\int_{\Sigma}\op{Str}\langle z|e^{-tD^2_s}|z\rangle d\mu_g=\lim_{t\to 0}-t\int_{\Sigma}\op{Str}\left\langle z|\frac{\p D_s^2}{\p s}e^{-tD_s^2}|z\right\rangle d\mu_g=0.
	\end{align*}
Hence
$$\lim_{t\to 0}\int_{\Sigma}\op{Str}\langle z|e^{-tD^2}|z\rangle d\mu_g=\lim_{t\to 0}\int_{\Sigma}\op{Str}\langle z|e^{-t(D^{\mc{F}})^2}|z\rangle d\mu_g.$$
\end{proof}
Now we will calculate the term  $\lim_{t\to 0}\int_{\Sigma}\op{Str}\langle z|e^{-t(D^{\mc{F}})^2}|z\rangle d\mu_g$.  Firstly, we need find the Clifford connection $\n^{\mc{F}}$ on $\mc{F}$ such that $D^{\mc{F}}=c\circ\n^{\mc{F}}$. The adjoint operator $(\n^{\mc{E}})^*$ is given by
$$(\n^{\mc{E}})^{*}=-2g(z)^{-1}i_{\frac{\p}{\p z}}\n^{\mc{E}}_{\frac{\p}{\p \b{z}}}-2g(z)^{-1}i_{\frac{\p}{\p \b{z}}}\n^{\mc{E}}_{\frac{\p}{\p {z}}}.$$
The Dirac operator $D^{\mc{F}}$ induces a Clifford action of $T^*\Sigma$ on $\mc{F}$ by
$$[D^{\mc{F}},f]=c(df)$$
for any smooth function $f$. Note that $\wedge^-=\wedge^{1,0}T^*\Sigma\otimes \mc{E}^{-}\oplus\wedge^{0,1}T^*\Sigma\otimes \mc{E}^+$, so
$$c(df)=-2g(z)^{-1}\frac{\p f}{\p\b{z}}i_{\frac{\p}{\p z}}-2g(z)^{-1}\frac{\p f}{\p z}i_{\frac{\p}{\p\b{z}}}$$
when acting on $\wedge^-$, and 
$$c(df)=\p f\otimes p^-+\b{\p}f\otimes p^+$$
when acting on $\mc{E}=\mc{E}^+\oplus \mc{E}^-$. Since 
$$\mc{F}=\mc{E}\oplus \wedge^-=\mc{E}\oplus \wedge^{1,0}T^*\Sigma\otimes \mc{E}^{-}\oplus\wedge^{0,1}T^*\Sigma\otimes \mc{E}^+,$$
there exists a natural connection on $\mc{F}$ induced from the connections on $\mc{E}$ and $T^*\Sigma$, we denote this connection by $\n^{\mc{F}}$.
\begin{lemma}
$\n^{\mc{F}}$ is a Clifford connection and
$$D^{\mc{F}}=c(dz)\n^{\mc{F}}_{\frac{\p}{\p z}}+c(d\b{z})\n^{\mc{F}}_{\frac{\p}{\p\b{z}}}.$$ 	
\end{lemma}
\begin{proof}
	$\n^{\mc{F}}$ is a Clifford connection if $[\n^{\mc{F}}_X,a]=\n_Xa$ for any local section $a$ of Clifford bundle $ C(\Sigma\cup_{\p\Sigma}\Sigma)$. Let $\sigma=c(\bullet)\cdot 1:C(\Sigma\cup_{\p\Sigma}\Sigma)\to \wedge^*T^*\Sigma$ denote the symbol map, which identifies $C(\Sigma\cup_{\p\Sigma}\Sigma)$ with $\wedge^*T^*\Sigma$.
	By a direct checking, one has 
	$$[\n^{\mc{F}}_{\frac{\p}{\p z}},c(dz)]=-\p_z\log g(z)c(dz)=\n_\frac{\p}{\p z}c(dz),\quad [\n^{\mc{F}}_{\frac{\p}{\p z}},c(d\b{z})]=0=\n_{\frac{\p}{\p z}}c(d\b{z}),$$
	and so
	\begin{align*}
	[\n^{\mc{F}}_{\frac{\p}{\p z}},c(dz) c(d\b{z})]	
	=[\n^{\mc{F}}_{\frac{\p}{\p z}},c(dz)] c(d\b{z})=\n_\frac{\p}{\p z}c(dz)c(d\b{z})=\n_\frac{\p}{\p z}(c(dz)c(d\b{z})).
	\end{align*}
	For any smooth function $f$, one has 
$$[\n^{\mc{F}}_X,c(f)]=\n_X c(f).$$	
Thus $\n^{\mc{F}}$ is a Clifford connection. If $s\in \Gamma(\Sigma\cup_{\p\Sigma}\Sigma,\mc{E})$, then 
\begin{align*}
c(dz)\n^{\mc{F}}_{\frac{\p}{\p z}}s+c(d\b{z})\n^{\mc{F}}_{\frac{\p}{\p\b{z}}}s&=dz\otimes \n^{\mc{F}}_{\frac{\p}{\p z}}s^-+d\b{z}\otimes \n^{\mc{F}}_{\frac{\p}{\p\b{z}}}s^+\\
&=\pi^-\n^\mc{F}s=D^{\mc{F}}s.
\end{align*}
If $dz\otimes s^-\in\wedge ^-$, then 
\begin{align*}
	&\quad c(dz)\n^{\mc{F}}_{\frac{\p}{\p z}}(dz\otimes s^-)+c(d\b{z})\n^{\mc{F}}_{\frac{\p}{\p\b{z}}}(dz\otimes s^-)\\
	&=-2g(z)^{-1}i_{\frac{\p}{\p z}}\n^{\mc{F}}_{\frac{\p}{\p \b{z}}}dz\otimes s^-\\
	&=-2g(z)^{-1}\n^{\mc{F}}_{\frac{\p}{\p \b{z}}}s^-=(\n^{\mc{F}})^*(dz\otimes s^-)=D^{\mc{F}}(dz\otimes s^-).
\end{align*}
	Similarly, for $d\b{z}\otimes s^+\in \wedge^-$, one has $$(c(dz)\n^{\mc{F}}_{\frac{\p}{\p z}}+c(d\b{z})\n^{\mc{F}}_{\frac{\p}{\p\b{z}}})(d\b{z}\otimes s^+)=D^{\mc{F}}(d\b{z}\otimes s^+).$$
	The proof is complete.
\end{proof}
The Clifford module $\mc{F}$ has the following decomposition
\begin{align*}
\mc{F}=\mc{F}^+\oplus \mc{F}^-=\mc{E}^+\oplus\mc{E}^-\oplus \mc{E}^+\otimes \wedge^{0,1}T^*\Sigma\oplus \mc{E}^-\otimes \wedge^{1,0}T^*\Sigma,	
\end{align*}
where $\mc{F}^+=\mc{E}^+\oplus\mc{E}^-$ and $\mc{F}^-=\mc{E}^+\otimes \wedge^{0,1}T^*\Sigma\oplus \mc{E}^-\otimes \wedge^{1,0}T^*\Sigma$. Denote
$$S=S^+\oplus S^-=\mb{C}\oplus\wedge^{0,1}T^*\Sigma=\wedge^{0,*}T^*\Sigma.$$
 Then the complex module $\mc{F}$ is isomorphic to
$$\mc{F}\cong \mc{W}\otimes S$$
where $\mc{W}=\mc{E}^+\oplus \mc{E}^-\otimes (\wedge^{0,1}T^*\Sigma)^*\cong \mc{E}^+\oplus \mc{E}^-\otimes\wedge^{1,0}T^*\Sigma$. Let $\Gamma$ be the chirality operator, which is an element  in $C(\Sigma)\otimes \mb{C}\cong \op{End}(S)$, and is $+\op{Id}$ when acting on $S^+=\mb{C}$, is $-\op{Id}$ when acting on $S^-=\wedge^{0,1}T^*\Sigma$. Thus it induces an endomorphism of $\mc{F}\cong \mc{W}\otimes S$ by the action
$\op{Id}_{\mc{W}}\otimes \Gamma$, we also denote it by $\Gamma$. Thus 
\begin{align*}
\Gamma=\op{Id}_{\mc{E}^+}\oplus-\op{Id}_{\mc{E}^-}\oplus -\op{Id}_{\mc{E}^+\otimes \wedge^{0,1}T^*\Sigma}	\oplus \op{Id}_{\mc{E}^-\otimes \wedge^{1,0}T^*\Sigma}.
\end{align*}
The Clifford connection $\n^{\mc{F}}$ is given by 
\begin{multline*}
\n^{\mc{F}}=\n^{\mc{E}^+}\oplus \n^{\mc{E}^-}\oplus (\n^{\mc{E}^+}\otimes \op{Id}_{\wedge^{0,1}T^*\Sigma}+\op{Id}_{\mc{E}^+}\otimes \n^{\wedge^{0,1}T^*\Sigma})\\\oplus ( \n^{\mc{E}^-}\otimes \op{Id}_{\wedge^{1,0}T^*\Sigma}+\op{Id}_{\mc{E}^-}\otimes \o{ \n^{\wedge^{0,1}T^*\Sigma}}),
\end{multline*}
and the curvature is 
\begin{multline*}
	(\n^{\mc{F}})^2=R^{\mc{E}^+}\oplus R^{\mc{E}^-}\oplus (R^{\mc{E}^+}\otimes \op{Id}_{\wedge^{0,1}T^*\Sigma}+R^{T^{1,0}\Sigma}\cdot\op{Id}_{\mc{E}^+\otimes\wedge^{1,0}T^*\Sigma} )\\\oplus ( R^{\mc{E}^-}\otimes \op{Id}_{\wedge^{1,0}T^*\Sigma}-R^{T^{1,0}\Sigma}\cdot\op{Id}_{\mc{E}^-\otimes\wedge^{1,0}T^*\Sigma} ),
\end{multline*}
where $R^{T^{1,0}\Sigma}$ is a two-form on $\Sigma$, and $\frac{i}{2\pi}R^{T^{1,0}\Sigma}$ denotes the first Chern form of $T^{1,0}\Sigma$.
Denote by $R^{\mc{F}}$ the action of the Riemannian curvature $R$ of $\Sigma$ on the bundle $\mc{F}$, which is given by
\begin{align*}
R^{\mc{F}}&:=\frac{1}{4}\left(R\frac{\p}{\p z},\frac{\p}{\p\b{z}}\right)c(dz)c(d\b{z})+\frac{1}{4}\left(R\frac{\p}{\p\b{z}},\frac{\p}{\p z}\right)	c(d\b{z})c(dz)\\
&=\frac{g(z)}{8}R^{T^{1,0}\Sigma}(c(dz)c(d\b{z})-c(d\b{z})c(dz))\\
&=\frac{1}{4}R^{T^{1,0}\Sigma}(-\op{Id}_{\mc{E}^+}\oplus \op{Id}_{\mc{E}^-}\oplus\op{Id}_{\mc{E}^+\otimes \wedge^{0,1}T^*\Sigma}\oplus-\op{Id}_{\mc{E}^-\otimes \wedge^{1,0}T^*\Sigma}).
\end{align*}
From \cite[Proposition 3.43]{BGV}, the curvature $F^{\mc{F}/S}$ is given by
\begin{align*}
F^{\mc{F}/S}&=(\n^{\mc{F}})^2-R^{\mc{F}}\\
&=	(R^{\mc{E}^+}+\frac{1}{4}R^{T^{1,0}\Sigma}\op{Id}_{\mc{E}^+})\oplus(R^{\mc{E}^-}-\frac{1}{4}R^{T^{1,0}\Sigma}\op{Id}_{\mc{E}^-})\\
&\quad\oplus (R^{\mc{E}^+}\otimes \op{Id}_{\wedge^{0,1}T^*\Sigma}+\frac{3}{4}\op{Id}_{\mc{E}^+\otimes \wedge^{0,1}T^*\Sigma} R^{T^{1,0}\Sigma})\\
&\quad \oplus( R^{\mc{E}^-}\otimes \op{Id}_{\wedge^{1,0}T^*\Sigma}-\frac{3}{4}\op{Id}_{\mc{E}^-\otimes \wedge^{1,0}T^*\Sigma} R^{T^{1,0}\Sigma}).
\end{align*}
Thus $\Gamma F^{\mc{F}/S}$ is 
\begin{align*}
\Gamma F^{\mc{F}/S}
&=	(R^{\mc{E}^+}+\frac{1}{4}R^{T^{1,0}\Sigma}\op{Id}_{\mc{E}^+})\oplus(-R^{\mc{E}^-}+\frac{1}{4}R^{T^{1,0}\Sigma}\op{Id}_{\mc{E}^-})\\
&\quad\oplus (-R^{\mc{E}^+}\otimes \op{Id}_{\wedge^{0,1}T^*\Sigma}-\frac{3}{4}\op{Id}_{\mc{E}^+\otimes \wedge^{0,1}T^*\Sigma} R^{T^{1,0}\Sigma})\\
&\quad \oplus( R^{\mc{E}^-}\otimes \op{Id}_{\wedge^{1,0}T^*\Sigma}-\frac{3}{4}\op{Id}_{\mc{E}^-\otimes \wedge^{1,0}T^*\Sigma} R^{T^{1,0}\Sigma}).
\end{align*}
Hence the supertrace $\op{Str}_{\mc{F}/S}(F^{\mc{F}/S})$ is 
\begin{align*}
	\op{Str}_{\mc{F}/S}(F^{\mc{F}/S})&=\frac{1}{2}\op{Str}_{\mc{F}}(\Gamma F^{\mc{F}/S})=\frac{1}{2}\op{Tr}_{\mc{F}^+}(\Gamma F^{\mc{F}/S})-\frac{1}{2}\op{Tr}_{\mc{F}^-}(\Gamma F^{\mc{F}/S})\\
	&=\frac{1}{2}\left(\op{Tr}(R^{\mc{E}^+})+\frac{1}{4}p R^{T^{1,0}\Sigma}\right)+\frac{1}{2}\left(-\op{Tr}(R^{\mc{E}^-})+\frac{1}{4}q R^{T^{1,0}\Sigma}\right)\\
	&\quad -\frac{1}{2}\left(-\op{Tr}(R^{\mc{E}^+})-\frac{3}{4}p R^{T^{1,0}\Sigma}\right)-\frac{1}{2}\left(\op{Tr}(R^{\mc{E}^-})-\frac{3}{4}q R^{T^{1,0}\Sigma}\right)\\
	&=\op{Tr}(R^{\mc{E}^+})-\op{Tr}(R^{\mc{E}^-})+\frac{p+q}{2}R^{T^{1,0}\Sigma}.
\end{align*}
By the local index theorem, see e.g. \cite[Theorem 8.34]{Melrose}, one has 
\begin{align*}
	&\quad \lim_{t\to 0}\op{Str}\langle z|e^{-t(D^{\mc{F}})^2}|z\rangle d\mu_g \\
	&=\left[(2\pi i)^{-1}\det\left(\frac{R/2}{\sinh(R/2)}\right)\op{Str}_{\mc{F}/S}(\exp(-F^{\mc{F}/S}))\right]^{(1,1)}\\
	&=\frac{i}{2\pi}\op{Str}_{\mc{F}/S}(F^{\mc{F}/S}),
\end{align*}
since $\widehat{A}(\Sigma)=\det\left(\frac{R/2}{\sinh(R/2)}\right)\in A^{4*}(\Sigma,\mb{R})$. Thus,
\begin{align}\label{Ind-}
\begin{split}
\int_{\Sigma}\alpha_-(z)d\mu_g &=\frac{i}{2\pi}\int_{\Sigma}	(\op{Tr}(R^{\mc{E}^+})-\op{Tr}(R^{\mc{E}^-})+\frac{p+q}{2}R^{T^{1,0}\Sigma})\\
&=\int_\Sigma \left(c_1(\mc{E}^+,\n^{\mc{E}^+})-c_1(\mc{E}^-,\n^{\mc{E}^-})+\frac{p+q}{2}c_1(T^{1,0}\Sigma,\n^{T^{1,0}\Sigma})\right).
\end{split}
\end{align}
Similarly, one has
\begin{align}\label{Ind+}
\int_{\Sigma}\alpha_+(z)d\mu_g =\int_\Sigma\left(-c_1(\mc{E}^+,\n^{\mc{E}^+})+c_1(\mc{E}^-,\n^{\mc{E}^-})+\frac{p+q}{2}c_1(T^{1,0}\Sigma,\n^{T^{1,0}\Sigma})\right).
\end{align}
Therefore,
\begin{align}\label{4.1}
	\int_{\Sigma}\alpha_{-}(z)d\mu_g-\int_{\Sigma}\alpha_{+}(z)d\mu_g=2\int_\Sigma\left(c_1(\mc{E}^+,\n^{\mc{E}^+})-c_1(\mc{E}^-,\n^{\mc{E}^-})\right).
\end{align}

Note that on $\p\Sigma\times[0,1]$, $\n^{\mc{E}^+}=\n|_{\mc{E}^+}$ is a flat connection, which follows that $c_1(\mc{E}^+,\n^{\mc{E}^+})=0$ on $\p\Sigma\times [0,1)$. Denote by
\begin{align}\label{FCC}
	[c_1(\mc{E}^+,\n^{\mc{E}^+})]_c\in \mathrm{H}^2_{\text{dR,comp}}(\Sigma_o ,\mb{R})
\end{align}
the de Rham cohomology class of $c_1(\mc{E}^+,\n^{\mc{E}^+})$ with compact support, see e.g. \cite[Chapter 1]{BT} for the definition of de Rham cohomology with compact support,  where $\Sigma_o:=\Sigma\backslash\p\Sigma$. 
On the other hand, since $$\int_{\Sigma}c_1(T^{1,0}\Sigma,\n^{T^{1,0}\Sigma})=\chi(\Sigma)=2-2g-n,$$ where $n$ denotes the number of components in $\p\Sigma$, so

\begin{prop}\label{prop410}
	For  any peripheral connection $\n$ on $\mc{E}$, one has
	\begin{equation*}
  \int_{\Sigma}\alpha_{\pm}(z)d\mu_g=\mp\int_\Sigma \left(c_1(\mc{E}^+,\n|_{\mc{E}^+})-c_1(\mc{E}^-,\n|_{\mc{E}^-})\right)+\frac{\dim E}{2}\chi(\Sigma).
  \end{equation*}     
\end{prop}
 
Substituting \eqref{4.1} into \eqref{sign}, one gets
\begin{align}\label{sign1}
\op{sign}(\mc{E},\Omega)=2\int_\Sigma\left(c_1(\mc{E}^+,\n^{\mc{E}^+})-c_1(\mc{E}^-,\n^{\mc{E}^-})\right)+h_{\infty}(\wedge^-)-h_{\infty}(\wedge^+)+\eta(A_{\mbf{J}}).	
\end{align}

\subsubsection{Limiting values of extended $\mathrm{L}^2$-sections}
\label{limiting}

In this subsection, we will calculate the terms $h_{\infty}(\wedge^{+})$, $h_{\infty}(\wedge^{-})$, and show that $h_{\infty}(\wedge^-)=h_{\infty}(\wedge^+)$.

By Atiyah-Patodi-Singer's index theorem \cite[Theorem 3.10]{APS}, one has
\begin{align}\label{3.8}
\op{Index}(d^-_P)=	\int_{\Sigma}\alpha_{-}(z)dz-\frac{\dim \mathrm{H}^0(\p\Sigma,\mc{E})- \eta(A_{\mbf{J}})}{2}.
\end{align}
On the other hand, we have
\begin{align}\label{3.9}
\op{Index}(d^-_P)+h_{\infty}(\wedge^-)=\mathrm{L}^2\op{Index}(d^-).	
\end{align}
Following \cite[(3.20)--(3.25)]{APS},  we consider the operator $(d^-)^*$, then 
$$(d^-)^*=-(\sigma^-)^{-1}\left(\frac{\p}{\p u}+\sigma^-A_{\mbf{J}}(\sigma^-)^{-1}\right).$$
Since $\eta(\sigma^-A_{\mbf{J}}(\sigma^-)^{-1})=\eta(A_{\mbf{J}})$, so
\begin{align*}
\op{Index}(d^-)^*_P+h_{\infty}(\mc{E})=\mathrm{L}^2\op{Index}(d^-)^*=-\mathrm{L}^2\op{Index}(d^-),	
\end{align*}
\begin{align*}
	\op{Index}(d^-)^*_P=-\int_{\Sigma}\alpha_-(z)dz-\frac{\dim \mathrm{H}^0(\p\Sigma,\mc{E})+\eta(A_{\mbf{J}})}{2}.
\end{align*}
Combining with the above equalities, we have
\begin{equation}\label{hinfE}
  h_{\infty}(\mc{E})+h_{\infty}(\wedge^-)=\dim \mathrm{H}^0(\p\Sigma,\mc{E}).
\end{equation}

Denote by $\ms{K}^-$ the set of all  extended $\mathrm{L}^2$-solutions of $(d^-)^*\phi=0$ in  $\wedge^-$, that is,  for any $\phi\in \ms{K}^-$, one has  $d^*\phi=0$ and $\phi$ is with valued in $\wedge^-$, and in the cylinder $\p\Sigma\times (-\infty,u_0]$ for some large negative $u_0$, we can write 
$$\phi=\psi+\theta$$
where $\psi=\psi_0+\psi_1du\in \op{Ker}(\sigma^-A_{\mbf{J}}(\sigma^-)^{-1})$ and $\theta\in \op{Ker}(d^-)^*\cap \mathrm{L}^2(\widehat{\Sigma},\wedge^-)$ is a $\mathrm{L}^2$-section in $\wedge^-$ (hence decaying exponentially). From Proposition \ref{prop3}, one has $d\theta=0$. For any $\mbf{J}\in \mc{J}(\mc{E},\Omega)$ and extend it to the vector bundle  $\mc{E}$ over $\widehat{\Sigma}$ such that  $\mbf{J}=\mbf{J}(x)$ on $\p\Sigma\times (-\infty,u_0]$. By the definition of $\sigma^-$, then $[d,\sigma^-]=0$. Since $\psi\in \op{Ker}(\sigma^-A_{\mbf{J}}(\sigma^-)^{-1})$ and by \eqref{KerAJ} so $(\sigma^-)^{-1}\psi\in \op{Ker}(A_{\mbf{J}})=\op{Ker}d$. Hence $d\psi=d\sigma^-(\sigma^-)^{-1}\psi=\sigma^-d((\sigma^-)^{-1}\psi)=0$, which follows all elements of $\ms{K}^-$ are harmonic.
 Denote by $\delta^-:\mathscr{K}^-\to \mathrm{H}^1(\Sigma,\mc{E})$ the natural map, then  
$$*\mbf{J}(\psi_0+\psi_1du)=-(\psi_0+\psi_1du),\quad *\mbf{J}(\theta)=-\theta.$$
If moreover, $\psi_0=0$, then $*\mbf{J}(\psi_1du)=-\psi_1du$. However,  since $*du=-\frac{dx}{|dx|}$, so we conclude that 
$\psi_1=0$, and so $\phi=\theta\in\op{Ker}(d^-)^*\cap \mathrm{L}^2(\widehat{\Sigma},\wedge^-)$, then 
$$\op{Ker}(\iota^*\delta^-)=\op{Ker}(d^-)^*\cap \mathrm{L}^2(\widehat{\Sigma},\wedge^-),$$
where $\iota^*: \mathrm{H}^{1}(\Sigma,\mc{E})\to \mathrm{H}^{1}(\p\Sigma,\mc{E})$ is the  induced map on cohomology by restriction.
By the definition of limiting values of extended $\mathrm{L}^2$-sections \cite{APS}, $h_{\infty}(\wedge^-)$ is the dimension of subspace of all  $\psi$, so $h_{\infty}(\wedge^-)=\dim (\mathscr{K}^-/\op{Ker}(d^-)^*\cap \mathrm{L}^2(\widehat{\Sigma},\wedge^-))$, and we have 
\begin{align*}
h_{\infty}(\wedge^-)=\dim (\mathscr{K}^-/\op{Ker}(\iota^*\delta^-))=\dim\op{Im}(\iota^*\delta^-)\leq \dim\op{Im}(\iota^*).
\end{align*}
\begin{lemma}
We have 
$$\dim\op{Im}(\iota^*)=\dim \mathrm{H}^0(\p\Sigma,\mc{E})-\dim \mathrm{H}^0(\Sigma,\mc{E}). $$	
\end{lemma}
\begin{proof}
From the following exact sequence
\begin{align*}
\cdots \to \mathrm{H}^{1}(\Sigma,\mc{E}) \stackrel{\iota^{*}}{\longrightarrow} \mathrm{H}^{1}(\p\Sigma,\mc{E}) \stackrel{\alpha^{*}}{\rightarrow} \mathrm{H}^{2}(\Sigma,\p\Sigma,\mc{E}) \stackrel{\beta^{*}}{\rightarrow} \mathrm{H}^{2}(\Sigma,\mc{E}) \to 0, 
\end{align*}
one has
$$ \mathrm{H}^1(\p\Sigma,\mc{E})/\op{Im}\iota^*\simeq \dim \mathrm{H}^1(\p\Sigma,\mc{E})/\op{Ker}\alpha^*\simeq \op{Im}\alpha^*\simeq \op{Ker}\beta^*,$$
and 
$$\mathrm{H}^2(\Sigma,\p\Sigma,\mc{E})/\op{Ker}\beta^*\simeq\op{Im}\beta^*\simeq \mathrm{H}^2(\Sigma,\mc{E}).$$
Hence 
\begin{align*}
\dim\op{Im}(\iota^*)
&=\dim \mathrm{H}^1(\p\Sigma,\mc{E})-(\dim \mathrm{H}^2(\Sigma,\p\Sigma,\mc{E})-\dim \mathrm{H}^2(\Sigma,\mc{E}))\\
&=\dim \mathrm{H}^0(\p\Sigma,\mc{E})-(\dim \mathrm{H}^0(\Sigma,\mc{E})-\dim \mathrm{H}^2(\Sigma,\mc{E}))\\	
&=	\dim \mathrm{H}^0(\p\Sigma,\mc{E})-\dim \mathrm{H}^0(\Sigma,\mc{E}),
\end{align*}
where the second equality uses Poincar\'e duality, and the last equality follows from the fact that $\dim \mathrm{H}^2(\Sigma,\mc{E})=0$.	
\end{proof}
Hence
\begin{align}\label{hinf}
	h_{\infty}(\wedge^-)\leq \dim\op{Im}(\iota^*)=\dim \mathrm{H}^0(\p\Sigma,\mc{E})-\dim \mathrm{H}^0(\Sigma,\mc{E}).
\end{align}

On the other hand, we can also consider the term $h_{\infty}(\mc{E})$. Denote by $\mathscr{K}_0^-$ the set of all extended  $\mathrm{L}^2$ solutions of $d^-\phi=0$ in $\mc{E}$, so that for any $\phi\in \mathscr{K}_0^-$, one has $d\phi=0$. In the cylinder $\p\Sigma\times (-\infty,u_0]$, we can write
$\phi=\psi+\theta$
where $\psi\in\op{Ker}(A_{\mbf{J}})$ is a harmonic section on $\p\Sigma$ and $\theta$ is a $\mathrm{L}^2$ harmonic section. 
From \eqref{H0} and \eqref{KerAJ}, one has $d\psi=d\theta=0$. Hence $d\phi=0$ for any $\phi\in \ms{K}^-_0$.
 Denote by $\delta_0^-:\mathscr{K}_0^-\to \mathrm{H}^0(\Sigma,\mc{E})$ the natural map, then 
$$\op{Ker}(\iota_0^*\delta_0^-)=\op{Ker}(d)\cap \mathrm{L}^2(\widehat{\Sigma},\mc{E}),$$
where $\iota_0^*$ is defined by
$$
\cdots \to0 \stackrel{\beta^{*}}{\rightarrow} \mathrm{H}^{0}(\Sigma,\mc{E}) \stackrel{\iota_0^{*}}{\longrightarrow} \mathrm{H}^{0}(\p\Sigma,\mc{E}) \rightarrow \cdots
$$
Since $h_{\infty}$ is the dimension of the space of all $\psi$, so
 $h_{\infty}(\mc{E})=\dim (\mathscr{K}_0^-/\op{Ker}(d)\cap \mathrm{L}^2(\widehat{\Sigma},\mc{E}))$, and we have
\begin{align}\label{hE}
\begin{split}
h_{\infty}(\mc{E})&=\dim (\mathscr{K}_0^-/\op{Ker}(\iota_0^*\delta_0^-))=\dim\op{Im}(\iota^*_0\delta_0^-)\\
&\leq \dim\op{Im}(\iota_0^*)=\dim \mathrm{H}^0(\Sigma,\mc{E}).
\end{split}
\end{align}
From \eqref{hinfE}, \eqref{hinf} and \eqref{hE}, we obtain
\begin{equation*}
  h_{\infty}(\wedge^-)=\dim \mathrm{H}^0(\p\Sigma,\mc{E})-\dim \mathrm{H}^0(\Sigma,\mc{E})
\end{equation*}
Similarly, one has 
\begin{equation}\label{hpm}
  h_{\infty}(\wedge^+)=\dim \mathrm{H}^0(\p\Sigma,\mc{E})-\dim \mathrm{H}^0(\Sigma,\mc{E})=h_{\infty}(\wedge^-).
\end{equation}
Substituting \eqref{hpm} into \eqref{sign1}, we obtain a formula for signature:
\begin{thm}\label{thmsign2}
	The signature is given by
	\begin{equation}\label{sign2}
	\op{sign}(\mc{E},\Omega)= 2\int_\Sigma\left(c_1(\mc{E}^+,\n|_{\mc{E}^+})-c_1(\mc{E}^-,\n|_{\mc{E}^-})\right)+\eta(A_{\mbf{J}}).
\end{equation}
\end{thm}
\begin{rem}
The above theorem was proven by Atiyah \cite[(3.1)]{Atiyah} under the assumption that the representation on each component of the boundary is elliptic. He then proceeded to prove the existence and uniqueness of a section of a covering space of the unitary group that specifies a trivialization of a suitable line bundle whose relative first Chern number is equal to the signature, \cite[Theorem 2.13]{Atiyah}. However, this existential statement does not easily provide a formula for the value of this discontinuous section on arbitrary elements. This will be done in Section \ref{Atiyahsigma}.
\end{rem}

\section{Toledo invariants}\label{Tol}

In this section, we recall the definition of the Toledo invariant for surfaces with boundary, which is given by Burger,  Iozzi and Wienhard \cite[Section 1.1]{BIW}. We show that 
the Toledo invariant can be expressed as the integration of first Chern forms with compact support over the surface.

\subsection{Definition of Toledo invariant}\label{Toledo}

Let $\Sigma$ be a connected oriented surface with boundary $\p\Sigma$, and $\phi:\pi_1(\Sigma)\to G$ be a surface group representation into a Lie group $G$ which is of Hermitian type. Burger,  Iozzi and  Wienhard \cite[Section 1.1]{BIW} introduced the definition of Toledo invariant $\op{T}(\Sigma,\phi)$, which generalizes
 the Toledo invariant for closed surfaces. 

A Lie group $G$ is of {\it Hermitian type} if it is connected, semisimple with finite center and no compact factors, and if the associated symmetric space is Hermitian. Let $G$ be a group of Hermitian type so that in particular the associated symmetric space $\mathscr{X}$ is Hermitian of noncompact type, then $\mathscr{X}$ carries a unique Hermitian (normalized) metric of minimal holomorphic sectional curvature $-1$. The associated K\"ahler form $\omega_{\mathscr{X}}$ is in $\Omega^2(\ms{X})^G$ the space of $G$-invariant $2$-forms on $\ms{X}$. A Lie group $G$ is of type $(\mathrm{RH})$ if it is connected reductive with compact center and the quotient $G / G_{c}$ by the largest connected compact normal subgroup $G_{c}$ is of Hermitian type.
 By the van Est isomorphism \cite{Van},
$\Omega^2(\ms{X})^G\cong \mathrm{H}^2_c(G,\mb{R})$,  
where $\mathrm{H}^{\bullet}_c(G,\mb{R})$ denotes the continuous cohomology of the group $G$ with $\mb{R}$-trivial coefficients, there exists a unique class $\kappa_G\in \mathrm{H}^2_c(G,\mb{R})$ corresponding to the K\"ahler form $\omega_{\ms{X}}$, and thus gives rise to a bounded K\"ahler class $\kappa^b_G\in \mathrm{H}^2_{c,b}(G,\mb{R})$ by the  isomorphism \cite{BO},
$\mathrm{H}^2_{c}(G,\mb{R})\cong \mathrm{H}^2_{c,b}(G,\mb{R})$, 
where $\mathrm{H}^{\bullet}_{c,b}(G,\mb{R})$ denotes the bounded continuous cohomology. In fact, $\kappa^b_G\in  \mathrm{H}^2_{c,b}(G,\mb{R})$ is defined by a bounded cocycle 
\begin{equation}\label{bounded cocyle}
  c(g_0,g_1,g_2)=\frac{1}{2\pi}\int_{\triangle(g_0x, g_1x, g_2x)} \omega_{\ms{X}},
\end{equation}
 where $\triangle(g_0x, g_1x, g_2x)$ is a geodesic triangle with ordered vertices $g_0x, g_1x,g_2x$ for some base point $x\in \ms{X}$.

 By Gromov isomorphism \cite{Gromov}, one has 
$$\phi_b^*(\kappa^b_G)\in \mathrm{H}^2_b(\pi_1(\Sigma),\mb{R})\cong \mathrm{H}^2_b(\Sigma,\mb{R}).$$
The canonical map $j_{\p\Sigma}:\mathrm{H}^2_b(\Sigma,\p\Sigma,\mb{R})\to \mathrm{H}^2_b(\Sigma,\mb{R})$ from singular bounded cohomology relative to $\p\Sigma$ to singular bounded cohomology is an isomorphism. Then the Toledo invariant is defined as 
$$\op{T}(\Sigma,\phi)=\langle j^{-1}_{\p\Sigma}\phi_b^*(\kappa^b_G),[\Sigma,\p\Sigma]\rangle,$$
where   $j^{-1}_{\p\Sigma}\phi_b^*(\kappa^b_G)$ is considered as an ordinary relative cohomology class and $[\Sigma,\p\Sigma]\in \mr{H}_2(\Sigma,\p\Sigma,\mb{Z})\cong\mb{Z}$ denotes the relative fundamental class.

\subsection{Invariant K\"ahler potentials}\label{IKP}

In this subsection, we introduce a family of differential $1$-forms $\alpha_W$ on the symmetric space $\op{D}^{\op{I}}_{p,q}$, parametrized by points $W$ on the closure $\o{\mr{D}^{\op{I}}_{p,q}}$. The form $\alpha_W$ is a primitive of the K\"ahler form and is invariant under the stabilizer of $W$ in $\mr{U}(p,q)$. The key feature is that $\alpha_W$ defines a bounded $1$-cochain. Therefore it can be used to modify the pull-back of the K\"ahler form by an equivariant map in order to make it compactly supported, without changing its bounded cohomology class.

In order to find formula (\ref{invariant Kahler potential}) for $\alpha_W$, we start from the classical formula for the K\"ahler potential invariant under the stabilizer of a point $J$ of the symmetric space. Then we let $J$ tend to infinity and observe that, up to an additive constant, it converges as $J$ converges to a boundary point $W$. 

\medskip

Every $L\in \mr{U}(p,q)$ acts on the bounded symmetric domain of type I
$$\op{D}^{\op{I}}_{p,q}=\{W\in M(p,q,\mb{C}), I_q-W^*W>0\}$$
holomorphically, and the action extends continuously to $\o{\mr{D}^{\op{I}}_{p,q}}$. By Brouwer's fixed point theorem, there exists a fixed point in $\o{\mr{D}^{I}_{p,q}}$. Let 
\begin{align}\label{kahlermetric}
\omega_{\op{D}^{\op{I}}_{p,q}}=	-2i\p\b{\p}\log\det(I-W^*W).
\end{align}
denote the invariant K\"ahler metric (Bergman metric) on $\op{D}^{\op{I}}_{p,q}$ with minimal holomorphic sectional curvature $-1$, see e.g. \cite{KM1, Mok}.

For any point $W_0\in \o{\mr{D}^{\mr{I}}_{p,q}}$, denote by  
\begin{align*}
\begin{split}
  K_{W_0}:=\{L\in \mr{U}(p,q):L(W_0)=W_0\}
 \end{split}
\end{align*}
the isotropy group of $W_0$. 
For any $L\in \mr{U}(p,q)$, we write 
\begin{align*}
\begin{split}
  L:=\left(\begin{matrix}
  a& b\\
  c& d
\end{matrix}\right)
 \end{split}\in \mr{U}(p,q),
\end{align*}
i.e. $L^*\left(\begin{matrix}
  I_p&0 \\
  0&-I_q 
\end{matrix}\right)L=\left(\begin{matrix}
  I_p&0 \\
  0&-I_q 
\end{matrix}\right)$,
from which it follows that 
\begin{align}\label{relation on U(p,q)}
\begin{split}
   L^{-1}=\left(\begin{matrix}
  I_p&0 \\
  0&-I_q 
\end{matrix}\right)L^*\left(\begin{matrix}
  I_p&0 \\
  0&-I_q 
\end{matrix}\right)=\left(\begin{matrix}
  a^*&-c^* \\
  -b^*&d^* 
\end{matrix}\right).
 \end{split}
\end{align}
It acts on $W\in \op{D}^{\op{I}}_{p,q}$ by 
$$L(W)=(aW+b)(cW+d)^{-1}.$$
One can refer to \cite[Page 65-68, Section (2.2)]{Mok} for the bounded symmetric domain $\op{D}^I_{p,q}$. If moreover, $L(W_0)=W_0$, then $L^{-1}(W_0)=W_0$, and so 
$$a^*W_0-c^*=W_0(-b^*W_0+d^*).$$
The complex conjugate transpose gives
\begin{align}\label{fix condition}
\begin{split}
  W_0^*a-c=-(W_0^*b-d)W_0^*.
 \end{split}
\end{align}
Now we define a smooth function $\psi_{W_0}=\psi_{W_0}(W)$ on $\mr{D}^I_{p,q}$ by 
\begin{equation}\label{invariant Kahler potential}
  \psi_{W_0}:=-\log\left(|\det(W_0^*W-I_q)|^{-2}\det(I_q-W^*W)\right),
\end{equation}
which is a smooth real function on $\mr{D}^I_{p,q}$. Moreover, it satisfies 
\begin{equation*}
  i\p\b{\p}\psi_{W_0}=\frac{1}{2}\omega_{\op{D}^{\op{I}}_{p,q}},
\end{equation*}
i.e. $\psi_{W_0}$ is a K\"ahler potential of the K\"ahler form of $\frac{1}{2}\omega_{\op{D}^{\op{I}}_{p,q}}$. On the other hand, since
\begin{align*}
\begin{split}
&\quad|\det(W_{0}^*L(W)-I_q)|^{-2}\det(I_q-L(W)^*L(W))\\
&=\left(|\det(W_{0}^*(aW+b)-(cW+d))|^{-2}|\det(cW+d)|^2\right)\\
&\quad \cdot\left( \det(cW+d)|^{-2}\det(I_q-W^*W)\right)\\
&=|\det((W_0^*a-c)W+W_0^*b-d)|^{-2}\det(I_q-W^*W)\\
&=|\det(-(W_0^*b-d)W_0^*W+W_0^*b-d)|^{-2}\det(I_q-W^*W)\\
&=|\det(d-W_0^*b)|^{-2}|\det(W_0^*W-I_q)|^{-2}\det(I_q-W^*W),
 \end{split}
\end{align*}
where  the third equality follows from \eqref{fix condition}, the first equality follows from the fact $\det(I_q-L(W)^*L(W))=|\det(cW+d)|^{-2}\det(I_q-W^*W)$. In fact, 
\begin{align*}
\begin{split}
  &\quad \det(I_q-L(W)^*L(W))\\
  &=\det(I_q-(W^*c^*+d^*)^{-1}(W^*a^*+b^*)(aW+b)(cW+d)^{-1}) \\
  &=\det((W^*c^*+d^*)^{-1})\det((cW+d)^{-1})\det((W^*c^*+d^*)(cW+d)-(W^*a^*+b^*)(aW+b))\\
  &=|\det(cW+d)|^{-2}\det(W^*(c^*c-a^*a)W+W^*(c^*d-a^*b)+(d^*c-b^*a)W+d^*d-b^*b)\\
  &=|\det(cW+d)|^{-2}\det(I_q-W^*W),
 \end{split}
\end{align*}
where the last equality follows from \eqref{relation on U(p,q)}.
Hence
\begin{align*}
\begin{split}
  (L^*\psi_{W_0})(W)&=\psi_{W_0}(L(W))\\
  &=-\log \left(|\det(W_0^*L(W)-I_q)|^{-2}\det(I_q-L(W)^*L(W))\right)\\
  &=-\log|\det(d-W_0^*b)|^{-2}-\log|\det(W_0^*W-I_q)|^{-2}\det(I_q-W^*W)\\
  &=\log|\det(d-W_0^*b)|^{2}+\psi_{W_0}(W),
 \end{split}
\end{align*}
which means that $\psi_{W_0}$ is a $L$-invariant (up to a constant) function. In one word, we have
\begin{prop}\label{invariant prop}
For any $W_0\in \o{\mr{D}^{\mr{I}}_{p,q}}$, there exists a $K_{W_0}$-invariant (up to a constant) K\"ahler potential $\psi_{W_0}$ for $\frac{1}{2}\omega_{\op{D}^{\op{I}}_{p,q}}$.	
\end{prop}
{\begin{rem}\label{zero}
In fact, for any classical Hermitian symmetric space $\mathscr{X}$, and for any $W\in \overline{\mathscr{X}}$, there exists a ($\mr{Stab}_W=K_W$)-invariant (up to a constant) 
 K\"ahler potential $\psi_W$ (Prop. \ref{invariant prop} and corresponding paragraphs in the other cases) with the following property:
 for $\alpha=d^c\psi_W$, $\int_\gamma \alpha=0$ for any geodesic $\gamma$ passing through $W$. 
 \begin{prop}\label{potential}For $\alpha=d^c\psi_W$, where $d^c:=-i(\p-\b{\p})$, then $\int_\gamma \alpha=0$ for any geodesic $\gamma$ passing through $W$ 
 and          $||\alpha||_\infty\leq \mathrm{rank}(\mathscr{X})\pi$.
 \end{prop}
 \begin{proof} For any $W\in \overline{\mathscr{X}}$, let $\psi_W$ be any $K_W$-invariant (up to a constant) K\"ahler potential. Let $\gamma(t)$ be a geodesic with $\gamma(0)=Q\in \mathscr{X}$ and $\gamma(\infty)=W$. Then $\gamma$ is contained in a maximal flat $F$ which is a totally real subspace. The following argument is basically due to Domic-Toledo \cite{DT}. The difficult case is when $W$ is an ideal point. If $W$ is regular,   $K_W$ is $MAN$ where $N$ is a minimal parabolic group. If $J$ denotes a complex structure, then $J\gamma'(t)$ is tangent to the orbit $N\gamma(t)$ since $J\gamma'(t)$ is orthogonal to the geodesic and $N\gamma(t)$ contains all the directions orthogonal to $F$.
 If $W$ is singular,  $K_W$ is $M'A'N'$ where $N'$ contains $N$ \cite{Eb}. Hence $J\gamma'(t)$ is tangent to the orbit $N'\gamma(t)$.
 Then $$d^c\psi_W(\gamma'(t))=d\psi_W(J\gamma'(t))=0$$ since $\psi_W$ is $K_W$-invariant. This shows that $\alpha=d^c\psi_W$ is zero along $\gamma$.
 
 Hence, for any $Q,R\in \mathscr{X}$ with geodesic $\gamma(Q,R)$ connecting them,
 $$\int_{\gamma(Q,R)}\alpha=\int_{\gamma(Q,R)} d^c\psi_W=\int_{\triangle(W,Q,R)} dd^c\psi_W=\int_{\triangle(W,Q,R)} \omega.$$
By \cite[Theorem 1]{DT}, one has
 $$\left|\int_{\gamma(Q,R)} \alpha\right|=\left|\int_{\triangle(W,Q,R)} \omega \right|\leq \sup_{\triangle}\left|\int_\triangle \omega\right|\leq \mathrm{rank}(\mathscr{X})\pi.$$
 This shows that $||\alpha||_\infty\leq \mathrm{rank}(\mathscr{X})\pi$, where the $\ell^\infty$-norm $\|\bullet\|_\infty$ is defined by \eqref{infinity norm}.
 \end{proof}
\end{rem}}
\subsection{Relation to the pullback forms with compact support}

Firstly, we will recall some definitions on the cohomology group of a topological space with a group action, we refer to \cite{Kim} and the references therein. Let $X$ be a topological space and $G$ be a group acting continuously on $X$.  For any $k>0$, one can define the space
$$F^k_{\text{alt}}(X,\mb{R})=\{f:X^{k+1}\to \mb{R}| f \text{ is alternating}\}.$$
Let $F^k_{\text{alt}}(X,\mb{R})^G$ denote the subspace of $G$-invariant functions, where the action of $G$ on $F^k_{\text{alt}}(X,\mb{R})$ is given by 
$$(g\cdot f)(x_0,\ldots,x_k)=f(g^{-1}x_0,\ldots,g^{-1}x_k),$$
for any $f\in F^k_{\text{alt}}(X,\mb{R})$ and $g\in G$. The natural coboundary operator $\delta_k:F^k_{\text{alt}}(X,\mb{R})\to F^{k+1}_{\text{alt}}(X,\mb{R})$ is given by
$$(\delta_k f)(x_0,\ldots,x_{k+1})=\sum_{i=0}^{k+1}(-1)^if(x_0,\cdots,\hat{x}_i,\ldots,x_{k+1}),$$
which also gives a coboundary operator on the complex $F^*_{\text{alt}}(X,\mb{R})^G$. The cohomology $\mathrm{H}^*(X;G,\mb{R})$ is defined as the cohomology of this complex. Define $F^*_{\text{alt},b}(X,\mb{R})$ as the subspace of $F^*_{\text{alt}}(X,\mb{R})$ consisting of bounded alternating functions. The coboundary operator restricts to the complex $F^*_{\text{alt},b}(X,\mb{R})^G$ and so it defines a cohomology, denoted by $\mathrm{H}^*_b(X;G,\mb{R})$,
 see \cite{Dupre} and also \cite[Section 3]{Kim}. In particular, for a manifold $X$, $\mathrm{H}^*_b(\wt{X};\pi_1(X),\mb{R})\cong \mathrm{H}^*_b(\pi_1(X),\mb{R})$.
 
 Similarly, if $G$ is a semisimple Lie group and $X$ is the associated symmetric space, one can also define the complex for the continuous (resp. bounded) and alternating functions, we denote this complex by $C^*_c(X,\mb{R})_{\text{alt}}$ (resp. $C^*_{c,b}(X,\mb{R})_{\text{alt}}$). Then the continuous cohomology $\mathrm{H}^*_c(G,\mb{R})$ (resp. $\mathrm{H}^*_{c,b}(G,\mb{R})$) can be isomorphically computed by the cohomology of $G$-invariant complex $C^*_c(X,\mb{R})_{\text{alt}}^G$ (resp. $C^*_{c,b}(X,\mb{R})_{\text{alt}}^G$), see \cite[Chapitre III]{Gui} and \cite[Corollary 7.4.10]{Monod}. 

If $X$ is a countable CW-complex, then one can define the cohomology  groups $\mathrm{H}^*_b(X,\mb{R})$ and  $\mathrm{H}^*_b(X,A,\mb{R})$ associated with the complex $C^*_b(X,\mb{R})$ of  bounded real-valued cochains on $X$ and  the subcomplex 
$C^*_b(X,A,\mb{R})$ of the bounded cochains  that vanish on simplices with image contained in $A$, repsectively.
Let $C^k_b(\wt{X},\mb{R})_{\text{alt}}$ denote the complex of bounded, alternating real-valued Borel functions on $\wt{X}^{k+1}$, then the cohomology of the $\pi_1(X)$-invariant complex $C^*_b(\wt{X},\mb{R})^{\pi_1(X)}_{\text{alt}}$
is isomorphic to $\mathrm{H}^*_b(X,\mb{R})$, see \cite{Iva} and also \cite[Section 2]{Kim}.

Let $\Sigma$ be a  connected oriented surface with boundary $\p\Sigma$, $\Sigma_o:=\Sigma\backslash \p\Sigma$. Consider a representation $\phi:\pi_1(\Sigma)\to G$ where $G=\op{U}(p,q)$. Denote by $\mathscr{X}:=G/K$ the associated symmetric space, which can identified with the bounded symmetric domain $\op{D}^{\op{I}}_{p,q}$ of type $\op{I}$, we denote $\omega=\omega_{\op{D}^{\op{I}}_{p,q}}$ for  simplicity.  The K\"ahler form $\omega$ gives the cohomology classes $\kappa_G\in \mathrm{H}^2_c(G,\mb{R})$ and $\kappa_G^b\in \mathrm{H}^2_{c,b}(G,\mb{R})$ which both correspond to the cochain $c_\omega$ defined by \eqref{bounded cocyle}.

The natural inclusion $C^*_{c,b}(\mathscr{X},\mb{R})_{\text{alt}}\subset F^*_{\text{alt},b}(\mathscr{X},\mb{R})$ induces a homomorphism $i_G:\mathrm{H}^*_{c,b}(G,\mb{R})\to \mathrm{H}^*_b(\mathscr{X};G,\mb{R})$. Then we have the following commutative diagram:
\begin{equation*}
\begin{CD}
\mathrm{H}^2_b(\mathscr{X};G,\mb{R}) @>f^*_b>> \mathrm{H}^2_b(\wt{\Sigma};\pi_1(\Sigma),\mb{R})\cong \mathrm{H}^2_b(\Sigma,\mb{R})\\
@AA i_GA @AA i_\Sigma A\\
\mathrm{H}^2_{c,b}(G,\mb{R}) @>\phi^*_b>> \mathrm{H}^2_b(\pi_1(\Sigma),\mb{R})
\end{CD}
\end{equation*}
see \cite[Page 58]{Kim}, where $f^*_b$ is induced from any $\phi$-equivariant map $f:\wt{\Sigma}\to\mathscr{X}$, $i_\Sigma$ is the Gromov isomorphism. 
 Then the cochain  representing the class $i_\Sigma\phi^*_b(\kappa_G^b)=f^*_bi_G(\kappa_G^b)$ is given by
{\begin{equation}\label{strsimplex}
   \frac{1}{2\pi}\int_{\op{Str}(f)(\sigma)}\omega,
\end{equation}}
where $\sigma\in C_2(\wt{\Sigma},\mb{R})$ is any two dimensional singular simplex on $\wt{\Sigma}$, $\op{Str}(f)(\sigma):=\Delta(fv_1,fv_2,fv_3)$ denotes the geodesic $2$-simplex, and $v_1,v_2,v_3$ are the vertices of $\sigma$. 
Denote $$[f^*\omega]_b:=2\pi i_\Sigma\phi^*_b(\kappa_G^b)=2\pi f^*_bi_G(\kappa_G^b).$$

We assume that $\p\Sigma=\cup_{i=1}^n c_i$, where each $c_i$ is a connected component of the boundary $\p\Sigma$. For any representation $\phi:\pi_1(\Sigma)\to \mr{U}(p,q)$, we denote by
$L_i:=\phi(c_i)\in \op{U}(p,q)$ the representation of boundary component $c_i$, and 
\begin{equation}\label{Fixpoint2}
  \alpha_i:=d^c\psi_{i},
\end{equation}
where $d^c:=-i(\p-\b{\p})$, $dd^c=2i\p\b{\p}$, and $\psi_{i}$ is given by \eqref{invariant Kahler potential}. By Proposition \ref{invariant prop}, each $\alpha_i$ is $L_i$-invariant.
 Let $\chi_i=\chi_i(u):\Sigma\to[0,1]$ be any smooth cut-off function on $\Sigma$, which is equal to $1$ near $c_i$ and vanishes outsides a small neighborhood of $c_i$. For example, one can take $\chi_i(u)$ satisfying
\begin{align*}
\chi_i(u)=
\begin{cases}
& 1,\quad u\in c_i\times [0,1/2];\\
&0, \quad u\in\Sigma\backslash (c_i\times [0,3/4]).	
\end{cases}
\end{align*}
For any $\phi$-equivariant map $f:\wt{\Sigma}\to \mr{D}^{\mr{I}}_{p,q}$, the differential form 
$$f^*\omega-\sum_{i=1}^n d(\chi_i f^*\alpha_i)$$
descends to a well-defined form on $\Sigma$, and has compact support in $\Sigma_o$. Hence it defines a class 
$$\left[f^*\omega-\sum_{i=1}^nd(\chi_i f^*\alpha_i)\right]_c\in \mr{H}^2_{\text{dR,comp}}(\Sigma_o,\mb{R})$$
in the de Rham cohomology group with compact support.

 On the other hand, for any $\phi$-equivariant map $f:\wt{\Sigma}\to \mr{D}^{\mr{I}}_{p,q}$, and for each $i$, let $f_i:\wt{\Sigma}\to \op{D}^I_{p,q}\cup\{W_i\}$ be a $\phi(c_i)$-equivariant smooth map   such that $f_i=f$ in a small neighborhood $\wt{I_{1/2}}$ of $\wt{c_i}$ and is constant $W_i$ outside $\wt{I_{3/4}}$, where $I_a:=c_i\times [0,a)$ and $W_i$ is a fixed point of $L_i$, where $\wt{\bullet}=\pi^{-1}(\bullet)$ denotes the lifting of $\bullet$, $\pi:\wt{\Sigma}\to\Sigma$ is the covering map. In fact, if $W_i$ is a fixed point of $\phi(c_i)$, then $W_i$ gives a constant section of the associated bundle $\wt{I_1}\times_{\phi(c_i)}(\op{D}^I_{p,q}\cup \{W_i\})\to I_1$. The $\phi$-equivariant map $f$ also gives a section 
 $f|_{I_1}:I_1\to \wt{I_1}\times_{\phi(c_i)}\op{D}^I_{p,q}\subset \wt{I_1}\times_{\phi(c_i)}(\op{D}^I_{p,q}\cup \{W_i\})$ by restriction. Hence we can construct a smooth section $f_i:I_1\to \wt{I_1}\times_{\phi(c_i)}(\op{D}^I_{p,q}\cup \{W_i\})$ such that $f_i=f$ near  $c_i$, and $f_i\equiv W_i$ outside a small collar neighborhood of $c_i$, which also can be viewed as a $\phi(c_i)$-equivariant map $f_i:\wt{I_1}\to\op{D}^I_{p,q}\cup\{W_i\}$. Moreover, the equivariant map $f_i$ can be chosen such that the norm of the differential $(f_i)_*$ is exponentially decaying near the boundary  $\partial(f_i^{-1}(W_i))$ of  $f_i^{-1}(W_i)$. Hence $(f_i^*\alpha_i)(p), p\in \wt{I_1}\backslash f_i^{-1}(W_i)$ converges to zero as $p$ goes to $\partial f_i^{-1}(W_i)$, which can be extended to a $\phi$-equivariant one-form on $\wt{\Sigma}$ by  zero extension, we denote this one-form also by $f_i^*\alpha_i$ for convenience. 
 
 \begin{rem}
Here is how $f_i$ is constructed. 
Let $W_0$ be a fixed point of $L_i=\phi(c_i)$.  If $W_0 \in \mathscr{X}$, then $L_i$
is elliptic, and one can construct $f_i$ as a constant map. Hence suppose that $W_0$ is an ideal point.  

Consider a Busemann function $B=B_{W_0}$ based at $W_0$ with the corresponding geodesic flow $\Phi_t$ pointing toward $W_0$, i.e.,
$$\frac{d(\Phi_t)}{dt}=\nabla B,$$ and $|\nabla B|=1$.

We give coordinates on $\tilde{I_1}$ as $(s, t)$ where $s$ is  the parametrization of $\tilde{c_i}$ and $t$ is a parameter for $I_t$.

\begin{align*}
\begin{split}
f_i(s,t)  = \begin{cases}
 \Phi_{\frac{t-\frac{1}{2}}{\frac{3}{4}-t}}\circ f(s,t)	& t\in [\frac{1}{2}, \frac{3}{4})\\
 	f(s,t)	&t \in [0, \frac{1}{2}]\\
 	W_0   & t\in [\frac{3}{4},1].
 \end{cases}
 \end{split}
\end{align*}
 Then at $t=1/2$, $f_i=f$ and as $t\to \frac{3}{4}$, 
$\lim_{t\ra \frac{3}{4}}f_i(s,t)=\Phi_\infty\circ f(s,\frac{3}{4})=W_0$.
Geometrically, $f_i$ maps a segment $s \times [0,\frac{3}{4}]$ to an infinite arc
from $f(s, 0)$ to $W_0$.

Since we can take any $\phi$-eqivariant map, by perturbing $f$ a little bit near $c_i$, 
we may assume that $\tilde{I_1}$ is mapped into the orbit  $N_{W_0}f(\tilde c_i)$ where $N_{W_0}$ is the horospherical subgroup fixing $W_0$.
Then $(f_i^* \alpha_i)(v)=\alpha_i((\Phi_u)_*(f_*(v)))$ on $t\in [\frac{1}{2}, \frac{3}{4})$, and since the flow lines of $\Phi_u$ are geodesics converging to $W_0$, by  (\ref{contraction}), $(\Phi_u)_*(f_*(v))$ tends to zero exponentially fast as $t\ra\frac{3}{4}$.
This shows that $f_i^*\alpha_i$ is supported on a small neighborhood of $\tilde{c_i}$ and zero elsewhere.
 \end{rem}
 
 Note that the form $f^*\omega-\sum_{i=1}^n d(f_i^*\alpha_i)$ is  a $\phi$-equivariant differential form on $\wt{\Sigma}$ and vanishes near $\wt{\p\Sigma}$, so it descends to a
  differential form on $\Sigma$ with compact support. Moreover, $\sum_{i=1}^n(f_i^*\alpha_i-\chi_if^*\alpha_i)$ also descends to a one-form on $\Sigma$ with compact support, hence
 \begin{align*}
\begin{split}
  \left[f^*\omega-\sum_{i=1}^n d(f_i^*\alpha_i)\right]_c=\left[f^*\omega-\sum_{i=1}^nd(\chi_i f^*\alpha_i)\right]_c\in \mr{H}^2_{\text{dR,comp}}(\Sigma_o,\mb{R}).
 \end{split}
\end{align*}
The differential form $f^*\omega-\sum_{i=1}^nd(f_i^*\alpha_i)$  defines a cochain as follows: 
\begin{align*}
\begin{split}
\left (f^*\omega-\sum_{i=1}^nd(f_i^*\alpha_i)\right)(\sigma_2)&:=\int_{\mathrm{Str}(f)(\sigma_2)}\omega-\sum_{i=1}^n\int_{\mathrm{Str}(f_i)(\sigma_2)}d\alpha_i\\
&=\int_{\mathrm{Str}(f)(\sigma_2)}\omega-\sum_{i=1}^n\int_{\mathrm{Str}(f_i)(\sigma_2)}\omega,
 \end{split}
\end{align*}
for any singular $2$-simplex $\sigma_2\in C_2(\wt{\Sigma},\mb{R})$.
The $\ell^\infty$-norm for a cochain is defined by 
\begin{align}\label{infinity norm}
\begin{split}
  \|\bullet\|_\infty:=\sup_{\sigma\in C_*(\wt{\Sigma},\mb{R})}|\bullet(\sigma)|.
 \end{split}
\end{align}
Then the $\ell^\infty$-norm of the cochain defined by $f^*\omega-\sum_{i=1}^n d( f_i^*\alpha_i)$ is given by
\begin{align*}
\begin{split}
  \left\|f^*\omega-\sum_{i=1}^nd(f_i^*\alpha_i)\right\|_\infty&=\sup_{\sigma_2\in  C_2(\wt{\Sigma},\mb{R})}\left|\left(f^*\omega-\sum_{i=1}^nd(f_i^*\alpha_i)\right)(\sigma_2)\right|\\
  &\leq \sup_{\sigma_2\in  C_2(\wt{\Sigma},\mb{R})}\left\|\int_{\mathrm{Str}(f)(\sigma_2)}\omega\right\|+ \sum_{i=1}^n\sup_{\sigma_2\in  C_2(\wt{\Sigma},\mb{R})}\left\|\int_{\mr{Str}(f_i)(\sigma_2)}\omega\right\|\\
  &\leq (n+1)\|\omega\|_{\infty}\leq (n+1)\min\{p,q\}\pi,
 \end{split}
\end{align*}
where the last inequality follows from \cite[Theorem 1]{DT}. So the cochain is bounded and defines a bounded class 
\begin{align*}
\begin{split}
  \left[f^*\omega-\sum_{i=1}^n d(f_i^*\alpha_i)\right]_b\in \mr{H}^2_b(\Sigma,\mb{R}).
 \end{split}
\end{align*}
 Note that $\sum_{i=1}^nf_i^*\alpha_i$ also defines a cochain by 
\begin{align*}
\begin{split}
  \left(\sum_{i=1}^nf_i^*\alpha_i\right)(\sigma_1)=\sum_{i=1}^n\int_{\mathrm{Str}(f_i)(\sigma_1)}\alpha_i,
 \end{split}
\end{align*}
 for any singular $1$-simplex $\sigma_1\in C_1(\wt{\Sigma},\mb{R})$.
 This cochain is also bounded, i.e. $\|\sum_{i=1}^nf_i^*\alpha_i\|_\infty<+\infty$. In fact, 
\begin{align*}
\begin{split}
  \left|\sum_{i=1}^n\int_{\mathrm{Str}(f_i)(\sigma_1)}\alpha_i\right|
 \leq \sum_{i=1}^n\left|\int_{\mathrm{Str}(f_i)(\sigma_1)}\alpha_i\right|\leq \sum_{i=1}^n\|\alpha_i\|_\infty\leq n\cdot\min\{p,q\}\pi 
  \end{split}
\end{align*}
where the last inequality follows from Proposition \ref{potential}.
Thus
\begin{align*}
\begin{split}
 [f^*\omega]_b=  \left[f^*\omega-d\left(\sum_{i=1}^n f_i^*\alpha_i\right)\right]_b=\left[f^*\omega-\sum_{i=1}^n d(f_i^*\alpha_i)\right]_b\in \mr{H}^2_b(\Sigma,\mb{R}).
 \end{split}
\end{align*}

There exist the following several natural maps 
\begin{equation*}
\mr{H}_b^2(\Sigma,\mb{R})\stackrel{j_{\p\Sigma}^{-1}}{\longrightarrow}	\mathrm{H}^2_b(\Sigma,\p\Sigma,\mb{R})\stackrel{c}{\longrightarrow}\mathrm{H}^2(\Sigma,\p\Sigma,\mb{R})\stackrel{j_o^{-1}}{\longrightarrow}\mathrm{H}^2_{\text{comp}}(\Sigma_o,\mb{R})\stackrel{D}{\longrightarrow}
\mathrm{H}^2_{\text{dR,comp}}(\Sigma_o,\mb{R})
\end{equation*}
where $j_{\p\Sigma}$ is the natural isomorphism induced from the inclusion  $C^2_b(\Sigma,\p\Sigma,\mb{R})\to C^2_b(\Sigma,\mb{R})$,  see \cite[\S 2.2, (2.e)]{BIW},
{ $c$ is the canonical map induced from the inclusion $C^2_b(\Sigma,\p\Sigma,\mb{R})\subset C^2(\Sigma,\p\Sigma,\mb{R})$, $j_o$ is the natural map from the singular cohomology with compact support to the relative  singular cohomology, which is an isomorphism.}  Recall that the singular  cohomology with compact support
is the complex  $C^*_{\text{comp}}(\Sigma,\mb{R})$ of all cochains have compact support, 
 a cochain $u \in C^{*}(\Sigma_o, \mb{R})$ has compact support if and only if there exists a compact set $K \subset \Sigma_o$ such that $u \in C^{*}(\Sigma_o, \Sigma_o-K, \mb{R})$,
see \cite[Chapter IX,\S 3]{WM}.  The map $D$ is the de Rham map between the singular cohomology with compact support and de Rham cohomology with compact support, which is also an isomorphism, see e.g. \cite[Appendix, Page 261]{WM}. Under these canonical maps  we have
\begin{align*}
\begin{split}
  Dj_o^{-1}cj^{-1}_{\p\Sigma}([f^*\omega]_b)&=Dj_o^{-1}c j^{-1}_{\p\Sigma}\left(\left[f^*\omega-\sum_{i=1}^n d(f_i^*\alpha_i)\right]_b\right)\\
  &=  \left[f^*\omega-\sum_{i=1}^n d(f_i^*\alpha_i)\right]_c\\
  &=\left[f^*\omega-\sum_{i=1}^nd(\chi_i f^*\alpha_i)\right]_c\in \mr{H}^2_{\text{dR,comp}}(\Sigma_o,\mb{R}).
 \end{split}
\end{align*}
Hence the Toledo invariant can be given by
\begin{align}\label{Toledo1}
\begin{split}
\op{T}(\Sigma,\phi)&=\langle j^{-1}_{\p\Sigma}i_\Sigma\phi^*_b(\kappa_G^b),[\Sigma,\p\Sigma]\rangle\\
&=\langle cj^{-1}_{\p\Sigma}i_\Sigma\phi^*_b(\kappa_G^b),[\Sigma,\p\Sigma]\rangle\\
&=\frac{1}{2\pi}\int_\Sigma\left(f^*\omega-\sum_{i=1}^nd(\chi_i f^*\alpha_i)\right).
\end{split}
\end{align}
 By the same proof as in \cite[Proposition-definition 4.1]{KM}, the de Rham cohomology class $[f^*\omega-\sum_{i=1}^nd(\chi_i f^*\alpha_i)]_c$ with compact 
support depends only on the conjugate class of the representation $\phi$ (independent of $f$), and following \cite{KM}, we set
\begin{equation}
  \left[\phi^*\omega\right]_c:=\left[f^*\omega-\sum_{i=1}^nd(\chi_i f^*\alpha_i)\right]_c\in \mathrm{H}^2_{\text{dR,comp}}(\Sigma_o,\mb{R}).
\end{equation}
Hence 
\begin{equation}\label{Toledo invariant compact form}
  \op{T}(\Sigma,\phi)=\frac{1}{2\pi}\int_\Sigma\left(f^*\omega-\sum_{i=1}^nd(\chi_i f^*\alpha_i)\right)=\frac{1}{2\pi}\int_\Sigma \left[\phi^*\omega_{\op{D}^{\op{I}}_{p,q}}\right]_c.
\end{equation}

\begin{rem}
When $\p\Sigma\neq\emptyset$ and $G=\op{PU}(1,m)$, Koziarz and Maubon \cite[Proposition-Definition 4.1]{KM} introduced the invariant $\frac{1}{2\pi}\int_\Sigma \left[\phi^*\omega\right]_c$ by using the de Rham cohomology with compact support. In fact, the invariant can be shown to be equal to the Toledo invariant defined by  Burger-Iozzi-Wienhard, see \cite[Remark 6]{BIW}.

\end{rem}

\subsection{Relation to the first Chern class} \label{RFCC}

Consider the bounded symmetric domain of type $\mr{I}$
$$ \op{D}^{\op{I}}_{p,q}=\{W\in M(p,q,\mb{C}), I_q-W^*W>0\}.$$
One can check that $W^*\in \op{D}^{\op{I}}_{q,p}$. Denote by $E=\mb{C}^{p+q}$ and let $\Omega$ be a Hermitian form on $E$ with matrix given by
\begin{align*}
\begin{split}
  I_{p,q}:=\left(\begin{matrix}
  I_p&0 \\
  0&-I_q 
\end{matrix}\right)
 \end{split}
\end{align*}
i.e. for any two column vectors $X,Y\in E$, $\Omega(X,Y)=X^\top I_{p,q} \o{Y}$. We consider the trivial Hermitian vector bundle 
$$(F,\Omega):=\op{D}^{\op{I}}_{p,q}\times (E,\Omega)\to \op{D}^{\op{I}}_{p,q}.$$
Recall that  $\mc{J}(F,\Omega)$ is the space of all smooth sections $\mbf{J}$ of $\mr{End}(F)$ which satisfy the following conditions:
  \begin{itemize}
  \item[(i)] $\mbf{J}^2=-\op{Id}$;
  \item[(ii)] $\Omega(\mbf{J}\cdot,\mbf{J}\cdot)=\Omega(\cdot,\cdot)$;
  \item[(iii)] $i\Omega(\cdot, \mbf{J}\cdot)$ is positive definite. 
\end{itemize}
For the trivial Hermitian vector bundle $(F,\Omega)$, there exists the following canonical linear transformation  
\begin{align*} 
\begin{split}
  \mbf{J}_{\mr{I}}(W):=i\left(\begin{matrix}
  (I_p-WW^*)^{-1}(I_p+WW^*)& -2(I_p-WW^*)^{-1}W\\
  2(I_q-W^*W)^{-1}W^*&-(I_q-W^*W)^{-1}(I_q+W^*W) 
\end{matrix}\right).
 \end{split}
\end{align*}
\begin{prop}
$\mbf{J}_{\mr{I}}\in \mc{J}(F,\Omega)$.	
\end{prop}
\begin{proof}
By using $W(I_q-W^*W)^{-1}=(I_p-WW^*)^{-1}W$, one can check that 
	\begin{align*}
\begin{split}
  \left(\begin{matrix}
  -(I_p-WW^*)^{-1}(I_p+WW^*)& 2(I_p-WW^*)^{-1}W\\
  -2(I_q-W^*W)^{-1}W^*&(I_q-W^*W)^{-1}(I_q+W^*W) 
\end{matrix}\right)^2=\left(\begin{matrix}
  I_p&0 \\
  0& I_q
\end{matrix}\right).
 \end{split}
\end{align*}
Hence $\mbf{J}_{\mr{I}}^2=-\op{Id}$. By a direct checking, one has 
$${\mbf{J}}^*_{\mathrm{I}}I_{p,q}+I_{p,q}\mbf{J}_{\mr{I}}=0,$$
which is equivalent to $\mbf{J}^\top_FI_{p,q}\o{\mbf{J}_{\mr{I}}}=I_{p,q}$.
On the other hand, one has
\begin{align*}
\begin{split}
  -iI_{p,q}\mbf{J}_{\mr{I}} &=\left(\begin{matrix}
  (I_p-WW^*)^{-1}(I_p+WW^*)& -2(I_p-WW^*)^{-1}W\\
  -2(I_q-W^*W)^{-1}W^*&(I_q-W^*W)^{-1}(I_q+W^*W) 
\end{matrix}\right)\\
&=\left(\begin{matrix}
  I_p&0 \\
  -2W^*(I_p+WW^*)^{-1}& I_q
\end{matrix}\right)\cdot\\
&\left(\begin{matrix}
 (I_p-WW^*)^{-1}(I_p+WW^*)  &0 \\
 0 &  (I_q-W^*W)(I_q+W^*W)^{-1}
\end{matrix}\right)\cdot\\
&\left(\begin{matrix}
  I_p& -2(I_p+WW^*)^{-1}W \\
  0& I_q 
\end{matrix}\right),
 \end{split}
\end{align*}
which follows that $-iI_{p,q} \mbf{J}_{\mr{I}}>0$. By conjugation, one has $iI_{p,q}\o{\mbf{J}_{\mr{I}}}>0$. Hence $i\Omega(\cdot, \mbf{J}_{\mr{I}}\cdot)$ is positive definite. So $\mbf{J}_{\mr{I}}\in \mc{J}(F,\Omega)$.
\end{proof}
\begin{rem}\label{rem3.4}
The almost complex structure $\mbf{J}_{\mr{I}}$ is a smooth map $\mbf{J}_{\mr{I}}:\mr{D}^{\mr{I}}_{p,q}\to \mc{J}(E,\Omega)$, which is also an isomorphism between $\mr{D}^{\mr{I}}_{p,q}$ and $\mc{J}(E,\Omega)$. For any $Z\in \mr{U}(p,q)$ and $J\in \mc{J}(E,\Omega)$, then 
$ZJ Z^{-1}\in \mc{J}(E,\Omega)$. By the bijection $\mbf{J}_{\mr{I}}$, it induces an action on $\mr{D}^{\mr{I}}_{p,q}$ by
\begin{align*}
\begin{split}
  Z(W):=\mbf{J}^{-1}_{\mr{I}}(Z\mbf{J}_{\mr{I}}(W)Z^{-1})=(Z_1W+Z_2)(Z_3W+Z_4)^{-1},
 \end{split}
\end{align*}
where $Z=\left(\begin{matrix}
  Z_1&Z_2 \\
  Z_3&Z_4 
\end{matrix}\right)\in \op{U}(p,q)$. The isomorphism $\op{U}(p,q)/(\op{U}(p)\times \op{U}(q))\cong \mc{J}(E,\Omega)$ is given by $Z\cdot (\op{U}(p)\times \op{U}(q))\mapsto ZJ_0 Z^{-1}$, where $J_0=iI_{p,q}$ and $\op{U}(p)\times \op{U}(q)\cong \{Z\in \op{U}(p,q): ZJ_0Z^{-1}=J_0\}$.
\end{rem}
Denote
\begin{align*}
\begin{split}
  V:=\left(\begin{matrix}
  I_p & -W \\
  W^*& -I_q 
\end{matrix}\right).
 \end{split}
\end{align*}
Then $\mbf{J}_{\mr{I}}$ can be decomposed as the following form
\begin{align*}
\begin{split}
  \mbf{J}_{\mr{I}}=V\left(\begin{matrix}
  iI_p& 0\\
  0& -iI_q
\end{matrix}\right)V^{-1}.
 \end{split}
\end{align*}
\begin{rem}
Let $a$, $b$ be two Hermitian matrices satisfying
\begin{align*}
\begin{split}
  a^2=(I_p-WW^*)^{-1},\quad b^2=(I_q-W^*W)^{-1}.
 \end{split}
\end{align*}
Then the following matrix 
\begin{align*}
\begin{split}
  \wt{V}:=V\left(\begin{matrix}
  a&0 \\
  0&b 
\end{matrix}\right)=\left(\begin{matrix}
  a&-Wb \\
  W^*a&-b 
\end{matrix}\right)
 \end{split}
\end{align*}
is in $\op{U}(p,q)$ with $\wt{V}(0)=W$. Moreover $\mbf{J}_{\mr{I}}=V(iI_{p,q})V^{-1}=\wt{V}(iI_{p,q})\wt{V}^{-1}$.
\end{rem}
Now we can define a connection on $F$ by 
\begin{align*}
\begin{split}
  \n^F=V\left(\begin{matrix}
  d+(I_p-WW^*)^{-1}\b{\p}(I_p-WW^*)& 0\\
  0& d+ (I_p-WW^*)^{-1}\p(I_p-WW^*)
\end{matrix}\right)V^{-1}.
 \end{split}
\end{align*}
Then $[\n^F,\mbf{J}_{\mr{I}}]=0$. By a direct calculation, one has $\n^F=d+C$ where $C$ is given by
\begin{align*}
\begin{split}
  C=V\left(\begin{matrix}
  0&(I_p-WW^*)^{-1}dW \\
  (I_q-W^*W)^{-1}dW^*&0 
\end{matrix}\right)V^{-1}.
 \end{split}
\end{align*}
Then $C^*\Omega+\Omega C=0$, which follows that $\n^F$ preserves the Hermitian form $\Omega$.

Let ${F}=F^+\oplus F^-$ be the decomposition of $F$ corresponding to the $\pm i$-eigenspaces of $\mbf{J}_{\mr{I}}$. Denote by $\{e_1,\cdots, e_{p+q}\}$ the standard basis of $E$, and set 
\begin{align*}
\begin{split}
  (f_1,\cdots,f_{p+q}):=(e_1,\cdots, e_{p+q})V.
 \end{split}
\end{align*}
Then $\{f_1,\cdots, f_p\}$ forms a basis of $F^+$, while $\{f_{p+1},\cdots, f_{p+q}\}$ is a basis of $F^-$.
Denote $\n^{F^+}:=\n^F|_{F^+}$ and $\n^{F^-}:=\n^F|_{F^-}$. Since $[\n^F,\mbf{J}_{\mr{I}}]=0$, so $\n^F=\n^{F^+}\oplus \n^{F^-}$. The first Chern forms of $F^+$ and $F^-$ can be given by
\begin{align*}
\begin{split}
  c_1(F^+,\n^{F^+})=\frac{i}{2\pi}\p\b{\p}\log\det(I_p-WW^*)=-\frac{1}{2}\cdot\frac{1}{2\pi}\omega_{\mr{D}^I_{p,q}}
 \end{split}
\end{align*}
and 
\begin{align*}
\begin{split}
  c_1(F^-,\n^{F^-})=\frac{i}{2\pi}\b{\p}\p\log\det(I_p-WW^*)=\frac{1}{2}\cdot\frac{1}{2\pi}\omega_{\mr{D}^I_{p,q}}.
 \end{split}
\end{align*}

For any representation $\phi:\pi_1(\Sigma)\to \op{U}(E,\Omega)$, and any $\mbf{J}\in \mc{J}(\mc{E},\Omega)$, it gives a $\phi$-equivariant map from $\wt{\Sigma}$ into $\mc{J}(E,\Omega)$, which is also denoted by $\mbf{J}$. Using the following identification
$$\mbf{J}_{\mr{I}}: \mr{D}^{\mr{I}}_{p,q}\xrightarrow{\cong}\mc{J}(E,\Omega),$$
see Remark \ref{rem3.4},  $\mbf{J}\in \mc{J}(\mc{E},\Omega)$ defines a $\phi$-equivariant map by 
\begin{align*}
\begin{split}
  \wt{\mbf{J}}:\wt{\Sigma}\to \mr{D}^{\mr{I}}_{p,q},\quad z\mapsto W=\wt{\mbf{J}}(z)=\mbf{J}^{-1}_{\mr{I}}(\mbf{J}(z))
 \end{split}
\end{align*}
which is also equivalent to a smooth section of the associated bundle $\wt{\Sigma}\times_\phi \mr{D}^{\mr{I}}_{p,q}\to \Sigma$, we denote it also by $\wt{\mbf{J}}$. Consider the following vector bundle
\begin{align*}
\begin{split}
  F_\phi:=\wt{\Sigma}\times_\phi F=\wt{\Sigma}\times_\phi(\mr{D}^{\mr{I}}_{p,q}\times E)
 \end{split}
\end{align*}
over $\wt{\Sigma}\times_\phi\mr{D}^{\mr{I}}_{p,q}$. By pullback, $\wt{\mbf{J}}^*F_\phi$ is a vector bundle over $\Sigma$. The almost complex structure $\mbf{J}_{\mr{I}}$ gives rise to a canonical almost complex structure on $F_\phi$ by
\begin{align*}
\begin{split}
  \mbf{J}_{F_\phi}([z,W,e]):=[z,W,\mbf{J}_{\mr{I}}(W)e]
 \end{split}
\end{align*}
for any $[z,W,e]\in \wt{\Sigma}\times_\phi(\mr{D}^{\mr{I}}_{p,q}\times E)$. Then $(\wt{\mbf{J}}^*F_\phi,\Omega)$ is a Hermitian vector bundle over $\Sigma$, which is equipped with a pullback almost complex structure $\wt{\mbf{J}}^*\mbf{J}_{F_\phi}$. Moreover, one has the following isomorphism
\begin{align}\label{tau isomorphism}
\begin{split}
  \tau:\wt{\mbf{J}}^*F_\phi\to \mc{E},\quad \tau:(z_0,[z,W,e])\mapsto [z,e],
 \end{split}
\end{align}
where $\wt{\mbf{J}}(z_0)=[z,W]\in \wt{\Sigma}\times_\phi \mr{D}^{\mr{I}}_{p,q}$. The isomorphism commutes with the almost complex structures, i.e. $\tau\circ (\wt{\mbf{J}}^*\mbf{J}_{F_\phi})=\mbf{J}\circ \tau$.

 With respect to the almost complex structure $\mbf{J}_{F_\phi}$, the complex vector bundle $F_\phi$ has the following  decomposition
\begin{align*}
\begin{split}
  F_\phi=F_\phi^+\oplus F_{\phi}^-=(\wt{\Sigma}\times_\phi F^+)\oplus (\wt{\Sigma}\times_\phi F^-),
 \end{split}
\end{align*}
which corresponds to the $+i,-i$-eigensapces of $\mbf{J}_{F_\phi}$ respectively, where the action $\phi$ on $F^\pm$ is defined by $\phi(\gamma)(W,f):=(\phi(\gamma)W,\phi(\gamma)f)$, which is well-defined since $\mbf{J}_I(\phi(\gamma)W)=\phi(\gamma)\mbf{J}_I(W)\phi(\gamma)^{-1}$ for any $\gamma\in \pi_1(\Sigma)$. The connection $\n^F$ induces a natural connection $\n^{F_\phi}$ on $F_\phi$, in terms of the local frame $\{[z,W,f_j], i=1,\cdots,p+q\}$, the connection $\n^{F_\phi}$ is given by 
\begin{align*}
\begin{split}
 \n^{F_\phi}= \left(\begin{matrix}
  d+(I_p-WW^*)^{-1}\b{\p}(I_p-WW^*)& 0\\
  0& d+ (I_p-WW^*)^{-1}\p(I_p-WW^*)
\end{matrix}\right).
 \end{split}
\end{align*}
Similarly, $[\n^{F_\phi},\mbf{J}_{F_\phi}]=0$ and $\n^{F_\phi}$ preserves the Hermitian form $\Omega$. The first Chern form of $(F_\phi^+,\n^{F_\phi}|_{F^+_\phi})$ is 
\begin{align*}
\begin{split}
  c_1(F^+_\phi,\n^{F_\phi}|_{F^+_\phi})=\frac{i}{2\pi}\p\b{\p}\log\det(I_p-WW^*)=-\frac{1}{2}\cdot\frac{1}{2\pi}\omega_{\mr{D}^I_{p,q}}.
 \end{split}
\end{align*}
Since the two-form $\omega_{\mr{D}^I_{p,q}}$ is invariant under the $\op{U}(p,q)$-group, so it defines a well-defined two-form on $\wt{\Sigma}\times_\phi {\mr{D}^I_{p,q}}$, we denote it also by $\omega_{\mr{D}^I_{p,q}}$. Hence $\wt{\mbf{J}}^*\omega_{\mr{D}^I_{p,q}}$ is a two-form on $\wt{\Sigma}$. On the other hand, $\wt{\mbf{J}}^*\omega_{\mr{D}^I_{p,q}}$ is a $\phi$-equivariant two-form on $\wt{\Sigma}$, and also descends to a two-form on $\Sigma$. Moreover
\begin{align}\label{Chern forms 3}
\begin{split}
  \wt{\mbf{J}}^*  c_1(F^+_\phi,\n^{F_\phi}|_{F^+_\phi})=\frac{i}{2\pi}\op{Tr}_{F^+_\phi}((\wt{\mbf{J}}^*\n^{F_\phi})^2)=-\frac{1}{4\pi}\wt{\mbf{J}}^*\omega_{\mr{D}^I_{p,q}}.
 \end{split}
\end{align}

When restricted to a small collar neighborhood of $\p\Sigma$,  $\mc{E}\cong p^*(\mc{E}|_{\p\Sigma})$ near the boundary, where 
	$p:\p\Sigma\times [0,1]\to \p\Sigma$ denotes the natural projection. For any  $J\in \mc{J}(\mc{E}|_{\p\Sigma},\Omega)$, then $p^*J\in \mc{J}(\mc{E}|_{\p\Sigma\times[0,1]},\Omega)$.
	Denote 
	\begin{align}\label{J0space}
	\begin{split}
	\mc{J}_o(\mc{E},\Omega)=\{\mbf{J}\in \mc{J}(\mc{E},\Omega)|  &\mbf{J}=p^*J
	\text{ on a small collar neighborhood}\\
	&\text{ of } \p\Sigma, \text{ where } J \in \mc{J}(\mc{E}|_{\p\Sigma},\Omega)\}.	
	\end{split}
	\end{align}
	 For any $\mbf{J}\in \mc{J}_o(\mc{E},\Omega)$, one has  $\mbf{J}=\mbf{J}(x)$ depends only on $x$ near $\partial\Sigma$.
\begin{prop}\label{proppullbackconnection}
	 For any $\mbf{J}\in \mc{J}_o(\mc{E},\Omega)$, the pullback connection $\wt{\mbf{J}}^*\n^{F_\phi}$ is a peripheral connection on $\wt{\mbf{J}}^*F_\phi$.
\end{prop}
\begin{proof}
	 The pullback almost complex structure $\mbf{J}$ depends only on $x$ near $\partial\Sigma$, so $\wt{\mbf{J}}^*\n^{F_\phi}$ has the form $d+C(x)dx$. Since $\n^{F_{\phi}}$ is a complex linear connection on $(F_\phi,\mbf{J}_{F_\phi})$, so 
	$$[\wt{\mbf{J}}^*\n^{F_\phi},\wt{\mbf{J}}^*\mbf{J}_{F_\phi}]=\wt{\mbf{J}}^*[\n^{F_\phi},\mbf{J}_{F_\phi}]=0.$$
 By the definition of $\n^F$, it preserves the Hermitian  form $\Omega$, so the induced connection $\n^{F_\phi}$ also preserves the Hermitian form $\Omega$ on $F_\phi$, hence the pullback connection $\wt{\mbf{J}}^*\n^{F_\phi}$ preserves $\Omega$. Thus the connection $\wt{\mbf{J}}^*\n^{F_\phi}$ is a peripheral connection on $\wt{\mbf{J}}^*F_\phi$.
\end{proof}
\begin{prop}\label{propchern}
For any $\mbf{J}\in \mc{J}_o(\mc{E},\Omega)$, then $\tau \wt{\mbf{J}}^*\n^{F_\phi}\tau^{-1}$ is a peripheral connection on $\mc{E}$ and
	\begin{align*}
		\int_{\Sigma}c_1(\mc{E}^+,\tau \wt{\mbf{J}}^*\n^{F_\phi}\tau^{-1}|_{\mc{E}^+})=-\frac{1}{4\pi}\int_\Sigma \wt{\mbf{J}}^*\omega_{\mr{D}^{\mr{I}}_{p,q}}.
	\end{align*}
\end{prop}
\begin{proof}
Since  $\wt{\mbf{J}}^*\n^{F_\phi}$ is a peripheral connection and by \eqref{tau isomorphism}, so $\tau \wt{\mbf{J}}^*\n^{F_\phi}\tau^{-1}$ is a peripheral connection on $\mc{E}$. On the other hand, one has
	\begin{align*}
		\int_{\Sigma}c_1(\mc{E}^+,\tau \wt{\mbf{J}}^*\n^{F_\phi}\tau^{-1}|_{\mc{E}^+})&=\int_\Sigma c_1(\wt{\mbf{J}}^*F^+_\phi, \wt{\mbf{J}}^*\n^{F_\phi}|_{\wt{\mbf{J}}^*F^+_\phi})=-\frac{1}{4\pi}\int_\Sigma \wt{\mbf{J}}^*\omega_{\mr{D}^{\mr{I}}_{p,q}},
	\end{align*}
which completes the proof.
\end{proof}
Similarly, we have 
\begin{align*}
		\int_{\Sigma}c_1(\mc{E}^-,\tau \wt{\mbf{J}}^*\n^{F_\phi}\tau^{-1}|_{\mc{E}^-})=\frac{1}{4\pi}\int_\Sigma \wt{\mbf{J}}^*\omega_{\mr{D}^{\mr{I}}_{p,q}}.
	\end{align*}
By \eqref{Toledo invariant compact form}, for any $\mbf{J}\in \mc{J}_o(\mc{E},\Omega)$, the Toledo invariant can be given by 
\begin{align}\label{To1}
\begin{split}
 &\quad \mr{T}(\Sigma,\phi) =\frac{1}{2\pi}\int_\Sigma\left(\wt{\mbf{J}}^*\omega_{\mr{D}^{\mr{I}}_{p,q}}-\sum_{i=1}^nd(\chi_i \wt{\mbf{J}}^*\alpha_i)\right)\\
  &=\int_{\Sigma}\left(c_1(\mc{E}^-,\tau \wt{\mbf{J}}^*\n^{F_\phi}\tau^{-1}|_{\mc{E}^-})-c_1(\mc{E}^+,\tau \wt{\mbf{J}}^*\n^{F_\phi}\tau^{-1}|_{\mc{E}^+})\right)-\frac{1}{2\pi}\sum_{i=1}^n\int_\Sigma d(\chi_i \wt{\mbf{J}}^*\alpha_i).
 \end{split}
\end{align}
From \eqref{sign2}, one has
\begin{align}\label{sign3}
\begin{split}
  \op{sign}(\mc{E},\Omega)&=-2 \mr{T}(\Sigma,\phi)-\frac{1}{\pi}\sum_{i=1}^n\int_\Sigma d(\chi_i \wt{\mbf{J}}^*\alpha_i)+\eta(A_{\mbf{J}})\\
  &=-2 \mr{T}(\Sigma,\phi)-\frac{1}{\pi}\sum_{i=1}^n\int_{c_i}\wt{\mbf{J}}^*\alpha_i+\eta(A_{\mbf{J}}).
 \end{split}
\end{align}
Inspired by the above formula, we define
\begin{defn}[Rho invariant of the boundary] 
\label{rhoboundary}
For any representation $\phi:\pi_1(\Sigma)\to \op{U}(E,\Omega)$, the \emph{rho invariant of the boundary} is defined by 
\begin{equation}\label{rho}
  \rro_\phi(\p\Sigma)=-\frac{1}{\pi}\sum_{i=1}^n\int_{c_i}\wt{\mbf{J}}^*\alpha_i+\eta(A_{\mbf{J}})
\end{equation}
	for any $\mbf{J}\in \mc{J}_o(\mc{E},\Omega)$. The rho invariant $\rro_\phi(\p\Sigma)$ is  independent of $\mbf{J}\in\mc{J}_o(\mc{E},\Omega) $ by \eqref{sign3}.
\end{defn}
\begin{rem}\label{remrho}
From the formula \eqref{sign3} for signature, the rho invariant $\rro_\phi(\p\Sigma)$ is independent of the choice of equivariant map $\mbf{J}\in \mc{J}_o(\mc{E},\Omega)$, and just depends on the representation of the boundary $\p\Sigma$.

 The rho invariant was originally introduced by Atiyah, Patodi and Singer \cite[Theorem 2.4]{APSII} for positive definite Hermitian forms. In particular, if the representation of boundary can be extended to a unitary representation of the whole manifold, then the signature of the representation can be expressed in terms of the signature of the trivial representation and the rho invariant. For the case of surfaces, the rho invariant defined in \eqref{rho} is a natural generalization to the group  $\op{U}(p,q)$. One can also refer to \cite{Bohn, KL, Toffoli} etc. for the development of rho invariant. 
\end{rem}
Therefore the proof of the first part of Theorem \ref{0} is completed.

\section{The rho invariant}\label{Rho}

In the previous section, we have encountered a global invariant of a surface, christened ``rho invariant of the boundary'' in Definition \ref{rhoboundary}. In this section, we express it as a sum of contributions of individual boundary components. We define the rho invariant of a triple $(L,\mbf{J},W)$ of an element of $L\in\mathrm{U}(E,\Omega)$, an equivariant splitting $\mbf{J}$ of the corresponding flat bundle over the circle and a fixed point $W$ of $L$ on the closure of the symmetric space of $\mathrm{U}(E,\Omega)$. It is the sum of an integral term $\iota(L,\mbf{J},W)$ and an eta invariant term $\eta(L,\mbf{J},W)$. Then we check that
\begin{itemize}
  \item $\rro(L,\mbf{J},W)$ depends only on $L$ (whereas both $\iota$ and $\eta$ depend nontrivially on $\mbf{J}$), 
  \item $\rro$ is continuous away from unitary matrices admitting $1$ as an eigenvalue,
  \item $\rro$ is a class function (i.e. it depends only on the conjugacy class of its argument).
\end{itemize}

Then we proceed to the proof of Theorem \ref{1}.

We first prove a structure theorem for unitary endomorphisms: they split as the orthogonal direct sum of three types, hyperbolic-unipotents, elliptic-unipotents, and unipotents.
\begin{itemize}
  \item The rho invariant of a hyperbolic-unipotent endomorphism vanishes for symmetry reasons.
  \item The rho invariant of a semi-simple elliptic endomorphism can be directly computed from definitions. Indeed, by a suitable choice of $\mbf{J}$, the boundary operator $A_{\mbf{J}}$ becomes a constant coefficient linear ODE, its spectrum can be explicitly computed from the spectrum of $L$.
  \item The case of elliptic-unipotent endomorphisms follows by a continuity argument.
\end{itemize}
The unipotent case is handled indirectly, by examining the dependence on $\theta\in\br$ of $\rro(e^{i\theta}L)$. Indeed, for $e^{i\theta}\not=1$, $e^{i\theta}L$ is an elliptic-unipotent whose semi-simple part $e^{i\theta}Id$ is easy to treat. Thus $\rro(L)$ appears as the discontinuity at $0$ of the signature of a finite dimensional Hermitian form $H(\theta)$ depending on $\theta$.

This discontinuity can be evaluated when $1$ is a simple eigenvalue of $L$. Indeed, in this case, the Implicit Function Theorem applies, and the unique small eigenvalue of $H(\theta)$ is a smooth function of $\theta$ whose asymptotic behavior at $\theta=0$ can be analyzed: it crosses $0$ only if the dimension is odd. So $\rro$ can be computed for certain normal forms.

Fortunately, every unipotent unitary endomorphism splits orthogonally as a direct sum of unipotents with $1$ as a simple eigenvalue, in the normal form previously alluded to. This completes the proof of Theorem \ref{1}.

\subsection{Basic properties of the rho invariant}

\subsubsection{Definition}

\begin{defn}
For $L\in \mr{U}(E,\Omega)$, defining a bundle $\mc{E}$ over $S^1$, and $\mbf{J}\in \mc{J}(\mc{E},\Omega)$, pick a fixed point $W$ of $L$ in the closure $\overline{\mathrm{D}^\mathrm{I}_{p,q}}$ of the symmetric domain $\mathrm{D}^\mathrm{I}_{p,q}$ and the associated $L$-invariant primitive $\alpha_W$. Set
\begin{equation}
  \rro(L,\mbf{J},W)=\iota(L,\mbf{J},W)+\eta(L,\mbf{J},W):=-\frac{1}{\pi}\int_{S^1}\wt{\mbf{J}}^*\alpha_W+\eta(A_{\mbf{J}}).
\end{equation} 
\end{defn}

With this notation, the previously defined ``rho invariant of the boundary'' is given by
$$
\rro_\phi(\partial\Sigma)=\sum_{\text{boundary component }c}\rro(\phi(c)).
$$
Therefore Theorem \ref{0} states that, given a compact oriented surface with nonempty boundary $\Sigma$ and a homomorphism $\phi:\pi_1(\Sigma)\to \mathrm{U}(E,\Omega)$, with corresponding flat vector bundle $\mc E$ over $\Sigma$,
\begin{align*}
\op{sign}(\mc E,\Omega)=-2\mr{T}(\Sigma,\phi)+\sum_{\text{boundary component }c}\rro(\phi(c)).
\end{align*}

\subsubsection{The rho invariant depends only on holonomy}

\begin{lemma}
On $K=S^1\times[0,1]$, for any $L\in \mr{U}(E,\Omega)$, $\op{sign}(K,\mc{E},\Omega)+2 \mr{T}(K,\phi)=0$.
\end{lemma}

\begin{proof}
We use a $\mbf{J}\in \mc{J}(K,\mc{E},\Omega)$ which is constant in the $[0,1]$ direction. Theorem \ref{0} gives $\op{sign}(K,\mc{E},\Omega)+2 \mr{T}(K,\phi)=\rro(\mbf{J}_{|S^1\times\{1\}},L)-\rro(\mbf{J}_{|S^1\times\{0\}},L)=0$.
\end{proof}

\begin{lemma}
On $S^1$, for any $L\in \mr{U}(E,\Omega)$ and $W$, $\rro(\mbf{J},L,W)$ does not depend on $\mbf{J}$ or $W$.
\end{lemma}

\begin{proof}
Given $\mbf{J}_0,\mbf{J}_1\in \mc{J}(S^1,\mc{E},\Omega)$, and $L$-invariant $W_0,W_1\in \overline{\mr{D}^{\mr{I}}_{p,q}}$, extend $\mbf{J}_0,\mbf{J}_1$ into $\mbf{J}\in \mc{J}(K,\mc{E},\Omega)$. Theorem \ref{0} gives 
$$
\rro(\mbf{J}_1,L,W_1)-\rro(\mbf{J}_0,L,W_0)=\op{sign}(K,\mc{E},\Omega)+2 \mr{T}(K,\phi)=0,
$$
according to the previous Lemma.
\end{proof}

\begin{cor}
$\rro$ is a conjugacy-invariant function on $\mr{U}(E,\Omega)$.
\end{cor}

\subsubsection{Continuity of the rho-invariant}

The eta invariant is continuous unless $A_{\mbf{J}}$ has a kernel, i.e. $1\in \mr{sp}(L)$. View $\rro=\rro(L,W)$ as a function on $F\times\o{\mr{D}^{\mr{I}}_{p,q}}$, where $F=\{L\in \mr{U}(E,\Omega)\,;\,1\notin \mr{sp}(L)\}$. Since $\alpha_W$ depends continuously on $W$, it is continuous. So is its restriction to the subset of pairs $(L,W)$ such that $L(W)=W$. Since there it depends only on $L$, $\rro(L)$ depends continuously on $L$ provided $1\notin \mr{sp}(L)$. 

\subsection{Splitting into types}

\begin{defn}
\label{defhyperell}
Let $(E,\Omega)$ be a complex vector space with non-degenerate Hermitian form $\Omega$. Say that $L\in \op{U}(E,\Omega)$ is \emph{hyperbolic-unipotent} if all its eigenvalues $\lambda$ satisfy $|\lambda|\neq 1$.
Say that $L$ is \emph{elliptic-unipotent} if all its eigenvalues $\lambda$ satisfy $|\lambda|=1$ and $\lambda\neq 1$. If $L$ has a single eigenvalue $1$, $L$ is \emph{unipotent}.
\end{defn}

\begin{prop}
\label{dechyperellu}
Let $(E,\Omega)$ be a complex vector space with non-degenerate Hermitian form $\Omega$ and let $L\in \op{U}(E,\Omega)$. There exists a unique $\Omega$-orthogonal decomposition 
$$
(E,\Omega,L)=(E_{hu},\Omega_{hu},L_{hu})\oplus 
(E_{eu},\Omega_{eu},L_{eu})\oplus (E_{u},\Omega_{u},L_{u})
$$
such that $L_{hu}$ is hyperbolic-unipotent, $L_{eu}$ is elliptic-unipotent and $L_u$ is unipotent. \\
Futhermore
\begin{align*}
\rro(L)=\rro(L_{hu})+\rro(L_{eu})+\rro(L_u).
\end{align*}
\end{prop}

\begin{proof}
The complex vector space $E$ splits into characteristic subspaces of $L$, $E=\bigoplus_\lambda E_\lambda$, $E_\lambda=\mathrm{ker}(L-\lambda I_m)^N$ for $N\geq\mathrm{dim}(E)=m$. 

If $\lambda,\mu$ are eigenvalues and $\lambda\bar\mu\not=1$, then $E_\lambda$ and $E_\mu$ are $\Omega$-orthogonal. Indeed, given $v\in E_\lambda$, the normalized powers $\lambda^{-k}L^k(v)=:P_v(k)$ depend polynomially on $k$. Therefore, if $w\in E_\mu$,
$$
(\lambda\bar\mu)^{-k}\Omega(L^k(v),L^k(w))=\Omega(P_v(k),P_w(k)):=Q(k)
$$
is a scalar polynomial in $k$. Since $\Omega$ is $L$-invariant, the function $k\mapsto (\lambda\bar\mu)^{k}Q(k)$ is constant. If $\lambda\bar\mu\not=1$, this can happen only if $Q$ vanishes identically. In particular, $Q(0)=\Omega(v,w)=0$.

Therefore the subspaces $\tilde E_{\lambda}:=E_\lambda +E_{1/\bar\lambda}$ are mutually orthogonal when distinct.

Let
\begin{align*}
E_{hu}=\bigoplus_{\lambda\,;\,|\lambda|\not=1}\tilde E_{\lambda},\quad 
E_{eu}=\bigoplus_{\lambda\,;\,|\lambda|=1,\,\lambda\not=1}\tilde E_{\lambda},\quad
E_u=E_1.
\end{align*}
These spaces are $L$-invariant and pairwise $\Omega$-orthogonal. In particular, the restriction of $\Omega$ to each summand is non-degenerate. Therefore the rho invariants of summands are well defined. Additivity follows from the definition, since $\mbf{J}$ and $W$ can be chosen to split accordingly.

\end{proof}

\subsubsection{The hyperbolic-unipotent case}

\begin{lemma}
Let $L$ be a hyperbolic-unipotent element of $\mathrm{U}(E,\Omega)$. Then $L$ commutes with an involution $I$ such that $I^*\Omega=-\Omega$ and $I$ fixes a point in the symmetric space of $\mathrm{U}(E,\Omega)$. 
Furthermore, $\rro(L)=0$.
\end{lemma}

\begin{proof}
Let $\Delta\subset\op{spectrum}(L)$ be a subset that contains exactly one element of each pair of eigenvalues $\{\lambda,1/\bar\lambda\}$ of $L$. Set
\begin{align*}
F_+=\bigoplus_{\lambda\in\Delta}E_\lambda,\quad F_-=\bigoplus_{\lambda\not\in\Delta}E_\lambda .
\end{align*}
Then $F_+$ and $F_-$ are totally isotropic, and $E=F_+\oplus F_-$.
Define $I\in\op{End}(E)$ by $I=\pm1$ on $F_{\pm}$. Then $I^2=1$ and $I^*\Omega=-\Omega$. 

Since 
$$
\Omega(v, -IJ I v)=I^*\Omega(Iv, -JIv)=-\Omega(Iv, -JIv)=\Omega(Iv, J Iv),
$$
the formula $J\mapsto -IJI$ defines an action on the space $\mathcal{J}(E,\Omega)$. This action is isometric. Indeed, $I$ defines an automorphism of the group $G=\mathrm{U}(E,\Omega)$, hence an isometry of $G$ equipped with the pseudo-Riemannian structure defined by the Killing form. Also, $I$ maps maximal compact subgroups to maximal compact subgroups, thus it preserves the fibration $G\to \mathcal{J}(E,\Omega)$. Since the metric on the base is induced by the fibration, $I$ preserves it.

Since $\mathcal{J}(E,\Omega)$ is a symmetric space of noncompact type, and $I^2=1$, the fixed point set $\mr{Fix}(I)$ of $I$ is a nonempty subsymmetric space, which is $L$-invariant. One can choose the equivariant map $\mbf{J}$ to lie in $\mr{Fix}(I)$ and the fixed point $W$ to lie in the closure of $\mr{Fix}(I)$.

Since $I$ commutes with $L$, it defines an automorphism, and still denoted by $I$, of the flat bundle $\mc{E}\to S^1$. Since $I\mbf{J}I=-\mbf{J}$, this automorphism induces an orthogonal transformation on the space of sections of $\mc{E}$. One can transport the differential operator $A_{\mbf{J}}$ with $I$, and get an orthogonally equivalent operator $I^*A_{\mbf{J}}$. Since $I\mbf{J}=-\mbf{J}I$, $I^*A_{\mbf{J}}=-A_{\mbf{J}}$. Therefore the set of eigenvalues (with multiplicities) of $A_{\mbf{J}}$ is symmetric, its eta function $\eta(s)$ vanishes identically, and $\eta(A_{\mbf{J}})=0$.

The isometry $I$ of $\mathcal{J}(E,\Omega)$ is anti-holomorphic. Therefore, it changes the sign of the K\"ahler form, $I^*\omega=-\omega$. Since $I$ fixes $W$, it changes the sign of $\alpha_W$, $I^*\alpha_W=-\alpha_W$. Hence
\begin{align*}
\iota(L,\mbf{J},W)=\iota(ILI,I\mbf{J}I,IWI)=\iota(L,-\mbf{J},W)=-\iota(L,\mbf{J},W),
\end{align*} 
showing that $\iota(L,\mbf{J},W)=0$. So $\rro(L)=\iota(L,\mbf{J},W)+\eta(A_{\mbf{J}})=0$.
\end{proof}

\subsubsection{The $1$-dimensional elliptic case}

Let $E=\bc$ and $\Omega$ be a non-degenerate Hermitian form on $E$, i.e. $\Omega(z,w)=\Omega \op{Re}(z\bar w)$ where $\Omega$ is a nonzero real number. Fix $\theta\in(0,2\pi)$ and let $L\in \op{U}(1)$ denote multiplication with $e^{i\theta}$. Let $\mc E$ denote the corresponding flat complex line bundle over $S^1$. Its sections correspond to functions $s:\br\to\bc$ such that $s(x+2\pi)=e^{-i\theta}s(x)$. The space $\mc J(E,\Omega)$ has a single element, $J(z)=iz$ if $\Omega>0$ or $J(z)=-iz$ if $\Omega<0$. Since $J$ commutes with $L$, one can take a constant equivariant map $\mbf{J}(x)=J$. The operator $A_{\mbf{J}}$ is $J\frac{d}{dx}$. A real number $\lambda$ is an eigenvalue of $A_{\mbf{J}}$ if the differential equation $J\frac{ds}{dx}=\lambda s(x)$ has a nonzero solution $s$ such that $s(2\pi)=e^{-i\theta}s(0)$. The solutions are the elements of $-Ji(\frac{\theta}{2\pi}+\bz)$.

The following classical Lemma
\begin{lemma}[{\cite[Lemma 2.10]{APS}}]
\label{eta11}
$$
\lim_{s\to 0}\left[\left(\frac{\theta}{2\pi}\right)^{-s}+\sum_{k=1}^{\infty}\left(\frac{1}{|k+\frac{\theta}{2\pi}|^s}-	\frac{1}{|k-\frac{\theta}{2\pi}|^s}\right)\right]=1-\frac{\theta}{\pi},
$$
\end{lemma}
yields
\begin{align}
\label{eta1}
\eta(A_{\mbf{J}})=-Ji(1-\frac{\theta}{\pi})=\op{sgn}(\Omega)(1-\frac{\theta}{\pi}).
\end{align}

\subsubsection{The elliptic case}

\begin{lemma}\label{rhoell}
If $L\in \mathrm{U}(E,\Omega)$ is elliptic-unipotent, its semi-simple part $S$ fixes a point in the symmetric space of $\mathrm{U}(E,\Omega)$. Furthermore, there is a unique $\op{Ad}_S$-invariant element $B\in \mathfrak{u}(E,\Omega)$ such that $\exp(2\pi iB)=S$ and $B$ has all its eigenvalues in $(0,2\pi)$. Let $\{e_j\}$ be an $\Omega$-orthonormal basis of eigenvectors of $B$. Then
\begin{align*}
 \rro(S)=\op{sign}(\Omega)-2\sum_j \Omega(Be_j,e_j).
\end{align*}
\end{lemma}

\begin{proof} 
Let $L=SU$ be the Jordan decomposition of $L$. Since the eigenvalues of $L$ are unit complex numbers and $S$ is semi-simple, the subgroup generated by $S$ is relatively compact. Therefore, it fixes a point $J$ in the symmetric space $\mathcal{J}(E,\Omega)$. 

One can use the constant map $\mbf{J}=J$ and $W=J$. For this choice, $\iota(L,\mbf{J},W)=0$. Since the stabilizer of $J$ is a conjugate of $\mathrm{U}(p)\times\mathrm{U}(q)$, one can assume that $S=S_+ \oplus S_- \in\mathrm{U}(p)\times\mathrm{U}(q)$ and use the expression (\ref{eta1}), found in Subsection \ref{Appeta},
\begin{align*}
\eta(S,\mbf{J},W)=\eta(S_+,\mbf{J},W)+\eta(S_-,\mbf{J},W),
\end{align*}
where, if the eigenvalues of $S_\pm$ are written $e^{i\theta_{\pm,j}}$ with $\theta_{\pm,j}\in(0,2\pi)$,
\begin{align*}
\eta(S_+,\mbf{J},W)=\sum_{j}\left(1-\frac{\theta_{+,j}}{\pi}\right),\quad \eta(S_-,\mbf{J},W)=-\sum_{j}\left(1-\frac{\theta_{-,j}}{\pi}\right).
\end{align*}
Let $\{e_{\pm,j}\}$ be an orthonormal basis of eigenvectors, i.e. $\Omega(e_{\pm,j},e_{\pm,j})=\pm 1$ and $Se_{\pm,j}=\exp(i\theta_{\pm,j})e_{\pm,j}$. Define $B\in \op{End}(E)$ by $Be_{\pm,j}=\frac{\theta_{\pm,j}}{2\pi}e_{\pm,j}$. Then $\exp(2\pi i B)=S$ and
\begin{align*}
\eta(S,\mbf{J},W)=p-q-2\sum_{j}\pm\frac{\theta_{\pm,j}}{2\pi}=p-q-2\sum_{\pm,j} \Omega(Be_{\pm,j},e_{\pm,j}).
\end{align*}
\end{proof}

\subsubsection{The elliptic-unipotent case}

By definition, an elliptic-unipotent element $L$ of $\mr{U}(p,q)$ does not have $1$ as an eigenvalue. Write $L=SU$ where $S$ is semi-simple elliptic and $U$ is unipotent and commutes with $S$. Then $S$ belongs to the closure of the conjugacy class of $L$ (one can apply \cite[Proposition 8.3]{KimPansu} to the homomorphism from $\bz$ generated by $L$), on which $\rro$ is continuous, hence constant. It follows that $\rro(L)=\rro(S)$.

\subsection{The unipotent case}

Let $L\in \mr{U}(p,q)$ be unipotent. The idea is to study $\rro$ along the curve $\theta\mapsto e^{i\theta}L$, for $\theta$ close to $0$. When $\theta\not=0$, $e^{i\theta}L$ is elliptic-unipotent. Formula \ref{rhoell} gives
\begin{align}\label{rhoellu}
\rro(e^{i\theta}L)=\begin{cases}
(p-q)\left(1-\frac{\theta}{\pi}\right) & \text{if }\theta>0, \\
(p-q)\left(1-\frac{\theta+2\pi}{\pi}\right) & \text{otherwise}.
\end{cases}
\end{align}
Note that all $e^{i\theta}L$ define the same automorphism of the domain, hence the same $1$-form $\alpha_W$ serves for all of them, so the term $-\frac{1}{\pi}\int_{S^1}\wt{\mbf{J}}^*\alpha_W$ does not depend on $\theta$. There remains to study the eta invariant term. We show that it can be expressed in terms of a finite dimensional Hermitian form.

\subsubsection{Reduction to a finite dimensional spectral problem}

Given $\epsilon>0$, for a self-adjoint first order differential operator $A$ on $S^1$, denote by
\begin{align*}
\eta^\epsilon_A(s)=\sum_{\lambda\in\mr{sp}(A),\,|\lambda|>\epsilon}\frac{\mr{sign}(\lambda)}{\lambda^s}.
\end{align*}
Then for every $A_0$, provided $\pm\epsilon$ are away from the spectrum of $A_0$, one can pick a neighborhood of $A_0$ on which $\eta^\epsilon$ is a continuous function of $s$ and $A$. On this neighborhood,
\begin{align*}
\eta(A)=\eta^\epsilon_{A}(0)+\sum_{\lambda\in\mr{sp}(A),\,0<|\lambda|<\epsilon}\mr{sign}(\lambda),
\end{align*}
where $A\mapsto\eta^\epsilon_A$ is a continuous function of $A$ and the sum has finitely many terms. Therefore the discontinuity of the eta invariant at $A_0$ has to do with the signs of the finitely many eigenvalues below level $\epsilon$.

\subsubsection{Algebraic expression of the spectrum}

Since the unipotent radical of $\mr{SU}(p,q)$ is a simply connected nilpotent Lie group, there is a unique nilpotent element $B\in\mathfrak{su}(p,q)$ such that $\exp(2\pi B)=L$. Pick a $J_0\in \mc{J}(E,\Omega)$, and set
$$
\mbf{J}(x)=\exp(-xB)J_0\exp(xB).
$$
Then the spectrum of $A_{\mbf{J},e^{i\theta}L}=\mbf{J}\frac{d}{dx}$ near zero is the set of real numbers $\sigma$ near zero with the eigenvector sections
$$s(x)=\exp(-xB) \exp(x(-\sigma J_0+B))s(0),$$ satisfying
$$s(2\pi)=\exp(-2\pi B)\exp(2\pi(-\sigma J_0+B))s(0)=e^{-i\theta I}\exp(-2\pi B) s(0),$$
such that $-\sigma J_0+B+i\frac{\theta}{2\pi}\,Id$ has a nontrivial kernel, with a multiplicity equal to the dimension of this kernel. 

 Alternatively, it is the spectrum near zero of the self-adjoint operator $-(B +i\theta\, Id)J_0$. We conclude that for $\epsilon>0$ avoiding the spectrum of $BJ_0$, the function
\begin{align*}
\theta\mapsto &\eta(A_{\mbf{J},e^{i\theta}L})-\sum_{\sigma\in\mr{sp}(-(B +i\theta\,Id)J_0),\,0<|\sigma|<\epsilon}\mr{sign}(\sigma)\\
&=\eta(A_{\mbf{J},e^{i\theta}L})+\sum_{\sigma\in\mr{sp}((B +i\theta\,Id)J_0),\,0<|\sigma|<\epsilon}\mr{sign}(\sigma)
\end{align*}
is continuous in a sufficiently small neighborhood of $0$. Since the finite sum
\begin{align*}
\sum_{\sigma\in\mr{sp}((B +i\theta\,Id)J_0),\,|\sigma|>\epsilon}\mr{sign}(\sigma)
\end{align*}
is constant in a sufficiently small neighborhood of $0$, and $\rro=\eta+$constant for our elliptic-unipotent family, we can rephrase the conclusion as follows.

\begin{prop}\label{continuous}
The function 
\begin{align*}
\theta\mapsto \rro(e^{i\theta}L)+\sum_{\sigma\in\mr{sp}((B +i\theta\,Id)J_0),\,|\sigma|>0}\mr{sign}(\sigma)
\end{align*}
is continuous at $0$.
\end{prop}

\subsubsection{Expression in terms of a Hermitian form}

By assumption on $J_0$, the Hermitian form $(u,v)\mapsto H(u,v)=i\Omega(u,J_0v)$ is positive definite. Therefore, the difference of the number of positive and negative eigenvalues of $(B +i\theta\,Id)J_0$ is equal to the signature of the Hermitian form $(u,v)\mapsto \Omega(i(B +i\theta\,Id)J_0 u,J_0 v)$. This suggests defining the following Hermitian form, which has the same signature (up to sign).

\begin{defn}
Given a nilpotent element $B\in\mathfrak{su}(E,\Omega)$, define the (possibly indefinite) Hermitian form $H_{B+i\theta}$ by
\begin{align*}
H_{B+i\theta}(u,v)=\Omega(i(B +i\theta\,Id)u,v).
\end{align*}
\end{defn}

\begin{prop}\label{thm}
Let $L=\exp(2\pi B)$ be a unipotent element of $\mr{U}(E,\Omega)$. Consider the function
$$
\theta\mapsto \sigma(\theta):=\op{sign}(H_{B+i\theta}).
$$
Here, the signature is the number of plus signs minus the number of minus signs among eigenvalues, irrelevant of the dimension of the kernel. Then
\begin{align*}
\rro(L)=-\sigma(0)+\sigma(0^+)+\op{sign}(\Omega)=-\sigma(0)+\sigma(0^-)-\op{sign}(\Omega).
\end{align*}
\end{prop}

Note that the function $\sigma$ is a conjugacy invariant of $B$.

\subsubsection{Proof of Proposition \ref{thm}}

Proposition \ref{continuous} gives that
$$
\theta\mapsto\rro(e^{i\theta}L)+\sigma(\theta)
$$
is continuous at $0$. Therefore, using Formulae \ref{rhoellu},
\begin{align*}
\rro(L)+\sigma(0)&=\rro(e^{i0^+}L)+\sigma(0^+)\\
&=p-q+\sigma(0^+).
\end{align*}
Similarly,
\begin{align*}
\rro(L)+\sigma(0)&=\rro(e^{i0^-}L)+\sigma(0^-)\\
&=-(p-q)+\sigma(0^-).
\end{align*}

\subsubsection{Examples}
\begin{itemize}
  \item[(1)] If $L$ is the identity, then $B=0$, $H_{i\theta}=i\Omega(i\theta \cdot,\cdot)=-\theta\Omega$, hence $\sigma(\theta)=-\op{sgn}(\theta)\op{sign}(\Omega)$. Here, the sign function $\op{sgn}$ takes values $-1$, $0$ and $1$. Therefore $\sigma(0)=0$, $\sigma(0^+)=-\op{sign}(\Omega)$, $\sigma(0^-)=\op{sign}(\Omega)$, hence
$$
\rro(\mr{Id})=\op{sign}(\Omega)-\op{sign}(\Omega)=0.
$$

\item[(2)] Let $L$ be unipotent in $\mathrm{U}(1,1)$. On $E=\bc^2$, choose $\Omega=\begin{pmatrix}
1 & 0 \\
0 & -1 
\end{pmatrix}$. Then $B=U^{-1}\begin{pmatrix}
0 & \mu \\
0 & 0
\end{pmatrix}U=\frac{1}{2}\begin{pmatrix}
	i\mu & i\mu\\
	-i\mu & -i\mu
\end{pmatrix}
\in\mathfrak{su}(E,\Omega)$, where $U$ is defined in \eqref{U}. The matrix of $H_{B+i\theta}$ is $-\frac{1}{2}\begin{pmatrix}
\mu+\theta & \mu   \\
\mu & \mu-\theta
\end{pmatrix}$. Its characteristic polynomial is $X^2+\mu X-\frac{1}{4}\theta^2=0$, which has roots
$$X=-\frac{\mu}{2}\pm\frac{1}{2}\sqrt{\mu^2+\theta^2},$$
 hence
\begin{itemize}
  \item if $\theta\not=0$, two eigenvalues of opposite signs, $\sigma(\theta)=0$ ;
  \item if $\theta=0$, eigenvalues $0$ and $-\mu$, $\sigma(0)=-\op{sgn}(\mu)$.
\end{itemize}
Proposition \ref{thm} gives $\rro(\exp(2\pi B))=\op{sgn}(\mu)$, which is consistent with Table (\ref{etadim2}).

\end{itemize}

\begin{rem}
\label{modZ}
We shall see in Lemma \ref{rotrho} that modulo $\bz$, the rho invariant coincides with minus twice Burger-Iozzi-Wienhard's rotation number. 
\end{rem}

\subsubsection{The rho invariant of nilpotents whose kernel is $1$-dimensional}
\label{1dimensional}

The matrix of $H_{B+i\theta}$ is $i(B+i\theta\,Id)^\top \Omega$. Its determinant is equal to $\op{det}(\Omega)(-\theta)^n$, $n=\op{dim}(E)$. When $\theta=0$, its kernel coincides with $B$'s kernel. If $\op{dim}\op{Ker}(B)=1$, when $\theta=0$, the derivative at $0$ of the characteristic polynomial 
$$
P(x,\theta):=\op{det}(xI-i(B+i\theta\,Id)^\top \Omega)
$$ 
does not vanish. Therefore, according to the Implicit Function Theorem, the unique eigenvalue $\lambda(\theta)$ which is close to $0$ varies smoothly with $\theta$ in a neighborhood of $\theta=0$. Writing
$$
P(x,\theta)=\sum_{k=0}^n a_k(\theta)x^k,
$$
we know that $a_0(\theta)=\op{det}(\Omega)\theta^n$ and $a_1(0)\not=0$. Assuming that $\lambda(\theta)=c\theta^k+o(\theta^k)$ for some $0<k<n$ and $c\not=0$, we see that
$$
0=P(\lambda(\theta),\theta)=a_1(0)c\theta^k+o(\theta^k),
$$
contradiction. Hence the Taylor expansion of $\lambda(\theta)$ starts with $\lambda(\theta)=c\theta^n+o(\theta^n)$, and 
$$
0=P(\lambda(\theta),\theta)=\op{det}(\Omega)\theta^n+a_1(0)c\theta^n+o(\theta^n).
$$
Therefore $c=-\frac{\op{det}(\Omega)}{a_1(0)}$ and $\lambda(\theta)=cf(\theta)^n$ for some smooth function $f$ such that $f(0)=0$ and $f'(0)=1$.

\subsubsection{Discussion}
\label{discussion}

If $n$ is even and $c>0$, $\lambda(\theta)$ does not change sign near $0$, it merely disappears from the count when $\theta=0$, therefore $\sigma(0^+)=\sigma(0^-)=\sigma(0)+1$. Therefore Proposition \ref{thm} implies that $\op{sign}(\Omega)=0$ and
$$
\rro(\exp(2\pi B))=1.
$$
If $n$ is even and $c<0$, $\sigma(0^+)=\sigma(0^-)=\sigma(0)-1$. In this case, $\op{sign}(\Omega)=0$ and
$$
\rro(\exp(2\pi B))=-1.
$$

If $n$ is odd and $c>0$, $\lambda(\theta)$ changes sign near $0$, $\sigma(0^+)-\sigma(0)=\sigma(0)-\sigma(0^-)=1$. Therefore $\op{sign}(\Omega)=-1$ and
$$
\rro(\exp(2\pi B))=0.
$$
If $n$ is odd and $c<0$, $\sigma(0^+)-\sigma(0)=\sigma(0)-\sigma(0^-)=-1$. Therefore $\op{sign}(\Omega)=1$ and
$$
\rro(\exp(2\pi B))=0.
$$

There remains to compute the sign of $c=c(E,\Omega,B)$.

\subsubsection{Examples of nilpotent elements of $\mathfrak{su}(p,q)$}
\label{exnil}

Let $B\in\mathfrak{su}(E,\Omega)$ be nilpotent. Let $(e_j)_{1\leq j\leq n}$ be a Jordan basis for $B$, i.e. for all $j$, $Be_j=e_{j-1}$ or $0$. Let 
$$
Z=\{j\in\{1,\ldots,n\}\,;\,Be_j=0\}.
$$
In this basis, the matrix of $\Omega$ has entries $(\omega_{k,\ell})$, and 
\begin{align*}
\forall k\notin Z,~\forall\ell\notin Z,&\quad \omega_{k,\ell-1}=-\omega_{k-1,\ell},\\
\forall k\in Z,~\forall\ell\notin Z,&\quad \omega_{k,\ell-1}=0,\\
\forall k\notin Z,~\forall\ell\in Z,&\quad \omega_{k-1,\ell}=0.
\end{align*}
For instance, if $B$ has only one Jordan block, i.e. $Z=\{1\}$, then the matrix $\Omega$ is anti-lower-triangular, and along each nonzero antidiagonal, the same number arises with alternating signs. The simplest examples are the anti-diagonal matrix with entries alternatively equal to $i$ and $-i$ (if $n$ is even) or with entries alternatively equal to $1$ and $-1$ with $1$ (resp. $-1$) on the diagonal (if $n$ is odd). Its signature is $0$ (if $n$ is even) or $1$ (resp. $-1$) (if $n$ is odd). We shall see in the next paragraph that every indecomposable nilpotent element of some $\mathfrak{su}(p,q)$ is conjugate under $\mr{SU}(p,q)$ to one of these types.

Let us compute the relevant derivative $a_1(0)$ for those examples.

\begin{lemma}
Let $D$ be an $n\times n$ matrix whose entries $a_{j,k}$ vanish except those along the second antidiagonal, i.e. when $j+k=n+2$. Let $P(x)=\op{det}(xI-D)$ denote the characteristic polynomial of $D$. Then
$$
\frac{\partial P}{\partial x}(0)=(-1)^{n-1+\lfloor \frac{n-1}{2} \rfloor}\prod_{j=2}^n d_{j,n+2-j}.
$$
\end{lemma}

\begin{proof}
In general,
$$
\frac{\partial P}{\partial x}(0)=\op{Trace}(\op{adj}(-D)),
$$
where the entries of the adjugate $\op{adj}(M)$ of a matrix $M$ are signed cofactors. Here, we need to compute only diagonal cofactors $\op{adj}(M)_{j,j}=\op{det}(M_j)$, which all come with a plus sign. Each $(n-1)\times (n-1)$-matrix $M_j$ is anti-triangular. All but the first one $M_1$ have at least one zero entry on their diagonal, so only $\op{det}(M_1)$ can be nonzero. Its diagonal entries are all entries of the nonzero antidiagonal, whence the announced formula, up to a sign. When all $m_{j,n+2-j}=1$, $M_{1}$ is the matrix of the permutation that exchanges $j$ and $n-1-j$. It has $\lfloor \frac{n-1}{2} \rfloor$ $2$-cycles, plus a $1$-cycle if $n$ is even, therefore its signature is $(-1)^{\lfloor (n-1)/2 \rfloor}$. When we substitute $M$ with $-D$, an extra factor $(-1)^{n-1}$ shows up.
\end{proof}

When $\Omega$ has only one nonzero antidiagonal,
$$
\omega_{j,n+1-j}=(-1)^{j+1} \epsilon,
$$
where $\epsilon\in\{1,i,-1,-i\}$, the Hermitian matrix $D=iB^\top\Omega$ has only one nonzero antidiagonal, the second one, with nonzero entries $d_{j,n+2-j}=(-1)^{j}i\epsilon$. Therefore
$$
a_1(0)=\frac{\partial P}{\partial x}(0)=(-1)^{n-1+\lfloor \frac{n-1}{2} \rfloor}(-1)^{\lfloor \frac{n-1}{2} \rfloor}(i\epsilon)^{n-1}=(-i\epsilon)^{n-1}.
$$
Also,
\begin{align*}
\op{det}(\Omega)=\epsilon^n,
\end{align*}
and $\epsilon=\omega_{1,n}=\Omega(B^{n-1}e_n,e_n)$, so
\begin{align}
\label{c}
\begin{split}
c&=-\frac{\op{det}(\Omega)}{\frac{\partial P}{\partial x}(0)}=-\frac{(-\epsilon)^n}{(i\epsilon)^{n-1}}=-(i^{n-1}\epsilon)=-(i^{n-1}\Omega(B^{n-1}e_n,e_n))\\
&=-\Omega((iB)^{n-1}e_n,e_n).
\end{split}
\end{align}

\subsubsection{Nilpotent conjugacy classes in $\mathfrak{su}(p,q)$}

The following result is borrowed from N. Burgoyne and R. Cushman's work \cite{BurgoyneCushman}.

\begin{prop}
\label{BurgoyneCushman}
Consider complex vector spaces $E$ equipped with non-degenerate Hermitian forms $\Omega$ and nilpotent skew-hermitian endomorphisms $N$. 
\begin{enumerate}

  \item Any such triple $(E,\Omega,N)$ is a direct sum of indecomposables. 
  
  \item A triple is indecomposable if and only if $(E,N)$ is a single Jordan block. Then $\bar E=E/NE$ is $1$-dimensional. If $n=\op{dim}(E)$, the Hermitian form
$$
(u,v)\mapsto \tau_{n-1}(u,v)=\Omega((iN)^{n-1}u,v)
$$
induces a non-degenerate Hermitian form $\bar\tau$ on $\bar E$.
  
  \item Two indecomposable triples are isomorphic if and only if they have the same dimension $n$ and the quotients $(\bar E,\bar\tau)$ have equal signatures.

\end{enumerate}

\end{prop}

For the reader's convenience, a detailed proof of Proposition \ref{BurgoyneCushman} is provided in the Appendix, Subsection \ref{nilpotentsu}.

\subsubsection{Rho invariants of nilpotent elements of $\mathfrak{su}(p,q)$}

\begin{defn}
\label{defsign}
Let $B\in\mathfrak{su}(E,\Omega)$ be a single Jordan block of even dimension $n$. Its \emph{sign} $\op{sgn}(E,\Omega,B)\in\{-1,1\}$ is the signature of the non-degenerate Hermitian form induced by $\tau_{n-1}:(u,v)\mapsto \Omega((iB)^{n-1}u,v)$ on the $1$-dimensional space $E/BE$.
\end{defn}

Piecing together the above results and computations, we get

\begin{thm}
\label{value}
Let $B\in\mathfrak{su}(E,\Omega)$ be nilpotent. Then $E$ admits an orthogonal decomposition into $B$-invariant subspaces $E_j$ which are single Jordan blocks,
$$
(E,\Omega,B)=\bigoplus_j (E_j,\Omega_j,B_j).
$$
Furthermore,
$$
\rro(\exp(2\pi B))=\sum_{\op{dim}(E_j)\,even} -\op{sgn}(E_j,\Omega_j,N_j).
$$
In particular, if the signature of $\Omega$ is $(p,q)$, $|\rro(\exp(2\pi B))|\leq\min\{p,q\}$.
\end{thm}

\begin{proof}
Split $(E,\Omega,B)$ into indecomposable blocks $(E_j,\Omega_j,B_j)$, according to Proposition \ref{BurgoyneCushman}. Then $\rro(\exp(2\pi B))=\sum_j \rro(E_j,\Omega_j,\exp(2\pi B_j))$. Each $(E_j,\Omega_j,B_j)$ is isomorphic to one of the examples of paragraph \ref{exnil}. Its rho invariant is equal to $-1$, $0$ or $1$ according to the discussion of paragraph \ref{discussion}.
\begin{itemize}
\item If $n_j=\op{dim}(E_j)$ is odd, $\rro(E_j,\Omega_j,\exp(2\pi B_j))=0$.
\item If $n_j$ is even, then $\rro(E_j,\Omega_j,B_j)$ is the sign of the parameter $c(E_j,\Omega_j,B_j)$ introduced in paragraph \ref{1dimensional}. This parameter is given in Equation \ref{c}: $c=-\tau_{n_j-1}(e,e)$ for some nonzero vector $e$, hence the sign of $c$ is opposite to the sign defined in \ref{defsign}.
\end{itemize}

Finally, the bound on the rho invariant follows from that fact that the pieces which contribute to $\rro$ have vanishing signature.
\end{proof}

\begin{rem}
We note that for $L\in \op{U}(p,q)$, $\rro(L)$ is an integer if and only if the semi-simple part of the elliptic-unipotent summand admits a conjugate in a subgroup $\mathrm{SU}(p')\times \mathrm{SU}(q')\subset \mathrm{U}(p,q)$. 
\end{rem}

\section{Atiyah's signature cocycle and section $\sigma$}
\label{Atiyahsigma}

In this section, we relate Meyer and Atiyah's cocycle with rotation numbers and rho invariants, and we compute Atiyah's section $\sigma$ for the group $\mr{U}(p,q)$.

\subsubsection{The signature cocycle}\label{cocycle}

Let $\Sigma_3$ be the three-hole sphere. Let us fix an orientation of $\Sigma_3$, a base-point $*\in \Sigma_3$ and loops $c_1,c_2,c_3$ based at $*$ which represent the oriented boundary components in such a way that $c_1c_2c_3$ is null homotopic. Given $A$ and $B\in \mr{U}(E,\Omega)$, let $\phi:\pi_1(\Sigma_3,*)\to\mr{U}(E,\Omega)$ be a representation with $\phi(c_1)=A^{-1}$, $\phi(c_2)=B^{-1}$ and $\phi(c_3)=(A^{-1}B^{-1})^{-1}=BA$.\footnote{Atiyah's convention for turning a representation into a flat bundle differs from ours. This is why we specify $A^{-1}$, $B^{-1}$ instead of $A$ and $B$ in the representation.}{}
Atiyah's \emph{signature cocycle} is the symmetric function $\mr{sign}:\mr{U}(E,\Omega)\times \mr{U}(E,\Omega)\to \bz$ defined as follows: the integer $\mr{sign}(A,B)$ is the signature of the flat unitary bundle $(\mc{E},\Omega)$ on $\Sigma_3$ associated with $\phi$.

\subsubsection{Cocycles and central extensions}

Atiyah's section $\sigma$ arises from the following general fact relating sections of central extensions, $2$-cocycles and $1$-cochains.

\begin{lemma}
\label{b}
Let $Z$ be an abelian group, let $0\to Z\stackrel{j}{\to} H\stackrel{p}{\to} G\to 1$ be a central extension of groups. 
\begin{enumerate}
  \item If $\sigma:G\to H$ is an arbitrary section of $p$, the function $c:G\times G\to Z$ defined by
  \begin{align*}
c(g,g')=j^{-1}(\sigma(g)\sigma(g')\sigma(gg')^{-1})
\end{align*}
is a $2$-cocycle, and any other choice of section leads to a cohomologous $2$-cocycle. Hence the cohomology class $[c]\in H^2(G,Z)$ that classifies the extension is well defined. Conversely, every $2$-cocycle in the classifying class is obtained from some section $\sigma:G\to H$.
  \item If $c:G\times G\to Z$ is a $2$-cocycle classifying the extension, then the $2$-cocycle $p^*c$ is exact. If in addition $\mr{Hom}(H,Z)=0$, there is a unique $1$-cochain $b:H\to Z$ such that $db=-p^*c$. Furthermore, $c$ is obtained from the section $\sigma$ uniquely determined by the condition $b\circ\sigma=0$. 
  \item If $Z=\bz$, if no nonzero multiple of $c$ is a coboundary and $\mr{Hom}(G,\bz)=0$, then $\mr{Hom}(H,\bz)=0$.
\end{enumerate}
\end{lemma}

\begin{proof}
1. The given section defines a bijection of $H$ with $Z\times G$. In these coordinates, the multiplication reads
\begin{align*}
(z,g)(z',g')=(z+z'+c(g,g'),gg').
\end{align*} 
The associativity of this law is equivalent to the cocycle equation. An other section $\sigma'$ can be written $\sigma'=(j\circ f)\sigma$ where $f:G\to Z$ is arbitrary. Its $2$-cocycle is
\begin{align*}
c'(g,g')&=j^{-1}((j\circ f)(g)\sigma(g)(j\circ f)(g')\sigma(g')(j\circ f)(gg')\sigma(gg')^{-1})\\
&=j^{-1}((j\circ f)(g)(j\circ f)(g')(j\circ f)(gg')^{-1})j^{-1}(\sigma(g)\sigma(g')\sigma(gg')^{-1})\\
&=f(g)+f(g')-f(gg')+c(g,g')=df(g,g')+c(g,g').
\end{align*}
Conversely, every $2$-cocycle in the classifying class is of the form $c+df$, hence arises from the section $(j\circ f)\sigma$.

2. By assumption, $H$ is the set $Z\times G$ equipped with the multiplication $(z,g)(z',g')=(z+z'+c(g,g'),gg')$.
Define $b:H\to Z$ by $b(z,g)=z$. Then 
\begin{align*}
db((z,g),(z',g'))&=b(z,g)+b(z',g')-b((z,g)(z',g'))\\
&=z+z'-z-z'-c(g,g')=-c(g,g').
\end{align*}
The section $\sigma$ defined by $b\circ\sigma=0$ is $\sigma(g)=(0,g)$ in our notation. The map $j$ is $j(z)=(z,e)$. The corresponding $2$-cocycle is
\begin{align*}
j^{-1}(\sigma(g)\sigma(g')\sigma(gg')^{-1})
&=j^{-1}((0,g)(0,g')(0,gg')^{-1})\\
&=j^{-1}((c(g,g'),gg')(-c(gg',g'^{-1}g^{-1}),g'^{-1}g^{-1}))\\
&=j^{-1}((c(g,g')-c(gg',g'^{-1}g^{-1})+c(gg',g'^{-1}g^{-1}),e))\\
&=c(g,g').
\end{align*}

If $b'$ is an other $1$-cochain on $H$ such that $db'=-c$, then $d(b'-b)=0$, $b'-b:H\to Z$ is a homomorphism. If $\mr{Hom}(H,Z)=0$, $b'=b$, whence the uniqueness of $b$. 

3. Up to adding a coboundary, one can assume that $c(e,e)=0$, and hence that $c(e,g)=0$ for all $g\in G$. Let $\psi:H\to\bz$ be a homomorphism. Then $\psi\circ j:\bz\to\bz$ is a homomorphism, so there exists $n\in\bz$ such that $\psi\circ j=n\,Id$. For all $z\in \bz$ and $g\in G$, 
\begin{align*}
\psi(z,g)=\psi(z+c(e,g),g)=\psi((z,e)(0,g))=\psi(j(z)(0,g)=\psi\circ j(z)+\psi(0,g).
\end{align*}
For $g,g'\in G$,
\begin{align*}
\psi(c(g,g'),gg')=\psi((0,g)(0,g'))=\psi(0,g)+\psi(0,g').
\end{align*}
Let $f(g)=\psi(0,g)$. Then
\begin{align*}
df(g,g')&=\psi(0,g)+\psi(0,g')-\psi(0,gg')\\
&=\psi(c(g,g'),gg')-\psi(0,gg')\\
&=\psi\circ j(c(g,g'))=n\,c(g,g').
\end{align*}
By assumption, $n\,c(g,g')$ is not a coboundary unless $n=0$. Thus $n=0$, $\psi$ descends to a homomorphism $G\to\bz$, which vanishes by assumption, so $\psi=0$. We conclude that $Hom(H,\bz)=0$. 

\end{proof}

The item (3) of Lemma \ref{b} applies to our situation, since $H^2(\mathrm{U}(p,q),\bz)$ is a free abelian group, and the signature cocycle represents a nonzero cohomology class, see paragraph \ref{b2} below. Furthermore, $\mr{Hom}(\mr{U}(p,q),\bz)=0$.

\begin{defn}[Atiyah]
\label{defatiyah}

The signature cocycle $\mr{sign}$ defined in paragraph \ref{cocycle} determines a central extension 
$$
0\to \bz\stackrel{j_2}{\longrightarrow}\mr{U}(p,q)_2\stackrel{p_2}{\longrightarrow} \mr{U}(p,q)\to 1.
$$
There is a unique section $\sigma:\mr{U}(p,q)\to \mr{U}(p,q)_2$ such that for all $L,L'\in \mr{U}(p,q)$,
\begin{align*}
\mr{sign}(L,L')=j_2^{-1}(\sigma(L)\sigma(L')\sigma(LL')^{-1}).
\end{align*}
We shall call it \emph{Atiyah's section $\sigma$}.
\end{defn}

According to Lemma \ref{b}, 
\begin{enumerate}

  \item The pulled-back cocycle $p_2^*\mr{sign}$ on $\mr{U}(p,q)_2$ is a coboundary. There is a unique $1$-cochain $b_2$ on $\mr{U}(p,q)_2$ such that
$$
p_2^*\mr{sign}=-db_2.
$$

  \item Atiyah's section $\sigma$ is uniquely determined by the requirement $b_2\circ\sigma=0$.
  
  \item On fibers of $p_2$, $b_2$ restricts to isomorphisms to $\bz$,
$$
\forall n\in\bz,~\forall g\in \mr{U}(p,q)_2, \quad b_2(j_2(n)g)=n+b_2(g).
$$

\end{enumerate}

The goal of the next paragraphs is to determine $b_2$. This will be achieved in paragraph \ref{b2}.

\subsubsection{Rotation numbers}

The general notion, for locally compact groups $G$, is due to \cite{BIW}. Given a bounded Borel cohomology class $\kappa\in \hat{\mr{H}}_{cb}^2(G,\bz)$, the corresponding \emph{rotation number} is a continuous map $\mr{Rot}_\kappa: G\to  \br/\bz$ defined as follows. For $g\in G$, let $B$ denote the closed subgroup generated by $g$. Since $\mr{H}^2_{cb}(B,\br)=0$, the long exact sequence arising from the exponential short exact sequence $0\to\bz\to\br\to \br/\bz\to 0$ gives $\hat{\mr{H}}_{cb}^2(B,\bz)\simeq \mr{Hom}_{c}(B,\br/\bz)$, whence a homomorphism $f_{B,\kappa}:B\to\br/\bz$, and a number
$$
\mr{Rot}_\kappa(g)=f_{B,\kappa}(g).
$$

\subsubsection{Integrality}

Let $(E,\Omega)$ be a complex vector space equipped with a nondegenerate indefinite Hermitian form $\Omega$. We are interested in the rotation number associated with a suitable multiple of the bounded Borel cohomology class $\kappa$ of the K\"ahler form $\omega$ of the Hermitian symmetric space of $\mr{U}(E,\Omega)$. From general principles, it follows that some multiple of $\kappa$ is integral (see \cite[bottom of page 526 and Proposition 7.7]{BIW}). Theorem \ref{0} suggests that $2\kappa$ is integral. Indeed, it indicates that the $2$-cocycle $\mathrm{sign}+d\rro$ is a representative of $-2\kappa$. So does $\mathrm{sign}$, which is integer-valued. This is indeed the case.

\begin{lemma}\label{integrality}

Let the symmetric space $\ms{X}$ of $G=\mr{U}(E,\Omega)$ be equipped with the invariant metric whose minimal holomorphic sectional curvature is equal to $-1$. Then 
twice
the K\"ahler bounded cohomology class is integral, i.e. $$
2\kappa\in\hat{\mr{H}}_{cb}^2(\mr{U}(E,\Omega),\bz).
$$

\end{lemma}

\begin{proof}
Let $(E_1,\Omega_1)$ and $(E_2,\Omega_2)$ be nondegenerate Hermitian spaces, and let $(E,\Omega)=(E_1,\Omega_1)\oplus (E_2,\Omega_2)$. The corresponding embedding $\ms{X}_1\times \ms{X}_2\to \ms{X}$ between symmetric spaces is isometric, totally geodesic and holomorphic, as is visible on Equation \ref{kahlermetric}. Hence the normalized K\"ahler form of $\ms{X}$ restricts to the normalized K\"ahler forms on the factors. 

Fix an origin $o\in \ms{X}$. For $L,L'\in \mr{U}(E,\Omega)$, let $\Delta_o(L,L')$ denote the geodesic simplex with vertices $o,Lo,LL'o$. Recall (Equation \ref{bounded cocyle}) that the K\"ahler bounded cohomology class is represented by the following bounded real valued cocycle,
\begin{align*}
\kappa_o(L,L')=\frac{1}{2\pi}\int_{\Delta_o(L,L')}\omega.
\end{align*} 
If $o$ is chosen in $\ms{X}_1$ and $L_1,L'_1\in \mr{U}(E_1,\Omega_1)$, $\Delta_o(L_1,L'_1)$ serves as a geodesic simplex for both $\ms{X}_1$ and $\ms{X}$, so the restriction of the K\"ahler bounded cohomology class $\kappa$ of $\mr{U}(E,\Omega)$ to $\mr{U}(E_1,\Omega_1)$ is the K\"ahler bounded cohomology class of $\mr{U}(E_1,\Omega_1)$. 

This reduces the integrality question to the case when $\mr{sign}(\Omega)=0$, so $\ms{X}$ is of tube type. In this case, Clerc (\cite{Clerc}) shows that 
twice 
the K\"ahler bounded cohomology class is integral. Indeed, the $2$-cocycle 
$2\kappa_o$
converges, as $o$ tends $\Gamma$-radially to a point of the Shilov boundary, to the integer-valued \emph{generalized Maslov $2$-cocycle}. Since, as $o$ varies in $\ms{X}$, all these cocycles are cohomologous, so is the limiting cocycle.

\end{proof}

\begin{defn}
For $L\in\mr{U}(E,\Omega)$, we denote the rotation number associated to twice the bounded cohomology class of the normalized K\"ahler form by $\mr{Rot}(E,\Omega,L)$, or simply by $\mr{Rot}(L)$ when the context is clear.
\end{defn}

\subsubsection{Properties}

We shall use the following properties of rotation numbers:
\begin{enumerate}
  \item If $B<\mathrm{U}(p,q)$ is a closed amenable subgroup, the restriction of $\mathrm{Rot}$ to $B$ is a group homomorphism.
  \item If $(E,\Omega,L)=(E_1,\Omega_1,L_1)\oplus (E_2,\Omega_2,L_2)$ and $\Omega_1, \Omega_2$ are indefinite Hermitian forms,
\begin{align*}
\mathrm{Rot}(E,\Omega,L)=\mathrm{Rot}(E_1,\Omega_1,L_1)+\mathrm{Rot}(E_2,\Omega_2,L_2).
\end{align*}
  \item For every $L\in\mr{U}(E,\Omega)$, $\mathrm{Rot}(E,-\Omega,L)=-\mathrm{Rot}(E,\Omega,L)$.
  \item $\mathrm{Rot}$ is a conjugacy invariant: $\mathrm{Rot}(CLC^{-1})=\mathrm{Rot}(L)$.
  \item If $L=L_eL_kL_u$ is the generalized Jordan decomposition of $L$, then (\cite[Theorem 11]{BIW})
\begin{align}\label{rotelliptic}
\mathrm{Rot}(L)=\mathrm{Rot}(L_e).
\end{align}
\end{enumerate}

\begin{proof}
Only item (2) needs some explanation. Let $L,L'\in\mr{U}(E,\Omega)$ preserve the splitting $E=E_1\oplus E_2$. As was observed in the proof of Lemma \ref{integrality}, the product of symmetric spaces $\ms{X}_1\times \ms{X}_2$ embeds in $\ms{X}$ and the ambient K\"ahler form $\omega$ restricts to $pr_1^*\omega_1+pr_2^*\omega_2$. The geodesic simplex $\Delta(L,L')\subset \ms{X}_1\times \ms{X}_2$ projects onto both factors to the geodesic simplices $\Delta(L_1,L'_1)\subset \ms{X}_1$ and $\Delta(L_2,L'_2)\subset \ms{X}_2$. If follows that
\begin{align*}
\kappa(L,L')&=\int_{\Delta(L,L')}pr_1^*\omega_1+\int_{\Delta(L,L')}pr_2^*\omega_2\\
&=\int_{\Delta(L_1,L'_1)}\omega_1+\int_{\Delta(L_2,L'_2)}\omega_2\\
&=\kappa_1(L_1,L'_1)+\kappa_2(L_2,L'_2).
\end{align*}
Let $B$ denote the closed subgroup generated by $L=L_1\oplus L_2$, and $B_i$ the corresponding subgroup for $L_i$. Then $B\subset B_1 \times B_2$. The long exact sequence yields a homomorphism
\begin{align*}
f_{B_1\times B_2,\kappa_1+\kappa_2}=f_{B_1,\kappa_1}\circ pr_1+f_{B_2,\kappa_2}\circ pr_2:B_1\times B_2\to\br/\bz,
\end{align*} 
whose restriction to $B$ is equal to $f_{B,\kappa}$. Therefore \begin{align*}
\mathrm{Rot}(L)=f_{B,\kappa}(L)=f_{B_1,\kappa_1}(L_1)+f_{B_2,\kappa_2}(L_2)=\mathrm{Rot}(L_1)+\mathrm{Rot}(L_2).
\end{align*}

\end{proof}

\subsubsection{Rotation numbers and Toledo invariants}

Let $\Sigma$ be a compact oriented surface with boundary, let $*\in\Sigma$ be a basepoint. Let $a_1,b_1,\ldots,a_g,b_g,c_1,\ldots,c_n$ be loops based at $*$ such that $c_j$ represent the oriented boundary components, in such a way that the fundamental group of $\Sigma$ be presented by 
$$
\langle a_1,b_1,\ldots,a_g,b_g,c_1,\ldots,c_n\,|\,(\prod[a_i,b_i])(\prod c_j)=1\rangle.
$$ 
Since $\pi_1(\Sigma,*)$ is free, every representation $\phi:\pi_1(\Sigma,*)\to\mr{U}(E,\Omega)$ admits a lift $\tilde\phi :\pi_1(\Sigma,*)\to\widetilde{\mr{U}(E,\Omega)}$ to the universal covering group of $\mr{U}(E,\Omega)$. The continuous function $\mr{Rot}:\mr{U}(E,\Omega)\to\br/\bz$ admits a unique continuous lift $\wt{\mr{Rot}}:\widetilde{\mr{U}(E,\Omega)}\to\br$ mapping the neutral element to $0$.

According to \cite[Theorem 12]{BIW}, the Toledo invariant of $\phi$ is given by 
\begin{align}\label{Toledo-rotation}
\begin{split}
  2\mr{T}(\Sigma,\phi)=-\sum_{j=1}^n\wt{\mr{Rot}}(\wt{\phi}(c_j)).
 \end{split}
\end{align}
(remember that $\mr{Rot}$ is associated to \emph{twice} the K\"ahler bounded cohomology class).
Let $(\mathcal{E},\Omega)$ denote the flat unitary bundle over $\Sigma$ associated to $\phi$. In combination with Theorem \ref{0}, Equation \ref{Toledo-rotation} gives 
\begin{align}\label{sign-rotation-rho}
\begin{split}
  \mr{sign}(\mathcal{E},\Omega)=\sum_{j=1}^n \tilde b(\wt{\phi}(c_j)),
 \end{split}
\end{align}
where $\tilde b:\widetilde{\mr{U}(E,\Omega)}\to\br$ is the function defined by 
\begin{align*}
\tilde b:=\wt{\mr{Rot}}+\rro\circ \tilde p,
\end{align*}
and $\tilde p:\widetilde{\mr{U}(E,\Omega)}\to\mr{U}(E,\Omega)$ is the covering map.

Equation \eqref{sign-rotation-rho} yields
\begin{align}\label{signformula}
\begin{split}
  \mr{sign}(A,B)&=-\tilde b(\tilde A)-\tilde b(\tilde B)+\tilde b(\tilde A\tilde B).
 \end{split}
\end{align}
In other words, if $\tilde b$ is viewed as an $1$-cochain on $\widetilde{\mathrm{U}(E,\Omega)}$,
\begin{align}\label{db}
p^{*}\mathrm{sign}=-d\tilde b.
\end{align}

\subsubsection{Computing rotation numbers}

\begin{lemma}\label{rotrho}
The function $b:\mr{U}(E,\Omega)\to\br/\bz$ defined by $b:=\mr{Rot}+\rro$ vanishes. It follows that $\tilde b:\widetilde{\mr{U}(E,\Omega)}\to\br$ is integer-valued.	
\end{lemma}

\begin{proof}
The first step is to show that if $\mr{sign}(\Omega)=0$, then $\mr{Rot}$ vanishes on the center of $\mr{U}(E,\Omega)$. Let $C\in \mr{End}_\bc(E)$ be a $\bc$-linear map such that $C^*\Omega=-\Omega$. Let $L\in\mr{U}(E,\Omega)$. Then $CLC^{-1}\in\mr{U}(E,-\Omega)$ and 
$$
\mr{Rot}(E,L,\Omega)=\mr{Rot}(E,CLC^{-1},-\Omega)=-\mr{Rot}(E,CLC^{-1},\Omega).
$$
If $L=u\,\mr{Id}_E$ for some unit complex number $u$, $CLC^{-1}=L$, hence $\mr{Rot}(E,L,\Omega)=-\mr{Rot}(E,L,\Omega)$ mod $\bz$. Thus $\mr{Rot}(E,L,\Omega)=0$ or $\frac{1}{2}$ mod $\bz$. Since $\mr{Rot}$ is continuous and $\mr{Rot}(\mr{Id}_E)=0$, $\mr{Rot}(u\,\mr{Id}_E)=0$ mod $\bz$ for all unit complex numbers $u$.

Given $L\in \mr{U}(E,\Omega)$, write 
\begin{align*}
\begin{split}
L=\bar u\,U_1
 \end{split}
\end{align*}
for some $U_1\in\mr{SU}(E,\Omega)$ and $u\in\mb{C}$, $|u|=1$. Since  $\mr{SU}(E,\Omega)$ is the commutator subgroup of $\mr{U}(E,\Omega)$, there exist $A$ and $B\in\mr{U}(E,\Omega)$ such that 
	\begin{align*}
\begin{split}
  U_1=[B,A]=BAB^{-1}A^{-1}.
 \end{split}
\end{align*}
Hence 
\begin{align*}
\begin{split}
  [A,B]L(u\,\mr{Id}_{E})=\mr{Id}_{E}.
 \end{split}
\end{align*}
Let $\Sigma_1$ be the two-hole torus. Let $a$ and $b$ be a meridian and a parallel loop on the torus, let $c_1$ and $c_2$ be loops representing the oriented boundary components, in such a way that the fundamental group of $\Sigma_1$ be presented by $\langle a,b,c_1,c_2\,|\,[a,b]c_1c_2=1\rangle$. Consider the representation $\phi:\pi_1(\Sigma_1)\to\mr{U}(E,\Omega)$ with 
$$
\phi(a)=A,\quad \phi(b)=B,\quad \phi(c_1)=L,\quad \phi(c_2)=u\,\mr{Id}_{E}.
$$
Let $(\mc{E},\Omega)$ denote the associated flat unitary vector bundle over $\Sigma_1$. Let $\tilde\phi$ be a lift of $\phi$ to $\widetilde{\mathrm{U}(E,\Omega)}$, let $\tilde L=\tilde\phi(L)$ and $\widetilde{u\,\mr{Id}_{E}}=\tilde\phi(u\,\mr{Id}_{E})$. Equation \ref{sign-rotation-rho} gives 
\begin{align*}
\begin{split}
  \mr{sign}(\mc{E},\Omega)=\tilde b(\tilde L)+\tilde b(\widetilde{u\,\mr{Id}_{E}}),
 \end{split}
\end{align*}
hence
\begin{align*}
b(L)+b(u\,\mr{Id}_{E})=0 \mod\bz.
\end{align*}

If $L\in \mr{SU}(E,\Omega)$, we can take $u=1$, and $b(L)=0\mod\bz$.
For a general $L\in \mr{U}(E,\Omega)$, we let $E'=E\oplus \bc^2$, $\Omega'=\Omega\oplus\Omega''$, where $\Omega''$ is indefinite, of signature $(1,1)$, and $L'=L\oplus u \,\mr{Id}_{\bc^2}$, where the unit complex number $u$ is chosen so that $L'\in \mr{SU}(E',\Omega')$. Since both $\rro$ and $\mathrm{Rot}$ are additive under orthogonal direct sums, and both vanish mod $\bz$ on $(\bc^2,u\,\mr{Id}_{\bc^2},\Omega'')$, 
\begin{align*}
\begin{split}
\boldsymbol{\rho}(L')=\boldsymbol{\rho}(L),\quad  \mr{Rot}(L')=\mr{Rot}(L)+\mr{Rot}(u\,Id_{\bc^2})=\mr{Rot}(L)\mod\bz.
 \end{split}
\end{align*}
Hence
\begin{align*}
\begin{split}
b(L)= \mr{Rot}(L)+\boldsymbol{\rho}(L)=\mr{Rot}(L')+\boldsymbol{\rho}(L')=0\mod\bz,
 \end{split}
\end{align*}
which completes the proof.

\end{proof}

\begin{lemma}
\label{rotdet}
If $L=L_+\oplus L_-\in\mathrm{U}(p)\times\mathrm{U}(q)<\mathrm{U}(p,q)$, 
$$
e^{2\pi i\mathrm{Rot}(L)}=\left(\frac{\mathrm{det}(L_+)}{\mathrm{det}(L_-)}\right)^2.
$$
\end{lemma}

With Equation \ref{rotelliptic}, in principle this determines $\mathrm{Rot}$ on all of $\mathrm{U}(p,q)$.

\begin{proof}
By conjugacy-invariance, it suffices to compute $\mathrm{Rot}$ on the maximal torus $T=U(1)\times\cdots\times U(1)$. If $L=(e^{2\pi ix_1},\ldots,e^{2\pi ix_{p+q}})$, Lemma \ref{rhoell} gives
\begin{align*}
\rro(L)=\sum_{j=1}^{p}(1-2\{x_j\})-\sum_{j=1}^{q}(1-2\{x_{p+j}\})\mod\bz,
\end{align*}
thus
\begin{align*}
e^{2\pi i\rro(L)}
&=\exp(2\pi i(-2\sum_{j=1}^{p}x_j+2\sum_{j=1}^{q}x_{p+j}))\\
&=\left(\frac{\mathrm{det}(L_-)}{\mathrm{det}(L_+)}\right)^2.
\end{align*}
Lemma \ref{rotrho} shows that 
\begin{align*}
e^{2\pi i\mathrm{Rot}(L)}=e^{-2\pi i\rro(L)}=\left(\frac{\mathrm{det}(L_+)}{\mathrm{det}(L_-)}\right)^2.
\end{align*}

\end{proof}

\subsubsection{The $1$-cochain on the central extension $\mr{U}(p,q)_2$} 
\label{b2}

\begin{prop}
\label{propsigma}
The primitive $b_2:\mr{U}(p,q)_2\to\bz$ of the pulled-back signature cocycle, 
$$
p_2^*\mr{sign}=-db_2,
$$
is given by
\begin{align}\label{bcochain}
\begin{split}
b_2=\mr{Rot}_2+\rro\circ p_2,
 \end{split}
\end{align}
where $\mr{Rot}_2:\mr{U}(p,q)_2\to\br$ is a continuous function, and $\mr{Rot}_2$ is the continuous lift of the rotation number $\mr{Rot}$ that satisfies $\mr{Rot}_2\circ j_2(z)=z$ for $z\in\bz$.

\end{prop}

\begin{proof}
Atiyah has determined the cohomology class of $\mr{sign}$ in 
$$
\mr{H}^2(\mr{U}(p,q),\bz)\simeq \mr{Hom}(\pi_1(\mr{U}(p,q)),\bz)\simeq \mr{Hom}(\pi_1(\mr{U}(p)\times \mr{U}(q)),\bz)\simeq\bz\oplus\bz.
$$
This class has coordinates $(2,-2)$. Therefore, if $\lambda:\mb{Z}^2\to\mb{Z}$ denotes the homomorphism given by 
\begin{align*}
\begin{split}
  \lambda(m,n)=2m-2n,
 \end{split}
\end{align*}
the group $\mr{U}(p,q)_2$ can be obtained as an associated bundle 
\begin{align*}
\begin{split}
  \mr{U}(p,q)_2:=\wt{\mr{U}(p,q)}\times_\lambda \mb{Z}=\wt{\mr{U}(p,q)}\times \mb{Z}/\pi_1(\mr{U}(p,q)),
 \end{split}
\end{align*}
where $\pi_1(\mr{U}(p,q))$ acts diagonally, on $\mr{U}(p,q)$ by deck transformations, and on $\bz$ by translations via $\lambda$.

The standard generator of $\pi_1(\mr{U}(p))$ is represented by the arc $t\mapsto \exp(2\pi it/p)I_p$ followed by an arc joining $\exp(2\pi i/p)I_p$ to $I_p$ in $\mr{SU}(p)$. The determinant $\mr{det}:\mr{U}(p)\to \mr{U}(1)$ maps this homotopy class to the identity homotopy class of $\mr{U}(1)$. Therefore the expression found in Lemma \ref{rotdet} for the rotation number $\mr{Rot}$ indicates that 
$$
(\mr{Rot})_\sharp:\pi_1(\mr{U}(p,q))\to\pi_1(\br/\bz)=\bz
$$
is given by $\lambda$. It follows that there exists a continuous function $\mr{Rot}_2:\mr{U}(p,q)_2\to\br$ such that $\mr{Rot}_2\circ j_2(z)=z$ for $z\in\bz$ and
$$
\mr{Rot}\circ p_2=\mr{Rot}_2 \mod\bz.
$$

The restriction to $\wt{\mr{U}(p,q)}\times\{0\}$ of the pull-back of $\mr{Rot}_2$ to $\wt{\mr{U}(p,q)}\times \mb{Z}$ is $\wt{\mr{Rot}}$. This pull-back satisfies $\wt{\mr{Rot}}\circ\tilde j(z)=z$ for $z\in\bz$.
Then the identity $\tilde p^*\mr{sign}=-d(\wt{\mr{Rot}}+\rro\circ\tilde p)$ on $\wt{\mr{U}(p,q)}$ (Equation \ref{db}) implies that
$$
p_2^*\mr{sign}=-d(\mr{Rot}_2+\rro\circ p_2)
$$
on $\mr{U}(p,q)_2$. Thus, by uniqueness,
$$
b_2=\mr{Rot}_2+\rro\circ p_2.
$$
\end{proof}

\subsubsection{Concrete realization of the universal covering of $\mr{U}(p,q)$}\label{universal covering}

According to \cite[Theorem 3.4]{Neretin}, every $L\in \mr{U}(p,q)$ has the following unique decomposition
\begin{align*}
\begin{split}
  L=\begin{pmatrix}
  U_1&0 \\
  0&U_2 
\end{pmatrix}S(W),
 \end{split}
\end{align*}
 where  $U_1\in \mr{U}(p)$, $U_2\in \mr{U}(q)$ and 
 \begin{align*}
\begin{split}
  S(W)=\left(\begin{array}{cc}
\left(1-WW^{*}\right)^{-1 / 2} & \left(1-W W^{*}\right)^{-1 / 2} W \\
W^{*}\left(1-W W^{*}\right)^{-1 / 2} & \left(1-W^{*} W\right)^{-1 / 2}
\end{array}\right)
 \end{split}
\end{align*}
for $W\in \mr{D}_{p,q}^{\mr{I}}$.  Hence 
$\mr{U}(p,q)$ is homeomorphic to $\mr{U}(p)\times \mr{U}(q)\times \mr{D}_{p,q}^{\mr{I}}$. The universal covering $\wt{\mr{U}(p,q)}$ of $\mr{U}(p,q)$ can be given by
\begin{align*}
\begin{split}
  \wt{\mr{U}(p,q)}=\mb{R}\times\mb{R}\times \mr{SU}(p)\times \mr{SU}(q)\times \mr{D}_{p,q}^{\mr{I}}
 \end{split}
\end{align*}
 with the projection 
 \begin{align*}
\begin{split}
 & \cal P:\wt{\mr{U}(p,q)}=\mb{R}\times\mb{R}\times \mr{SU}(p)\times \mr{SU}(q)\times \mr{D}_{p,q}^{\mr{I}}\to \mr{U}(p,q)\\
 &\cal P(x,y,U_1,U_2,W)=\begin{pmatrix}
  e^{2\pi i x} U_1&0 \\
  0& e^{2\pi i y}U_2
\end{pmatrix}S(W),
 \end{split}
\end{align*}
where $U_1\in \mr{SU}(p)$ and $U_2\in \mr{SU}(q)$.
\begin{rem}
Every element $L\in\mr{U}(p,q)$ can be written
\begin{align*}
\begin{split}
  L=\begin{pmatrix}
  e^{2\pi i x} U_1&0 \\
  0& e^{2\pi i y}U_2
\end{pmatrix}S(W),
 \end{split}
\end{align*}
with $U_1\in \mr{SU}(p)$ and $U_2\in \mr{SU}(q)$. Then any lift of $L$ has the form 
\begin{align*}
\begin{split}
  (x+\frac{k_1}{p},y+\frac{k_2}{q},e^{-2\pi i\frac{k_1}{p}}U_1,e^{-2\pi i\frac{k_2}{q}}U_2,W)\in \wt{\mr{U}(p,q)}
 \end{split}
\end{align*}
for some $k_1, k_2\in\mb{Z}$.

\end{rem}

\subsubsection{Concrete realization of the central extension $\mr{U}(p,q)_2$} 

Let $\lambda:\mb{Z}^2\to\mb{Z}$ be the homomorphism given by $\lambda(m,n)=2m-2n$. Recall that
\begin{align*}
\begin{split}
  \mr{U}(p,q)_2:=\wt{\mr{U}(p,q)}\times_\lambda \mb{Z}=\wt{\mr{U}(p,q)}\times \mb{Z}/\sim.
 \end{split}
\end{align*}
We can give a concrete expression for the equivalence relation,
\begin{align*}
\begin{split}
  (x+\frac{m}{p},y+\frac{n}{q},e^{-2\pi i\frac{m}{p}}U_1,e^{-2\pi i\frac{n}{q}}U_2,W,k)\sim (x,y,U_1,U_2,W,2m-2n+k).
 \end{split}
\end{align*}
The group $\mr{U}(p,q)_2$ is a central extension of $\mr{U}(p,q)$ by $\bz$, with exact sequence
\begin{align*}
\begin{split}
\mb{Z}  \stackrel{j_2}{\longrightarrow}	\mr{U}(p,q)_2\stackrel{p_2}{\longrightarrow} \mr{U}(p,q),
 \end{split}
\end{align*}
given by 
\begin{align*}
  p_2([x,y,U_1,U_2,W,k])=
  \cal P(x,y,U_1,U_2,W)=\begin{pmatrix}
  e^{2\pi i x} U_1&0 \\
  0& e^{2\pi i y}U_2
\end{pmatrix}S(W).
  \end{align*}
The isomorphism $j_2:\mb{Z}\to \mr{Ker}(p_2)$ is given by
\begin{align*}
\begin{split}
  j_2(k)=\left[0,0,I_p,I_q,0,k\right].
 \end{split}
\end{align*}

\subsubsection{Atiyah's section $\sigma$}

Remember that Atiyah's section $\sigma$ is the unique section $\sigma:\mr{U}(p,q)\to \mr{U}(p,q)_2$ such that for all $A,B\in \mr{U}(p,q)$,
\begin{align}\label{sign-sigma}
\begin{split}
  j_2(\mr{sign}(A,B))=\sigma(A)\sigma(B)\sigma(AB)^{-1}.
 \end{split}
\end{align}

\begin{thm}
\label{thmsigmabis}
In the notation of Proposition \ref{propsigma}, Atiyah's section $\sigma:\mr{U}(p,q)\to \mr{U}(p,q)_2$ is uniquely determined by the equation $b_2\circ\sigma=0$. In the notation of paragraph \ref{universal covering}, Atiyah's section is given by the following formula. 
\begin{align}\label{sigma-section}
\begin{split}
\sigma(A)=[\wt{A},-\wt{\mr{Rot}}(\wt{A})-\boldsymbol{\rho}(A)],
 \end{split}
\end{align}
where $\wt{A}$ is an arbitrary lift of $A$ in $\wt{\mr{U}(p,q)}$. 
\end{thm}

The first claim is a restatement of Lemma \ref{b}.

Equation \eqref{sigma-section} is unambiguous. Indeed, as in paragraph \ref{universal covering}, we can assume that
\begin{align*}
\begin{split}
 A=\begin{pmatrix}
  e^{2\pi i x} U_1&0 \\
  0& e^{2\pi i y}U_2
\end{pmatrix}S(W),
 \end{split}
\end{align*}
and a lift $\wt{A}$ of $A$ is given by
\begin{align*}
\begin{split}
  \wt{A}=(x+\frac{k_1}{p},y+\frac{k_2}{q},e^{-2\pi i\frac{k_1}{p}}U_1,e^{-2\pi i\frac{k_2}{p}}U_2,W)\in \wt{\mr{U}(p,q)}
 \end{split}
\end{align*}
 for some $k_1,k_2\in\mb{Z}$. Then 
\begin{align*}
\begin{split}
  &\quad[x+\frac{k_1}{p},y+\frac{k_2}{q},e^{-2\pi i\frac{k_1}{p}}U_1,e^{-2\pi i\frac{k_2}{q}}U_2,W, -\wt{\mr{Rot}}(\wt{A})-\boldsymbol{\rho}(A)]\\
  &=[x+\frac{k_1}{p},y+\frac{k_2}{q},e^{-2\pi i\frac{k_1}{p}}U_1,e^{-2\pi i\frac{k_2}{q}}U_2,W,-2(px+k_1-qy-k_2)-\boldsymbol{\rho}(A)]\\
  &=[x,y,U_1,U_2,W,-2(px-qy)-\boldsymbol{\rho}(A)],
 \end{split}
\end{align*}
which is independent of $k_1,k_2$. 

On the other hand, given $A,B\in \mr{U}(p,q)$, one can pick lifts $\tilde A$ and $\tilde B$ in $\wt{\mr{U}(p,q)}$. The product $\wt{AB}=\tilde A\tilde B$ can be used as a lift of $AB$. Then
\begin{align*}
\begin{split}
  \sigma(A)\sigma(B)\sigma(AB)^{-1}&=[\wt{A}\wt{B}(\wt{AB})^{-1}, \mr{sign}(A,B)]\\
  &=[e,\mr{sign}(A,B)]=j_2(\mr{sign}(A,B)),
 \end{split}
\end{align*}
which means that the section defined by \eqref{sigma-section} is exactly Atiyah's section $\sigma$.

\section{Milnor-Wood type inequality}\label{AMW}

In this section, we will prove a Milnor-Wood type inequality, in the form of an estimate on the signature of a flat unitary bundle. We first express the dimension of the vector space $\op{Im}(\mathrm{H}^1(\Sigma,\p \Sigma,\mc{E})\to \mathrm{H}^1(\Sigma,\mc{E}))$ on which the Hermitian form $iQ$ is defined in terms of the Euler characteristic and the dimension of the space of flat sections. Then we prove that this space vanishes for a dense set of representations. This provides the link between signature and Euler characteristic. In view of the Milnor-Wood inequality for closed surfaces, the right hand side $(p+q)|\chi(\Sigma)|$ does not seem to be sharp. However, by considering positive definite Hermitian forms ($q=0$), we shall see that our Milnor-Wood type estimate on signature cannot be improved to $\min\{p,q\}|\chi(\Sigma)|$ in general.

\subsection{Milnor-Wood type inequality}
Let $\phi:\pi_1(\Sigma)\to \op{U}(E,\Omega)$ be a representation into the $\mr{U}(p,q)$-group $\mr{U}(E,\Omega)$, where $E=\mb{C}^{p+q}$, we have 
$$\op{sign}(\mc{E},\Omega)=-2\op{T}(\Sigma,\phi)-\frac{1}{\pi}\sum_{i=1}^n\int_{c_i}\wt{\mbf{J}}^*\alpha_i+\eta(A_{\mbf{J}})=-2\op{T}(\Sigma,\phi)+\rro_\phi(\p\Sigma).$$
\begin{lemma}
	The indices of $d^-_P$ and $d^+_P$ can be given by
	\begin{equation*}
\op{Index}(d^{\mp}_P)= \pm\frac{1}{2}\op{sign}(\mc{E},\Omega)+\frac{\dim E}{2}\chi(\Sigma)-\frac{\dim \mathrm{H}^0(\p\Sigma,\mc{E})}{2}.
\end{equation*}
\end{lemma}
\begin{proof}
For the index of $d^-_P$, by \eqref{APSformula}, one has
\begin{align*}
\op{Index}(d^{-}_P)&=\int_{\Sigma}\alpha_{{-}}(z)d\mu_g-\frac{\eta(A^-_{\mbf{J}})+\dim\op{Ker}A^{-}_{\mbf{J}}}{2}\\
&=\int_{\Sigma}\alpha_{{-}}(z)d\mu_g+\frac{\eta(A_{\mbf{J}})}{2}-\frac{\dim \mathrm{H}^0(\p\Sigma,\mc{E})}{2}\\
&=\frac{\dim E}{2}\chi(\Sigma)+\int_\Sigma \left(c_1(\mc{E}^+,\n|_{\mc{E}^+})-c_1(\mc{E}^-,\n|_{\mc{E}^-})\right)+\frac{\eta(A_{\mbf{J}})}{2}-\frac{\dim \mathrm{H}^0(\p\Sigma,\mc{E})}{2}\\
&=\frac{\dim E}{2}\chi(\Sigma)+\frac{1}{2}\op{sign}(\mc{E},\Omega)-\frac{\dim \mathrm{H}^0(\p\Sigma,\mc{E})}{2},
\end{align*}
where the second equality by $\eta(A_{\mbf{J}}^-)=-\eta(A_{\mbf{J}})$ and \eqref{KerAJ}, the third equality by Proposition \ref{prop410}, and the last equality by Theorem \ref{thmsign2}. Similarly,  we can also obtain
\begin{equation*}
  \op{Index}(d^{+}_P)=\frac{\dim E}{2}\chi(\Sigma)-\frac{1}{2}\op{sign}(\mc{E},\Omega)-\frac{\dim \mathrm{H}^0(\p\Sigma,\mc{E})}{2}.
\end{equation*}
\end{proof}
From the above lemma and using $\mathrm{L}^2\op{Index}(d^{\pm})=\op{Index}(d^{\pm}_P)+h_{\infty}(\wedge^\pm)$, one has
\begin{align*}
\begin{split}
\pm\frac{1}{2}\op{sign}(\mc{E},\Omega)&=-\frac{\dim E}{2}\chi(\Sigma)+\frac{\dim \mathrm{H}^0(\p\Sigma,\mc{E})}{2}+\op{Index}(d^{\mp}_P)\\
&=-\frac{\dim E}{2}\chi(\Sigma)+\frac{\dim \mathrm{H}^0(\p\Sigma,\mc{E})}{2}-h_{\infty}(\wedge^{\mp}) +\mathrm{L}^2\op{Index}(d^{\mp}).
\end{split}
\end{align*}
Here the $\mathrm{L}^2$-index is given by 
\begin{align*}
\begin{split}
\mathrm{L}^2\op{Index}(d^{\pm})&=\dim\op{Ker}(d^\pm)\cap \mathrm{L}^2(\widehat{\Sigma},\mc{E})-\dim\op{Ker}(d^\pm)^*\cap \mathrm{L}^2(\widehat{\Sigma},\wedge^\pm)\\
&=\dim\mathscr{H}^0(\widehat{\Sigma},\mc{E})-\dim\op{Ker}(d^\pm)^*\cap \mathrm{L}^2(\widehat{\Sigma},\wedge^\pm)\\
&=\dim\op{Im}(\mathrm{H}^0(\Sigma,\p\Sigma,\mc{E})\to \mathrm{H}^0(\Sigma,\mc{E}))-\dim\op{Ker}(d^\pm)^*\cap \mathrm{L}^2(\widehat{\Sigma},\wedge^\pm),
\end{split}
\end{align*}
where the second equality follows from (\ref{H0}) and the last equality follows from \cite[Proposition 4.9]{APS}. Since
$\mathrm{H}^0(\Sigma,\p\Sigma,\mc{E})=\{0\}$, so  $$\mathrm{L}^2\op{Index}(d^{\pm})= -\dim\op{Ker}(d^\pm)^*\cap \mathrm{L}^2(\widehat{\Sigma},\wedge^\pm).$$
On the other hand, by \eqref{hpm}, one has
$$\dim \mathrm{H}^0(\p\Sigma,\mc{E})-h_{\infty}(\wedge^{\pm})=\dim \mathrm{H}^0(\Sigma,\mc{E}).$$
Hence
\begin{multline}\label{5.6}
 \pm\frac{1}{2}\op{sign}(\mc{E},\Omega)	=-\frac{\dim E}{2}\chi(\Sigma)-\frac{\dim \mathrm{H}^0(\p\Sigma,\mc{E})}{2}\\
+\dim \mathrm{H}^0(\Sigma,\mc{E})-\dim\op{Ker}(d^\mp)^*\cap \mathrm{L}^2(\widehat{\Sigma},\wedge^\pm).
\end{multline}
\begin{rem}
From the above equality \eqref{5.6}, one gets immediately 
\begin{multline}\label{relative dimension}
	\dim\op{Im}(\mathrm{H}^1(\Sigma,\p \Sigma,\mc{E})\to \mathrm{H}^1(\Sigma,\mc{E}))\\=-\dim E\cdot\chi(\Sigma)-\dim \mathrm{H}^0(\p\Sigma,\mc{E})+2\dim \mathrm{H}^0(\Sigma,\mc{E}).
\end{multline}
In fact, by Proposition \ref{prop3.2}, one has
$$\op{sign}(\mc{E},\Omega)=\dim\op{Ker}(d^+)^*\cap \mathrm{L}^2(\widehat{\Sigma},\wedge^+)-\dim\op{Ker}(d^-)^*\cap \mathrm{L}^2(\widehat{\Sigma},\wedge^-).$$ 
 On the other hand, by Proposition \ref{prop3} and \cite[Proposition 4.9]{APS}, one has
\begin{multline*}
	\dim\op{Im}(\mathrm{H}^1(\Sigma,\p \Sigma,\mc{E})\to \mathrm{H}^1(\Sigma,\mc{E}))\\=\dim\op{Ker}(d^-)^*\cap \mathrm{L}^2(\widehat{\Sigma},\wedge^-)+\dim\op{Ker}(d^+)^*\cap \mathrm{L}^2(\widehat{\Sigma},\wedge^+).
\end{multline*}	
Hence we obtain \eqref{relative dimension} by combining with \eqref{5.6}.
\end{rem}

\begin{rem}
 If $\p\Sigma=\emptyset$, then $-\frac{1}{2}\op{sign}(\mc{E},\Omega)=\op{T}(\Sigma,\phi)$ and we have 
\begin{align*}
	\pm\op{T}(\Sigma,\phi)&=-\frac{\dim E}{2}\cdot\chi(\Sigma)+\op{Index}(d^{\pm})\\
	&=-\frac{\dim E}{2}\cdot\chi(\Sigma)+\dim \mathrm{H}^0(\Sigma,\mc{E})-\dim\op{Ker}(d^{\pm})^*.
	\end{align*}
	The second equality follows from the observation: if $d^-a=0$, then $d a=d^+a$, and so 
$d a=*\mbf{J}d a$, which follows that $(d^+)^*d^+a=d^*d a=0$. A similar argument holds for $d^+$.
	\end{rem}
\begin{lemma}\label{lemma3}
If $\Sigma$ is a  surface with  genus $g\geq 1$, then
the set of all representations $\phi$ such that $\{v\in E: \phi(a_i)v=v=\phi(b_i)v,1\leq i\leq g\}=\{0\}$ is dense in $\op{Hom}(\pi_1(\Sigma),\op{U}(E,\Omega))$, where $a_i,b_i\in \pi_1(\Sigma),1 \leq i\leq g$ denote the  generators in the interior of $\Sigma$.	
\end{lemma}
\begin{proof}
Suppose that the boundary of $\Sigma$ is the union $\partial \Sigma=\bigsqcup_{j=1}^q c_j$ of oriented circles, the fundamental group of $\Sigma$ is 
$$
\pi_{1}(\Sigma)=\left\langle a_{1}, b_{1}, \ldots, a_{g}, b_{g}, c_{1}, \ldots, c_{q}: \prod_{i=1}^{g}\left[a_{i}, b_{i}\right] \prod_{j=1}^{q} c_{j}=e\right\rangle.
$$
Then  $\op{Hom}(\pi_1(\Sigma),\op{U}(E,\Omega))$ is the space of all homomorphisms  with the relation 
\begin{align}\label{relation}
\prod_{i=1}^{g}\left[\phi(a_{i}), \phi(b_{i})\right] \prod_{j=1}^{q} \phi(c_{j})=\phi(e)=\op{Id}.	
\end{align}
If $1$ is an eigenvalue of $\phi(a_1)\in \op{U}(E,\Omega)$ and $\phi(b_1)\in \op{U}(E,\Omega)$, then by a small perturbation, we can find two families of  elements  $A_{\epsilon}, B_{\epsilon}\in \op{U}(E,\Omega)$  such that 	$1$ is not the an eigenvalue for $\epsilon>0$ and
	$$\lim_{\epsilon\to 0}A_{\epsilon}=\phi(a_1),\quad \lim_{\epsilon\to 0}B_{\epsilon}=\phi(b_1),\quad [A_{\epsilon},B_{\epsilon}]=[\phi(a_1),\phi(b_1)],$$
	for $\epsilon>0$ small enough. In fact, the linear transformations $\sigma\in \op{U}(E,\Omega)$ and $\tau\in \op{U}(E,\Omega)$ with $1$ as an eigenvalue is equivalent to $\det(\sigma-\mr{Id})=0$ and $\det(\tau-\mr{Id})=0$, which defines a variety $H_1$ in $\op{U}(E,\Omega)\times \op{U}(E,\Omega)$, and $[\sigma,\tau]=[\phi(a_1),\phi(b_1)]$ defines a variety $H_2$ in $\op{U}(E,\Omega)\times \op{U}(E,\Omega)$, in fact it is given by the set of solutions of these polynomial equations $\sigma\tau-[\phi(a_1),\phi(b_1)]\tau\sigma=0$. 
	From the definitions of $H_1$ and $H_2$, one sees easily that 
	$H_2$ is not contained in $H_1$, i.e. $H_2\nsubseteq H_1$, so $H_1\cap H_2\subsetneqq H_2$, which means the subvariety  $H_1\cap H_2$ is a proper subvariety of $H_2$. Thus we can find such $A_{\epsilon},B_{\epsilon}$ in $H_2\backslash (H_1\cap H_2)$.	
	
	Now we can take $\phi_{\epsilon}$ by 
	$$\phi_{\epsilon}(a_1)=A_{\epsilon},\quad \phi_{\epsilon}(b_1)=B_{\epsilon},\quad \phi_{\epsilon}=\phi,\,\text{ on other generators.}$$ 
	Then $\phi_{\epsilon}\in \op{Hom}(\pi_1(\Sigma),\op{U}(E,\Omega))$ and has no global fixed point, and 
	$\lim_{\epsilon\to 0}\phi_{\epsilon}=\phi.$
	Especially, $\{\phi_{\epsilon}\}$ is equal to $\phi$ on the boundary, i.e. $\phi_{\epsilon}(c_i)=\phi(c_i)$ for any $\epsilon>0$.
\end{proof}
For any $s\in \mathrm{H}^0(\Sigma,\mc{E})$, it can be viewed as a $\phi$-equivariant map $s=s^i(x)e_i\in E$, and 
$0=ds=ds^i(x)e_i,$
so $s^i(x)=s^i$ is constant, and $s=s^ie_i$ is a constant vector. The $\phi$-equivariant condition is 
$s^i(\gamma x)e_i=\phi(\gamma)^{-1}s^i(x)e_i,$
which follows that 
$s=\phi(\gamma)^{-1}s.$
Thus we obtain
$$\mathrm{H}^0(\Sigma,\mc{E})\cong \{s\in E:s=\phi(\gamma)s,\forall \gamma\in \pi_1(\Sigma)\}.$$	
\begin{thm}\label{thmsign}
The signature satisfies the following Milnor-Wood type inequality:
$$|\op{sign}(\mc{E},\Omega)|\leq\dim E \cdot |\chi(\Sigma)|=(p+q)|\chi(\Sigma)|. $$	
\end{thm}
\begin{proof}
For  the number of boundary components $n\geq 2$, since 
\begin{equation}\label{5.7}
  \dim \mathrm{H}^0(\p\Sigma,\mc{E})\geq n\dim \mathrm{H}^0(\Sigma,\mc{E})\geq 2\dim \mathrm{H}^0(\Sigma,\mc{E}),
\end{equation}
 so by (\ref{5.6}),
$\pm\op{sign}(\mc{E},\Omega)\leq \dim E|\chi(\Sigma)|.$

For $n\leq 1$, by Lemma \ref{lemma3}, if the genus $g\geq 1$, for any representation $\phi$, there exists a family of representations $\phi_\epsilon$ with $\phi_\epsilon(\pi_1(\p\Sigma))=\phi(\pi_1(\p\Sigma))$, such that each $\phi_\epsilon$ has no global fixed point, and $\lim_{\epsilon\to 0}\phi_\epsilon=\phi$. Then $\dim \mathrm{H}^0(\Sigma,\mc{E}_\epsilon)=0$ for any $\epsilon>0$. 
Since the representations $\phi_\epsilon$ are fixed on the boundary, the eta invariant and $\int_{c_i}{\mbf{J}}^*\alpha_i$ are fixed, hence the rho invariant is fixed, $\rro_{\phi_\epsilon}(\p\Sigma)=\rro_\phi(\p\Sigma)$.
By \cite[Corollary 8.11]{BIW}, the Toledo invariant is also fixed, i.e.  $\op{T}(\Sigma,\phi_\epsilon)=\op{T}(\Sigma,\phi)$. Hence  the signature associated with $\phi_\epsilon$  is independent of $\epsilon$, and equals $\op{sign}(\mc{E},\Omega)$. Thus
\begin{align}\label{siginequality}
\begin{split}
\pm\frac{1}{2}\op{sign}(\mc{E},\Omega) &=
\mp\op{T}(\Sigma,\phi_\epsilon)\pm\frac{\rro_{\phi_\epsilon}(\p\Sigma)}{2}\\
& =	-\frac{\dim E}{2}\chi(\Sigma)-\frac{\dim \mathrm{H}^0(\p\Sigma,\mc{E})}{2} -\dim_{\mb{C}}\op{Ker}(d^\mp)^*\cap \mathrm{L}^2(\widehat{\Sigma},\wedge^\mp)\\
&\leq 	-\frac{\dim E}{2}\chi(\Sigma), 
\end{split}
\end{align}
from which it follows that $|\op{sign}(\mc{E},\Omega)|\leq\dim E \cdot |\chi(\Sigma)|$.  If $g=0$ and $n\leq 1$, then $\pi_1(\Sigma)$ is trivial, and so $\op{sign}(\mc{E},\Omega)=0$.

Hence, we obtain 
\begin{equation}\label{sign4}
  |\op{sign}(\mc{E},\Omega)|\leq \dim E\cdot\max\{-\chi(\Sigma),0\}\leq \dim E\cdot |\chi(\Sigma)|.
\end{equation}
The proof is complete.
\end{proof}

Without loss of generality, we assume $p\leq q$. For any representation $\phi:\pi_1(\Sigma)\to \op{U}(p,p)\times \op{U}(q-p)$, then 
\begin{align*}
\begin{split}
  \mc{E}=\mc{E}_1\oplus \mc{E}_2=(\wt{\Sigma}\times_{\phi_1} E_1)\oplus (\wt{\Sigma}\times_{\phi_2}E_2),
 \end{split}
\end{align*}
where $E=E_1\oplus E_2=\mb{C}^{2p}\oplus\mb{C}^{q-p}$, and $\phi_1:\pi_1(\Sigma)\to \op{U}(p,p)$ is defined as the projection of the image of $\phi$ on $\op{U}(p,p)$, while $\phi_2:\pi_1(\Sigma)\to\op{U}(q-p)$ is  the projection of the image of $\phi$ on $\op{U}(q-p)$. Hence 
$$\widehat{\mr{H}}^1(\Sigma,\mc{E})=\widehat{\mr{H}}^1(\Sigma,\mc{E}_1)\oplus \widehat{\mr{H}}^1(\Sigma,\mc{E}_2),$$
the flat vector bundles $\mc{E}_1$ and $\mc{E}_2$ are orthogonal to each other with respect to $\Omega$. From the definition of signature, see  Section \ref{Definition of signature}, one has 
\begin{align*}
\begin{split}
   \op{sign}(\mc{E},\Omega)= \op{sign}(\mc{E}_1,\Omega|_{\mc{E}_1})+\op{sign}(\mc{E}_2,\Omega|_{\mc{E}_2}).
 \end{split}
\end{align*}

Now we consider the case that $\Sigma$ is closed. In this case, $\op{sign}(\mc{E}_2,\Omega|_{\mc{E}_2})=0$ since the Toledo invariant $\op{T}(\Sigma,\phi_2)$ vanishes for any unitary representation. By Theorem \ref{thmsign}, one has
\begin{align*}
\begin{split}
  |\op{T}(\Sigma,\phi)|&=\frac{1}{2}|\op{sign}(\mc{E},\Omega)|=\frac{1}{2}| \op{sign}(\mc{E}_1,\Omega|_{\mc{E}_1})|\\
  &\leq \frac{1}{2}(p+p)|\chi(\Sigma)|=p|\chi(\Sigma)|=\min\{p,q\}|\chi(\Sigma)|.
 \end{split}
\end{align*}
On the other hand, by \cite[Theorem 6.7]{BGG},  the images of all maximal representations of Toledo invariant are in $\op{U}(p,p)\times \op{U}(q-p)\subset\op{U}(p,q)$. Hence, by using Atiyah-Patodi-Singer index theorem, we obtain the following Milnor-Wood ineqaulity
	\begin{align*}
\begin{split}
   |\op{T}(\Sigma,\phi)|\leq \min\{p,q\}|\chi(\Sigma)|,
 \end{split}
\end{align*}
which was originally proved by  A. Domic and D. Toledo \cite{DT}.

\begin{rem}
	For the case $\p\Sigma\neq\emptyset$ and $G$ is a group of Hermitian type, it was proved by Burger, Iozzi and  Wienhard \cite[Theorem 1 (1)]{BIW} that 
	\begin{equation*}
  |\op{T}(\Sigma,\phi)|\leq \op{rank}(G/K)|\chi(\Sigma)|.
\end{equation*}
The above inequality was also generalized to the higher dimensional case by using the isometric isomorphism of $j_{\p\Sigma}$, see \cite[Theorem 1, Corollary 2]{BBFIPP} or \cite[Theorem 1.2]{KK1}. More precisely, 
\begin{align*}
\begin{split}
  |\op{T}(\Sigma,\phi)|&=|\langle j^{-1}_{\p\Sigma}i_\Sigma\phi^*_b(\kappa_G^b),[\Sigma,\p\Sigma]\rangle|\\
  &\leq \|j^{-1}_{\p\Sigma}i_\Sigma\phi^*_b(\kappa_G^b)\|\cdot\|[\Sigma,\p\Sigma]\|_1\\
  &=\|\phi^*_b(\kappa_G^b)\|\cdot\|[\Sigma,\p\Sigma]\|_1\\
  &\leq \|\kappa_G^b\|\cdot\|[\Sigma,\p\Sigma]\|_1\\
  &=\op{rank}(G/K)\cdot|\chi(\Sigma)|,
 \end{split}
\end{align*}
 where the third equality holds since $j_{\p\Sigma}$ is an isometric isomorphism and the Gromov isomorphism $i_\Sigma$ is also isometric, the fourth inequality  by the fact that the pullback is norm decreasing, and the last equality since  $\|\kappa_G^b\|=\frac{\op{rank}(G/K)}{2}$ and $\|[\Sigma,\p\Sigma]\|_1=2|\chi(\Sigma)|$.
\end{rem}

\subsection{Surface group representations in $\mr{U}(p)\times\mr{U}(q)$}\label{ex-unitary}
 We consider a representation $\phi:\pi_1(\Sigma)\to \op{U}(p)\times \op{U}(q)\subset \op{U}(p,q)$. In this case, $\phi(\gamma)\in \op{U}(p)\times \op{U}(q)$ has the form 
 \begin{align*}
\begin{split}
  \left(\begin{matrix}
  a&0 \\
  0&d 
\end{matrix}\right),
 \end{split}
\end{align*}
where $a\in \op{U}(p)$ and $d\in \op{U}(q)$, $\gamma\in \pi_1(\Sigma)$. The flat bundle $(\mc{E},\Omega)$ is a Hermitian (indefinite) vector bundle, where 
$\Omega$ is a Hermitian form and $\Omega$ is given by the matrix $I_{p,q}$ with respect to the standard basis of $E$.
Denote $\mbf{J}:=iI_{p,q}$, then $i\Omega(\cdot,\mbf{J}\cdot)$ is positive definite. Since $[\mbf{J}, \phi(\gamma)]=0$ for any $\gamma\in \pi_1(\Sigma)$ and $\mbf{J}$ is constant, so $\mbf{J}\in \mc{J}_o(\mc{E},\Omega)$. Note that $[d,\mbf{J}]=0$, so $d$ is a peripheral connection on $\mc{E}$. Therefore
\begin{align*}
\begin{split}
  \op{sign}(\mc{E},\Omega)=2\int_\Sigma\left(c_1(\mc{E}^+,d|_{\mc{E}^+})-c_1(\mc{E}^-,d|_{\mc{E}^-})\right)+\eta(A_{\mbf{J}})=\eta(A_{\mbf{J}}),
 \end{split}
\end{align*}
since $d|_{\mc{E}^\pm}$ is flat. The Toledo invariant  and rho invariant are given by
\begin{align*}
\begin{split}
  \op{T}(\Sigma,\phi)=0,\quad \rro_\phi(\p\Sigma)=\eta(A_{\mbf{J}}).
 \end{split}
\end{align*}
For the boundary component $c_i$, we assume that
\begin{align*}
\begin{split}
  \phi(c_i)=\left(\begin{matrix}
  U&0 \\
  0&V 
\end{matrix}\right)\left(\begin{matrix}
  \op{diag}(e^{i\theta_{i,1}},\cdots, e^{i\theta_{i,p}})& 0\\
 0 &  \op{diag}(e^{i\theta_{i,p+1}},\cdots, e^{i\theta_{i,p+q}})
\end{matrix}\right)\left(\begin{matrix}
  U^{-1}&0 \\
  0&V^{-1}
\end{matrix}\right),
 \end{split}
\end{align*}
where $U\in \op{U}(p)$ and $V\in \op{U}(q)$, $\theta_{i,l}\in [0,2\pi)$, $1\leq l\leq p+q$, $1\leq i\leq n$. From Section \ref{Appeta}, the set of all eigenvalues (with multiplicities) of $A_{\mbf{J}}=\mbf{J}\frac{d}{dx}$ is given by
\begin{align*}
\begin{split}
  \left\{\frac{\theta_{i,1}}{2\pi}+k_{i,1},\cdots,\frac{\theta_{i,p}}{2\pi}+k_{i,p},-\frac{\theta_{i,p+1}}{2\pi}+k_{i,p+1},\cdots,-\frac{\theta_{i,p+q}}{2\pi}+k_{i,p+q},1\leq i\leq n,k_{i,l}\in \mb{Z}\right\}.
 \end{split}
\end{align*}
From the definition of eta invariant, then 
\begin{align}\label{sign5}
\begin{split}
  \op{sign}(\mc{E},\Omega)= \eta(A_{\mbf{J}})=\sum_{i=1}^n\left(\sum_{j\in \{\theta_{i,j}\neq 0\}}\left(1-\frac{\theta_{i,j}}{\pi}\right)-\sum_{l\in \{\theta_{i,l}\neq 0\}}\left(1-\frac{\theta_{i,l}}{\pi}\right)\right).
 \end{split}
\end{align}
From the above formula, one can obtain a bound for signature
\begin{align*}
\begin{split}
  |\op{sign}(\mc{E},\Omega)| &\leq \sum_{i=1}^n\sum_{j\in \{\theta_{i,j}\neq 0\}}\left|1-\frac{\theta_{i,j}}{\pi}\right|+\sum_{i=1}^n\sum_{l\in \{\theta_{i,l}\neq 0\}}\left|1-\frac{\theta_{i,l}}{\pi}\right|\\
  &<np+nq=n(p+q).
 \end{split}
\end{align*}
For the case of $g\geq 1$, then 
\begin{align*}
\begin{split}
   |\op{sign}(\mc{E},\Omega)|<n(p+q)\leq |2g-2+n|(p+q)=(p+q)|\chi(\Sigma)|.
 \end{split}
\end{align*}
For the case of $g=0$, in this case, one has
\begin{align*}
\begin{split}
  |\op{sign}(\mc{E},\Omega)|\leq \left|\sum_{i=1}^n\sum_{j\in \{\theta_{i,j}\neq 0\}}\left(1-\frac{\theta_{i,j}}{\pi}\right)\right|+\left|\sum_{i=1}^n\sum_{l\in \{\theta_{i,l}\neq 0\}}\left(1-\frac{\theta_{i,l}}{\pi}\right)\right|.
 \end{split}
\end{align*}
where the first term in the right hand side of the above inequality is exactly the absolute value of the signature of a flat $\op{U}(p)$-Hermitian vector bundle. Hence, by Theorem \ref{thmsign}, we can conclude the following inequality
\begin{align}\label{inequality for pant}
\begin{split}
  \left|\sum_{i=1}^n\sum_{j\in \{\theta_{i,j}\neq 0\}}\left(1-\frac{\theta_{i,j}}{\pi}\right)\right|\leq (p+0)|\chi(\Sigma)|=p|2g-2+n|=p(n-2).
 \end{split}
\end{align}
Similarly, one has 
\begin{align*}
\begin{split}
  \left|\sum_{i=1}^n\sum_{l\in \{\theta_{i,l}\neq 0\}}\left(1-\frac{\theta_{i,l}}{\pi}\right)\right|\leq q(n-2).
 \end{split}
\end{align*}
 For a surface with genus zero, then it is a $2$-sphere with $n$ discs deleted, we denote it by $\Sigma_n$. We can cut it off  into $(n-2)$ union of $\Sigma_3$, and the signature of the surface is exactly the sum of the signature of these $\Sigma_3$. Hence, the inequality \eqref{inequality for pant} follows from the case $n=3$, i.e.
 \begin{align}\label{etainequality}
\begin{split}
  \left|\sum_{i=1}^3\sum_{j\in \{\theta_{i,j}\neq 0\}}\left(1-\frac{\theta_{i,j}}{\pi}\right)\right|\leq p,
 \end{split}
\end{align}
which can be proved directly by using the results from the multiplicative Horn problem, see Subsection \ref{Horn}.

In particular, if we take $q=0$, then 
 \begin{align*}
\begin{split}
   \op{sign}(\mc{E},\Omega)=\sum_{i=1}^n\left(\sum_{j\in \{\theta_{i,j}\neq 0\}}\left(1-\frac{\theta_{i,j}}{\pi}\right)\right),
 \end{split}
\end{align*}
which does not vanish in general. Since the Toledo invariant is bounded by $\min\{p,q\}|\chi(\Sigma)|$, one might think that $| \op{sign}(\mc{E},\Omega)|$ is also bounded from above by $2\min\{p,q\}|\chi(\Sigma)|$. From the above discussion, we find that the signature can not be bounded by the constant  $2\min\{p,q\}|\chi(\Sigma)|$ in general.

\section{Surface group representations in $\op{SO}^*(2n)$}\label{SO}

In this section, we will consider the surface group representations in $\op{SO}^*(2n)$, one can refer to \cite[Page 71-74, Section (2.4)]{Mok} for the bounded symmetric domain of type $\op{II}$. Recall that 
\begin{align*}
\begin{split}
  \op{SO}^*(2n):=\left\{M\in \mathrm{SL}(2n, \mathbb{C}): {M}^{\top} M=I_{2n} ,\,M^* {J}_{{n}}{M}={J}_{n}\right\},
 \end{split}
\end{align*}
where $I_{2n}$ is the identity matrix and $J_n:=\left(\begin{matrix}
  0&I_n \\
  -I_n&0 
\end{matrix}\right)$, see \cite[Page 445]{Hel}. Each element in the group $\op{SO}^*(2n)$ leaves invariant the skew Hermitian form 
$$-z^{1} \bar{z}^{n+1}+z^{n+1} \bar{z}^{1}-z^{2} \bar{z}^{n+2}+z^{n+2} \bar{z}^{2}-\cdots-z^{n} \bar{z}^{2 n}+z^{2 n} \bar{z}^{n}.$$
This group $\op{SO}^*(2n)$ is isomorphic to the the following group 
\begin{align*}
\begin{split}
  G_o:=\left\{\mathrm{M} \in \mathrm{SL}(2 n, \mathbb{C}): M^* {I}_{n,n} M={I}_{{n}, {n}} ,\, {M}^{\top} S_{\mathrm{n}} {M}={S}_{{n}}\right\}=\mr{O}(n,\mb{C})\cap \op{SU}(n,n),
 \end{split}
\end{align*}
where $\mr{O}(n,\mb{C})$ is the complex orthogonal group with respect to $S_n=\left(\begin{matrix}
  0&I_n \\
  I_n&0 
\end{matrix}\right)$, and 
$
  I_{n,n}=\left(\begin{matrix}
  I_n&0 \\
  0&-I_n 
\end{matrix}\right).
$
The isomorphism $G_o\cong \op{SO}^*(2n)$ is given by $M\mapsto UMU^{-1}$, where $U=\frac{1}{\sqrt{2}}\left(\begin{matrix}
  I_n& iI_n \\
  iI_n&I_n 
\end{matrix}\right)$. The group $G_o$ has the following form 
\begin{align*}
\begin{split}
  G_o=\left\{M=\left(\begin{matrix}
  A& B\\
  -\o{B}&\o{A} 
\end{matrix}\right): M\in \op{SU}(n,n)\right\}.
 \end{split}
\end{align*}
The isotropy group $K\in G_o$ at the origin $o$ (i.e. the $n\times n$ zero matrix) is given by 
\begin{align*}
\begin{split}
  K=\left\{M=\left(\begin{matrix}
  U&0 \\
  0&\o{U} 
\end{matrix}\right): M\in \op{SU}(n,n)\right\}\cong \op{U}(n).
 \end{split}
\end{align*}
The  bounded symmetric domain of type $\op{II}$ is defined as 
\begin{align*}
\begin{split}
  \op{D}^{\op{II}}_n:=\{W\in \op{D}^{\op{I}}_{n,n}: W^{\top}=-W \}.
 \end{split}
\end{align*}
The group $G_o$ acts transitively on $ \op{D}^{\op{II}}_n$ with the isotropy group $K$ at the origin $0$, hence 
\begin{align*}
\begin{split}
   \op{D}^{\op{II}}_n\cong G_o/K,
 \end{split}
\end{align*}
see e.g. \cite[Page 74]{Mok}.

Let $\Sigma$ be a surface with boundary, and consider surface group representations  into the group $\op{SO}^*(2n)$. By the isomorphism $G_o\cong \op{SO}^*(2n)$, we will identify the groups $G_o$ and $\op{SO}^*(2n)$. Hence, we just need to consider the surface group representations in the group $G_o$.

Since  $G_o\subset \op{SU}(n,n)\subset \op{U}(n,n)$, so we can define the signature as in Section \ref{Definition of signature}. For any representation $\phi:\pi_1(\Sigma)\to G_o$, denote $\mc{E}=\wt{\Sigma}\times_\phi E$, where $E=\mb{C}^{2n}$, and let $\Omega$ be the Hermitian form given by \eqref{Hermitian form}, in terms of matrices, $\Omega$ is represented by the matrix $I_{n,n}$. Denote by $\op{sign}(\mc{E},\Omega)$ the signature of the flat Hermitian vector bundle $(\mc{E},\Omega)$.
From Theorem \ref{thmsign2}, the signature can be given by 
	\begin{equation*}
	\op{sign}(\mc{E},\Omega)= 2\int_\Sigma\left(c_1(\mc{E}^+,\n^{\mc{E}^+})-c_1(\mc{E}^-,\n^{\mc{E}^-})\right)+\eta(A_{\mbf{J}}),
\end{equation*}
for any $\mbf{J}\in \mc{J}(\mc{E},\Omega)$. Let 
\begin{align*}
\mc{J}_{\op{II}}(E,\Omega):=\{J\in\mc{J}(E,\Omega): J^\top S_n J=S_n\}=\mc{J}(E,\Omega)\cap G_o.	
\end{align*}
The group $G_o$ acts on $\mc{J}_{\op{II}}(E,\Omega)$ by $ZJZ^{-1}$, the action is transitive and the isotropy group at $iI_{n,n}$ is exactly $K$. Hence 
$$\mc{J}_{\op{II}}(E,\Omega)\cong G_o/K\cong \op{D}^{\op{II}}_n.$$
In fact, this isomorphism can be given by $\mbf{J}_{\mr{II}}=\mbf{J}_{\mr{I}}|_{\op{D}^{\op{II}}_n}:\op{D}^{\op{II}}_n\to \mc{J}_{\op{II}}(E,\Omega)$, i.e.
\begin{align*}
\begin{split}
  \mbf{J}_{\mr{II}}(W):=i\left(\begin{matrix}
  (I_n+W\o{W})^{-1}(I_n-W\o{W})& -2(I_n+W\o{W})^{-1}W\\
  -2(I_n+\o{W}W)^{-1}\o{W}&-(I_n+\o{W}W)^{-1}(I_n-\o{W}W) 
\end{matrix}\right)
 \end{split}
\end{align*}
for any $W\in \op{D}^{\op{II}}_{n}$. Denote $\mc{J}_{\op{II}}(\mc{E},\Omega):=C^{\infty}(\Sigma,\wt{\Sigma}\times_\phi \mc{J}_{\op{II}}(E,\Omega))$, and set
\begin{align*}
	\begin{split}
	\mc{J}_{\op{II},o}(\mc{E},\Omega)=\{\mbf{J}\in \mc{J}_{\op{II}}(\mc{E},\Omega)|  &\mbf{J}=p^*J
	\text{ on a small collar neighborhood}\\
	&\text{ of } \p\Sigma, \text{ where } J \in \mc{J}_{\op{II}}(\mc{E}|_{\p\Sigma},\Omega)\},
	\end{split}
	\end{align*}
where 
	$p:\p\Sigma\times [0,1]\to \p\Sigma$ denotes the natural projection. Let $\omega_{\op{D}^{\op{II}}_{n}}$ be the invariant K\"ahler metric on $\op{D}^{\op{II}}_{n}$ with the minimal holomorphic sectional curvature is $-1$, then 
	\begin{align*}
\begin{split}
    \omega_{\op{D}^{\op{II}}_{n}}=\frac{1}{2}\omega_{\op{D}^{\op{I}}_{n,n}}|_{\op{D}^{\op{II}}_{n}}=	-i\p\b{\p}\log\det(I_n+\o{W}W),
 \end{split}
\end{align*}
see e.g. \cite[Lemma 5.5]{KM1}.
	Similar to Section \ref{Tol} and \eqref{To1}, we obtain
\begin{align*}
\begin{split}
&\quad2 \mr{T}(\Sigma,\phi) =\frac{1}{2\pi}\int_\Sigma\left(\wt{\mbf{J}}^*(\omega_{\op{D}^{\op{I}}_{n,n}}|_{\op{D}^{\op{II}}_{n}})-\sum_{i=1}^qd(\chi_i \wt{\mbf{J}}^*\alpha_i)\right)\\
  &=\int_{\Sigma}\left(c_1(\mc{E}^-,\tau \wt{\mbf{J}}^*\n^{F_\phi}\tau^{-1}|_{\mc{E}^-})-c_1(\mc{E}^+,\tau \wt{\mbf{J}}^*\n^{F_\phi}\tau^{-1}|_{\mc{E}^+})\right)-\frac{1}{2\pi}\sum_{i=1}^q\int_\Sigma d(\chi_i \wt{\mbf{J}}^*\alpha_i).
 \end{split}
\end{align*}
for any $\mbf{J}\in \mc{J}_{\op{II},o}(\mc{E},\Omega)$, where $q$ denotes the number of connect components of $\p\Sigma$, $\wt{\mbf{J}}:\wt{\Sigma}\to \mc{J}_{\op{II}}(E,\Omega)\cong \op{D}^{\op{II}}_n$ is the $\phi$-equivariant map given by $\mbf{J}$,
 $\alpha_i=d^c\psi_i$ and 
$$\psi_i=-\log \left(|\det(\o{W_i}W+I_n)|^{-2}\det(I_n+\o{W}W)\right),$$
where $W_{i}\in \o{\op{D}^{\op{II}}_n}$ is a fixed point of $\phi(c_i)$, which is an invariant (up to a constant) K\"ahler potential under the isotropy group $K_{W_i}$ of $W_i$. Hence,
\begin{align*}
\begin{split}
  \op{sign}(\mc{E},\Omega)=-4\op{T}(\Sigma,\phi)+\rro_\phi(\p\Sigma),
 \end{split}
\end{align*}
where the rho invariant is given by 
\begin{align*}
\begin{split}
  \rro_\phi(\p\Sigma)=-\frac{1}{\pi}\sum_{i=1}^q\int_\Sigma d(\chi_i \wt{\mbf{J}}^*\alpha_i)+\eta(A_{\mbf{J}}).
 \end{split}
\end{align*}
By Theorem \ref{thmsign}, one has
\begin{align*}
\begin{split}
  | \op{sign}(\mc{E},\Omega)|\leq 2n |\chi(\Sigma)|=\dim E |\chi(\Sigma)|.
 \end{split}
\end{align*}
In particular, if the surface $\Sigma$ is closed, then 
$$|\op{T}(\Sigma,\phi)|=\frac{1}{4} | \op{sign}(\mc{E},\Omega)|\leq \frac{n}{2}|\chi(\Sigma)|.$$
In particular, if $n$ is even, then $|\op{T}(\Sigma,\phi)|\leq \frac{n}{2}|\chi(\Sigma)|=\left[\frac{n}{2}\right]|\chi(\Sigma)|$, which is exactly the Milnor-Wood inequality proved by A. Domic and D. Toledo \cite{DT}.

\section{Surface group representations in $\op{Sp}(2n,\mb{R})$}\label{SP}

 In this section, we can deal with the case of $\op{Sp}(2n,\mb{R})$  by using the results from $\op{U}(p,q)$-case and discuss some Milnor-Wood inequalities for $\mr{Sp}(2,\mb{R})$.
 
 \subsection{Signature and Toledo invariant}
 
  Recall that 
\begin{align*}
\begin{split}
  \op{Sp}(2n,\mb{R})=\{M\in \op{GL}(2n,\mb{R}): M^\top J_nM=J_n\}=\op{Sp}(2n,\mb{C})\cap \op{GL}(2n,\mb{R}),
 \end{split}
\end{align*}
where $J_n=\left(\begin{matrix}
  0&I_n \\
  -I_n&0 
\end{matrix}\right)$,
which is isomorphic to the following group
\begin{align*}
\begin{split}
  G_o=\op{Sp}(2n,\mb{C})\cap \op{SU}(n,n)=\left\{M=\left(\begin{matrix}
  A&B \\
  \o{B}&\o{A} 
\end{matrix}\right):M\in \op{SU}(n,n)\right\}.
 \end{split}
\end{align*}
This isomorphism is given by 
\begin{align*}
\begin{split}
  \Phi: G_o\to \op{Sp}(2n,\mb{R}),\quad N\mapsto UNU^{-1},
 \end{split}
\end{align*}
where $U=\left(\begin{matrix}
  -iI_n& iI_n\\
  I_n&I_n 
\end{matrix}\right)$. One can also refer to \cite[Page 68-71, Section 2.3]{Mok} on the isomorphism between $G_o$ and $\op{Sp}(2n,\mb{R})$, and some basic facts on the bounded symmetric domain of type $\op{III}$. The isotropy subgroup of $G_o$ at the origin $o$ (i.e. the $n\times n$ zero matrix) is 
\begin{align*}
\begin{split}
  K=\left\{X=\left(\begin{matrix}
  U&0 \\
  0&\o{U} 
\end{matrix}\right):U\in\op{U}(n)\right\}\cong \op{U}(n).
 \end{split}
\end{align*}
The group $G_o$ acts on the bounded symmetric domain $\op{D}^{\op{III}}_n$ of type $\op{III}$ transitively, where 
\begin{align*}
\begin{split}
  \op{D}^{\op{III}}_n:=\{W\in \op{D}^{\op{I}}_{n,n}: W^{\top}=W \}.
 \end{split}
\end{align*}
Hence 
\begin{align*}
\begin{split}
    \op{D}^{\op{III}}_n\cong G_o/K\cong \op{Sp}(2n,\mb{R})/\op{U}(n).
 \end{split}
\end{align*}

Denote by $\Omega_0$ the real symplectic form on $\mb{R}^{2n}$ with the matrix form is $J_n$. For any $\phi_0:\pi_1(\Sigma)\to \op{Sp}(2n,\mb{R})\subset \op{Sp}(2n,\mb{C})$, it induces a representation $\phi:\pi_1(\Sigma)\to G_o$, i.e. 
$$\phi(\cdot)=U^{-1}\phi_0(\cdot)U.$$
Denote $\mc{E}_0=\wt{\Sigma}\times_{\phi_0}\mb{R}^{2n}$, $\mc{E}=\wt{\Sigma}\times_\phi E$, $E=\mb{C}^{2n}$ and $(\mc{E}_0)_{\mb{C}}=\mc{E}_0\otimes \mb{C}=\wt{\Sigma}\times_{\phi_0}\mb{C}^{2n}$. The relation between $\phi_0$ and $\phi$ gives an isomorphism between the bundles $\mc{E}$ and $(\mc{E}_0)_{\mb{C}}$, i.e. 
\begin{align*}
\begin{split}
  (\mc{E}_0)_{\mb{C}}\cong \mc{E},\quad a\mapsto U^{-1}a.
 \end{split}
\end{align*}
For any $[a], [b]\in \widehat{\mr{H}}^1(\Sigma,(\mc{E}_0)_{\mb{C}})$, one has
\begin{align*}
\begin{split}
  \Omega_0([a],\o{[b]})&= (U^\top\Omega_0\o{U})([U^{-1}a],\o{[U^{-1}b]})\\
  &=(U\Omega_0\o{U})([U^{-1}a],\o{[U^{-1}b]}).
 \end{split}
\end{align*}
Since the matrix  form of $U\Omega_0\o{U}$ is 
$
  UJ_n\o{U}=-iI_{n,n},
$
which represents the Hermitian form $-i\Omega$ exactly. Hence 
\begin{align}\label{sp6.1}
\begin{split}
   \Omega_0([a],\o{[b]})=-i\Omega([U^{-1}a],[U^{-1}b]).
 \end{split}
\end{align}
The signature $\op{sign}(\mc{E}_0,\Omega_0)$ of the flat symplectic vector bundle $(\mc{E}_0,\Omega_0)$ is defined as the signature of the symmetric quadratic form $Q_{\mb{R}}(\cdot,\cdot)=\int_\Sigma \Omega_0(\cdot\cup \cdot)$ on the space $\widehat{\mr{H}}^1(\Sigma,\mc{E}_0)$. Moreover, it also can be given by 
\begin{align*}
\begin{split}
  \op{sign}(\mc{E}_0,\Omega_0)=\op{sign}((\mc{E}_0)_{\mb{C}},\Omega_0)
 \end{split}
\end{align*}
since $\widehat{\mathrm{H}}^1(\Sigma,(\mc{E}_0)_{\mb{C}})=\widehat{\mathrm{H}}^1(\Sigma,\mc{E}_0)\otimes\mb{C}$. Here the signature $\op{sign}((\mc{E}_0)_{\mb{C}},\Omega_0)$ is defined as the signature of Hermitian form $Q_{\mb{C}}(\cdot,\cdot)=\int_\Sigma \Omega_0(\cdot\cup \o{\cdot})$ on the space $\widehat{\mr{H}}^1(\Sigma,(\mc{E}_0)_{\mb{C}})$. From \eqref{sp6.1}, one has 
\begin{align*}
\begin{split}
  Q_{\mb{C}}([a],[b])=-i\int_\Sigma\Omega([U^{-1}a],[U^{-1}b]).
 \end{split}
\end{align*}
Recall that the signature $\op{sign}(\mc{E},\Omega)$ of flat Hermitian vector bundle $(\mc{E},\Omega)$ is defined as the signature of the Hermitian form $i\int_\Sigma\Omega(\cdot,\cdot)$, which follows that 
\begin{align}\label{relation of signatures}
\begin{split}
 \op{sign}(\mc{E}_0,\Omega_0)=-\op{sign}(\mc{E},\Omega). 
 \end{split}
\end{align}
For any representation $\phi_0:\pi_1(\Sigma)\to \op{Sp}(2n,\mb{R})$, and denote $\phi=U^{-1}\phi_0 U$. By Theorem \ref{thmsign2}, one has 
\begin{equation*}
	\op{sign}(\mc{E},\Omega)= 2\int_\Sigma\left(c_1(\mc{E}^+,\n^{\mc{E}^+})-c_1(\mc{E}^-,\n^{\mc{E}^-})\right)+\eta(A_{\mbf{J}}),
\end{equation*}
for any $\mbf{J}\in \mc{J}(\mc{E},\Omega)$. Let 
\begin{align*}
\mc{J}_{\op{III}}(E,\Omega):=\{J\in\mc{J}(E,\Omega): J^\top J_n J=J_n\}=\mc{J}(E,\Omega)\cap G_o.	
\end{align*}
The group $G_o$ acts on $\mc{J}_{\op{III}}(E,\Omega)$ by $ZJZ^{-1}$, the action is transitive and the isotropy group at $iI_{n,n}$ is exactly $K$. Hence 
$$\mc{J}_{\op{III}}(E,\Omega)\cong G_o/K\cong \op{D}^{\op{III}}_n.$$
In fact, this isomorphism can be given by $\mbf{J}_{\mr{III}}:=\mbf{J}_{\mr{I}}|_{\op{D}^{\op{III}}_n}:\op{D}^{\op{III}}_n\to \mc{J}_{\op{III}}(E,\Omega)$, that is
\begin{align}\label{isomorphism DIII}
\begin{split}
  \mbf{J}_{\mr{III}}(W)=i\left(\begin{matrix}
  (I_n-W\o{W})^{-1}(I_n+W\o{W})& -2(I_n-W\o{W})^{-1}W\\
  2(I_n-\o{W}W)^{-1}\o{W}&-(I_n-\o{W}W)^{-1}(I_n+\o{W}W) 
\end{matrix}\right).
 \end{split}
\end{align}
for any $W\in \op{D}^{\op{III}}_{n}$. 
\begin{rem}\label{rem5.2}
	The isomorphism $\mbf{J}_{\mr{III}}$ induces an isomorphism $$\mbf{J}_{\mr{III},0}:=-U\mbf{J}_{\mr{III}}U^{-1}:\op{D}^{\op{III}}_n\to \mc{J}(\mb{R}^{2n},\Omega_0),$$ one can check that it is a bijection, where $\mc{J}(\mb{R}^{2n},\Omega_0)$ denotes the subset of $\mr{Sp}(2n,\mb{R})$ such that $J^2=-I_{2n}$ and $\Omega_0(\cdot, J\cdot)>0$. For any $Z\in \op{Sp}(2n,\mb{R})$, then  $Z\mbf{J}_{\mr{III},0}Z^{-1}\in \mc{J}(\mb{R}^{2n},\Omega)$, and the induced action on $\op{D}^{\op{III}}_n$ is given by
	\begin{equation}\label{Action}
  Z(W)=(Z_1W+Z_2)(\o{Z_2}W+\o{Z_1})^{-1}\in \op{D}^{\op{III}}_n,
\end{equation}
 where $Z_1$ and $Z_2$ is defined by
	 $Z=U\left(\begin{matrix}
Z_1 & Z_2\\
\o{Z_2} & \o{Z_1}	
\end{matrix}
\right)U^{-1}$.
 The isomorphism $\op{Sp}(2n,\mb{R})/\op{U}(n)\cong\mc{J}(\mb{R}^{2n},\Omega)$ is given by $Z\cdot \op{U}(n)\mapsto ZJZ^{-1}$, where $J=\left(\begin{matrix}
0 & -I_n\\
I_n & 0	
\end{matrix}
\right)$ and $\op{U}(n)\cong \{Z\in \op{Sp}(2n,\mb{R}): ZJZ^{-1}=J\}$.
\end{rem}
 Denote $\mc{J}_{\op{III}}(\mc{E},\Omega):=C^{\infty}(\Sigma,\wt{\Sigma}\times_\phi \mc{J}_{\op{III}}(E,\Omega))$, and set
\begin{align*}
	\begin{split}
	\mc{J}_{\op{III},o}(\mc{E},\Omega)=\{\mbf{J}\in \mc{J}_{\op{III}}(\mc{E},\Omega)|  &\mbf{J}=p^*J
	\text{ on a small collar neighborhood}\\
	&\text{ of } \p\Sigma, \text{ where } J \in \mc{J}_{\op{III}}(\mc{E}|_{\p\Sigma},\Omega)\},
	\end{split}
	\end{align*}
where 
	$p:\p\Sigma\times [0,1]\to \p\Sigma$ denotes the natural projection. Let $\omega_{\op{D}^{\op{III}}_{n}}$ be the invariant K\"ahler metric on $\op{D}^{\op{III}}_{n}$ with the minimal holomorphic sectional curvature is $-1$, then 
	\begin{align*}
\begin{split}
    \omega_{\op{D}^{\op{III}}_{n}}=\omega_{\op{D}^{\op{I}}_{n,n}}|_{\op{D}^{\op{III}}_{n}}=	-2i\p\b{\p}\log\det(I_n-\o{W}W),
 \end{split}
\end{align*}
see e.g. \cite[Lemma 5.4]{KM1}. Similar to Section \ref{Tol} and \eqref{To1}, we obtain
\begin{align*}
\begin{split}
 \mr{T}(\Sigma,\phi) &=\frac{1}{2\pi}\int_\Sigma\left(\wt{\mbf{J}}^*(\omega_{\op{D}^{\op{I}}_{n,n}}|_{\op{D}^{\op{III}}_{n}})-\sum_{i=1}^qd(\chi_i \wt{\mbf{J}}^*\alpha_i)\right)\\
  &=\int_{\Sigma}\left(c_1(\mc{E}^-,\tau \wt{\mbf{J}}^*\n^{F_\phi}\tau^{-1}|_{\mc{E}^-})-c_1(\mc{E}^+,\tau \wt{\mbf{J}}^*\n^{F_\phi}\tau^{-1}|_{\mc{E}^+})\right)\\
  &\quad-\frac{1}{2\pi}\sum_{i=1}^q\int_\Sigma d(\chi_i \wt{\mbf{J}}^*\alpha_i).
 \end{split}
\end{align*}
for any $\mbf{J}\in \mc{J}_{\op{III},o}(\mc{E},\Omega)$, where $q$ denotes the number of connected components of $\p\Sigma$, $\wt{\mbf{J}}:\wt{\Sigma}\to \mc{J}_{\op{III}}(E,\Omega)\cong \op{D}^{\op{III}}_n$ is a $\phi$-equivariant map given by $\mbf{J}$,
 $\alpha_i=d^c\psi_i$ and 
$$\psi_i=-\log \left(|\det(\o{W_i}W-I_n)|^{-2}\det(I_n-\o{W}W)\right),$$
where $W_{i}\in \o{\op{D}^{\op{III}}_n}$ is a fixed point of $\phi(c_i)$, which is an invariant (up to a constant) K\"ahler potential under the isotropy group $K_{W_i}$ of $W_i$.
\begin{rem}\label{rempotential}
 Denote by 
\begin{align*}
\mb{H}_n=\{Z\in \mb{C}^{n\times n}|Z=Z^\top, \op{Im}Z>0\}	
\end{align*}
the Siegel upper half plane, and 
\begin{align*}
\Psi:\op{D}^{\op{III}}_n\to \mb{H}_n,\quad Z=\Phi(W)=i(I_n-W)(I_n+W)^{-1}	
\end{align*}
the identification between $\op{D}^{\op{III}}_n$ and $\mb{H}_n$. The induced action of $L\in \op{Sp}(2n,\mb{R})$ on $\mb{H}_n$ is the generalized M\"obius transformation. Moreover, 
$$\det\op{Im}Z=|\det(I_n+W)|^{-2}\det(I_n-\o{W}W).$$
 Suppose 
 \begin{align*}
 L=U\left(\begin{matrix}
a & b\\
\b{b} & \b{a}	
\end{matrix}
\right)U^{-1}=\left(\begin{matrix}
L_1 & L_2	\\
L_3 & L_4
\end{matrix}
\right)\in \op{Sp}(2n,\mb{R}),	
 \end{align*}
 where $L_3=\frac{i}{2}(a-\b{a})+\frac{i}{2}(\b{b}-b)$. Then  $W_0=-I_n$ is a fixed point of $L$  is equivalent to 
 $b-\o{b}=a-\o{a}$, and so $L_3=0$. Hence $L(-I_n)=-I_n$ if and only if $L$ has the following matrix form 
  \begin{align*}
 L=\left(\begin{matrix}
L_1 & L_2	\\
0 & L_4
\end{matrix}
\right)\in \op{Sp}(2n,\mb{R}).	
 \end{align*}
Thus  the K\"ahler potential can be given by
 \begin{equation}
  \psi=-\log\det\op{Im}Z.
\end{equation}
\end{rem}
\vspace{3mm}
 Hence,
\begin{equation}\label{Tole}
\begin{split}
\op{sign}(\mc{E}_0,\Omega_0)=-  \op{sign}(\mc{E},\Omega)=2\op{T}(\Sigma,\phi)-\rro_\phi(\p\Sigma),
 \end{split}
\end{equation}
where the rho invariant is given by 
\begin{align*}
\begin{split}
  \rro_\phi(\p\Sigma)=-\frac{1}{\pi}\sum_{i=1}^q\int_\Sigma d(\chi_i \wt{\mbf{J}}^*\alpha_i)+\eta(A_{\mbf{J}}).
 \end{split}
\end{align*}
Note that $\op{T}(\Sigma,\phi)=\op{T}(\Sigma,\phi_0)$ due to the isomorphism $G_o\cong\op{Sp}(2n,\mb{R})$ and $\phi=U^{-1}\phi_0 U$. Denote 
\begin{align}
\begin{split}
  \rro_{\phi_0}(\p\Sigma):=\frac{1}{\pi}\sum_{i=1}^q\int_\Sigma d(\chi_i \wt{\mbf{J}_0}^*\alpha_i)+\eta(A_{\mbf{J}_0})
 \end{split}
\end{align}
for any $\mbf{J}_0\in \mc{J}_o(\mc{E}_0,\Omega_0)$, where $\mc{J}_o(\mc{E}_0,\Omega_0)$ is defined as
\begin{align}\label{J0spacebis}
	\begin{split}
	\mc{J}_o(\mc{E}_0,\Omega_0)=\{\mbf{J}\in \mc{J}(\mc{E}_0,\Omega_0)|  &\mbf{J}=p^*J
	\text{ on a small collar neighborhood}\\
	&\text{ of } \p\Sigma, \text{ where } J \in \mc{J}( \mc{E}_0|_{\p\Sigma},\Omega_0)\},
	\end{split}
	\end{align}
	and $ \mc{J}(\mc{E}_0,\Omega_0):=C^{\infty}(\Sigma,\wt{\Sigma}\times_{\phi_0}\mc{J}(\mb{R}^{2n},\Omega_0))$.
	From the definition of $\mbf{J}_{\mr{III},0}$, see Remark \ref{rem5.2}, one has
\begin{align*}
\begin{split}
\mbf{J}_{\mr{III},0}(W)=-U\mbf{J}_{\mr{I}}(W)U^{-1},
 \end{split}
\end{align*}
where $\mbf{J}_{\mr{I}}(W)$ is given by \eqref{isomorphism DIII}.
For any $\mbf{J}\in \mc{J}_{\op{III},o}(\mc{E},\Omega)$, denote $$\mbf{J}_0=U(-\mbf{J})U^{-1},$$ one can check that
$
  \mbf{J}_0\in \mc{J}_o(\mc{E}_0,\Omega_0).
$
For any tangent vector $X\in T\wt{\Sigma}$, one has
 \begin{align}\label{equals of two J}
\begin{split}
  (\wt{\mbf{J}}^*\alpha)(X)&= \alpha(\wt{\mbf{J}}_*X)=\alpha((\mbf{J}^{-1}_{\mr{I}}\circ \mbf{J})_*X)\\
  &=\alpha(((\mbf{J}_{\mr{I}}^{-1}\circ(-U^{-1}\mbf{J}_0 U))_*X)\\
  &=\alpha((\mbf{J}^{-1}_{\mr{III},0}\circ \mbf{J}_0)_*X)\\
  &=\alpha((\wt{\mbf{J}}_0)_*X)=(\wt{\mbf{J}}_0^*\alpha)(X),
 \end{split}
\end{align}
which follows that $\wt{\mbf{J}}^*\alpha=\wt{\mbf{J}}_0^*\alpha$. 
For the eta invariants, one has
\begin{align}\label{-eta}
\begin{split}
  \eta(A_{\mbf{J}})=-\eta(A_{\mbf{J}_0}).
 \end{split}
\end{align}
Thus,
\begin{align}\label{-rho}
\begin{split}
   \rro_\phi(\p\Sigma)=- \rro_{\phi_0}(\p\Sigma),
 \end{split}
\end{align}
which implies that 
\begin{align*}
\begin{split}
  \op{sign}(\mc{E}_0,\Omega_0)=2\op{T}(\Sigma,\phi_0)+\rro_{\phi_0}(\p\Sigma).
 \end{split}
\end{align*}
By Theorem \ref{thmsign}, one has
\begin{align*}
\begin{split}
 |\op{sign}(\mc{E}_0,\Omega_0)|= | \op{sign}(\mc{E},\Omega)|\leq 2n |\chi(\Sigma)|=\dim E |\chi(\Sigma)|.
 \end{split}
\end{align*}
In particular, if the surface $\Sigma$ is closed, then 
$$|\op{T}(\Sigma,\phi_0)|=\frac{1}{2} | \op{sign}(\mc{E}_0,\Omega_0)|\leq n|\chi(\Sigma)|,$$
 which is exactly the Milnor-Wood inequality for the real symplectic group  proved by Turaev \cite{Tur}.
\begin{rem}\label{remunitary}
	Following \cite{APSII}, if we consider the unitary representation, i.e. $$\phi:\pi_1(\Sigma)\to \op{U}(n)=\{A+iB\in \op{U}(n)\}\cong\left\{Z=\left(\begin{matrix}
A & B\\
-B & A	
\end{matrix}
\right)\in\op{Sp}(2n,\mb{R})\right\}.$$ Then $[Z,J]=0$ where $J=\left(\begin{matrix}
0 &-I_n\\
I_n &0	
\end{matrix}
\right)$ is the standard complex structure. Hence $J\in \mc{J}_o(\mc{E}_0,\Omega_0)$, and so 
\begin{align*}
  \op{T}(\Sigma,\phi)
  &=\frac{1}{2\pi}\int_\Sigma \wt{{J}}^*\omega_{\op{D}^{\op{III}}_n}-\frac{1}{2\pi}\sum_{i=1}^q\int_{c_i}\wt{{J}}^*\alpha_i=0,
\end{align*}
and \begin{equation*}
  \rro_\phi(\p\Sigma)=\frac{1}{\pi}\sum_{i=1}^q\int_{c_i}\wt{{J}}^*\alpha_i+\eta(A_{{J}})=\eta(A_J).
\end{equation*}
Hence 
$$\op{sign}(\mc{E},\Omega)=  \rro_\phi(\p\Sigma)=\eta(A_{{J}}),$$
which is is agrement with \cite[Theorem 2.2, Theorem 2.4]{APSII}.

\end{rem}

\subsection{Improved Milnor-Wood inequalities for $\op{Sp}(2,\mb{R})$}\label{Class}

We show that for the target group $\op{Sp}(2,\mb{R})$, an upper bound on Toledo invariant can be given that sometimes improves upon previously known bounds.
We consider representations $\phi:\pi_1(\Sigma)\to \op{Sp}(2,\mb{R})$. Burger-Iozzi-Wienhard's Milnor-Wood inequality reads
$$
|\mr{T}(\Sigma,\phi)|\leq|\chi(\Sigma)|.
$$
We shall introduce a boundary contribution which makes the right hand side smaller. Here, $\{x\}=x-\lfloor x \rfloor$ denotes the fractional part of a real number $x$.

\begin{prop}\label{propmod2}
For representations $\phi:\pi_1(\Sigma)\ra \op{Sp}(2,\mb{R})$,
\begin{equation}\label{TI1}
  \op{T}(\Sigma,\phi) \leq |\chi(\Sigma)|+1 - \sum_{c\,;\,\phi(c) \text{ elliptic}}\left \{\frac{\rro(\phi(c))}{2}\right\}
\end{equation}
\end{prop}

\begin{proof}
We need to understand rho invariants of elements of $\op{Sp}(2,\br)$. By definition, one embeds $\op{Sp}(2,\br)$ into $\op{U}(1,1)$ and computes the rho invariant there. An elliptic element $R(\theta)=\begin{pmatrix}
\cos\theta &  -\sin\theta  \\
\sin\theta  &  \cos\theta
\end{pmatrix}$ is mapped to a $\bc$-linear map $L$ of $\bc^2$ with eigenvalues $e^{i\theta}$ and $e^{-i\theta}=e^{i(2\pi-\theta)}$ and eigenvectors $e_+$ and $e_-$ such that $\Omega(e_+,e_+)=1$ and $\Omega(e_-,e_-)=-1$, hence its rho invariant is
\begin{align*}
\rro(L)=(1-\frac{\theta}{\pi})-(1-\frac{2\pi-\theta}{\pi})=2(1-\frac{\theta}{\pi}).
\end{align*}
A unipotent element $\exp(N)=\begin{pmatrix}
1 &  \mu  \\
0  &  1
\end{pmatrix}$ is mapped to a unipotent $\bc$-linear map $L'$ of $\bc^2$ with the same matrix. The $\op{Sp}(2,\br)$-invariant symplectic structure on $\br^2$ is the determinant. The corresponding Hermitian form on $\bc^2$ is $\Omega(u,v)=-\mathrm{det}(iu,\bar v)$. The Hermitian form $\tau$ is 
\begin{align*}
\tau(u,v)=\Omega((iN)u,v)=\mu u_2\overline{v_2},
\end{align*}
the sign of the Hermitian form $\bar\tau$ induced on $\bc^2/\op{Im}(N)$ is the sign of $\mu$, hence its rho invariant is
\begin{align*}
\rro(L')=-\op{sgn}(\mu).
\end{align*}
(for a double-check of these calculations, see the Appendix, Subsection \ref{Appeta}).

We start with the signature formula and Milnor-Wood type inequality (Theorem \ref{thm0.5})
\begin{align*}
\op{T}(\Sigma,\phi) &=\frac{1}{2}\op{sign}(\mc E)- \frac{1}{2}\sum_{c}\rro(\phi(c))\\
&\leq |\chi(\Sigma)|- \sum_{c} \frac{\rro(\phi(c))}{2}.
\end{align*}
A boundary contribution $\rro(\phi(c))$ is negative when 
\begin{itemize}
  \item either $\phi(c)$ is elliptic with angle $\theta(c)>\pi$,
  \item or $\phi(c)$ is unipotent with negative sign.
\end{itemize}
Changing $\phi(c)$ with $-\phi(c)$ replaces 
\begin{itemize}
  \item in the elliptic case, $\frac{\rro(\phi(c))}{2}=1-\frac{\theta(c)}{\pi}$ with $\frac{\rro(-\phi(c))}{2}=1-\frac{\theta(c)-\pi}{\pi}=\{\frac{\rro(\phi(c))}{2}\}$,
  \item in the unipotent case, $\frac{\rro(\phi(c))}{2}=-\frac{1}{2}$ with $0$, since $-\phi(c)$ is an elliptic-unipotent with angle $\theta=\pi$ and $\frac{\rro(-\phi(c))}{2}=0$.
\end{itemize}
The idea is to modify $\phi$ by replacing an even number of boundary holonomies $\phi(c)$ with $-\phi(c)$. This is compatible with the standard presentation \eqref{relation} of $\pi_1(\Sigma)$. This does not change the Toledo invariant, since both representations define the same action on the symmetric space. This allows us to get rid of all negative boundary contributions but possibly one. Whence the extra term of $1$ in the right-hand side.
\end{proof}

\begin{rem}\label{improve}
Let us consider the set of boundaries \(\{c_\alpha\}_{\alpha\in I}\) where \(\rho(\phi(c_\alpha))>0\). By altering the representations at the maximal even subset of these boundaries from \(\phi(c_\alpha)\) to \(-\phi(c_\alpha)\), we obtain a new representation \(\widetilde{\phi}\). This representation \(\widetilde{\phi}\) retains the same Toledo invariant as the original representation \(\phi\). Following the same argumentation presented in Proposition \ref{propmod2}, we conclude that
\begin{equation*}
\op{T}(\Sigma,\phi)\geq   -|\chi(\Sigma)|-1 + \sum_{\phi(c_k)\text{ is elliptic}}  \frac{\theta_{k}}{\pi},\quad \theta_k\in (0,\pi).
\end{equation*}
Combing with Proposition \ref{propmod2}, we obtain
\begin{equation*}
  -|\chi(\Sigma)|-1 + \sum_{\phi(c_j)\text{ is elliptic}}  \frac{\theta_{j}}{\pi}\leq\op{T}(\Sigma,\phi) \leq |\chi(\Sigma)|+1 - \sum_{\phi(c_k)\text{ is elliptic}}  \left(1-\frac{\theta_{k}}{\pi}\right),
\end{equation*}
where $\theta_k,\theta_j\in (0,\pi)$ such that $[R(\theta_k)]$ is conjugate to $[\phi(c_k)]\in \op{PSL}(2,\mb{R})=\op{SL}(2,\mb{R})/\{\pm I\}$, and $[\bullet]$ denotes the class in $\op{PSL}(2,\mb{R})$.	Note in the summation, if some
$\phi(c)=R(\theta')$ with angle $\theta'>\pi$, we take $\theta=\theta'-\pi$.
\end{rem}

\begin{rem}
	Note that the equalities in \eqref{TI1} can be attained. For example, we consider a cylinder $\Sigma=S^1\times [0,1]$, and the elliptic representation $\phi$ is given by $\phi(S^1\times \{0\})=R(\theta)$ and  $\phi(S^1\times \{1\})=R(2\pi-\theta)$ for some $\theta\in (0,\pi)$. In this case, $\op{T}(\Sigma,\phi)=\chi(\Sigma)=0$, $\theta_1=\theta$ and $\theta_2=\pi-\theta\in (0,\pi)$, and so 
	\begin{equation*}
  -|\chi(\Sigma)|-1+\sum_{k=1}^2\frac{\theta_k}{\pi}=\op{T}(\Sigma,\phi)=|\chi(\Sigma)|+1-\sum_{k=1}^2\left(1-\frac{\theta_k}{\pi}\right).
\end{equation*}
\end{rem}

\begin{rem}
Each element $L=R(\theta)\in \op{Sp}(2,\mb{R}),\theta\in (0,\pi)$	gives an automorphism $L_{\mb{D}}=e^{i(2\pi-2\theta)}$ acting on the unit disc $\mb{D}$. In fact, 
note that $\op{D}^{\op{III}}_1=\mb{D}=\{w\in \mb{C}| |w|<1\}$ is the unit disc in the complex plane, and  
\begin{align*}
L=U\left(\begin{matrix}
e^{-i\theta}	& 0\\
0& e^{i\theta}
\end{matrix}
\right)U^{-1},	
\end{align*}
 where $U$ and $U^{-1}$ are given by \eqref{U}, and by Remark \ref{rem5.2}, so 
 $$L_{\mb{D}}(w)=e^{-i\theta}we^{-i\theta}=e^{-2i\theta}w=e^{i(2\pi-2\theta)}w.$$
 If we denote $\phi_k=2\pi-2\theta_k\in (0,2\pi)$, then $\phi(c_k)_{\mb{D}}=e^{i\phi_k}$. By Proposition \ref{propmod2}, one has
 \begin{equation}\label{TI2}
  -|\chi(\Sigma)|-1 + \sum_{k=1}^q  \left(1-\frac{\phi_k}{2\pi}\right)\leq\op{T}(\Sigma,\phi) \leq |\chi(\Sigma)|+1 - \sum_{k=1}^q  \frac{\phi_k}{2\pi}.
\end{equation}

 If the representation $\phi$ is the holonomy of a cone hyperbolic surface $S$ with cone angles $\phi_k\in (0,2\pi)$,
 one can refer to \cite{Mc, Tro} for the definition of the surfaces with conical singularities,
  then $S$ can be identified with $\mb{D}/\phi(\pi_1(\Sigma))$. To get the cone point of cone angle $0<\varphi\leq 2\pi$, we need to identify the sector of angle $\varphi$ by the rotation of anlge $-(2\pi-\varphi)$, hence we need $R(\theta)$ such that $-(2\pi-\varphi)=-2\theta$. The induced representation on the boundary is conjugate to the rotation $R(\theta_k)\in \op{SL}(2,\mb{R})$ with $\theta_k\in (0,\pi)$, where $1\leq k\leq q$.   
Then the Toledo invariant is exactly  the area of $S$ and can be given by
  \begin{equation*}
  \op{T}(\Sigma,\phi)=\frac{1}{2\pi}\text{Area}(S)=-\left(\chi(\o{\Sigma})+\frac{1}{2\pi}\sum_{k=1}^q(\phi_k-2\pi)\right)=-\chi(\Sigma)-\sum_{k=1}^q\frac{\phi_k}{2\pi},
\end{equation*}
where $\o{\Sigma}\simeq S$ is a closed surface obtained by capping off the boundary of $\Sigma$.
Similarly, by conjugating $\phi$
with an orientation-reversing (anti-holomorphic) isometry $\tau$ of $\mb{D}$, then we obtain a representation $\phi_\tau$ with  $\op{T}(\Sigma,\phi_\tau)=-\op{T}(\Sigma,\phi)$. Hence,
\begin{align*}
	\op{T}(\Sigma,\phi_\tau)=\chi(\o{\Sigma})+\frac{1}{2\pi}\sum_{k=1}^q(\phi_k-2\pi)=\chi(\Sigma) + \sum_{k=1}^q  \frac{\phi_k}{2\pi}.
\end{align*}
\end{rem}

 Recall that $L$ is called parabolic if all eigenvalues of $L$ are $\pm 1$. The reason is as follows.
\begin{prop}
	If $L$ is parabolic, then it fixes a point at the Shilov boundary of $\op{D}^{\op{III}}_n$.
\end{prop}
\begin{proof}
From \cite[Theorem 1]{Gutt}, there exists a symplectic basis such that $L$ is  symplectic direct sum of matrices of the form
$$\mbf{L}|_{\mb{R}^{2r_j}}=\left(\begin{array}{cc}J\left(\lambda, r_{j}\right)^{-1} & C\left(r_{j}, s_{j}, \lambda\right) \\ 0 & J\left(\lambda, r_{j}\right)^{\top}\end{array}\right)\in \op{Sp}(2r_j,\mb{R})$$
where $C\left(r_{j}, s_{j}, \lambda\right):=J\left(\lambda, r_{j}\right)^{-1} \operatorname{diag}\left(0, \ldots, 0, s_{j}\right)$ with $s_{j} \in\{0,1,-1\}$,  $J(\lambda, r)$ is the elementary $r \times r$ Jordan matrix associated to $\lambda$.  By Remark \ref{rempotential}, $-I_{r_j}$ is a fixed point of $L|_{\mb{R}^{2r_j}}$. Hence $-I_n$ is a fixed point of $L$. With respect to  the other basis, the matrix of $L$ is $P\mbf{L}P^{-1}$
for
 some matrix $P\in\op{Sp}(2n,\mb{R})$. Hence $P(-I_n)$ is a fixed point of $L$, which is also at the Shilov boundary of $\op{D}^{\op{III}}_n$ since  the  Shilov boundary  is an orbit of the action of $\op{Sp}(2n,\mb{R})$ on   $\op{D}^{\op{III}}_n$.                  
\end{proof}
\begin{prop}\label{propmod3} Suppose $\Sigma$ is not a cylinder.
If there exists a boundary component $c$ such that the representation $\phi(c)$	has an eigenvalue $1$, then 
$|\op{sign}(\mc{E},\Omega)|<\dim E\cdot |\chi(\Sigma)|$. 
\end{prop}
\begin{proof}
For $g\geq 1$, from (\ref{siginequality}), one has 
	\begin{align*}
\begin{split}
\pm\op{sign}(\mc{E},\Omega)&\leq -\dim E\chi(\Sigma)-\dim \mathrm{H}^0(\p\Sigma,\mc{E}).
\end{split}
\end{align*}
If  there exists a boundary component $c$ such that the representation $\phi(c)$ has an eigenvalue $1$, then $\phi(c)$ fixes a nonzero vector in $E$, and so $\dim \mathrm{H}^0(\p\Sigma,\mc{E})\geq 1$. Hence 
\begin{equation}\label{signineq0}
  \pm\op{sign}(\mc{E},\Omega)\leq-\dim E\chi(\Sigma)-1<\dim E\cdot|\chi(\Sigma)|,
\end{equation}
which completes the proof for $g\geq 1$.

If $g=0$ and the number of boundary components $n\leq 1$,  $\pi_1(\Sigma)$ is trivial and the signature vanishes. For $g=0, n \geq 3$, $$\text{dim}\ \mathrm{H}^0(\partial\Sigma,\cal E)\geq 1+\text{dim}\ \mathrm{H}^0(\partial\Sigma\setminus \{c\},\cal E)\geq 1+2 \text{dim}\ \mathrm{H}^0(\Sigma,\cal E),$$ hence by (\ref{5.6}) $$\pm \text{sign}(\cal E,\Omega)\leq -\text{dim}\ E \chi(\Sigma)-1-\text{dim}\ \mathrm{H}^0(\Sigma,\cal E)< \text{dim}\ E\cdot |\chi(\Sigma)|.$$
\end{proof}

\begin{ex}
While for a surface with one boundary component $c$ satisfies $\phi(c)$ has an eigenvalue $-1$, then the signature $\op{sign}(\mc{E},\Omega)$ may attain the maximal $\dim E\cdot|\chi(\Sigma)|$. For example, we consider a surface $\Sigma_3$ with boundary components $p_1,p_2,p_3$, which is homeomorphic to a surface a $2$-sphere with $3$ discs deleted, and consider a representation $\phi:\pi_1(\Sigma_3)\to \op{SO}(2)\subset\op{Sp}(2,\mb{R})$ such that
$$\phi(p_1)=-I_2,\quad \phi(p_2)=R(\theta),\quad \phi(p_3)=R(\pi-\theta)$$
for some $\theta\in (0,\pi)$. From Remark \ref{remunitary} and \eqref{etadim2}, one has
$$\op{T}(\Sigma_3,\phi)=0,\quad \rro_\phi(\partial \Sigma_3)=2(1-\frac{\theta}{\pi})+2(1-\frac{\pi-\theta}{\pi})=2,$$
which follows that 
$$\op{sign}(\mc{E},\Omega)=2\op{T}(\Sigma_3,\phi)+ \rro_\phi(\partial \Sigma_3)=2=\dim E\cdot|\chi(\Sigma_3)|$$
since $\dim E=2$ and $\chi(\Sigma_3)=-(2\cdot 0-2+3)=-1$.

\end{ex}

\section{Surface group representations in $\op{SO}_0(n,2)$}\label{SO0}

In this section, we will consider the surface group representations in $\op{SO}_0(n,2)$, one can refer to \cite[Page 75-78, Section 2.5]{Mok} for the group $\op{SO}_0(n,2)$ and the bounded symmetric domain $\op{D}^{\op{IV}}_n$ of type $\op{IV}$. 

Let $(x^1,\cdots, x^{n+2})$ be the  Euclidean coordinates on $\mb{R}^{n+2}$  with respect to the standard basis $\{e_i\}_{1\leq i\leq n+2}$. Let $\Omega$ be an indefinite quadratic form of signature $(n,2)$ defined by 
\begin{align*}
\begin{split}
  \Omega(x,x)=(x^1)^2+\cdots (x^n)^2-(x^{n+1})^2-(x^{n+2})^2.
 \end{split}
\end{align*}
By complexification, denote $E=\mb{R}^{n+2}\otimes\mb{C}=\mb{C}^{n+2}$ and the quadratic form $\Omega$ can be $\mb{C}$-linearly extended to $E$. Consider on $\mb{C}^{n+2}$ the space of complex lines $L$ such that $\Omega|_L\equiv 0$, this space can be identified with the hyperquadric $Q^n\subset \mb{P}^{n+1}$ defined by the homogeneous equation 
$(w^1)^2+\cdots+(w^n)^2-(w^{n+1})^2-(w^{n+2})^2=0$, where $w^j=x^j+iy^j, 1\leq j\leq n+2$ denote the complex coordinates of $\mb{C}^{n+2}$. Denote $H(\cdot,\cdot)=\Omega(\cdot,\o{\cdot})$, then $H$ is a Hermitian form on $\mb{C}^{n+2}$. Define
\begin{align*}
\begin{split}
  D_0:=\left\{L\in Q^n: H|_L<0\right\}.
 \end{split}
\end{align*}
The condition $L\in D_0$ means
\begin{align}\label{condition of D0}
\begin{split}
   \begin{cases}
 	& \sum_{1\leq i\leq n}(w^i)^2-(w^{n+1})^2-(w^{n+2})^2=0, \\
 	&\sum_{1\leq i\leq n}|w^i|^2<|w^{n+1}|^2+|w^{n+2}|^2.
 \end{cases}
 \end{split}
\end{align}
By the  transformation of coordinates  $w=Uz$, i.e. 
\begin{align}\label{DIVU}
\begin{split}
  \left(\begin{matrix}
  w^1\\
  w^2\\
  \vdots\\
  w^n\\
  w^{n+1}\\
  w^{n+2}
\end{matrix}\right)=U\left(\begin{matrix}
  z^1\\
  z^2\\
  \vdots\\
  z^n\\
  z^{n+1}\\
  z^{n+2}
\end{matrix}\right), \quad\text{where } U=\left(\begin{matrix}
  I_n&0 & 0 \\
  0& \frac{1}{\sqrt{2}}& \frac{1}{\sqrt{2}}\\
  0&-\frac{i}{\sqrt{2}} & \frac{i}{\sqrt{2}}
\end{matrix}\right),
 \end{split}
\end{align}
then the condition \eqref{condition of D0} is reduced to 
\begin{align}\label{condition1}
\begin{split}
   \begin{cases}
 	& \sum_{1\leq i\leq n}(z^i)^2-2z^{n+1}z^{n+2}=0, \\
 	&\sum_{1\leq i\leq n}|z^i|^2<|z^{n+1}|^2+|z^{n+2}|^2.
 \end{cases}
 \end{split}
\end{align}
On the subset of $Q^n$ defined by $z^{n+1}\neq 0$, we can identify $(z^1,\cdots, z^n)\in \mb{C}^n$ with the point $[z^1,\cdots, z^n,1, \frac{1}{2}\sum_{1\leq i\leq n}(z^i)^2]$. The condition \eqref{condition1} implies that $|z^{n+2}|\neq 1$, see \cite[Page 76]{Mok}. Then the bounded symmetric domain $\op{D}^{\op{IV}}_n$ of type $\op{IV}$ is defined as the connected component containing the point $[0,\cdots,0,1,0]\in Q^n$. Hence 
\begin{align*}
\begin{split}
  \op{D}^{\op{IV}}_n=\left\{z=(z^1,\cdots, z^{n})^\top\in \mb{C}^n:\|z\|^2<2\text{ and } \|z\|^2<1+\left|\frac{1}{2}\sum_{i=1}^n (z^i)^2\right|^2\right\}.
 \end{split}
\end{align*}
Let $\op{SO}(n,2)$ be the real group acting on $\mb{R}^{n+2}$ and preserving the quadratic form $\Omega$, which induces an action on $Q^n$. Denote by $G_o=\op{SO}_0(n,2)$ the identity component of $\op{SO}(n,2)$. Then the group $G_o$ acts transitively on $ \op{D}^{\op{IV}}_n$ with the isotropy group $K$  isomorphic to $\op{SO}(n)\times\op{SO}(2)$. Hence 
\begin{align*}
\begin{split}
   \op{D}^{\op{IV}}_n\cong \op{SO}_0(n,2)/(\op{SO}(n)\times\op{SO}(2)).
 \end{split}
\end{align*}
For any $L\in \op{SO}_0(n,2)$, and for any $z\in \op{D}^{\op{IV}}_n$, then $$z_L=U^{-1}LU\left(\begin{matrix}
  z \\
  1\\
  \frac{1}{2} \sum_{i=1}^n(z^i)^2
\end{matrix}\right)$$
is a column vector in $\mb{C}^{n+2}$ with $z_L^{n+1}\neq 0$. Then the action of the group $G_o=\op{SO}_0(n,2)$ on $\op{D}^{\op{IV}}_n$ is given by 
\begin{align}\label{action of L}
\begin{split}
  L(z)=\left(\frac{z_L^1}{z_L^{n+1}},\cdots, \frac{z_L^n}{z_L^{n+1}}\right)^\top\in  \op{D}^{\op{IV}}_n.
 \end{split}
\end{align}
Let $\Sigma$ be a surface with boundary, and for any representation 
$\phi:\pi_1(\Sigma)\to \op{SO}_0(n,2)$, denote $\mc{E}=\wt{\Sigma}\times_\phi E=\wt{\Sigma}\times_\phi \mb{C}^{n+2}$. Then the following form 
\begin{align*}
\begin{split}
  Q_{\mb{C}}(\cdot,\cdot)=i\int_{\Sigma}H(\cdot\cup\cdot)=i\int_\Sigma\Omega(\cdot\cup\o{\cdot}),
 \end{split}
\end{align*}
is a non-degenerate Hermitian form on the space $\widehat{\mr{H}}^1(\Sigma,\mc{E})$. We define the signature $\op{sign}(\mc{E},\Omega)$ of flat vector bundle $(\mc{E},\Omega)$ associated with the representation $\phi:\pi_1(\Sigma)\to \op{SO}_0(n,2)$ as the signature of the  form $Q_{\mb{C}}(\cdot,\cdot)$. 

Note that the form $\int_\Sigma\Omega(\cdot\cup\o{\cdot})$ is real (i.e. $\o{\int_\Sigma\Omega([a]\cup\o{[b]})}=\int_{\Sigma}\Omega(\o{[a]}\cup[b])$), non-degenerate and skew-symmetric (i.e. ${\int_\Sigma\Omega([a]\cup\o{[b]})}=-\int_{\Sigma}\Omega(\o{[b]}\cup[a])$), so its eigenvalues have the form
  $$\pm\lambda_1,\pm\lambda_2,\cdots,\pm\lambda_N,$$ 
where $N=\frac{1}{2}\dim \widehat{\mr{H}}^1(\Sigma,\mc{E})$  and each $\lambda_i$ is  purely imaginary and nonzero, which follows that the numbers of positive and negative eigenvalues of the Hermitian form $Q_{\mb{C}}$ are equal. Hence 
\begin{equation}\label{complexsingature}
  \op{sign}(\mc{E},\Omega)=0.
\end{equation}

On the other hand, by using Atiyah-Patodi-Singer index  theorem, we can also give a precise formula for the signature $ \op{sign}(\mc{E},\Omega)$. 
 Denote by  $\mc{J}(E_{\mb{R}},\Omega)$ the subspace of  $\op{SO}_0(n,2)$ such that $J^2=\op{Id}$ and $\Omega(\cdot,J\o{\cdot})>0$. Set $\mc{J}(\mc{E}_{\mb{R}},\Omega)=C^{\infty}(\Sigma,\wt{\Sigma}\times_\phi\mc{J}(E_{\mb{R}},\Omega))$, $\mc{E}_{\mb{R}}:=\wt{\Sigma}\times_\phi E_{\mb{R}}$ where $E_{\mb{R}}=\mb{R}^{n+2}$. There is a canonical action of $\op{SO}_0(n,2)$ on 
$\mc{J}(E_{\mb{R}},\Omega)$ by $Z(J)=ZJZ^{-1}$ for any $Z\in \op{SO}_0(n,2)$, the action is transitive, and the isotropy group at the point $I_{n,2}\in \mc{J}(E_{\mb{R}},\Omega)$ is exactly $\op{SO}(n)\times \op{SO}(2)$. Hence 
\begin{align*}
\begin{split}
   \op{D}^{\op{IV}}_n\cong \op{SO}_0(n,2)/(\op{SO}(n)\times\op{SO}(2))\cong \mc{J}(E_{\mb{R}},\Omega).
 \end{split}
\end{align*}
For any $\mbf{J}\in \mc{J}(\mc{E}_{\mb{R}},\Omega)$, let $\mc{E}=\mc{E}^+\oplus\mc{E}^-$ be the decomposition of $\mc{E}$ corresponding to the $\pm i$-eigenspace of $i\mbf{J}$. 
With respect to $i\mbf{J}$, we call $\n$ is a peripheral connection on $\mc{E}$ if $\n$ is a real connection (i.e. $\n=\o{\n}$) and satisfies  the following conditions on a collar neighborhood of $\p\Sigma$:
	\begin{itemize}
\item[(i)] 	$\n=d+C(x)dx$ for some $C=C(x)\in A^0(\p\Sigma,\op{End}(\mc{E}))$;
\item[(ii)] $[\n,\mbf{J}]=0$;
\item[(iii)] $\n$ preserves the quadratic form $\Omega$.
\end{itemize}
Similar to Section \ref{Sig} and Theorem \ref{thmsign2}, we obtain 
\begin{align}\label{DIVsign}
\begin{split}
\op{sign}(\mc{E},\Omega)= 2\int_\Sigma\left(c_1(\mc{E}^+,\n^{\mc{E}^+})-c_1(\mc{E}^-,\n^{\mc{E}^-})\right)+\eta(A_{i\mbf{J}}),
 \end{split}
\end{align}
where $\n^{\mc{E}^+}=\n|_{\mc{E}^+}$, $\n^{\mc{E}^-}=\n|_{\mc{E}^-}$ and $\n$ is a peripheral connection on $\mc{E}$. 

For any $z\in  \op{D}^{\op{IV}}_n$, since the group $\op{SO}_0(n,2)$ acts transitively on $\op{D}^{\op{IV}}_n$, so there exists  $L\in \op{SO}_0(n,2)$ such that $L(0)=z$, where $L(0)$ is defined by \eqref{action of L} and $0$ denotes the origin in $\mb{C}^n$. We define 
\begin{align*}
\begin{split}
  \mbf{J}_{\mr{IV}}(z):=LI_{n,2}L^{-1}\in \mc{J}(E_{\mb{R}},\Omega).
 \end{split}
\end{align*}
In fact, the definition for $  \mbf{J}_{\mr{IV}}(z)$ is well-defined, for another $L'\in \op{SO}_0(n,2)$ with $L'(0)=z$, then $L^{-1}L'\in \op{SO}(n)\times \op{SO}(2)$, which follows that 
$L'I_{n,2}L'^{{-1}}=L(L^{-1}L'I_{n,2}L'^{-1}L)L^{-1}=LI_{n,2}L^{-1}$, so  $\mbf{J}_{\mr{IV}}(z)$ is independent of the choice of $L\in \op{SO}_0(n,2)$ with $L(0)=z$. One can check that $\mbf{J}_{\mr{IV}}=\mbf{J}_{\mr{IV}}(z):\op{D}^{\op{IV}}_n\to\mc{J}(E_{\mb{R}},\Omega)$ is an isomorphism. For any $z\in  \op{D}^{\op{IV}}_n$, we can take 
\begin{align*}
\begin{split}
  L=U\cdot \frac{1}{a}V\cdot U^{-1},\quad\text{where }V=\left(\begin{matrix}
  A& z& \o{z}\\
  \o{z}^\top& 1 & \frac{1}{2}\o{z^\top z}\\
  z^\top & \frac{1}{2}z^\top z & 1 
\end{matrix}\right) 
 \end{split}
\end{align*}
where $U$ is given by \eqref{DIVU}, 
$ A=a I_n+bz\o{z}^\top+b\o{z}z^\top+czz^\top+\o{c}\o{z}\o{z}^\top$
and 
\begin{align*}
\begin{split}
  a=\sqrt{1+|\frac{1}{2}z^\top z|^2-\|z\|^2},\quad b=\frac{a+1}{2(1+a-\frac{1}{2}\|z\|^2)},\quad c=-\frac{\o{z^\top z}}{4(1+a-\frac{1}{2}\|z\|^2)}.
 \end{split}
\end{align*}
One can check that $L$ is real and $L^\top I_{n,2}L=I_{n,2}$, so $L\in \op{SO}_0(n,2)$. Moreover, $L(0)=z$. Hence 
\begin{align*}
\begin{split}
  \mbf{J}_{\mr{IV}}(z)=UV I_{n,2}V ^{-1}U^{-1}.
 \end{split}
\end{align*}
Now we define a connection on the trivial bundle $F=\mr{D}^{\mr{IV}}_n\times \mb{C}^{n+2}$ by 
\begin{align*}
\begin{split}
  \n&=UV\left((d+\theta)\cdot I_{n+2}\right)V^{-1}U^{-1}\\
  &=d+UV(dV^{-1})U^{-1}+\theta I_{n+2}\\
  &=d+(UVU^{-1})d(UVU^{-1})^{-1}+\theta I_{n+2},
 \end{split}
\end{align*}
where $\theta=\frac{1}{2}d\log a^2$.
Then $\n$ is real and $[\n,\mbf{J}_{\mr{IV}}]=0$. Denote $\n=d+C$ where 
\begin{align*}
\begin{split}
  C=UV(dV^{-1})U^{-1}+\theta I_{n+2}.
 \end{split}
\end{align*}
Then 
\begin{align*}
\begin{split}
 {C}^\top I_{n,2}+I_{n,2}C&= \o{C}^\top I_{n,2}+I_{n,2}C\\
 &=UdV^{-1}VU^{-1}I_{n,2}+I_{n,2}(UV(dV^{-1})U^{-1})+2\theta I_{n,2}\\
 &=-UV^{-1}(dV I_{n,2} V+V I_{n,2} dV)V^{-1}U^{-1}+2\theta I_{n,2}\\
 &=-UV^{-1}d(VI_{n,2}V)V^{-1}U^{-1}+2\theta I_{n,2}.
 \end{split}
\end{align*}
By a direct calculation, one has $VI_{n,2}V=a^2I_{n,2}$. Hence 
\begin{align*}
\begin{split}
   {C}^\top I_{n,2}+I_{n,2}C&=-UV^{-1}da^2 I_{n,2}V^{-1}U^{-1}+2\theta I_{n,2}\\
   &=-d\log a^2 UI_{n,2}U^{-1}+2\theta I_{n,2}\\
   &=-d\log a^2I_{n,2}+2\theta I_{n,2}=0.
 \end{split}
\end{align*}
Thus, $\n$ preserves the quadratic form $\Omega$. Similar to Section \ref{RFCC}, one can define the vector bundles $F$, $F_{\phi}$, $F^{\pm}$, $F_{\phi}^{\pm}$. Denote $\n^{F^\pm}:=\n|_{F^{\pm}}$. Then the curvature of $\n^{F^\pm}$ vanishes, and
$
  c_1(F_{\phi}^+,\n^{F_{\phi}}|_{F_{\phi}^+})=0.
$
 Let 
\begin{align*}
	\begin{split}
	\mc{J}_{o}(\mc{E}_{\mb{R}},\Omega)=\{\mbf{J}\in \mc{J}(\mc{E}_{\mb{R}},\Omega)|  &\mbf{J}=p^*J
	\text{ on a small collar neighborhood}\\
	&\text{ of } \p\Sigma, \text{ where } J \in \mc{J}(\mc{E}_{\mb{R}}|_{\p\Sigma},\Omega)\},
	\end{split}
	\end{align*}
For any $\mbf{J}\in \mc{J}_{o}(\mc{E}_{\mb{R}},\Omega)$, then $\tau \wt{\mbf{J}}^*\n^{F_\phi}\tau^{-1}$ is a peripheral connection and
\begin{align}\label{pullbackc1}
\begin{split}
  c_{1}\left(\mathcal{E}^{+}, \tau \wt{\mbf{J}}^* \nabla^{F_{\phi}} \tau^{-1}|_{\mathcal{E}^{+}}\right)&=c_{1}\left(\wt{\mbf{J}}^* F_{\phi}^{+},\left.\wt{\mbf{J}}^* \nabla^{F_{\phi}}\right|_{\wt{\mbf{J}}^* F_{\phi}^{+}}\right)\\
  &=\wt{\mbf{J}}^* c_{1}\left(F_{\phi}^{+},\left.\nabla^{F_{\phi}}\right|_{F_{\phi}^{+}}\right)=\widetilde{\mathbf{J}}^{*} c_{1}\left(F^{+}, \nabla|_{F^{+}}\right).
 \end{split}
\end{align}
Here $\tau:\wt{\mbf{J}}^*F_\phi\to \mc{E}$ is an isomorphism, which is given by \eqref{tau isomorphism}. Set
$$(f_1,\cdots,f_{n+2})=(e_1,\cdots,e_{n+2})UV,$$
where $\{e_1,\cdots,e_{n+2}\}$ denotes the standard basis of $\mb{C}^{n+2}$. With respect to the basis $\{f_1,\cdots, f_{n+2}\}$, the matrix of $\mbf{J}$ is $I_{n,2}$, so that $\{f_1,\cdots, f_n\}$ forms a basis of $F^+$, while $\{f_{n+1},f_{n+2}\}$ is a basis of $F^-$. From the definition of $\nabla$, then 
\begin{align*}
\begin{split}
  \n(e_1,\cdots,e_{n+2})=(e_1,\cdots,e_{n+2})(UV(dV^{-1})U^{-1}+\theta I_{n+2}).
 \end{split}
\end{align*}
Hence 
\begin{align*}
\begin{split}
  \n(f_1,\cdots,f_{n+2})&=\n ((e_1,\cdots, e_{n+2})UV)\\
  &=(e_1,\cdots, e_{n+2})(UV(dV^{-1})U^{-1}+\theta I_{n+2})UV+(e_1,\cdots, e_{n+2})UdV\\
  &=(e_1,\cdots, e_{n+2})UV\cdot \theta=(f_1,\cdots, f_{n+2})\theta,
 \end{split}
\end{align*}
which means $\n=d+\theta \cdot \op{Id}_{F}$ with respect to the basis $\{f_1,\cdots, f_{n+2}\}$. Thus, with respect to the frame $\{f_1,\cdots, f_{n}\}$, the connection $\n|_{F^+}$ is given by 
\begin{align*}
\begin{split}
  \n|_{F^+}=d+\theta\cdot\mr{Id}_{F^+},
 \end{split}
\end{align*}
whose curvature is  
$$(d\theta+\theta\wedge \theta)\mr{Id}_{F^+}=0,$$
which follows that $c_{1}\left(F^{+}, \nabla|_{F^{+}}\right)=0$, and one has by \eqref{pullbackc1} 
\begin{align*}
\begin{split}
  c_1(\mc{E}^+,\tau \wt{\mbf{J}}^*\n^{F_\phi}\tau^{-1}|_{\mc{E}^+})=0.
 \end{split}
\end{align*}
Similarly, $ c_1(\mc{E}^-,\tau \wt{\mbf{J}}^*\n^{F_\phi}\tau^{-1}|_{\mc{E}^-})=0$. By \eqref{complexsingature} and \eqref{DIVsign}, one has 
\begin{align}\label{etazero}
\begin{split}
 \eta(A_{i\mbf{J}})=\op{sign}(\mc{E},\Omega)=0
 \end{split}
\end{align}
for any $\mbf{J}\in \mc{J}_{o}(\mc{E}_{\mb{R}},\Omega)$.
\begin{ex}
For the two groups $\mr{SL}(2,\mb{R})$ and $\mr{SO}_0(1,2)=\{(a_{ij})\in \op{SO}(1,2):a_{11}>0\}$, we have the following canonical map 
\begin{align*}
\begin{split}
  \Psi: \mr{SL}(2,\mathbb{R})\to\mr{SO}_0(1,2)
 \end{split}
\end{align*}
	\begin{align*}
\begin{split}
  \Psi\begin{pmatrix}
	a & b\\
	c& d
\end{pmatrix}=\left(\begin{matrix}
  \frac{1}{2}(a^2+b^2+c^2+d^2)& \frac{1}{2}(a^2-b^2+c^2-d^2)&-ab-cd \\
 \frac{1}{2}(a^2+b^2-c^2-d^2) & \frac{1}{2}(a^2-b^2-c^2+d^2)& cd-ab\\
 -ac-bd & bd-ac & ad+bc
\end{matrix}\right),
 \end{split}
\end{align*}
where $\begin{pmatrix}
	a & b\\
	c& d
\end{pmatrix}\in \mr{SL}(2,\mb{R})$, $ad-bc=1$.  The mapping $\Psi$ is a double covering  with kernel $\pm I_2$. It is well-known that there is a classification for the group $\mr{SL}(2,\mb{R})$, i.e. hyperbolic, elliptic and parabolic elements, see e.g.  Section \ref{Appeta}. By using the mapping $\Psi$, one can give a classification of $\mr{SO}_0(1,2)$  as follows:
\begin{itemize}
\item[(1)] $L$ is hyperbolic, i.e. $\mr{Tr}(L)>3$. In this case,  $\lambda\not\in S^1$and $\lambda\in \mb{R}$, $L$ has the normal form: $$
\Psi\left(\begin{array}{cc}
\lambda & 0 \\
0 & \frac{1}{\lambda}
\end{array}\right)=\left(\begin{matrix}
\
  \cosh\theta&\sinh\theta &0 \\
  \sinh\theta&\cosh\theta  &0\\
  0 &0 &1
\end{matrix}\right),\quad \theta=\log \lambda^2.
$$	
\item[(2)] $L$ is elliptic, i.e. $\mr{Tr}(L)\in (-1,3)$. In this case,  $L$ has the normal form
$$
\Psi\left(\begin{array}{cc}\cos\theta_1 & -\sin\theta_1 \\
\sin\theta_1 & \cos\theta_1
\end{array}\right)=\begin{pmatrix}
	1 & 0 &0\\
	0 & \cos(2\theta_1) & \sin(2\theta_1)\\
	0 & -\sin(2\theta_1) & \cos(2\theta_1)
\end{pmatrix},
$$	
where $\theta_1\in (0,\pi)\cup(\pi,2\pi)$.
\item[(3)] $L$ is parabolic, i.e. $\mr{Tr}(L)=3$. $L$ has the normal form: $$
\Psi\left(\begin{array}{ll}
\lambda & \mu \\
0 & \lambda
\end{array}\right)=\begin{pmatrix}
	\frac{\mu^2}{2}+1 & -\frac{\mu^2}{2} &-\lambda\mu\\
	\frac{\mu^2}{2}& 1-\frac{\mu^2}{2} &-\lambda\mu\\
	-\lambda\mu &\lambda\mu &1
\end{pmatrix},
$$	
where $\mu\in \mb{R}$, $\lambda=\pm 1$.
\end{itemize}
For any representation $\phi:\pi_1(S^1)\to \mr{SO}_0(1,2)$, we assume $L=\phi(S^1)\in \mr{SO}_0(1,2)$. 
If $L=\exp(2\pi B)$, then we can take 
\begin{align*}
\begin{split}
  \mbf{J}=\exp(-xB)I_{1,2}\exp(xB).
 \end{split}
\end{align*}
Similar to Section \ref{Appeta}, the eigenvalues (with multiplicities)  of $A_{i\mbf{J}}$ are given by the following equation:
	\begin{align*}
\begin{split}
  \exp(2\pi(-\sigma iI_{1,2}+B))e=e.
 \end{split}
\end{align*}

(1)  $L$ is hyperbolic, one can take $B$ as
$$B=\left(\begin{matrix}
  0&\frac{\theta}{2\pi}&0 \\
  \frac{\theta}{2\pi}&0&0\\
  0&0&0 
\end{matrix}\right).$$
Then the set of all eigenvalues (with multiplicities) of $A_{i\mbf{J}}$ is 
\begin{align*}
\begin{split}
  \left\{k, \pm\sqrt{k^2+(\frac{\theta}{2\pi})^2},k\in \mb{Z}\right\},
 \end{split}
\end{align*}
which is symmetric.
Hence $\eta(A_{i\mbf{J}})=0$.

(2) $L$ is elliptic, one can take $B$ as
\begin{align*}
\begin{split}
  B=\left(\begin{matrix}
  0&0 &0 \\
  0&0& \frac{\theta_1}{\pi}\\
  0&-\frac{\theta_1}{\pi}&0
\end{matrix}\right).
 \end{split}
\end{align*}
Then the set of eigenvalues (with multiplicities) of $A_{i\mbf{J}}$ is 
\begin{align*}
\begin{split}
  \left\{k, k\pm\frac{\theta_1}{\pi},k\in \mb{Z}\right\},
 \end{split}
\end{align*}
which is symmetric.
Hence $\eta(A_{i\mbf{J}})=0$.
 
 (3) $L$ is parabolic, one can take $B$ as
 \begin{align*}
\begin{split}
B=  \left(\begin{matrix}
  0&0 &-\frac{\lambda\mu}{2\pi}\\
  0& 0&-\frac{\lambda\mu}{2\pi}\\
  -\frac{\lambda\mu}{2\pi}&\frac{\lambda\mu}{2\pi}&0
\end{matrix}\right).
 \end{split}
\end{align*}
Then the set of eigenvalues (with multiplicities) of $A_{i\mbf{J}}$ is given by solving the equation 
$$\det(-\sigma i I_{1,2}+B+k iI_3)=0,\quad k\in \mb{Z},$$
see Remark \ref{remdetspectrum},
which is also equivalent to solve
\begin{align}\label{detequation}
\begin{split}
  \sigma^3+k\sigma^2-(k^2+2(\frac{\mu}{2\pi})^2)\sigma-k^3=0,\quad k\in\mb{Z}.
 \end{split}
\end{align}
If $\sigma_0$ is a solution of the above equation \eqref{detequation} with respect to $k=k_0\in\mb{Z}$, i.e. 
\begin{align*}
\begin{split}
  \sigma_0^3+k_0\sigma_0^2-(k_0^2+2(\frac{\mu}{2\pi})^2)\sigma_0-k_0^3=0,
 \end{split}
\end{align*}
which is also equivalent to 
\begin{align*}
\begin{split}
  (-\sigma_0)^3+(-k_0)(-\sigma_0)^2-((-k_0)^2+2(\frac{\mu}{2\pi})^2)(-\sigma_0)-(-k_0)^3=0,
 \end{split}
\end{align*}
which means that $-\sigma_0$ is  a solution of \eqref{detequation} with respect to $k=-k_0\in\mb{Z}$.
Hence the set   $ \mr{Eigen}\left(A_{i\mbf{J}}\right)$ of eigenvalues with  multiplicities
of the operator $A_{i\mbf{J}}$ is symmetric, which follows tha
$\eta(A_{i\mbf{J}})=0$.

\end{ex}

\vspace{5mm}

Denote by $\omega_{\op{D}^{\op{IV}}_n}$ the K\"ahler metric with the minimal holomorphic sectional curvature is $-1$, then 
\begin{align*}
\begin{split}
  \omega_{\op{D}^{\op{IV}}_n}=-2i\p\b{\p}\log \left(1+\left|\frac{1}{2}\sum_{i=1}^n (z^i)^2\right|^2-\sum_{i=1}^n\left|z^i\right|^2\right),
 \end{split}
\end{align*}
see e.g. \cite[Page 87]{Mok}. Similar to Section \ref{Tol} and \eqref{Toledo invariant compact form}, the Toledo invariant $\op{T}(\Sigma,\phi)$ can be given by 
\begin{align*}
\begin{split}
 \mr{T}(\Sigma,\phi) &=\frac{1}{2\pi}\int_\Sigma\left(\wt{\mbf{J}}^*\omega_{\op{D}^{\op{IV}}_{n}}-\sum_{i=1}^qd(\chi_i \wt{\mbf{J}}^*\alpha_i)\right)
 \end{split}
\end{align*}
for any $\mbf{J}\in \mc{J}_{o}(\mc{E}_{\mb{R}},\Omega)$, where $q$ denotes the number of connect components of $\p\Sigma$, $\wt{\mbf{J}}:\wt{\Sigma}\to  \op{D}^{\op{IV}}_n(\cong \mc{J}(E_{\mb{R}},\Omega))$ is the $\phi$-equivariant map given by $\mbf{J}$,
 $\alpha_i=d^c\psi_i$, $\psi_i$ is a $\phi(c_i)$-invariant (up to a constant)  K\"ahler potnetial with $d\alpha_i=\omega_{\op{D}^{\op{IV}}_n}$. In fact, for any $z_0\in \o{\op{D}^{\op{IV}}_n}$, the isotropy group of $z_0$ is  
 $K_{z_0}:=\{L\in \mr{SO}_0(n,2):L(z_0)=z_0\}$, then the $K_{z_0}$-invariant (up to a constant) K\"ahler potential $\psi_{z_0}$ with $dd^c\psi_{z_0}=\omega_{\op{D}^{\op{IV}}_n}$ can be given by
 \begin{align*}
\begin{split}
  \psi_{z_0}(z)=-\log\left(\left|-z_0^*z+1+\frac{1}{4}\o{z_0^\top z_0}\cdot z^\top z\right|^2(1+|\frac{1}{2}z^\top z|^2-\|z\|^2)\right).
 \end{split}
\end{align*}

\section{Appendix}

In this section, we will calculate the eta invariant and rho invariant for the group $\op{U}(1,1)$, and explain the classification of nilpotent conjugacy classes in $\op{U}(p,q)$. In the last subsection, we will give a geometric proof of the Milnor-Wood inequality for the bounded cohomology Toledo invariant.

\subsection{The eta invariant and the rho invariant for the group $\op{U}(1,1)$}\label{Appeta}

Every $L\in \op{U}(1,1)$ has the form 
\begin{align*}
\begin{split}
  L=e^{i\theta}L_1
 \end{split}
\end{align*}
where $\theta\in [0,2\pi)$, and  
\begin{align*}
\begin{split}
L_1:=  \left(\begin{matrix}
  a&b \\
  \b{b}&\b{a} 
\end{matrix}\right)\in \op{SU}(1,1),
 \end{split}
\end{align*}
i.e. $|a|^2-|b|^2=1$.
It is known that the group $\mr{SU}(1,1)$ is isomorphic to the special linear group $\op{SL}(2,\mb{R})$. More precisely, the isomorphism is given by
\begin{align*}
\begin{split}
  \Phi: \op{SU}(1,1)\to \op{SL}(2,\mb{R}),\quad 
   \end{split}\end{align*}
\begin{align*}
\begin{split}
\Phi\left(   \left(\begin{matrix}
  a&b \\
  \b{b}&\b{a} 
\end{matrix}\right)\right)=U \left(\begin{matrix}
  a&b \\
  \b{b}&\b{a} 
\end{matrix}\right)U^{-1}=\op{Re}\left(\begin{matrix}
a-b  & -ia-ib \\
  ia-ib&a+b 
\end{matrix}\right),
 \end{split}
\end{align*}
where 
\begin{align}\label{U}
\begin{split}
  U= \frac{1}{\sqrt{2}}\left(\begin{matrix}
  -i& i\\
  1& 1
\end{matrix}\right),\quad U^{-1}=\frac{1}{\sqrt{2}}\left(\begin{matrix}
  i& 1\\
  -i& 1
\end{matrix}\right).
 \end{split}
\end{align}

Hence, one can give the definitions of hyperbolic, parabolic and elliptic elements in $\op{SU}(1,1)$ through the isomorphism $\Phi$ and the classification of $\op{SL}(2,\mb{R})$. More precisely,  for any $L_1=\left(\begin{matrix}
  a&b \\
  \b{b}&\b{a} 
\end{matrix}\right)\in \op{SU}(1,1)$, it is called
\begin{itemize}
  \item[(i)] hyperbolic if $|\op{Re}(a)|>1$;
  \item[(ii)]elliptic if $|\op{Re}(a)|<1$;
  \item[(iii)]  parabolic if $\op{Re}(a)=\pm1$.
\end{itemize}
In this section, we always assume 
\begin{align*}
\begin{split}
  J=\left(\begin{matrix}
  i&0 \\
  0&-i 
\end{matrix}\right).
 \end{split}
\end{align*}
Note that 
\begin{align*}
\begin{split}
  \Phi(J)=UJU^{-1}=\left(\begin{matrix}
  0&1 \\
  -1&0 
\end{matrix}\right)=-\left(\begin{matrix}
  0&-1 \\
  1&0 
\end{matrix}\right)=:-J_0.
 \end{split}
\end{align*}
For any $L=e^{i\theta}L_1\in \op{U}(1,1)$, where $L_1=\left(\begin{matrix}
  a&b \\
  \b{b}&\b{a} 
\end{matrix}\right)\in \op{SU}(1,1)$,
 then 
\begin{align*}
\begin{split}
  L=\pm e^{i\theta}\exp(2\pi U^{-1}BU)=\pm e^{i\theta}U^{-1}\exp(2\pi B)U,
 \end{split}
\end{align*}
for some $B\in \mf{s}\mf{l}(2,\mb{R})$.
The canonical almost complex structure is given by 
\begin{align*}
\begin{split}
  \mbf{J}:&=\exp(-xU^{-1}BU)J\exp(xU^{-1}BU)=U^{-1}\exp(-xB)(UJU^{-1})\exp(xB)U\\
  &=-U^{-1}\exp(-xB)J_0\exp(xB) U=-U^{-1}\mbf{J}_0 U,
 \end{split}
\end{align*}
where $\mbf{J}_0:=\exp(-xB)J_0\exp(xB)$. Suppose $s\in A^0(\mb{R},E)^L=A^0(S^1,E_\phi)$ is an eigenvector of $A_{\mbf{J}}$ belongs to the eigenvalue $\sigma\in \mb{R}$, then 
$$\frac{d}{dx}s=-\sigma\mbf{J}s=U^{-1}\sigma \mbf{J}_0 Us=U^{-1}\sigma \exp(-xB)J_0\exp(xB) Us.$$
By solving the above ordinary differential equation, we obtain
\begin{align*}
\begin{split}
  s(x)=U^{-1}\exp(-xB)\exp(x(\sigma J_0+B))Us(0).
 \end{split}
\end{align*}
The $L$-equivariant condition $s(x+2\pi)=L^{-1}s(x)$ is equivalent to 
\begin{align*}
\exp(2\pi(\sigma J_0+B))Us(0)=
\begin{cases}
&e^{-i\theta}Us(0),\quad\,\,\,\, \text{ if } L=e^{i\theta}\exp(2\pi U^{-1}BU);\\
&-e^{-i\theta}Us(0),\quad\text{ if } L=-e^{i\theta}\exp(2\pi U^{-1}BU).\\ 	
\end{cases}
\end{align*}
If $s_1(x)=s_2(x)$ is an eigenvector of $A_{\mbf{J}}$, then
$$(\sigma_{s_1(x)}, s_1(0))=(\sigma_{s_2(x)}, s_2(0))\in\mb{R}\times(\mb{C}^{2}\backslash\{0\})$$
is a solution of the following  equation
\begin{align}\label{5.1}
\exp(2\pi(\sigma J_0+B))Ue=
\begin{cases}
&\exp(-{i\theta})Ue,\quad\,\,\,\, \text{ if } L=\exp({i\theta})\exp(2\pi U^{-1}BU);\\
&-\exp(-{i\theta})Ue,\quad\text{ if } L=-\exp({i\theta})\exp(2\pi U^{-1}BU).\\ 	
\end{cases}
\end{align}
For a subset $\mf{S}$ in the set of all solutions $\mb{R}\times(\mb{C}^{2}\backslash\{0\})$ of \eqref{5.1}, we call $\mf{S}$ is maximally $\mb{C}$-independent if for any eigenvalue $\sigma$, $\mf{S}_\sigma=\cup_{(\sigma,e)\in\mf{S}}\{e\}$ is  maximally  $\mb{C}$-linearly independent  in the set of all vectors associated with $\sigma$. 
	If $\mf{S}$ is a maximally $\mb{C}$-linearly independent subset in the set of all solutions of \eqref{5.1}, then the set of  eigenvalues (with multiplicities) of $A_{\mbf{J}}$  are given by the following disjoint union
	\begin{align*}
		\op{Eigen}(A_{\mbf{J}})=\bigsqcup_{(\sigma,e)\in\mf{S}}\{\sigma\}.
	\end{align*}
\begin{rem}\label{remdetspectrum}
If $\sigma$ is a solution of $\exp(2\pi(\sigma J_0+B))Ue=\exp(-i\theta)Ue$, which means that $2\pi k i$ is an eigenvalue of $2\pi(\sigma J_0+B)+i\theta I_2$, $k\in\mb{Z}$, which is equivalent to 
\begin{align}\label{detspectrum}
\begin{split}
  \det(2\pi(\sigma J_0+B)+i\theta I_2-2\pi ikI_2)=0.
 \end{split}
\end{align}
Hence, if $L=\exp({i\theta})\exp(2\pi U^{-1}BU)$, the set of eigenvalues (with multiplicities) of ${A_{\mbf{J}}}$ is given by  solving the  equation \eqref{detspectrum}. Similar for the case $L=-\exp({i\theta})\exp(2\pi U^{-1}BU)$, we just need to solve $  \det(2\pi(\sigma J_0+B)+i(\theta+\pi) I_2-2\pi ikI_2)=0.$ 
Note that $B\in \mf{s}\mf{l}(2,\mb{R})$, so $\mr{Tr}(B)=0$. If moreover, $2\pi(\sigma J_0+B)$ is diagonalizable, then  
\begin{align*}
\begin{split}
  2\pi(\sigma J_0+B)+i(\theta+2k\pi) I_2=P^{-1}\begin{pmatrix}
  \lambda_1&0 \\
  0&0 
\end{pmatrix}P.
 \end{split}
\end{align*}
Hence $\lambda_1=\mr{Tr}( 2\pi(\sigma J_0+B)+i(\theta+2k\pi) I_2)=2i(\theta+2k\pi)$, and so 
\begin{align*}
\begin{split}
  \exp(2\pi(\sigma J_0+B)+i\theta I_2)=P^{-1}\begin{pmatrix}
  e^{2\theta i}&0 \\
  0&0 
\end{pmatrix}P.
 \end{split}
\end{align*}
Thus, the $\sigma$ has the multiplicity $2$ if $\theta=0,\pi$, and $\sigma$ has the multiplicity $1$ if $\theta\neq 0,\pi$.
\end{rem}

The following lemma is useful in the calculation of eta invariant. 
\begin{lemma}\label{lemma2}
If there exists a constant $c_0>0$ such that  $(ak^2+bk+c)^{1/2}\geq l+c_0 k$ for any $k\geq 1$, where $l>0, a>0$, $s\in (0,1/2)$ and $b,c\in \mb{R}$, 
then
	$$\lim_{s\to 0}\sum_{k=1}^{\infty}\left(\frac{1}{((ak^2+bk+c)^{1/2}-l)^s}-	\frac{1}{((ak^2+bk+c)^{1/2}+l)^s}\right)=\frac{2l}{\sqrt{a}}.$$
\end{lemma}
\begin{proof}
Denote
\begin{align*}
	F(k):&=\frac{1}{((ak^2+bk+c)^{1/2}-l)^s}-	\frac{1}{((ak^2+bk+c)^{1/2}+l)^s}\\
	&=ls\int_{-1}^1\frac{d\theta}{((ak^2+bk+c)^{1/2}-\theta l)^{s+1}}.
\end{align*}
Then 
\begin{align}\label{FZeta}
\begin{split}
	&\quad \sum_{k=1}^{\infty}F(k)
	=\sum_{k=1}^{\infty}ls\int^1_{-1}\frac{d\theta}{((ak^2+bk+c)^{1/2}-\theta l)^{s+1}}\\
	&=\sum_{k=1}^{\infty}\frac{ls}{k^{s+1}}\int^1_{-1}\left(\frac{k^{s+1}}{((ak^2+bk+c)^{1/2}-\theta l)^{s+1}}-(\frac{1}{\sqrt{a}})^{s+1}\right)d\theta+\sum_{k=1}^{\infty}\frac{2ls}{k^{s+1}}(\frac{1}{\sqrt{a}})^{s+1}.
	\end{split}
\end{align}
Define a continuous function $f(x)=((a+bx+cx^2)^{1/2}-\theta lx)^{-(s+1)}$, $x\in [0,1]$. By assumption, one has  $0< f(x)\leq c_0^{-(s+1)}$. Thus
\begin{align*}
|f'(x)| &=(s+1)f(x)^{\frac{s+2}{s+1}}	|\frac{1}{2}(a+bx+cx^2)^{-1/2}(b+2cx)-\theta l|\leq C(s+1),
\end{align*}
which follows that 
$|f(\frac{1}{k})-f(0)|\leq C(s+1)\frac{1}{k}$,
and so
$$\left|\frac{k^{s+1}}{((ak^2+bk+c)^{1/2}-\theta l)^{s+1}}-(\frac{1}{\sqrt{a}})^{s+1}\right|\leq C(s+1)\frac{1}{k},$$
which implies that the first term in RHS of \eqref{FZeta} vanishes since $\lim_{s\to 0} s(s+1)\zeta(s+2)=0$, where $\zeta(s)=\sum_{k=1}^{\infty}k^{-s}$ is the zeta function. 
Hence
\begin{align*}
\lim_{s\to 0}\sum_{k=1}^{\infty}F(k)=\lim_{s\to 0}\sum_{k=1}^{\infty}\frac{2ls}{k^{s+1}}(\frac{1}{\sqrt{a}})^{s+1}=2l\lim_{s\to 0}s\zeta(s+1)(\frac{1}{\sqrt{a}})^{s+1}=\frac{2l}{\sqrt{a}},
\end{align*}
where the last equality follows from the fact $\lim_{s\to 0}s\zeta(s+1)=1$.
\end{proof}

For the group $\op{SU}(1,1)$, $L_1\in \op{SU}(1,1)$ has the following normal form:
\begin{itemize}
\item[(1)] $L_1$ is hyperbolic. In this case,  $\lambda\not\in S^1$ and $\lambda\in \mb{R}$, $L_1$ has the form: $$
U^{-1}\left(\begin{array}{cc}
\lambda & 0 \\
0 & \frac{1}{\lambda}
\end{array}\right)U;
$$	
\item[(2)] $L_1$ is elliptic. In this case, the eigenvalue $\lambda\in S^1\backslash\{\pm 1\}$,  $L_1$ is given by
$$
U^{-1}R(\theta_1)U=U^{-1}\left(\begin{array}{cc}\cos\theta_1 & -\sin\theta_1 \\
\sin\theta_1 & \cos\theta_1
\end{array}\right)U,
$$	
where $\theta_1\in (0,\pi)\cup(\pi,2\pi)$.
\item[(3)] $L_1$ is parabolic. $L_1$ has the following form: $$
U^{-1}\left(\begin{array}{ll}
\lambda & \mu \\
0 & \lambda
\end{array}\right)U,
$$	
where $\mu\in \mb{R}$, $\lambda=\pm 1$.
\end{itemize}

Now we will calculate the eta invariant $\eta(A_{\mbf{J}})$ and rho invariant $\rro_\phi(S^1)$. For the bounded symmetric domain of type $\op{I}$, we have the following isomorphism 
\begin{align*}
\begin{split}
  \mbf{J}_{\mr{III},0}:\mr{D}^{\op{III}}_{1}=\op{D}^{\mr{I}}_{1,1}\to \mc{J}(\mb{R}^2,-J_0),\quad \mbf{J}_{\mr{III},0}(W)=-U\mbf{J}_{\mr{I}}(W)U^{-1}.
 \end{split}
\end{align*}
 From \eqref{equals of two J}, one has
 \begin{align*}
\begin{split}
  \wt{\mbf{J}}^*\alpha=\wt{\mbf{J}_0}^*\alpha,
 \end{split}
\end{align*}
which follows that 
\begin{align}\label{rho invariant 10}
\begin{split}
  \rro_\phi(S^1)=-\frac{1}{\pi}\int_{S^1}\wt{\mbf{J}}^*\alpha+\eta(A_{\mbf{J}})=-\frac{1}{\pi}\int_{S^1}\wt{\mbf{J}_0}^*\alpha+\eta(A_{\mbf{J}}).
 \end{split}
\end{align}

(1) $\lambda\not\in S^1$. For the case $\lambda>0$, we take  
$$B=\frac{1}{2\pi}\log |\lambda|\left(\begin{array}{cc}
1 & 0 \\
0 & -1
\end{array}\right),$$
such that $L_1=\exp(2\pi B)$, and $L=\exp({i\theta})U^{-1}L_1U$.
Then the set of all eigenvalues (with multiplicities) of $A_{\mbf{J}}$ is 
\begin{align*}
\begin{split}
  \op{Eigen}(A_{\mbf{J}})= \begin{cases}
 \left\{\pm \sqrt{(\frac{\theta}{2\pi}+k)^2+(\frac{1}{2\pi}\log|\lambda|)^2},k\in \mb{Z}\right\} 	&\theta\neq 0,\pi\\
 	\bigsqcup_{l=1}^2 \left\{\pm \sqrt{k^2+(\frac{1}{2\pi}\log|\lambda|)^2},k\in \mb{Z}_{>0}\right\} 	\cup\{\pm \frac{1}{2\pi}\log|\lambda|\}&\theta=0\\
 	\bigsqcup_{l=1}^2 \left\{\pm \sqrt{(\frac{1}{2}+k)^2+(\frac{1}{2\pi}\log|\lambda|)^2},k\in \mb{Z}_{\geq 0}\right\} &\theta=\pi. 
 \end{cases}
 \end{split}
\end{align*}
Since the set $ \op{Eigen}(A_{\mbf{J}})$ is symmetric, so $\eta(A_{\mbf{J}})=0$.

If $\lambda<0$, then $\exp(2\pi B)=-L_1$, and  the set of all eigenvalues (with multiplicities) of $A_{\mbf{J}}$ is 
\begin{align*}
\begin{split}
  \op{Eigen}(A_{\mbf{J}})= \begin{cases}
 \left\{\pm \sqrt{(\frac{\theta-\pi}{2\pi}+k)^2+(\frac{1}{2\pi}\log|\lambda|)^2},k\in \mb{Z}\right\} 	&\theta\neq 0,\pi\\
 	\bigsqcup_{l=1}^2 \left\{\pm \sqrt{(-\frac{1}{2}+k)^2+(\frac{1}{2\pi}\log|\lambda|)^2},k\in \mb{Z}_{>0}\right\} 	&\theta=0\\
 	\bigsqcup_{l=1}^2 \left\{\pm \sqrt{k^2+(\frac{1}{2\pi}\log|\lambda|)^2},k\in \mb{Z}_{>0}\right\} \cup\{\pm \frac{1}{2\pi}\log|\lambda|\} &\theta=\pi.
 \end{cases}
 \end{split}
\end{align*}
Since the set $  \op{Eigen}(A_{\mbf{J}})$ is also symmetric, so $\eta(A_{\mbf{J}})=0$.

From Remark \ref{rempotential}, then 
\begin{equation}\label{potentialalpha}
  \alpha=-d^c\log\det\op{Im}Z=-d^c\log\op{Im}Z=\frac{1}{2\op{Im}Z}(dZ+d\o{Z}).
\end{equation}
The almost complex structure $\mbf{J}_0(x)$ is given by  {
\begin{equation*}
  \mbf{J}_0(x)=\exp(-xB)J\exp(xB)=\left(\begin{matrix}
0 & -|\lambda|^{-\frac{x}{\pi}}\\
|\lambda|^{\frac{x}{\pi}} &0	
\end{matrix}
\right).
\end{equation*}}
Then
$$W\circ \mbf{J}_0(x)=(2+|\lambda|^{\frac{x}{\pi}}+|\lambda|^{-\frac{x}{\pi}})^{-1}(|\lambda|^{\frac{x}{\pi}}-|\lambda|^{-\frac{x}{\pi}}).$$
Thus
$$Z\circ\mbf{J}_0(x)=i(1-W\circ \mbf{J}_0(x))(1+W\circ \mbf{J}_0(x))^{-1}$$
is purely imaginary, which follows that 
\begin{align*}
\wt{\mbf{J}_0}^*\alpha=	\frac{1}{2\op{Im}Z}(d(Z\circ \mbf{J}(x))+\o{d(Z\circ \mbf{J}(x))})=0.
\end{align*}
 Hence 
 $$\rro_\phi(S^1)=-\frac{1}{\pi}\int_{S^1}\wt{\mbf{J}_0}^*\alpha+\eta(A_{\mbf{J}})=\eta(A_{\mbf{J}})=0.$$
\begin{rem}\label{remhyperbolic}
For any $L\in \op{Sp}(2n,\mb{R})$ with  the following matrix form
\begin{align*}
L=\pm\left(\begin{matrix}
\exp(2\pi B) & 0\\
0 & \exp(-2\pi B^\top)	
\end{matrix}
\right),	
\end{align*}
where $B\in \mf{s}\mf{p}(2n,\mb{R})$.
One can check that $Z\circ \mbf{J}_0(x)$ is also purely imaginary, and so $\int_{S^1}\wt{\mbf{J}}_0^*\alpha=0$, 
 where $\mbf{J}_0(x)=\exp(-xB)J\exp(xB)$ and $J$ is the standard complex structure. Moreover, the set of all eigenvalues (with multiplicities) of $A_{\mbf{J}_0}$ is symmetric, so $\eta(A_{\mbf{J}_0})=0$. Hence $\rro_\phi(S^1)=0$.
\end{rem}

(2) $\lambda\in S^1\backslash\{\pm 1\}$. In this case, $L_1=R(\theta_1)$ and $B=\frac{\theta_1}{2\pi}J_0$, 
$
	\exp(2\pi(\sigma J_0+B))=R{(2\pi\sigma+\theta_1)}.
$
Hence the solutions of $\exp(2\pi(\sigma J_0+B))Ue=\exp(-i\theta)Ue$ are given by 
$$\sigma=-\frac{\theta_1}{2\pi}+k\pm \frac{\theta}{2\pi},\quad k\in\mb{Z}.$$
the set of all eigenvalues (with multiplicities) of $A_{\mbf{J}}$ is 
\begin{align*}
\begin{split}
  \op{Eigen}(A_{\mbf{J}})= \begin{cases}
 \left\{-\frac{\theta_1}{2\pi}+\frac{\theta}{2\pi}+k,-\frac{\theta_1}{2\pi}-\frac{\theta}{2\pi}+k,k\in \mb{Z}\right\} 	&\theta\neq 0,\pi\\
 	\bigsqcup_{l=1}^2 \left\{-\frac{\theta_1}{2\pi}+\frac{\theta}{2\pi}+k,k\in \mb{Z}\right\} 	&\theta=0,\pi.
 \end{cases}
 \end{split}
\end{align*}
Using Lemma \ref{eta11}, the eta invariant can be given by 
\begin{align*}
\begin{split}
  \eta(A_{\mbf{J}})&=\lim_{s\to 0}\left[\op{sgn}(\theta-\theta_1)\left|\frac{\theta-\theta_1}{2\pi}\right|^{-s}+\sum_{k=1}^\infty\left(\frac{1}{\left|k+\frac{\theta-\theta_1}{2\pi}\right|^s}-\frac{1}{\left|k-\frac{\theta-\theta_1}{2\pi}\right|^s}\right)\right]\\
  &+\lim_{s\to 0}\left[\op{sgn}(2\pi-\theta-\theta_1)\left|\frac{2\pi-\theta-\theta_1}{2\pi}\right|^{-s}+\sum_{k=1}^\infty\left(\frac{1}{\left|k+\frac{2\pi-\theta-\theta_1}{2\pi}\right|^s}-\frac{1}{\left|k-\frac{2\pi-\theta-\theta_1}{2\pi}\right|^s}\right)\right]\\
  &=\op{sgn}(\theta-\theta_1)-\frac{\theta-\theta_1}{\pi}-\op{sgn}(\theta+\theta_1-2\pi)-2+\frac{\theta+\theta_1}{\pi}\\
  &=\op{sgn}(\theta-\theta_1)-\op{sgn}(\theta+\theta_1-2\pi)-2+\frac{2\theta_1}{\pi}.
 \end{split}
\end{align*}
where $\op{sgn}(\theta)$ denotes the \emph{signum} function ($\op{sgn}(0)=0$, $\op{sgn}(x)=x/|x|$ otherwise). Since $[J,B]=0$, so
$$\mbf{J}_0(x)=\exp(-xB)J\exp(xB)=J,$$
and $\wt{\mbf{J}_0}^*\alpha=0$. Hence  $\rro_\phi(S^1)=-\frac{1}{\pi}\int_{S^1}\wt{\mbf{J}_0}^*\alpha+\eta(A_{\mbf{J}})=\eta(A_{\mbf{J}})$.

(3) $\lambda=\pm1$. In this case, $L_1$ is given by 
$$
\left(\begin{array}{ll}
\lambda & \mu \\
0 & \lambda
\end{array}\right),
$$	
where $\mu\in \mb{R}$, and $L_1=\lambda \exp(2\pi B)$, $B$ is given by
$$
B=\left(\begin{array}{cc}
0 & \frac{1}{2 \pi} \frac{\mu}{\lambda} \\
0 & 0
\end{array}\right).
$$
Thus
$$
2\pi(\sigma J_0+B)=\left(\begin{array}{cc}
0& -2\pi\sigma+ \frac{\mu}{\lambda} \\
2\pi\sigma & 0
\end{array}\right).
$$
For $\lambda=1$, the solution of $\exp(2\pi(\sigma J_0+B))Ue= \exp(-i\theta)U e$ is given by
\begin{align*}
\begin{split}
  \sigma=\frac{\mu}{4\pi}\pm \sqrt{\frac{\mu^2}{16\pi^2}+\left(\frac{\theta}{2\pi}+k\right)^2}, k\in \mb{Z}
 \end{split}
\end{align*}
For $\mu\neq 0$, the  set of all eigenvalues (with multiplicities) of $A_{\mbf{J}}$ is 
\begin{align*}
\begin{split}
  \op{Eigen}(A_{\mbf{J}})= \begin{cases}
 \left\{\frac{\mu}{4\pi}\pm \sqrt{\frac{\mu^2}{16\pi^2}+\left(\frac{\theta}{2\pi}+k\right)^2},k\in \mb{Z}\right\} 	&\theta\neq 0,\pi\\
 	\bigsqcup_{l=1}^2 \left\{\frac{\mu}{4\pi}\pm \sqrt{\frac{\mu^2}{16\pi^2}+k^2},k\in \mb{Z}_{>0}\right\}\cup \{0,\frac{\mu}{2\pi}\}	&\theta=0\\
 	\bigsqcup_{l=1}^2 \left\{\frac{\mu}{4\pi}\pm \sqrt{\frac{\mu^2}{16\pi^2}+\left(\frac{1}{2}+k\right)^2},k\in \mb{Z}_{\geq 0}\right\}	&\theta=\pi.
 \end{cases}
 \end{split}
\end{align*}
Thus, if
 $\lambda=1$, $\mu> 0$ and $\theta\neq 0,\pi$, then
\begin{align*}
&\quad \eta_{A_{\mbf{J}}}(s)
=\left|\frac{\mu}{4\pi}+\frac{1}{2\pi}\sqrt{\frac{\mu^2}{4}+\theta^2}\right|^{-s}-\left|\frac{\mu}{4\pi}-\frac{1}{2\pi}\sqrt{\frac{\mu^2}{4}+\theta^2}\right|^{-s}
\\
&+\sum_{k=1}^{\infty}\left[\left(\frac{\mu}{4\pi}+\frac{1}{2\pi}\left(\frac{\mu^2}{4}+(2k\pi+\theta)^2\right)^{\frac{1}{2}}\right)^{-s}	-\left(-\frac{\mu}{4\pi}+\frac{1}{2\pi}\left(\frac{\mu^2}{4}+(2k\pi+\theta)^2\right)^{\frac{1}{2}}\right)^{-s}\right]\\
&+\sum_{k=1}^{\infty}\left[\left(\frac{\mu}{4\pi}+\frac{1}{2\pi}\left(\frac{\mu^2}{4}+(2k\pi-\theta)^2\right)^{\frac{1}{2}}\right)^{-s}	-\left(-\frac{\mu}{4\pi}+\frac{1}{2\pi}\left(\frac{\mu^2}{4}+(2k\pi-\theta)^2\right)^{\frac{1}{2}}\right)^{-s}\right].
\end{align*}
By Lemma \ref{lemma2}, the eta invariant is 
$\eta(A_{\mbf{J}})=-\frac{\mu}{\pi}$. For $\theta=0$, one has $\eta(A_{\mbf{J}})=1-\frac{\mu}{\pi}$. For $\theta=\pi$, one has $\eta(A_{\mbf{J}})=-\frac{\mu}{\pi}$. In one word, if $\lambda=1$ and $\mu>0$, then 
\begin{align*}
\begin{split}
  \eta(A_{\mbf{J}})=1-\op{sgn}(\theta)-\frac{\mu}{\pi}.
 \end{split}
\end{align*}
 Similarly, If  $\lambda=1$ and $\mu< 0$, then  
$$\eta(A_{\mbf{J}})=-1+\op{sgn}(\theta)-\frac{\mu}{\pi}.$$
Hence, for $\lambda=1$, then 
\begin{align*}
\begin{split}
  \eta(A_{\mbf{J}})=\op{sgn}(\mu)\left(1-\op{sgn}(\theta)-\frac{|\mu|}{\pi}\right).
 \end{split}
\end{align*}
For $\lambda=-1$ and $\mu\neq 0$, then the   set of eigenvalues (with multiplicities) of $A_{\mbf{J}}$ is 
\begin{align*}
\begin{split}
  \op{Eigen}(A_{\mbf{J}})= \begin{cases}
 \left\{\frac{-\mu}{4\pi}\pm \sqrt{\frac{\mu^2}{16\pi^2}+\left(\frac{\theta-\pi}{2\pi}+k\right)^2},k\in \mb{Z}\right\} 	&\theta\neq 0,\pi\\
 	\bigsqcup_{l=1}^2 \left\{\frac{-\mu}{4\pi}\pm \sqrt{\frac{\mu^2}{16\pi^2}+\left(-\frac{1}{2}+k\right)^2},k\in \mb{Z}_{\geq 0}\right\} 	&\theta=0\\
 	 	\bigsqcup_{l=1}^2 \left\{\frac{-\mu}{4\pi}\pm \sqrt{\frac{\mu^2}{16\pi^2}+k^2},k\in \mb{Z}_{>0}\right\}\cup\{0,-\frac{\mu}{2\pi}\} 	&\theta=\pi.
 \end{cases}
 \end{split}
\end{align*}
Then the eta invariant is given by
$$ \eta(A_{\mbf{J}})=\op{sgn}(-\mu)\left(1-|\op{sgn}(\theta-\pi)|-\frac{|\mu|}{\pi}\right).$$
For $\mu=0$, the set of all eigenvalues (with multiplicities) of $A_{\mbf{J}}$ is symmetric, so $\eta(A_{\mbf{J}})=0$.

In this case, the almost complex structure is 
\begin{align*}
	\mbf{J}_0(x)=\exp(-xB)J\exp(xB)=\left(\begin{matrix}
b &-b^2-1\\
1 &-b	
\end{matrix}
\right),
\end{align*}
where $b=-\frac{x}{2\pi}\frac{\mu}{\lambda}$. Then 
{$$ W\circ  \mbf{J}_0(x)=\frac{2ib-b^2}{4+b^2}$$}
$$Z\circ \mbf{J}_0(x)=i(2+bi)^{-1}(2+b^2-ib)=i+\frac{b^3+4b}{b^2+4}={ i+b}.$$
By \eqref{potentialalpha}, one has 
\begin{equation*}
 \wt{\mbf{J}_0}^* \alpha=db=-\frac{1}{2\pi}\frac{\mu}{\lambda}dx.
\end{equation*}
Hence 
$$\frac{1}{\pi}\int_{S^1}\wt{\mbf{J}_0}^*\alpha=-\frac{1}{\pi}\frac{\mu}{\lambda},$$
and the rho invariant is 
\begin{align*}
\begin{split}
  \rro_\phi(S^1)=-\frac{1}{\pi}\int_{S^1}\wt{\mbf{J}_0}^*\alpha+\eta(A_{\mbf{J}})=\frac{1}{\pi}\frac{\mu}{\lambda}+\eta(A_{\mbf{J}}).
 \end{split}
\end{align*}
Therefore,  with the convention $\theta\in [0,2\pi), \theta_1\in (0,\pi)\cup(\pi,2\pi)$,
\begin{center}
\begin{align}\label{etadim2}
\begin{tabular}{| c | c | c| }
\hline
     $\lambda,\mu$  &  $\eta(A_{\mbf{J}})$  & $\rro_\phi(S^1)$\\
\hline
   $\lambda\not\in S^1$   & $0$ & $0$\\
\hline
   $\lambda\in S^1\backslash\{\pm1\}$    & $\op{sgn}(\theta-\theta_1)-2+\frac{2\theta_1}{\pi}\atop -\op{sgn}(\theta+\theta_1-2\pi)$ &$\op{sgn}(\theta-\theta_1)-2+\frac{2\theta_1}{\pi}\atop -\op{sgn}(\theta+\theta_1-2\pi)$\\
\hline
    $\mu=0$  & $0$ & $0$\\
\hline
$\lambda=1, \mu>0$ & $1-\op{sgn}(\theta)-\frac{\mu}{\pi}$ & $1-\op{sgn}(\theta)$\\
\hline
 $\lambda=1, \mu<0$ & $-1+\op{sgn}(\theta)-\frac{\mu}{\pi}$ & $-1+\op{sgn}(\theta)$\\
\hline
$\lambda=-1,\mu>0$&$-1+|\op{sgn}(\theta-\pi)|+\frac{\mu}{\pi}$ & $-1+|\op{sgn}(\theta-\pi)|$\\
\hline
$\lambda=-1,\mu<0$&$1-|\op{sgn}(\theta-\pi)|+\frac{\mu}{\pi}$ & $1-|\op{sgn}(\theta-\pi)|$\\
\hline
\end{tabular}	
\end{align}
\end{center}
Comparing with (\ref{sign5}),  if $k_1$ and $k_2$ are unique integers such that
$$\alpha_1=\theta-\theta_1+ 2 k_1\pi\in [0, 2\pi),\ \alpha_2=\theta+\theta_1+2k_2\pi \in [0, 2\pi)$$ for an elliptic element $e^{i\theta}\left(\begin{matrix}
  e^{-i\theta_1}&0 \\
  0& e^{i\theta_1} 
\end{matrix}\right)$,
then one can check that
$$\eta(A_{\mbf{J}})=\op{sgn}(\theta-\theta_1)-\op{sgn}(\theta+\theta_1-2\pi)-2+\frac{2\theta_1}{\pi}=
 \op{sgn}(\alpha_1)(1-\frac{\alpha_1}{\pi})-\op{sgn}(\alpha_2)(1-\frac{\alpha_2}{\pi}),$$ which is consistent with (\ref{sign5}).

\begin{rem}\label{rem5.1}
If we consider the representation $\phi_0:\pi_1(S^1)\to \op{Sp}(2,\mb{R})=\op{SL}(2,\mb{R})$, then the associated eta invariant and rho invariant correspond to the case of $\theta=0$ in \eqref{etadim2}. By \eqref{-eta} and \eqref{-rho}, one has
\begin{align*}
\begin{split}
  \eta(A_{\mbf{J}_0})=-  \eta(A_{\mbf{J}}),\quad \rro_{\phi_0}(S^1)=- \rro_\phi(S^1).
 \end{split}
\end{align*}
Hence the eta invariant and rho invariant for the representation $\phi_0:\pi_1(S^1)\to \op{Sp}(2,\mb{R})$ are given by 
\begin{center}
\begin{align}\label{etasp}
\begin{tabular}{| c | c | c| }
\hline
     $\lambda,\mu$  &  $\eta(A_{\mbf{J_0}})$  & $\rro_{\phi_0}(S^1)$\\
\hline
   $\lambda\not\in S^1$   & $0$ & $0$\\
\hline
   $\lambda\in S^1\backslash\{\pm1\}$    & $2(1-\frac{\theta_1}{\pi})$&$2(1-\frac{\theta_1}{\pi})$\\
\hline
    $\mu=0$  & $0$ & $0$\\
\hline
$\lambda=1, \mu>0$ & $-1+\frac{\mu}{\pi}$ & $-1$\\
\hline
 $\lambda=1, \mu<0$ & $1+\frac{\mu}{\pi}$ & $1$\\
\hline
$\lambda=-1$&$-\frac{\mu}{\pi}$ & $0$\\
\hline
\end{tabular}	
\end{align}
\end{center}

\end{rem}

\subsection{Connection with the multiplicative Horn problem}
\label{Horn}

When the Hermitian form is positive definite, i.e. $q=0$, the signature makes sense, the Toledo invariant vanishes. Our Milnor-Wood type inequality (Theorem \ref{thm0.5}) reads
\begin{align}
\label{q=00}
|\rro_\phi(\partial \Sigma)|\leq p|\chi(\Sigma)|.
\end{align}
By additivity of the signature, it suffices to establish it in case $\Sigma$ is a triply punctured sphere. In this case, the representation $\phi$ is determined by two unitary matrices $A$ and $B$. The third boundary holonomy, $C$, satisfies $ABC=I_p$.
We check that inequality (\ref{q=00}), which in this case reads
\begin{align}
\label{q=0}
|\rro(A)+\rro(B)+\rro(C)|\leq p,
\end{align} follows from the solution of the multiplicative Horn problem, \cite{AW}.

\medskip

Let $A\in\mr{U}(p)$. Let $(e^{i\theta_j})_{j=1,\ldots,p}$ denote the eigenvalues of $A$, normalized so that $\theta_j\in[0,2\pi)$. Then
\begin{align*}
\rro(A)&=\sum_{j\,;\,\theta_j\not=0} 1-\frac{\theta_j}{\pi}\\
&=\sum_{j=0}^{p} \op{sgn}(\theta_j)(1-\frac{\theta_j}{\pi}).
\end{align*}

\subsubsection{Matrices in $\op{SU}(p)$}

Let us translate the problem into more traditional notation. 

The eigenvalues of a matrix $A\in \op{SU}(p)$ can be uniquely written $e^{2i\pi\lambda_1},\ldots, e^{2i\pi\lambda_p}$, where
\begin{itemize}
  \item $\lambda_1-\lambda_p\leq 1$,
  \item $\lambda_1\ge \cdots \ge \lambda_p$,
  \item $\sum_{j=1}^p \lambda_j=0$.
\end{itemize}
Note that $\lambda_j(A^{-1})=-\lambda_{p-j}(A)$.

In our notation $e^{i\theta_1},\ldots, e^{i\theta_p}$, 
\begin{itemize}
  \item all $\theta_j\in[0,2\pi)$,
  \item $\theta_1\ge \cdots \ge \theta_p$,
  \item $m(A)=\frac{1}{2\pi}\sum_{j=1}^p \theta_j$ is an integer between $0$ and $p-1$.
\end{itemize}
The correspondance is as follows:
\begin{align*}
\lambda_1&=\frac{\theta_{m+1}}{2\pi},\ldots,\lambda_{p-m}=\frac{\theta_{p}}{2\pi},\\
\lambda_{p-m+1}&=\frac{\theta_{1}}{2\pi}-1,\ldots,\lambda_{p}=\frac{\theta_{m}}{2\pi}-1.
\end{align*}
We note that $m(A)$ is the number of negative reals among the $\lambda_j$'s and $m(A^{-1})$ is the number of positive reals among the $\lambda_j$'s. In particular, 
$$
m(A^{-1})=p-m(A)-\nu(A),
$$
where $\nu(A)$ is the number of $\lambda_j$'s which are equal to $0$. {Note that the numbers of $\{\lambda_i=0\}$ and $\{\theta_i=0\}$ are equal.}

Denote 
\begin{align*}
\widehat{m}(A):=m(A)+\frac{1}{2}\nu(A).	
\end{align*}
Note that
\begin{align*}
\begin{split}
  m(A^{-1})=p-m(A)-\nu(A)=p-(p-m(A^{-1})-\nu(A^{-1}))-\nu(A)
 \end{split}
\end{align*}
from which it follows that
\begin{align*}
\begin{split}
  \nu(A)=\nu(A^{-1}).
 \end{split}
\end{align*}
Hence 
\begin{align}\label{eq2.1}
\widehat{m}(A)+\widehat{m}(A^{-1})=m(A)+m(A^{-1})+\nu(A)	=p.
\end{align}

In our notation,
$$
\sum_{j=1}^{p}{\mr{sgn}(\theta_j(A))}(1-\frac{\theta_j(A)}{\pi})=p-\nu(A)-2m(A)=p-2\widehat{m}(A).
$$
Inequality (\ref{q=0}) states that if $A,B,C\in SU(p)$ satisfy $ABC=I$, then
$$
|p-2\widehat{m}(A)+p-2\widehat{m}(B)+p-2\widehat{m}(C)|\leq p,
$$
or equivalently,
$$
p\leq \widehat{m}(A)+\widehat{m}(B)+\widehat{m}(C)\leq 2p.
$$
Since $C=(AB)^{-1}$ and \eqref{eq2.1}, so it is equivalent to
\begin{align}\label{eq1.2}
\begin{split}
  0\leq\widehat{m}(A)+\widehat{m}(B)-\widehat{m}(AB)\leq p
 \end{split}
\end{align}

\medskip

Since 
$$
\lambda_{p-m(A)+1}(A)+\lambda_{p-m(B)+1}(B)<0,
$$
If $m(A)+m(B)\geq p+1$, 
the left hand side of Agnihotri-Woodward's \cite[inequality (8)]{AW}, implies that 
$$
\lambda_{p-m(A)+1+p-m(B)+1-1}(AB)<0.
$$
This implies that 
$$
p\geq p-m(A)+1+p-m(B)+1-1\geq p-m(AB)+1,
$$
i.e.
\begin{align}\label{eq2.2}
\begin{split}
  m(A)+m(B)-m(AB)\leq p.
 \end{split}
\end{align}
If instead $m(A)+m(B)\leq p$, \eqref{eq2.2} also holds.

Hence
\begin{align}\label{eq2.4}
\begin{split}
	&\quad \left( m(A)+m(B)-m(AB)\right)+\left( \nu(A)+\nu(B)-\nu(AB)\right)\\
	&=p-m(A^{-1})+p-m(B^{-1})-(p-m((AB)^{-1}))\\
	&=p-\left(m(A^{-1})+m(B^{-1})-m((B^{-1}A^{-1})\right)\geq 0
	\end{split}
\end{align}
since $m(A^{-1})+m(B^{-1})-m(B^{-1}A^{-1})\leq p$ by \eqref{eq2.2}.

On the other hand, if $m(A^{-1})+m(B^{-1})\leq p-1$, we can  take $i=p-m(A)-(\nu(A)-1)$, $j=p-m(B)-(\nu(B)-1)$, then
\begin{align}
\begin{split}
  \lambda_{i+j-1}(AB)\leq \lambda_i(A)+\lambda_j(B)\leq 0.
 \end{split}
\end{align}
This implies that
\begin{align*}
\begin{split}
  p\geq p-m(A)-(\nu(A)-1)+p-m(B)-(\nu(B)-1)-1\geq p-m(AB)-\nu(AB)+1,
 \end{split}
\end{align*}
which is equivalent to 
\begin{align*}
\begin{split}
  m(A^{-1})+m(B^{-1})-m((AB)^{-1})\geq 0.
 \end{split}
\end{align*}
If $m(A^{-1})+m(B^{-1})\geq p$, then the above inequality holds obviously. 
Hence
\begin{align}\label{eq2.5}
\begin{split}
  m(A)+m(B)-m(AB)\geq 0.
 \end{split}
\end{align}

Adding \eqref{eq2.4} and \eqref{eq2.5}, one gets
\begin{align}\label{eq2.7}
\begin{split}
&\quad 2\left( \widehat{m}(A)+\widehat{m}(B)-\widehat{m}(AB)\right)\\
&=\left[\left( m(A)+m(B)-m(AB)\right)+\left( \nu(A)+\nu(B)-\nu(AB)\right)\right]\\
&\quad +\left( m(A)+m(B)-m(AB)\right)\\
&\geq 0. 
\end{split}
\end{align}
From \eqref{eq2.7}, we have
\begin{align*}
\begin{split}
  0&\leq  \widehat{m}(A^{-1})+\widehat{m}(B^{-1})-\widehat{m}((AB)^{-1})\\
  &=p-\widehat{m}(A)+p-\widehat{m}(B)-(p-\widehat{m}(AB))\\
  &=p-(\widehat{m}(A)+\widehat{m}(B)-\widehat{m}(AB)),
 \end{split}
\end{align*}
which implies that
\begin{align}\label{eq1.6}
\begin{split}
  \widehat{m}(A)+\widehat{m}(B)-\widehat{m}(AB)\leq p.
 \end{split}
\end{align}
We have established both sides, \eqref{eq1.6} and \eqref{eq2.7}, of inequality \eqref{eq1.2}.

\subsubsection{The general case: matrices in $\mathrm{U}(p)$}

We have just proven that
\begin{align}\label{eqeta}
\begin{split}
  |\rro(A)+\rro(B)+\rro(C)|\leq p
 \end{split}
\end{align}
for any $A, B, C\in \mr{SU}(p)$ with $ABC=I_p$.
Now we are aiming to prove \eqref{eqeta} for any $A,B,C\in\mr{U}(p)$. It suffices to prove that 
\begin{align}\label{eqeta1}
\begin{split}
  \rro(A)+\rro(B)+\rro(C)\leq p.
 \end{split}
\end{align}
 Indeed, note that $\rro(A^{-1})=-\rro(A)$, so 
$$\rro(A^{-1})+\rro(B^{-1})+\rro(C^{-1})=-(\rro(A)+\rro(B)+\rro(C))\geq -p,$$
which implies that 
$$\rro(A)+\rro(B)+\rro(C)\geq -p.$$

 Denote by $\theta(A)\in [0,2\pi)$ the number satisfying
 \begin{align*}
\begin{split}
  \det A=e^{-i\theta(A)}.
 \end{split}
\end{align*}
Similarly, we can define $\theta(B),\theta(C)\in [0,2\pi)$ for $B,C$. Since $ABC=I_p$, 
\begin{align*}
\begin{split}
 \theta(A)+\theta(B)+\theta(C)\in\{0,2\pi,4\pi\}. 
 \end{split}
\end{align*}

\begin{itemize}
  \item If $ \theta(A)+\theta(B)+\theta(C)=0$, then $A,B,C\in \mr{SU}(p)$, and \eqref{eqeta1} follows from \eqref{eqeta}.
  \item If $ \theta(A)+\theta(B)+\theta(C)=2\pi$ and $\theta(A)\theta(B)\theta(C)\neq 0$, we denote 
      \begin{align*}
\begin{split}
\wt{A}:=  \left(\begin{array}{c:c}
A & 0 \\
\hdashline 0 & e^{i \theta(A)}
\end{array}\right),\quad \wt{B}:=  \left(\begin{array}{c:c}
B & 0 \\
\hdashline 0 & e^{i \theta(B)}
\end{array}\right),\quad \wt{C}:=  \left(\begin{array}{c:c}
C & 0 \\
\hdashline 0 & e^{i \theta(C)}
\end{array}\right).
 \end{split}
\end{align*}
Then $\wt{A},\wt{B},\wt{C}\in \mr{SU}(p+1)$ and $\wt{A}\wt{B}\wt{C}=I_{p+1}$, so 
\begin{align*}
\begin{split}
  \rro(\wt{A})+\rro(\wt{B})+\rro(\wt{C})\leq p+1.
 \end{split}
\end{align*}
By the definitions  of $\wt{A},\wt{B},\wt{C}$, one has
\begin{align*}
\begin{split}
\rro(\wt{A})+\rro(\wt{B})+\rro(\wt{C})
 & =\rro(A)+\rro(B)+\rro(C)+\mr{sgn}(\theta(A))\\
 &\quad+\mr{sgn}(\theta(B))+\mr{sgn}(\theta(C))-\frac{\theta(A)+\theta(B)+\theta(C)}{\pi}\\
 &=\rro(A)+\rro(B)+\rro(C)+1,
 \end{split}
\end{align*}
which implies \eqref{eqeta1}.

\item If $ \theta(A)+\theta(B)+\theta(C)=2\pi$ and $\theta(A)\theta(B)\theta(C)= 0$, without loss of generality, we assume that $\theta(A)=0$, then $\theta(B)>0, \theta(C)>0$. In this case, we denote 
\begin{align*}
\begin{split}
  \wt{A}:=\left(\begin{array}{c:cc}
A & 0 & 0 \\
\hdashline 0 & e^{i \pi} & 0 \\
0 & 0 & e^{i \pi}
\end{array}\right),\wt{B}=\left(\begin{array}{c:cc}
B & 0 & 0 \\
\hdashline 0 & e^{i \frac{\theta(B)}{2}} & 0 \\
0 & 0 & e^{i \frac{\theta(B)}{2}}
\end{array}\right),\wt{C}=\left(\begin{array}{c:cc}
C & 0 & 0 \\
\hdashline 0 & e^{i \frac{\theta(C)}{2}} & 0 \\
0 & 0 & e^{i \frac{\theta(C)}{2}}
\end{array}\right).
 \end{split}
\end{align*}
Then $\wt{A},\wt{B},\wt{C}\in \mr{SU}(p+2)$ and $\wt{A}\wt{B}\wt{C}=I_{p+2}$, so 
\begin{align*}
\begin{split}
  \rro(\wt{A})+\rro(\wt{B})+\rro(\wt{C})\leq p+2.
 \end{split}
\end{align*}
The rho invariants satisfy
\begin{align*}
\begin{split}
\rro(\wt{A})+\rro(\wt{B})+\rro(\wt{C})
 & =\rro(A)+\rro(B)+\rro(C)+2(\mr{sgn}(\pi)+\mr{sgn}(\theta(B)/2)\\
 &\quad +\mr{sgn}(\theta(C)/2) -\frac{\pi+\theta(B)/2+\theta(C)/2}{\pi})\\
 &=\rro(A)+\rro(B)+\rro(C)+2,
 \end{split}
\end{align*}
which implies \eqref{eqeta1}.

\item If $ \theta(A)+\theta(B)+\theta(C)=4\pi$, then $\theta(A)\theta(B)\theta(C)\neq 0$, we denote 
\begin{align*}
\begin{split}
  \wt{A}&=\left(\begin{array}{c:cc}
A & 0 & 0 \\
\hdashline 0 & e^{i \frac{\theta(A)}{2}} & 0 \\
0 & 0 & e^{i \frac{\theta(A)}{2}}
\end{array}\right),\wt{B}=\left(\begin{array}{c:cc}
B & 0 & 0 \\
\hdashline 0 & e^{i \frac{\theta(B)}{2}} & 0 \\
0 & 0 & e^{i \frac{\theta(B)}{2}}
\end{array}\right),\\
\wt{C}&=\left(\begin{array}{c:cc}
C & 0 & 0 \\
\hdashline 0 & e^{i \frac{\theta(C)}{2}} & 0 \\
0 & 0 & e^{i \frac{\theta(C)}{2}}
\end{array}\right).
 \end{split}
\end{align*}
Then
\begin{align*}
\begin{split}
  \rro(A)+\rro(B)+\rro(C)+2=\rro(\wt{A})+\rro(\wt{B})+\rro(\wt{C})\leq p+2
 \end{split}
\end{align*}
which implies \eqref{eqeta1}.
\end{itemize}

Therefore, we have completed the proof of \eqref{eqeta1}.

\subsection{Nilpotent conjugacy classes in $\mathfrak{su}(p,q)$}\label{nilpotentsu}

We provide details of the classification of nilpotent conjugacy classes in $\mathfrak{su}(p,q)$. This is a special case of the classification of conjugacy classes in classical groups, due to N. Burgoyne and R. Cushman \cite{BurgoyneCushman}.

\begin{defn}
\label{height}
Say a nilpotent element of $\mathfrak{su}(E,\Omega)$ has \emph{height} $m$, if $N^{m+1}=0$ and $N^m\not=0$. Say $N$ is \emph{uniform} if all its Jordan blocks have the same height.
\end{defn}

\begin{lemma}[{see \cite[Proof of Prop. 4]{BurgoyneCushman}}]
\label{P1}
Let $N\in\mathfrak{su}(E,\Omega)$ be nilpotent of height $m$. Then $E$ admits a decomposition $E=Y\oplus Z$ in $N$-invariant orthogonal subspaces such that $N_{|Y}$ is uniform of height $m$ and $Z\subset \op{Ker}(N^m)$.
\end{lemma}

\begin{proof}
Let $(e_j)_{1\leq j\leq n}$ be a Jordan basis for $N$, i.e. for each $j$, $Ne_{j}=e_{j-1}$ or $0$.
Let $F=\op{span}(\{e_j\,;\,N^m e_j \not=0\})$. Then $F$ is a complement to $\op{Ker}(N^m)$. Set
$$
Y=\bigoplus_{j=0}^{m} N^j F.
$$
Then $Y$ is $N$-invariant and uniform of height $m$. Let us show that $\Omega$ is non-degenerate on $Y$. Assume by contradiction that there exists a nonzero $x\in Y\cap Y^\perp$. Then $\Omega(x,(iN)^{j} y)=0$ for all $y\in Y$ and $j\geq0$. Write $x=\sum_{j=0}^{m} (iN)^j f_j=(iN)^{j_0} f+x_1$ where $j_0=\min\{j\,;\,f_j\not=0\}$, $f=f_{j_0}\in F$ and $x_1\in N^{j_0+1}Y$. Then $N^{m-j_0} x_1=0$, so, for all $y\in Y$, $\Omega(x_1,(iN)^{m-j_0} y)=\pm\Omega((iN)^{m-j_0} x_1,y)=0$. In particular, 
\begin{align*}
\forall y\in F,\quad  \Omega((iN)^m f,y)
&=\pm\Omega((iN)^{j_0}f,(iN)^{m-j_0} y)\\
&=\Omega(x-x_1,(iN)^{m-j_0} y)=0.
\end{align*}
Since $F$ is a complement to $\op{Ker}(N^m)$, $\Omega(f,(iN)^m e)=\Omega((iN)^m f, e)=0$ for all $e\in E$, so $N^m f=0$, contradiction. We conclude that $Z=Y^\perp$ is a $N$-invariant complement to $Y$ in $E$. Note that $\op{Im}(N^m)=N^m(F+\op{Ker}(N^m))\subset N^m F\subset Y$. It follows that $Z=Y^\perp\subset \op{Im}(N^m)^\perp\subset\op{Ker}(N^m)$.
\end{proof}

\begin{lemma}[{see \cite[Prop.~2]{BurgoyneCushman}}]
\label{P2}
Let $N\in\mathfrak{su}(E,\Omega)$ be uniform of height $m$. For $j\in\mb N$, let $\tau_j$ denote the Hermitian form on $E$ defined by
$$
\tau_j(u,v)=\Omega((iN)^j u,v).
$$
Then there exists a complement $F$ of $NE$ in $E$, such that
$$
E=\bigoplus_{j=0}^{m} N^j F ,
$$
and all Hermitian forms ${\tau_j}_{|F}$ vanish except $\tau_{m}$ which is non-degenerate.
\end{lemma}

\begin{proof}
Since $N$ has height $m$, $\tau_j=0$ if $j>m$. Furthermore, for all $u,v\in E$,
$$
\tau_j((iN)u,v)=\tau_j(u,(iN)v)=\tau_{j+1}(u,v).
$$

Let us start with the complement $F$ to $\op{Ker}(N^m)$ introduced in Lemma \ref{P1}. Since $N$ is uniform, $E=Y=\bigoplus_{j=0}^{m} N^j F$. Since $\Omega$ is non-degenerate, the kernel of $\tau_m$ equals $\op{Ker}(N^m)$, so $\tau_m$ is non-degenerate on $F$. We shall inductively improve $F$ until all $\tau_{j}$ but $\tau_{m}$ vanish on $F$. 
 
Let $k$ be the smallest $j < m$ such that $\tau_{j}\not=0$ on $F$. Let us compute, for $u,v\in F$,
\begin{align*}
\tau_k(u+(iN)^{m-k}u,v+(iN)^{m-k}v)
&=\tau_k(u,v)+2\tau_k((iN)^{m-k}u,v)\\
&+\tau_k((iN)^{m-k}u,(iN)^{m-k}v)\\
&=\tau_k(u,v)+2\tau_{k+m-k}(u,v)+\tau_{k+2(m-k)}(u,v)\\
&=\tau_k(u,v)+2\tau_{m}(u,v),
\end{align*}
since $k+2(m-k)=m+m-k>m$. This does not suffice to kill $\tau_k$. A correction, provided by a $\tau_m$-symmetric linear map $\phi:F\to F$, exists. Indeed, we want to solve
\begin{align*}
\tau_k(u+(iN)^{m-k}\phi(u),v+(iN)^{m-k}\phi(v))
&=\tau_k(u,v)+\tau_{m}(\phi(u),v)+\tau_{m}(u,\phi(v))\\
&=\tau_k(u,v)+2\tau_{m}(\phi(u),v)=0.
\end{align*}
Since $\tau_m$ is non-degenerate on $F$, this equation uniquely determines a $\tau_m$-symmetric linear map $\phi:F\to F$. Then $\tau_k$ vanishes on $(Id+(iN)^{m-k}\circ\phi)(F)$. If $k<j<m$, for all $u,v\in F$,
\begin{align*}
&\tau_j(u+(iN)^{m-k}\phi(u),v+(iN)^{m-k}\phi(v))\\
&=\tau_j(u,v)+\tau_{j+m-k}(\phi(u),v)+\tau_{j+m-k}(u,\phi(v))+\tau_{j+2(m-k)}(\phi(u),\phi(v))\\
&=\tau_j(u,v),
\end{align*}
since $j+m-k>m$ and $j+2(m-k)>m$. So the vanishing of $\tau_j$, $j>k$, is preserved. Therefore one more of the Hermitian forms $\tau_j$ vanishes on the image $F'=(Id+(iN)^{m-k}\circ\phi)(F)$. The Hermitian form $\tau_{m}$ is non-degenerate on every complement to its kernel, so on $F'$. Let us show that the sum $\sum_{j=0}^m (iN)^j F'$ is direct. If $f'_0,\ldots,f'_m\in F'$ and $\sum_{j=0}^m (iN)^j f'_j=0$, write $f'_j=f_j+(iN)^{m-k}\phi(f_j)$ for some $f_j\in F$. Then
\begin{align*}
\sum_{j=0}^{m-k-1} (iN)^j f_j +\sum_{j=m-k}^{m} (iN)^j (f_j+\phi(f_{j-m+k}))=0.
\end{align*}
This implies that $f_0$ vanishes, and recursively that all $f_j$ vanish as well. Thus $E=\bigoplus_{j=0}^{m}N^j F'$. Therefore we can replace $F$ with $F'$, winning the vanishing of $\tau_k$.

After finitely many steps, we get $F$ such that $\tau_j=0$ on $F$ for all $0\leq j\leq m-1$ and $\tau_{m}$ is non-degenerate on $F$.
\end{proof}

\begin{cor}
\label{P3}
Let $N\in\mathfrak{su}(E,\Omega)$ be a single Jordan block. There exists a Jordan basis for $N$ in which the matrix of $\Omega$ is antidiagonal.
\end{cor}

\begin{proof}
In this case, $\op{dim}(E/NE)=1$ and $N$ is $n-1$-uniform. Lemma \ref{P2} provides us with a $1$-dimensional complement $F$ of $NE$. Pick a nonzero vector $e_n\in F$. Set $e_j=(iN)^{n-j}e_n$, $j=1,\ldots,n$. This is a Jordan basis for $N$, in which the entries of $\Omega$ are given by
$$
\omega_{j,k}=\Omega((iN)^{n-j}e_n,(iN)^{n-k}e_n)=\Omega((iN)^{2n-j-k}e_n,e_n)=\tau_{2n-j-k}(e_n,e_n),
$$
which vanish unless $2n-j-k=n-1$, i.e. $j+k=n+1$. 
\end{proof}

\begin{lemma}[{\cite[Proof of Prop.~3]{BurgoyneCushman}}]
\label{P4}
Let $N\in\mathfrak{su}(E,\Omega)$ be uniform of height $m$. Then $E$ admits a decomposition $E=\bigoplus_j Y_j$ in $N$-invariant pairwise orthogonal subspaces such that each $(Y_j,N_{|Y_j})$ is a single Jordan block of height $m$.
\end{lemma}

\begin{proof}
According to Lemma \ref{P2}, there exists a complement $F$ of $NE$ in $E$ such that $E=F\oplus\cdots \oplus N^m F$ and all $\tau_j$ but $\tau_{m}$ vanish on $F$. 

Let $u,v\in F$ be linearly independent vectors such that $\tau_{m}(u,v)=0$. Then for all $j\in\bm N$, $\Omega((iN)^j u,v)=\tau_{j}(u,v)=0$. So the cyclic subspaces $C(u)=\op{span}(u,Nu,\cdot,N^m u)$ and $C(v)$ generated by $u$ and $v$ are $\Omega$-orthogonal. 

Let $(u_1,\ldots,u_r)$ be a $\tau_m$-orthogonal basis of $F$. Since $N$ is uniform of height $m$, $\op{Ker}(N)=N^m E$, $N$ is a bijection $N^j F\to N^{j+1}F$ for all $j<m$. Therefore $N^j u_1,\ldots,N^j u_r$ is a basis of $N^j F$, and the whole collection 
$$
\{N^j u_k\,;\,j=0,\ldots,m,\,k=1,\ldots,r\}
$$
is a basis of $E$. This shows that $E=C(u_1)\oplus\cdots\oplus C(u_r)$, this is the needed $N$-invariant and $\Omega$-orthogonal decomposition in Jordan blocks. 
\end{proof}

The above lemmata complete the proof of Proposition \ref{BurgoyneCushman}.

\subsection{The Milnor-Wood inequality for the Toledo invariant}

In this subsection, we will give a geometric proof for the Milnor-Wood inequality of Toledo invariant for general Hermitian symmetric spaces.

First we mention the classification of isometries for a symmetric space $X$ of noncompact type with $G=\op{Iso}^0(X)$. Let $\ell(\phi)=\inf_{x\in X} d_X(x,\phi(x))$ for $\phi\in G$. We say $\phi$ is (\cite[1.9.1]{Eb} )
\begin{enumerate}
\item  axial if $\ell(\phi)>0$ and realized in $X$.
\item elliptic if $\ell(\phi)=0$ and realized in $X$, i.e., it has a fixed point in $X$.
\item parabolic if $\ell(\phi)$ is not realized in $X$.
\end{enumerate}

If $L$ is parabolic, it has a fixed point at $X(\infty)$ (\cite[Prop. 4.1.1]{Eb} ). Hence it stabilizes a horosphere $H$ based at a fixed point of $L$. 
If $L\in G$ is parabolic, then it is an element of a horospherical subgroup $N_x$ for some point $x\in X(\infty)$, see \cite[Prop. 2.19.18 (5)]{Eb} . But $N_x$ has a property that  for $g\in N_x$
 $$\lim_{t\ra \infty} e^{-tX} g e^{tX}=id,$$ where $X\in \mathfrak p$ is a unit vector whose infinite end point is $x$ in the Cartan decomposition $\mathfrak g=\mathfrak t\oplus\mathfrak p$ at $p\in X$. This implies that for $g\in N_x$
 \begin{eqnarray}\label{contraction}
 \lim_{t\ra\infty} d(e^{tX} p, g e^{tX}p)=0,
 \end{eqnarray} which means that any two geodesic rays starting from $p$ and $gp$ pointing forwards $x$, get closer exponentially fast.

Let $\Sigma$ be a surface with $q$-boundary components and of genus $g$ of negative Euler number.
Considering a boundary component as a puncture, one can find an ideal triangulation $\triangle$ of $\Sigma$, which is just a maximal collection of disjoint essential arcs that are pairwise non-homotopic,
whose vertices are at punctures. If there are $F$ ideal triangles in $\triangle$, there
there are $\frac{3F}{2}$ edges since each edge is shared by the adjacent triangles.
Here by taking a  triangulation carefully, we  can assume that two adjacent triangles are distinct. Hence by the definition of the Euler number
$$F-\frac{3F}{2}=2-2g-q$$ where $q$ is the number of punctures. Hence there are $-4+4g+2q$ ideal triangles in $\triangle$.
Now  considering the {punctures} as boundaries, these ideal triangles wrap around each boundary component infinitely many times, and still denote this triangulation by $\triangle$.
\begin{figure}[hbt]
\begin{center}
\centerline{{\includegraphics{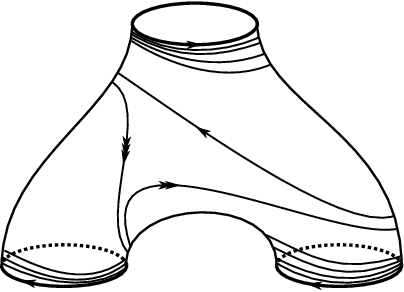}}}
\end{center}
\caption{Ideal triangulation of a pair of pants consisting of two ideal triangles }\label{infinite extreme point}
\end{figure}
Given a representation $\phi:\pi_1(\Sigma)\ra G$, 
the Toledo invariant for the associated Hermitian symmetric space $\mathscr X$ with a  K\"ahler form $\omega$  is given by the equation (\ref{Toledo invariant compact form}), i.e.,
$$  \op{T}(\Sigma,\phi)=\frac{1}{2\pi}\int_\Sigma\left(f^*\omega-\sum_{i=1}^qd(\chi_i f^*\alpha_i)\right)=\frac{1}{2\pi}\int_\Sigma f^*\omega-\frac{1}{2\pi}\sum_{i=1}^q\int_{c_i} f^*\alpha_i,$$  for any $\phi$-equivariant map $f$.

Since we have an identification between the symmetric space $\mathscr{X}$ and the compatible almost complex structures $\mathcal J(E, \Omega)$, for an elliptic boundary we can choose  a $\phi$-equivariant map $\wt{{\bf J}}:\widetilde\Sigma\ra \mathscr{X}\cong\mathcal J(E, \Omega)$ induced by an equivariant complex structure whose restriction to the elliptic boundary  is a constant complex structure $J$ fixed by $\phi(c_i)$ along the boundary, such that $J^*=0$ around the boundary. Hence $\int_{c_i} J^*\alpha_i=0$ for an elliptic boundary $c_i$.

Now we want to prove that the same holds for axial boundary.
Any axial isometry has an invariant real axis $l$ in $\mathscr{X}$, and $\alpha_i$ invariant under this axial isometry has a property that  $\int_l \alpha_i=0$ by Remark \ref{zero}. This shows that $\int_{c_i} f^*\alpha_i=\int_l \alpha_i=0$

Find a $\phi$-equivariant map $\wt{{\bf J}}:\widetilde\Sigma\ra \mathscr{X}$ induced by an equivariant complex structure.  Now we can straighten
$\wt{\bf J}$ such that each ideal triangle $\sigma$ in $\triangle$ is mapped to an ideal geodesic triangle in $D$. Now  a new
 surface $\Sigma'=\Sigma\cup \cup C_i$ is obtained from $\Sigma$ by attaching  a cone $C_i$ to each parabolic boundary $c_i$. One can extend $\op{Str}(\wt{\bf J})$ in an obvious way on each cone to the corresponding horoball neighborhood.
 In more details, 
$$\wt{\mbf{J}}:S^1\times [0,\infty)\ra  H,$$ where $H$ is a horoball  and $\wt{\mathbf{J}}$ maps each geodesic $x\times[0,\infty)$ to an arc-length parametrized geodesic from $\wt{\mathbf{J}}(x,0)$ to the base point of the horoball. 
Note here that the almost complex structure along the  boundary $c_i$ is an orbit
of $J$ under the one-parameter family generated by parabolic element $\phi(c_i)$, hence
the image under $\wt{\mathbf{J}}$ is also the orbit of the one-parameter family generated by parabolic element $\phi(c_i)$. This forces that $\wt{\mathbf{J}}(x \times [0,\infty))$ and $\wt{\mathbf{J}}(y\times [0,\infty))$ get closer exponentially fast as $t\rightarrow \infty$ for any two points $x$ and $y$ on $c_i$.

In more details,
$|\wt{\mathbf{J}}_*({\frac{\partial}{\partial x}}|_{(x,t)})|=e^{-t}|\wt{\mathbf{J}}_*({\frac{\partial}{\partial x}}|_{(x,0)})|$, and $|\wt{\mathbf{J}}_*({\frac{\partial}{\partial t}}|_{(x,t)})|=|\frac{\partial}{\partial t}|=1$, which makes 
$$\left|{\wt{\mathbf{J}}}^*\omega(\frac{\partial}{\partial x}|_{(x,t)},\frac{\partial}{\partial t})\right|=e^{-t}\left|\omega(\wt{\mathbf{J}}_*({\frac{\partial}{\partial x}}|_{(x,0)}),\frac{\partial}{\partial t})\right|\leq e^{-t}C$$ for some universal constant $C$.
Hence
\begin{eqnarray}\label{fast}
\int_{C_i}|{\wt{\mathbf{J}}}^*\omega|\leq \int_{0}^{2\pi}\int_{0}^\infty C e^{-t}<\infty,\end{eqnarray} which makes
 $\wt{\mbf{J}}^*(d\alpha_i)$ a $L^1$ form on  $C_i$. 
 
 If we decompose $C_i$ into two parts $C_i^1=S^1\times[0, t]\cup C_i^2=S^1\times [t,\infty)$, then $\wt{\mbf{J}}^*(d\alpha_i)$ being a $L^1$ form on  $C_i$ implies that
 $\int_{C_i^2} \wt{\mbf{J}}^*(d\alpha_i)\ra 0$ as $t\ra\infty$. Furthermore by noting that
 $|\wt{\mathbf{J}}_*({\frac{\partial}{\partial x}}|_{(x,t)})|=e^{-t}|\wt{\mathbf{J}}_*({\frac{\partial}{\partial x}}|_{(x,0)})|$, $\int_{S^1\times \{t\}} \wt{\mbf{J}}^*\alpha_i=\int_{S^1\times \{0\}} e^{-t}\wt{\mbf{J}}^*\alpha_i$,
   the ordinary Stokes' lemma holds to get for each parabolic boundary $c_i$
 $$\int_\Sigma d(\chi_i {Str(\wt{\bf J})}^*\alpha_i)=\int_{c_i} {\mr{Str}(\wt{\bf J})}^*\alpha_i=-\int_{C_i} d(\chi_i {\mr{Str}(\wt{\bf J})}^*\alpha_i).$$
This implies that
$$\op{T}(\Sigma,\phi)=\frac{1}{2\pi}\int_{\Sigma'} {\mr{Str}(\wt{\mbf{J}})}^*\omega.$$
Also note that we can deform $\triangle$ so that the triangles wrapping around the parabolic boundary $c_i$ can be straightened to the cone point of $C_i$ to include
$C_i$ in $\cup_{\sigma\in\triangle}\sigma$. We still denote the deformed triangulation by $\triangle$.
Since $\Sigma'=\sum_{\sigma_i\in\triangle} \sigma_i$,
 using $|\int_{\sigma_i} \mr{Str}(f)^* \omega|=| \int_{\mr{Str}(f)(\sigma_i)} \omega|\leq 2\pi (\frac{\mathrm{rank}(\mathscr{X})}{2}) $, which follows from the fact that the Gromov norm of $\kappa^b_G\in \mr{H}^2_{c,b}(G,\br)$ is $\frac{\mathrm{rank}(\mathscr{X})}{2}$ \cite{Cl}, we get the Milnor-Wood inequality
\begin{align*}
\begin{split}
  |\op{T}(\Sigma,\phi)|=&\left|\frac{1}{2\pi}\int_{\Sigma'} \mr{Str}(f)^*\omega\right|\\
  &\leq \frac{1}{2\pi}(-4+4g+2q)2\pi \left(\frac{\mathrm{rank}(\mathscr{X})}{2}\right)\\
  &=\mathrm{rank}(\mathscr{X}) |\chi(\Sigma)|.
 \end{split}
\end{align*}

\begin{rem}
First note that by \cite[Corollary 9.4]{Dj}, for any $a\in \op{Sp}(2n,\br)$, $a^2=\exp{X}$ for some $X\in \mathfrak g$. 

Let $\pi_1(\Sigma)=\langle c_1, c_2,\cdots, c_{q-1}, a_i, b_i \rangle$ where $c_i$ are boundary components, and $c_q=\Pi c_i \Pi [a_i,b_i]$. Take an index 2 subgroup $\Gamma< \pi_1(\Sigma)$
containing $c_i^2,\ i=1,\cdots, q-1$. Let $\Sigma_1=\widetilde \Sigma/\Gamma$ and $p:\Sigma_1\ra\Sigma$ be a covering map of degree 2.  Then for each boundary $c_i\ i=1,\cdots, q-1$, there exists a corresponding boundary $B_i$ of $\Sigma_1$ such that $p(B_i)=c_i^2$.  Now it is possible that
$p^{-1}(c_q)$ might be disjoint union of two circles $B_q, B_{q+1}$ which map to $c_q$ homeomorphically.  In this case, take a double cover $\Sigma_2$ of $\Sigma_1$, which corresponds to an index two subgroup of $\pi_1(\Sigma_1)$ containing $B_q^2, B_{q+1}^2$.
Then there exist two boundary components of $\Sigma_2$, $C_q,C_{q+1}$ which project to $B_q^2,
B_{q+1}^2$. Then the covering map $f$ from $\Sigma_2$ to $\Sigma$ has the property that each boundary component of $\Sigma_2$ projects to $c_i^{2k}$ for some $i=1,\cdots,q$.

Now for  this 4-fold covering map $f:\Sigma_2\ra\Sigma$, 
consider the induced representation $\phi_2=\phi\circ f_*:\pi_1(\Sigma_2)\ra \op{Sp}(2n,\br)$. Then for any boundary component $b$ of $\Sigma_2$, $\phi_2(b)=\phi(c_i^{2k})$ for some $c_i$, hence $\phi_2(b)=\exp(2\pi B)$ for some $B$ in the Lie algebra of $\op{Sp}(2n,\br)$. This allows us to define an almost complex structure
$\mbf{J}(x)=\exp(-xB)J\exp(xB)$ along any boundary of $\Sigma_2$. Now
if prove the Milnor-Wood inequality for $\phi_2$, then we prove the Milnor-Wood inequality for $\phi$ since both Toledo invariant and the Euler number for $\Sigma_2$ is $4$ times those of $\Sigma$. Hence we may assume that $\phi(c_i)=\exp(2\pi B_i)$ for some $B_i$.
The same construction works for other Hermitian Lie groups.
\end{rem}

\end{document}